\documentclass{acta_2024}
\usepackage{amssymb}
\usepackage{amsmath}

\usepackage{booktabs}

\def\R{\mathbb{R}}
\def\C{\mathbb{C}}
\def\F{\mathbb{F}}
\DeclareMathOperator{\mvec}{vec}
\DeclareMathOperator{\cvec}{cvec}
\DeclareMathOperator{\diag}{diag}
\DeclareMathOperator{\diagm}{diagm}

\DeclareMathOperator{\spann}{span}
\DeclareMathOperator{\traccia}{tr}
\DeclareMathOperator{\rank}{rank}
\DeclareMathOperator{\range}{range}

\usepackage[longnamesfirst]{natbib} \usepackage{har2nat}
\setcitestyle{comma,aysep={}} 
\shortcites{} 

\usepackage{newtxtext} 
\usepackage{newtxmath} 
\DeclareSymbolFont{CMlargesymbols}{OMX}{cmex}{m}{n} 
\DeclareMathDelimiter{(}{\mathopen} {operators}{"28}{CMlargesymbols}{"00}
\DeclareMathDelimiter{)}{\mathclose}{operators}{"29}{CMlargesymbols}{"01}
\DeclareMathAlphabet\mathcal{OMS}{cmsy}{m}{n} 
\SetMathAlphabet\mathcal{bold}{OMS}{cmsy}{b}{n} 
\pagerange{\pageref{firstpage}--\pageref{lastpage}}
\doi{XXXXXXXX}  
\usepackage[twoside,
	paperheight=247mm,
	paperwidth=174mm,
	marginratio={3:5,2:3},
	left=18mm,
	top=20mm]{geometry}
\usepackage{amssymb,mathtools}
\numberwithin{figure}{section}
\numberwithin{table}{section}
\usepackage[cmyk]{xcolor}
\usepackage{graphics,graphicx}
\usepackage[hang]{subfigure}
\usepackage[UKenglish]{babel}
\usepackage{enumerate} 
\newcommand{\ignore}[1]{}
\allowdisplaybreaks

\DeclarePairedDelimiter{\norm}{\lVert}{\rVert}
\DeclarePairedDelimiter{\abs}{\lvert}{\rvert}
\DeclarePairedDelimiter{\scalar}{\langle}{\rangle}
\newcommand{\subjectto}{\ensuremath{\ \text{s.t.}\ }}
\DeclareMathOperator{\grad}{grad}

\definecolor{brilliantrose}{rgb}{1.0, 0.33, 0.64}
\definecolor{myviolet}{rgb}{0.21, 0.0, 0.85}
\definecolor{amethyst}{rgb}{0.6, 0.4, 0.8}
\definecolor{carrotorange}{rgb}{0.93, 0.57, 0.13}
\definecolor{cutepink}{rgb}{1.0, 0.2, 0.6}

\definecolor{psA}{HTML}{2166AC}
\definecolor{psB}{HTML}{1B9E77}
\definecolor{psC}{HTML}{E6820E}
\definecolor{psD}{HTML}{D7263D}

\newtheorem{theorem}{Theorem}[section]
\newtheorem{lemma}[theorem]{Lemma}
\newtheorem{proposition}[theorem]{Proposition}
\newtheorem{corollary}[theorem]{Corollary}
\newtheorem{definition}[theorem]{Definition}
\newtheorem{remark}[theorem]{Remark}
\newtheorem{example}[theorem]{Example}

\usepackage{etoolbox}
\newcommand{\exampletombstone}{\rule{0.6em}{0.6em}}
\AtEndEnvironment{example}{\nobreak\hfill\exampletombstone}
\AtEndEnvironment{remark}{\nobreak\hfill\exampletombstone}

\usepackage{tikz,pgfplots}
\usetikzlibrary{angles,quotes,arrows.meta,calc,decorations.markings}

\title[Matrix Nearness Problems Revisited]{Matrix Nearness Problems Revisited}

\author[V. Noferini and F. Poloni]{Vanni Noferini\thanks{Corresponding author. Supported by a Research Council of Finland grant (decision number 370932).}\\
Department of Mathematics and Systems Analysis, Aalto University,\\ P.O. Box 11100, FI-00076, Aalto, Finland.\\
E-mail: {vanni.noferini@aalto.fi}
\and
Federico Poloni\thanks{Member of INDAM-GNCS. Supported by PRIN 2022ZK5ME7 funded by the Italian MUR.}\\
Department of Computer Science, University of Pisa,\\ Largo Pontecorvo 3, 56127 Pisa, Italy\\
E-mail: {federico.poloni@unipi.it}}

\begin{document}
\maketitle \label{firstpage}
 \vspace{2cm}
 
\begin{abstract}
Given a nonempty set of matrices $\mathcal{E} \subseteq \C^{m \times n}$ and a matrix $A \not\in\mathcal{E}$, a \emph{matrix nearness problem} (or matrix distance problem) consists in finding a matrix in $\mathcal{E}$ that is closest to $A$, in some norm, as well as the corresponding distance. This article is a survey of matrix nearness problems and of numerical algorithms for their solution. We pay special attention to modern related topics, such as Riemannian optimization and pseudospectra, that were not discussed in depth in previous reviews about this topic. Besides collecting known facts, we also prove some foundational results that, at least in the present form, we could not locate in the literature.
\end{abstract}

\tableofcontents

\section{Introduction}

A matrix nearness problem can be described very simply: We are given a matrix $A$ and a nonempty closed set $\mathcal{E}$, often described as the locus of a property which is not enjoyed by $A$. Can we find a matrix $X_\star \in \mathcal{E}$ nearest to $A$? If not, can we at least compute, or bound effectively, the distance from $\mathcal{E}$ to $A$? 

Such problems arise naturally across a range of fields. In some matrix nearness problems \citep[e.g.,][]{Higham2002}, the input matrix $A$ is supposed to have the desired property that characterizes the set $\mathcal{E}$, but in reality it does not because errors of various nature (modelling, measurement, numerical approximation\dots) have been made in forming $A$, and one seeks to correct these errors. On other occasions \citep[e.g.,][]{VanLoan1985} the property is on the contrary unwanted: The target set $\mathcal{E}$ contains matrices
associated with ill-posedness of a class of problems (linear systems, generalized eigenvalue problems, differential equations...), and thus the distance from it can be used as a proxy for the robustness of the instance of the problem associated with the input $A$. This interpretation connects matrix nearness problems with
the study of conditioning via the distance to ill-posedness
\cite{BurgisserCucker2013,Demmel87, Rump99}. This viewpoint is not completely uncontroversial; \citet{Trefethen2007note} discusses further subtleties, focusing in particular on structured perturbations.

There are many concrete applications of matrix nearness problems. We collect but a few illustrative examples in the non-exhaustive list below.
\begin{itemize}
    \item Control theory: nearest stable matrix \cite{Higham1989, Gillis2017, Guglielmi2018, Noferini2021a} or nearest unstable matrix \cite{VanLoan1985, Byers1988a, Kressner2006, Oracle, GuglielmiLubichBook};
    \item Finance: nearest correlation matrix \cite{Higham2002};
    \item Numerical algebraic geometry: approximate GCD \cite{Bini2010, Zeng2011, Usevich2017, Oracle};
    \item Numerical linear algebra: low-rank approximation \cite{Schmidt1907,EY36,Mirsky60}, distance to defective matrices \cite{malyshev1999,Alam2011},
or topics closely related to nearness problems such as matrix completion \cite{Keshavan2010,Vandereycken2013};
    \item Quantum information theory: nearest unitary gate or density matrix \cite{Nielsen_Chuang_2010};
    \item Signal processing and machine learning: nearest unitary matrix \cite{Higham1986} or improving the robustness of neural ODEs \cite{DGST25}.
\end{itemize}
From both the analytic and the numerical perspective, the difficulty of a matrix nearness problem is highly dependent on the choice of the target set $\mathcal{E}$, as well as on the choice of the bivariate function used to measure distances between matrices. The dependence on $\mathcal{E}$ is clear:
\begin{itemize}
    \item Certain problems admit a closed-form solution, say, the Hermitian matrix nearest to $A$ in the Frobenius distance is its Hermitian part $\frac12(A+A^*)$ \cite{FH55, Higham1989};
    \item Others can be solved by very robust algorithms, say, the nearest matrix of rank $\leq 1$ in the Frobenius distance can be easily computed via the Singular Value Decomposition (henceforth abbreviated SVD \cite{Schmidt1907, EY36, Mirsky60});
    \item Some more are provably difficult: For example, even deciding whether an affine subspace of real matrices contains a singular matrix is NP-hard \cite{ZZ23};
    \item And many more have been explored only marginally, if at all: For example, how do we find the matrix nearest to a given one and having a prescribed numerical range?
\end{itemize}
The choice of the distance can also dramatically influence the hardness of the problem. The Euclidean (Frobenius) distance combines two nice properties: It is unitarily invariant, and it leads to a convenient Hilbert space setting.  Arguably, these features make the Frobenius distance the most comfortable choice mathematically, but applications sometimes call for other choices, the most common of which are perhaps the distances induced by the spectral norm or other unitarily invariant norms.

We are not the first to review theoretical results and numerical algorithms related to matrix nearness problems. In particular, we would like to acknowledge one source of inspiration from the beginning. Starting very early on in his career, \citet{Higham1985, Higham1986, Higham1988} studied matrix nearness problems, and already in 1989 he wrote a highly influential survey \cite{Higham1989} that identified the most crucial and essential mathematical aspects for this class of problems, with his characteristic clarity and elegance. Our aim is not to replace Higham's review, but to complement it. To ensure self-containedness, we will inevitably have to overlap somewhat with the material in \citet{Higham1989}; however, we regard our paper as a natural sequel. We therefore focus primarily on developments that have taken place in the last four decades, and on advanced analytic and algorithmic aspects that were not treated in the earlier survey. 

We shall explore this subject along two dimensions: numerical methods for matrix nearness problems, and representative classes of problems to which these methods apply. Our treatment is (perhaps necessarily, to keep the length reasonable) selective, with
particular emphasis on the geometry of nearest-matrix projections, spectral constraints, and structured perturbations. A recurring theme that we have tried to emphasize throughout our review is how these viewpoints lead to useful reformulations and numerical algorithms. This editorial choice had some painful side effects, and we had to leave some important areas outside the scope of our detailed treatment; notably, among other topics, matrix completion \cite[Section~4.4]{WrightMa2022} and nonnegative matrix factorization \cite{GillisNMF}. 

The topic of matrix nearness problems ramifies into fairly many branches of mathematics, and to keep our treatment concise we will sometimes assume without justification basic results that can be found in textbooks such as (for example): \citet{HJtopics, Horn2013, Hogben2014} for matrix theory, \citet{TrefethenBau1997, GolubVLoan} for numerical linear algebra, \citet{Petersen2016, Lee2018} for differential geometry, \citet{Rockafellar1970, HL1993} for convex analysis, \citet{Rudin, RudinFunctional} for mathematical analysis, or \citet{Boumal2023,  Nocedal2006} for optimization.

As can be expected for a survey article, we have mostly collected results that have previously appeared elsewhere. Nonetheless, when a concise and instructive proof of a known result can be given, we have done so for the sake of self-containedness; our proofs are usually original and may sometimes depart significantly from the  classic ones. The article also includes results that, at least in the form given here, we have not been able to easily locate in the literature. These novelties include:
\label{intro:novelties}
\begin{itemize}
 \item Several characterizations of which nearness problems admit a common minimizer   for every distance induced by a unitarily invariant norm. Theorem~\ref{thm:dom} analyzes the general case and Theorem~\ref{thm:sunui} focuses on unitarily invariant sets; moreover, when the target set is a real affine subspace, we also obtain further equivalent conditions (Theorem~\ref{thm:projectorsnorm1} and Theorem~\ref{thm:bigthmminsubspaces}). Some auxiliary results that we developed to this goal, such as items (ii)-(iii) in Lemma~\ref{lem:subgradient} or
    Proposition~\ref{prop:genodd}, may also be new.
\item Theorem~\ref{thm:rank exactly n-1}, which shows that, under some differentiability assumptions of the chosen norm, the rank deficient matrix nearest to an input of full rank has rank drop precisely $1$; and Corollary~\ref{cor:allinD}, which under the same hypotheses on the norm analogously proves that any matrix with multiple eigenvalues and nearest to one with simple eigenvalues is defective.
  \item Theorem~\ref{thm:equivalentgeneral}, which describes global and local optimality criteria for reformulations of matrix nearness problems as optimization tasks over unitary matrices; on  the same theme, Proposition~\ref{prop:RGfQ} extends the corresponding gradient formula beyond what already exists in the literature.
    \item Theorem~\ref{thm:V82}, which is a description of the properties of optimal  perturbations when further linear constraints are added to a nearness problem, obtained via the theory of normal cones. This theorem also allows us to extend and generalize a recent result by \citet{GuglielmiLubichBook} for the structured  distance to singularity (Corollary~\ref{cor:V82 for singular}).
    \item Several examples that we developed to highlight specific phenomena or clarify the limitations of theoretical results; among others in this category, we may mention Example~\ref{ex:3}, the counterexample within Remark~\ref{rem:nolabel}, Example~\ref{ex:mustbeforall}, Example~\ref{ex:jordanblocks}, Example~\ref{ex:glarewrong}, and Example~\ref{ex:nearestdefective_diagonal}.
\end{itemize}

The structure of our review paper is the following. We first continue this introduction with an overview of basic properties of distance functions, nearest points, and normal cones, as well as a more formal definition of matrix nearness problems.  In Section~\ref{sec:uinorms} we discuss the choice of the distance, paying particular attention to  unitarily invariant norms and to the question of whether a matrix nearness problem admits a solution that simultaneously minimizes the distances induced by all these norms.  In Section~\ref{sec:directmethods} we discuss nearness problems that can be solved directly; our definition of ``direct methods" is broad and includes those that exploit a classic matrix decomposition of the input, e.g., the SVD\@. We then turn our attention to spectral nearness problems in Section~\ref{sec:spectralnearness}, discussing their connections with pseudospectra and with numerical methods based on singular value optimization, optimization over the manifold of unitary matrices, and gradient systems. In Section~\ref{sec:projections} we describe alternating projection methods for distances to intersections of sets. In the next two sections we study the additional complications (or, occasionally, simplifications) that arise when an extra constraint is added to a nearness problem, focusing on distance to singularity (Section~\ref{sec:structuredsingular}) and on spectral nearness (Section~\ref{sec:otherstruct}). In Section~\ref{sec:othernearness} we cover further matrix nearness problems beyond the previous classes, including distance to uncontrollability and distance to matrices with multiple eigenvalues. Finally, in Section~\ref{sec:beyond}, we briefly consider the extension of nearness problems from matrices to matrix pencils, polynomial matrices, matrix-valued functions, and tensors. We conclude the article by offering some final remarks in Section~\ref{sec:conc}.

\subsection{Distance from sets} \label{subsec:distsets}

We begin by giving a general overview of the analytic properties of distance functions to a set.  Consider a proper metric space $(\mathcal{M},d)$, i.e., a metric space such  that every closed ball in $\mathcal{M}$ is compact. Take a point $a \in \mathcal{M}$, and a set $\mathcal{E} \subseteq \mathcal{M}$. The distance from $a$ to $\mathcal{E}$ is defined as 
\begin{equation}\label{eq:distance}
    d(a,\mathcal{E}):=\inf_{x \in \mathcal{E}} d(a,x).
\end{equation}
Generally, the infimum may be not achieved at any point of $\mathcal{E}$, which is potentially problematic if one wants to study \eqref{eq:distance} algorithmically. Reassuringly, an effective medicine is to replace $\mathcal{E}$ by its closure $\overline{\mathcal{E}}$.
\begin{proposition}\label{prop:lauriwaswrong}
Let $\varnothing \neq \mathcal{E} \subseteq \mathcal{M}$ and $a \in \mathcal{M}$, where $\mathcal{M}$ is a proper metric space. Then:
    \begin{enumerate}
        \item There exists $x \in \overline{\mathcal{E}}$ such that $d(a,\mathcal{E})=d(a,\overline{\mathcal{E}})=d(a,x)$;
        \vskip1pt
        \item $\overline{\mathcal{E}}=\{a \in \mathcal{M} : d(a,\mathcal{E})=0\}$.
    \end{enumerate}
\end{proposition}
\begin{proof}
    \begin{enumerate}
        \item By definition of infimum, for all positive integers $n$ there exists a point $x_n \in \mathcal{E}$ such that $d(a,\mathcal{E})\leq d(a,x_n) < d(a,\mathcal{E})+\frac1n$. The sequence $(x_n)_n$ is contained in the closed ball $\{ z \in \mathcal{M} : d(a,z) \leq 1 + d(a,\mathcal{E}) \}$, which is compact by assumption. Hence, there exists a convergent subsequence $(x_{n_k})_k \rightarrow x \in \overline{\mathcal{E}}$ and $\displaystyle d(a,\mathcal{E}) = \lim_{k \rightarrow \infty} d(a,x_{n_k}) = d(a,x) \geq d(a,\overline{\mathcal{E}}).$
        Suppose for a contradiction that the last inequality is strict; by the first part of the argument, we can find $z \in \overline{\mathcal{E}}$ such that $d(a,z) = d(a,\overline{\mathcal{E}}) < d(a,\mathcal{E})=d(a,x)$. Take a sequence $(z_n)_n \subset \mathcal{E}$ with $d(z_n,z) < \frac1n < d(a,x)-d(a,z)$ for large enough $n$. Then, by the triangle inequality,
        \[ d(a,z_n) \leq d(a,z)+d(z,z_n) < d(a,x) \leq d(a,z_n). \]
        \item If $d(a,\mathcal{E})=0$, then  by the previous item $a \in \overline{\mathcal{E}}$. If $a \in \overline{\mathcal{E}}$, take a sequence $(x_n)_n \subset \mathcal{E}, \ x_n \rightarrow a$ to see that $0 \leq d(a,\mathcal{E}) \leq d(a,x_n) \rightarrow 0$.
    \end{enumerate}
\end{proof}
\begin{remark}\label{rem1}
    If the metric space $\mathcal{M}$ is not proper, it may not be true that there exists a point of $\overline{\mathcal{E}}$ where the distance is attained. For example, if $\mathcal{M}=\ell^1$ is the space of summable sequences, $x=0$ and $\mathcal{E}=\left\{\frac{k+2}{k+1} \ e_{k+1} \right\}_{k \in \mathbb{N}}$, which is closed, then $d(x,\mathcal{E})=1$ but the infimum is never attained.
\end{remark}
In solidarity with any reader who may have been understandably frightened by Remark~\ref{rem1}, let us immediately retreat to the comfort of finite dimensional vector spaces and consider the (proper) metric space $\mathcal{M}=\R^N$, endowed with the distance $d(x_1,x_2)=\| x_2-x_1\|_2$, the Euclidean length of the difference between the vectors $x_1$ and $x_2$. By Proposition~\ref{prop:lauriwaswrong}, the distance between a point and a set $\mathcal{E} \subseteq \R^N$ is equal to the distance between the point and the set's closure $\overline{\mathcal{E}}$, and moreover \emph{closed} sets always contain (at least) one element whose distance from an arbitrary point $a \in \R^N$ is equal to the distance of the set itself from $a$. For this reason, by replacing with no loss of generality a given set by its closure if necessary, one may focus on distances to nonempty closed sets. From now on, we shall do so throughout the paper.

 Given $a \in \R^N$ and a nonempty closed set $\mathcal{E} \subseteq \R^N$, by the previous observations we can both study the distance from $a$ to $\mathcal{E}$ and look for (at least) one point in $\mathcal{E}$ nearest to $a$:
\begin{equation}\label{eq:euclid}
  \mathrm{Compute} \  d(a,\mathcal{E}) = \min_{x \in \mathcal{E}} \|x-a\|_2, \qquad \mathrm{find} \ \mathrm{one} \ \mathrm{point} \ x_\star \in \mathrm{argmin}_{x \in \mathcal{E}} \| x-a\|_2.
\end{equation}

In the following, we study the differentiability of $d(a,\mathcal{E})$ as a function of the point $a$. For this (and for further purposes later on in the article), it is convenient to recall the definition and some properties of subgradients of a convex function.

\begin{definition}\cite{Rockafellar1970}
    Let $f(x): \R^N \rightarrow \R$ be a convex function. A vector $g \in \R^N$ is a \emph{subgradient} of $f$ at $x \in \R^N$ if, for every $z \in \R^N$, 
    \[  f(z)-f(x) \geq (z-x)^T g.  \]
\end{definition}

\begin{proposition}\cite{Rockafellar1970}
    Let $f(x): \R^N \rightarrow \R$ be a convex function. Then:
    \begin{enumerate}
        \item For every point $x \in \R^N$, the set of subgradients of $f$ at $x$ is not empty;
        \item The point $x \in \R^N$ is a global minimum of $f$ if and only if $0$ is a subgradient of $f$ at $x$;
        \item $f$ is differentiable at $x \in \R^N$ if and only if the set of subgradients of $f$ at $x$ has cardinality $1$, and in this case the Jacobian of $f$ at $x$ is the transpose of the (unique) subgradient of $f$ at $x$.
    \end{enumerate}
\end{proposition}
Since the function $z \mapsto z^2$ is increasing over the set of nonnegative real numbers, we can conveniently replace the distance by its square in the minimization problem \eqref{eq:euclid}. The advantage is that the gradient of $d(a,\mathcal{E})^2$, which is differentiable at almost every point $a \in \R^N$, has a very nice form. 

\begin{theorem}\label{thm:ae}
    Given a nonempty closed set $\mathcal{E} \subseteq \R^N$, consider the function $s_\mathcal{E}(a)=d(a,\mathcal{E})^2 = \min_{x \in \mathcal{E}} \| x-a \|_2^2$. Then, $s_\mathcal{E}$ is everywhere continuous and almost everywhere differentiable. In particular, $s_\mathcal{E}$ is differentiable at $a$ if and only if $a$ is such that the set of argument minima in \eqref{eq:euclid} consists of a single point $x_\star \in \R^N$; in that case, its gradient is $\nabla_a s_\mathcal{E}(a) = 2 (a-x_\star)$.
\end{theorem}
\begin{proof}
    Expanding $\|x-a\|_2^2=(x-a)^T(x-a)$, we have that $s_\mathcal{E}(a)=\|a\|_2^2-f(a)$, having defined
    \[ f(a):= \sup_{x \in \mathcal{E}} \ (2 a^T x - \|x\|_2^2).\]
    Hence, $f(a)=\|a\|_2^2-s_\mathcal{E}(a)$ is the supremum of a family of affine functions, therefore convex, therefore locally Lipschitz. By Rademacher's Theorem, it follows that $f(a)$ is everywhere continuous and almost everywhere differentiable. Since $\|a\|_2^2$ is a differentiable function, the same two properties hold for $s_\mathcal{E}(a)$.

    Fix now $a \in \R^N$ and observe that, for all $b \in \R^N$ and for all points $x$ where the minimum of \eqref{eq:euclid} is attained, we have 
    \[ f(b) \geq 2b^T x - \| x\|_2^2, \qquad f(a)=2a^Tx - \|x\|_2^2.\]
    Thus, $f(b) -f(a) \geq (b-a)^T (2x)$, implying that $2x$ is a subgradient of $f$ at $a$. Hence, if $a$ is such that there is more than one argument minimum of \eqref{eq:euclid}, then $f$ (and therefore $s_\mathcal{E}$) is not differentiable at $a$. Conversely, if $a$ is such that there is only one point $x_\star \in \mathcal{E}$ nearest to $a$, let $(a_n)_n$ be an arbitrary sequence such that $a_n \neq a$ for almost all $n$ and $a_n \to a$. Next, construct a sequence $(x_n)_n \subset \mathcal{E}$ in such a way that $x_n$ is a point nearest to $a_n$, for all $n$. Then, for large enough $n$,
    \[ \|x_n-a_n\|_2 \leq \|x_\star-a_n\|_2 \leq \|x_\star-a\|+1,\]
    and hence $(x_n)_n$ is bounded. Moreover, any convergent subsequence $(y_n)_n$ must converge to $x_\star$, because the limit $y \in \mathcal{E}$ must satisfy $\|y-a\|_2 \leq \|x_\star - a\|$ implying $y=x_\star$ by uniqueness. Suppose $x_n \not\to x_\star$, then there are $\varepsilon > 0$ and a subsequence $(x_{n_k})_k$ such that $\|x_{n_k}-x_\star\|_2 \geq \varepsilon$. Taking a convergent subsubsequence yields a contradiction, and hence $x_n \to x_\star$. Observe now that $2x_\star$ is a subgradient of $f$ at $a$ and $2x_n$ is a subgradient of $f$ at $a_n$. Hence,
\[ 2x_\star^T(a_n-a)\leq f(a_n)-f(a)\leq2x_n^T(a_n-a). \]
Therefore, using also the Cauchy-Schwarz inequality,
\[ 0\leq \frac{f(a_n)-f(a)-2x_\star^T(a_n-a)}{\|a_n-a\|_2} \leq \frac{2(x_n-x_\star)^T(a_n-a)}{\|a_n-a\|_2} \leq 2 \|x_n-x_\star\| \to 0,   \]
and hence $f$ is differentiable at $a$ and its gradient is $2 x_\star$. The differentiability and the gradient formula for $s_{\mathcal{E}}$ follow immediately.
\end{proof}
 As an immediate consequence of Theorem~\ref{thm:ae} we get Corollary~\ref{cor:erdos}, which was first proved by \citet{erdos} by more abstract arguments.
\begin{corollary}\label{cor:erdos}
    Given a nonempty closed set $\mathcal{E} \subseteq \R^N$, its \emph{medial axis} \[M(\mathcal{E}):=\{ a \in \R^N : \text{there is more than one point in $\mathcal{E}$ nearest to $a$}\}\] has Lebesgue measure $0$.
\end{corollary}

Famously, the even stronger statement $M(\mathcal{E})=\varnothing$ holds if and only if $\mathcal{E}$ is convex. In other words, if $\mathcal{E}$ is a convex set then every point $x$ has unique projection onto $\mathcal{E}$; whereas if it is not convex then exceptions exist but they are rare in the sense of Lebesgue.

Since \eqref{eq:euclid} is an optimization problem, it makes sense to also study the existence of the Hessian of $s_\mathcal{E}(x)$ in order to consider second order methods.  In the statement of Theorem~\ref{thm:alexandrov}, twice differentiable refers to the existence of a second order Taylor expansion. Moreover, in its proof and throughout the paper, $I$ denotes an identity matrix of appropriate size; when a possible ambiguity arises, we will specify its size by writing $I_n$.

\begin{theorem}\label{thm:alexandrov}
    Let $\mathcal{E} \subseteq \R^N$ be nonempty and closed. The function $s_\mathcal{E}(a)$ as defined in Theorem~\ref{thm:ae} is twice differentiable at almost every point $a$.
\end{theorem}
\begin{proof}
    Write $s_\mathcal{E}(a)=\|a\|_2^2 - f(a)$ as in the proof of Theorem~\ref{thm:ae}. Since $f(a)$ is convex, by Alexandrov's Theorem it has second derivatives almost everywhere. Whenever $f(a)$ is twice differentiable, so must be $s_\mathcal{E}(a)$ with Hessian $H_s = 2 I - H_f$, where $H_f$ denotes the Hessian of $f$.
\end{proof}

\begin{example}
    Generally, one cannot strengthen the statement of Theorem~\ref{thm:alexandrov} to conclude that the function $s_\mathcal{E}(a)$ is of class $C^2$ almost everywhere.   For a counterexample with $N=1$, let $\mathcal{E} \subsetneq \R$ be a Smith-Volterra-Cantor set \cite{fatcantor}. This is a closed subset of $[0,1]$ having empty interior, no isolated points, and positive Lebesgue measure. On each point $a$ in a connected component $(x_\ell,x_r)$ of $[0,1] \setminus \mathcal{E}$, it holds $s_\mathcal{E}(a) = \min \{ (a-x_\ell)^2, (a-x_r)^2\}$. Hence, the function is not twice differentiable at any midpoint of these intervals. While there are only countably many such midpoints, which does not contradict Theorem~\ref{thm:alexandrov} since a countable set has zero measure, they nevertheless  accumulate everywhere on the positive-measure set $\mathcal{E}$, excluding that $s_\mathcal{E}$ can be $C^2$ almost everywhere.
\end{example}

An explicit formula of the Hessian of $f(a)$, and hence of $s_\mathcal{E}(a)$, is elusive for general closed sets $\mathcal{E}$, but closed  form expressions are typically available when $\mathcal{E}$ is a manifold, making Theorem~\ref{thm:alexandrov} practically applicable. (See also \citet{Edelman1998} for the case of matrix optimization with orthogonality constraints, and \citet{Absil2013} for the more general topic of Riemannian Hessians in matrix optimization.) We record this in Theorem~\ref{thm:hessian}.

\begin{theorem}\label{thm:hessian}
    Let $\mathcal{E} \subseteq \R^N$ be an embedded Riemannian submanifold of real dimension $n$, and $a \not\in M(\mathcal{E})$. Denote by $x$ the unique projection of $a$ onto $\mathcal{E}$ and let  $s_\mathcal{E}$ be the squared distance function defined in Theorem~\ref{thm:ae}. Let $Q \in \R^{N \times n}$ be a matrix whose columns form an orthonormal basis of the tangent space of $\mathcal{E}$ at $x$, and let $W_u \in \R^{n \times n}$ be the matrix representing, in the basis $Q$, the Weingarten map \cite[Chapter 8]{Lee2018} in the direction of $u:=a-x$. 
    
    Assume in addition that $I-W_u$ is nonsingular. Then, $s_\mathcal{E}$ is twice differentiable at $a$, with Hessian
    \[ H_s(a) = 2 (I -Q(I- W_u)^{-1}Q^T).      \]
\end{theorem}

\begin{proof}
By the implicit function theorem, and since $I-W_u$ is nonsingular, for every point $b$ sufficiently close to $a$ the function mapping $b$ to the point on $\mathcal{E}$ nearest to $b$ is well defined and continuously differentiable (see the proof of Theorem~\ref{thm:ae} and \citet[Theorem~2]{LS21}).

    In view of Theorem~\ref{thm:ae} and defining $f(a)$ as in the proof of Theorem~\ref{thm:ae}, it suffices to prove that the Hessian of $\frac12 f(a)$ is $Q(I-W_{a-x})^{-1}Q^T$. Equivalently, we need to verify that such expression describes the Jacobian of the projection point $x$ with respect to $a$. Note first that a sufficiently small variation of $a$ in a direction normal to $\mathcal{E}$ at $x$ does not change its projection point; it thus suffices to consider perturbations in the tangent space $T_x \mathcal{E}$ to obtain a formula for $Q^T J_x(a) Q$, and then observe that $J_x(a)=QQ^T J_x(a) QQ^T$.
    
    By construction $u=a-x$ is normal to $\mathcal{E}$ at $x$. Differentiating with respect to $a$ yields $J_u(a)=I-J_x(a)$. On the other hand, by \cite[Proposition 8.4]{Lee2018}, for every tangential vector $v \in T_x \mathcal{E}$ it holds $\nabla_v u = - Q W_u (Q^T v)$. Hence, by the chain rule and noting that the normal component vanishes because $J_u(a)Q=Q-J_x(a)Q$ has columns in $T_x\mathcal E$, $J_u(a) Q = -QW_uQ^TJ_x(a)Q$. Since $Q^T Q=I$, we conclude that $Q^T J_x(a) Q  = (I-W_u)^{-1},$ yielding the statement.
\end{proof}

For example, for an affine subspace $W_u=0$ identically, and thus the formula in Theorem~\ref{thm:hessian} reduces to twice the orthogonal projector onto the normal space. More generally, manifolds are often given by constraints, that is, $\mathcal{E} = \{ z \in \R^N : h(z)=0   \}$ where $h: \R^N \mapsto \R^{N-n}$ is a submersion when restricted to $\mathcal{E}$, i.e., the Jacobian $J_h(x) \in \R^{(N-n)\times N}$ has full row rank for all $x \in \mathcal{E}$. It follows that the Gram matrix $G(x)=J_h(x) J_h(x)^T$ is everywhere invertible. In this case, by repeated application of the chain rule and having fixed an orthonormal basis $Q$ for $\ker J_h(x)$, we have that the matrix $W_u$ in Theorem~\ref{thm:hessian} can be expressed as
\begin{equation*}
    W_u = -Q^T \left( \sum_{i=1}^{N-n} c_i H_{h_i} \right)  Q , 
\end{equation*} 
where $H_{h_i} \in \R^{N \times N}$ is the Hessian matrix of the $i$-th component of $h(z)$, and the vector $c=G(x)^{-1} J_h(x) (a-x) \in \R^{N-n}$.

\begin{example}
Consider the manifold $\mathcal{E}$ of $2 \times 2$ real matrices (represented as vectors in $\R^4$ by vectorization, i.e., stacking their columns on top of each other) with singular values equal to $1$ and $0$. This is described by the constraints $h(z)=\begin{bmatrix}
    z_1^2 + z_2^2 + z_3^2 + z_4^2-1\\
    z_1 z_4 - z_2 z_3
\end{bmatrix}=0$. Take the point $a=\begin{bsmallmatrix}
    2\\
    9\\
    -6\\
    -2
\end{bsmallmatrix}$, whose unique projection is $x= \frac15 \begin{bsmallmatrix}
    2\\
    4\\
    -1\\
    -2
\end{bsmallmatrix}$, as can be easily verified (see  Theorem~\ref{thm:minimizerforunitinv} below) from an SVD of the matrix $A=\begin{bsmallmatrix}
    2 & -6\\
    9 & -2
\end{bsmallmatrix}$ so that $a=\mvec(A)$. 

The Jacobian matrix is $J_h(x) =  \frac15 \begin{bsmallmatrix}
    4&8&-2&-4\\
    -2&1&-4&2
\end{bsmallmatrix}$, and an orthonormal basis for its null space is $Q =  \frac{1}{5\sqrt{2}}\begin{bsmallmatrix}
    5 & -3\\
    0 & 4\\
    0 & 4\\
    5 & 3
\end{bsmallmatrix}$. Hence,
\[ W_u = -Q^T \left( \frac92 \cdot 2 I_4 + 5 \cdot \begin{bmatrix}
    0&0&0&1\\
    0&0&-1&0\\
    0&-1&0&0\\
    1&0&0&0
\end{bmatrix}   \right) Q = \begin{bmatrix}
    -14&0\\
    0 & -4
\end{bmatrix}.   \]
Hence, by Theorem~\ref{thm:hessian}, the Hessian of the square distance $s_\mathcal{E}$ at $a$ is the symmetric positive definite matrix
\[ H_s(a) = \frac{1}{375} \begin{bmatrix}
    698 & 36 & 36 & 2\\
    36 & 702 & -48 & -36\\
    36 & -48 & 702 & -36\\
    2 & -36 & -36 & 698
\end{bmatrix} \succ 0. \]
 \end{example}

\subsection{Matrix nearness problems}

Let $\F^{m \times n}$ be the vector space of $m \times n$ matrices with entries in the field $\F$, where $\F$ is either $\R$ or $\C$. We will endow it with the Frobenius norm 
\begin{equation}\label{eq:fro}
  \| A \|_F := \sqrt{\traccia(A^* A)} = \sqrt{\sum_{i=1}^m \sum_{j=1}^n |A_{ij}|^2 }. 
\end{equation}
and the associated distance $d(X,Y)=\|X-Y\|_F$. Throughout the paper, it is also convenient to denote $\nu:=\min\{m,n\}$. 

If $\F=\R$, the vectorization map $A \mapsto a=\mvec(A)$, where $\mvec(A)$ is the vector obtained by stacking the columns of $A$ on top of each other, induces a Hilbert space isomorphism $\R^{m \times n} \cong \R^{mn}$. Indeed, the scalar product $\langle A,B \rangle = \traccia(A^T B)$, which induces the Frobenius norm \eqref{eq:fro}, is mapped to the usual dot product $\langle a,b \rangle = a^T b$ of $\R^{mn}$. Hence, Theorem~\ref{thm:ae}, Corollary~\ref{cor:erdos}, Theorem~\ref{thm:alexandrov} and Theorem~\ref{thm:hessian} all apply to $\R^{m \times n}$. A similar machinery is useful when $\F=\C$, provided  that we view $\C^{m \times n}$ as a $2mn$-dimensional real vector space with inner product $\langle A,B \rangle = \Re (\traccia(A^*B))$; note that, for this choice of an inner product, the induced norm $\sqrt{\langle A,A \rangle}$ is still \eqref{eq:fro}, because $\traccia(A^*A) \in \R$. In this setting, the canonical Hilbert space isomorphism $\C^{m \times n} \rightarrow \R^{2mn}$ is given by
\begin{equation}\label{eq:cvec}
    A \mapsto a=\cvec(A):=\begin{bmatrix}
    \mvec(\Re A )\\
    \mvec(\Im A )
\end{bmatrix},
\end{equation}
where $\Re A$ and $\Im A$ denote, respectively, the elementwise real and imaginary part of the matrix $A$. Under this isomorphism, the above mentioned real scalar product is mapped to the usual dot product of $\R^{2mn}$. Throughout the paper, we often take this viewpoint and see $\C^{m \times n}$ as a \emph{real} Hilbert space. Again, all the results developed in the previous section still apply in this sense (note, in particular, that Theorem~\ref{thm:hessian} holds for \emph{real} submanifolds of $\C^{m \times n}$).

Given $A \in \F^{m \times n}$ and a nonempty closed set $\mathcal{E} \subseteq \F^{m \times n}$, which is often defined in terms of matrix properties (e.g., the set of all rank deficient matrices, or the set of all matrices with orthonormal columns), the corresponding \emph{matrix nearness problem} is defined as the task of finding the minimum $d(A,\mathcal{E})$ and (at least) one argument minimum $X_\star$ of the distance function $\|X-A\|_F$. It is in fact often convenient, at least from the perspective of a developer of numerical algorithms, to equivalently pose the problem with respect to the square distance function:
\begin{equation}\label{eq:matrixdistfrob}
   \mathrm{Compute} \ d(A,\mathcal{E})^2 = \min_{X \in \mathcal{E}} \| X-A \|_F^2,  \ \mathrm{find} \ \mathrm{one} \ \mathrm{matrix} \ X_\star \in \mathrm{argmin}_{X \in \mathcal{E}} \| X-A \|_F^2.   
\end{equation}  

\begin{remark}
An interesting question arises if one chooses $\F=\R$ for the input but $\F=\C$ for the target set. In other words, for which nearness problems is it true that every real input $A$  has a real nearest matrix, when the target set $\mathcal{E}$ includes complex matrices? \citet{NZ26b} provide some analysis in the following setting: Let $\mathcal{E}\subseteq\C^{n\times n}$ be closed, invariant under complex conjugation, and contain at least one real matrix. Hence, for an input $A\in\R^{n\times n}$, if $X_\star \in \mathcal{E}$ minimizes the distance then so must its complex conjugate matrix $\overline{X_\star}$. Therefore, if the minimizer is unique then it must be real; this happens, for example,  when $\mathcal{E}$ is convex.

More generally, \citet[Theorem~2.3]{NZ26b} show that, for every nonempty closed feasible set $\mathcal{E}$ and for Lebesgue-almost every real input, either the nearest matrix is unique and real, or none of the nearest matrices is real. For two problems of practical interest, the nearest normal and nearest Hurwitz
stable matrix\footnote{As the set of Hurwitz stable matrices is open, a nearest matrix must be sought in the closure of such set. In this paper, we call a matrix in this closure \emph{weakly Hurwitz stable}; see Section~\ref{sec:spectralnearness}.} problems, both situations happen with positive probability when $n\geq3$ and $A$ is sampled uniformly from the Frobenius unit sphere in $\R^{n\times n}$. For more details, see \citet[Sections~3 and~4]{NZ26b}.
\end{remark}

In the literature on matrix nearness problems, the seemingly most common convention is to call $X$ the minimizer, and $\Delta = X - A$ the corresponding perturbation of $A$. In this work, we reverse a sign, defining $\Delta := A - X$, so that $A = X + \Delta$ is a decomposition of $A$ into two summands: the minimizer $X\in\mathcal{E}$, and the perturbation $\Delta$. The reason for this choice is that with this definition $\Delta$ lies in a \emph{normal cone} of $\mathcal{E}$, a classical object from variational analysis. In the next subsection, we formally introduce normal cones and discuss their properties.

\subsection{Nearest points and normal cones} \label{subsec:normal cones}
Some tools from variational analysis are useful to formalize the idea that, if $x$ is the point in $\mathcal{E}$ nearest to $a$, then the difference $a-x$ must be orthogonal to $\mathcal{E}$. The \emph{regular} (also \emph{Fréchet}, or \emph{Bouligand}) \emph{normal cone} of a set $\mathcal{E} \subseteq \R^N$ at a point $x\in\mathcal{E}$ is
\[
N^r_x \mathcal{E} := \{\delta \in \R^N \colon \scalar{\delta, z-x} \leq o(\norm{z-x}) \text{ for all $z\in\mathcal{E}$ as $z\to x$}\},
\]
i.e., the set of vectors $\delta$ that form a right or obtuse angle with every possible tangent direction along which $z\to x$ in $\mathcal{E}$. For the regular cone, the following result holds.
\begin{theorem}[\protect{\citet[Theorem~6.12] {Rockafellar1998}}] \label{thm:reg min in normal}
    Let $\mathcal{E} \subseteq \mathbb{R}^N$, and $f$ be a real-valued function. Let $x\in\mathcal{E}$ be a local minimizer of $f$ in $\mathcal{E}$, and assume that $f$ is differentiable in a neighborhood of $x$. Then, $-\nabla f(x) \in N^r_x \mathcal{E}$.
\end{theorem}
In particular, with $f(x) = \frac12 \norm{x-a}_2^2$, Theorem~\ref{thm:reg min in normal} implies that a distance minimizer satisfies $a-x = \delta \in N^r_x(\mathcal{E})$.

More general results hold only for a different version of the normal cone. The \emph{limit} (also called basic, or Mordukhovich) \emph{normal cone} is
\[
N_x^\ell \mathcal{E} := \{\delta \in \mathbb{R}^N \colon \ \exists \ \text{$(x_k)_k\subset\mathcal{E},(\delta_k)_k \subset \R^N$ with $x_k \to x$, $\delta_k \to \delta$, $\delta_k \in N_{x_k}^r (\mathcal{E})$}\}.
\]
For instance, the following result holds.
\begin{theorem}[\protect{\citet[Theorem~8.15]{Rockafellar1998}}] \label{thm:nonreg min in normal}
    Let $\mathcal{E} \subseteq \mathbb{R}^N$ be closed, 
    and $f$ be a real-valued function.
    Let $x\in\mathcal{E}$ be a local minimizer of $f$ in $\mathcal{E}$, and assume $f$ is locally Lipschitz in a neighborhood of $x$. Then, there exists a subgradient $g \in \partial f(x)$ (the limit subdifferential) such that $-g \in N^\ell_x \mathcal{E}$.
\end{theorem}
(We have slightly changed the assumptions on $f$ with respect to \citet{Rockafellar1998}, to get a simpler statement.) Another remarkable property of the limit cone is the sum rule, which we record as Theorem~\ref{thm:sum rule}.
\begin{theorem}[\protect{\citet[Theorem~6.42]{Rockafellar1998}}] \label{thm:sum rule}
Let $\mathcal{E}_1,\mathcal{E}_2 \subseteq \R^N$ be closed sets. If  \emph{constraint qualification condition} $N^\ell_x\mathcal{E}_1 \cap (-N^\ell_x \mathcal{E}_2) = \{0\}$ is satisfied, then for every $x\in \mathcal{E}_1 \cap \mathcal{E}_2$ it holds
\begin{equation} \label{sum rule with inclusion}
    N^\ell_x (\mathcal{E}_1\cap \mathcal{E}_2) \subseteq N^\ell_x \mathcal{E}_1 + N^\ell_x\mathcal{E}_2.
\end{equation}
\end{theorem}
In general, $N_x^r \mathcal{E} \subseteq N_x^\ell \mathcal{E}$, without equality, but luckily there is a large class of ``tame" sets (known as \emph{Clarke regular} sets) for which the two definitions of normal cone coincide. Moreover, for Clarke regular sets Theorem~\ref{thm:sum rule} holds with an equality in~\eqref{sum rule with inclusion} rather than an inclusion. Notable examples of Clarke regular sets include:
\begin{itemize}
    \item If $\mathcal{E} = B+\mathcal{V}$ is an affine subspace of $\R^N$, then $N^r_x(\mathcal{E}) = N^\ell_x(\mathcal{E}) = \mathcal{V}^\perp$, the familiar orthogonal complement from linear algebra. For affine subspaces, the sum rule condition $(\mathcal{V}_1\cap \mathcal{V}_2)^\perp = \mathcal{V}_1^\perp + \mathcal{V}_2^\perp$ holds for any $x \in \mathcal{E}_1\cap \mathcal{E}_2$ without any further assumptions.
    \item If $\mathcal{E}$ is (locally around $x$) a $C^1$ embedded manifold, $N_x^r \mathcal{E} = N_x^\ell \mathcal{E} = N_x \mathcal{E}$, where $N_x \mathcal{E}$ is the orthogonal complement of the tangent space $T_x\mathcal{E}$. For two embedded manifolds, the property $T_x(\mathcal{E}_1 \cap \mathcal{E}_2) = T_x\mathcal{E}_1 \cap T_x\mathcal{E}_2$, which is equivalent to the sum rule $N_x(\mathcal{E}_1 \cap \mathcal{E}_2) = N_x \mathcal{E}_1 + N_x \mathcal{E}_2$, is called \emph{clean intersection} in the language of differential geometry \cite{LeeSmooth}; and the stronger condition $T_x(\mathcal{E}_1) + T_x(\mathcal{E}_2) = \R^N$, which is equivalent to $N_x \mathcal{E}_1 \cap (-N_x \mathcal{E}_2) = \{0\}$ and implies the former one, is called \emph{transversal intersection}.
    \item If $\mathcal{E}$ is a convex set, $N_x^r \mathcal{E} = N_x^\ell \mathcal{E} = N_x \mathcal{E}$, where 
    \begin{equation} \label{normal cone for convex E}
        N_x \mathcal{E} = \{\delta \in \mathbb{R}^N \colon \scalar{\delta,z-x}\leq 0 \text{ for all $z\in\mathcal{E}$} \},    
    \end{equation}
    and when $f$ is a convex function the condition $-g \in N_x \mathcal{E}$ is not only necessary but also sufficient to have a global minimum. In particular, for the squared distance function, this condition $a-x =\delta \in N_x \mathcal{E}$ can be interpreted geometrically as stating that $a$ and $\mathcal{E}$ lie on opposite sides of the supporting hyperplane $\{z\in\mathbb{R}^N \colon \scalar{\delta, z-x}=0\}$. For convex sets, a stronger version of Theorem~\ref{thm:sum rule} holds: If $\operatorname{relint}(\mathcal{E}_1) \cap \operatorname{relint}(\mathcal{E}_2) \neq \emptyset$, then $N_x(\mathcal{E}) = N_x(\mathcal{E}_1) + N_x(\mathcal{E}_2)$ \cite[Corollary~23.8.1]{Rockafellar1970}. Here, $\operatorname{relint}(\mathcal{E})$ denotes the \emph{relative interior} of $\mathcal{E}$, i.e., the interior of $\mathcal{E}$ inside the ambient space of its affine hull.
\end{itemize}

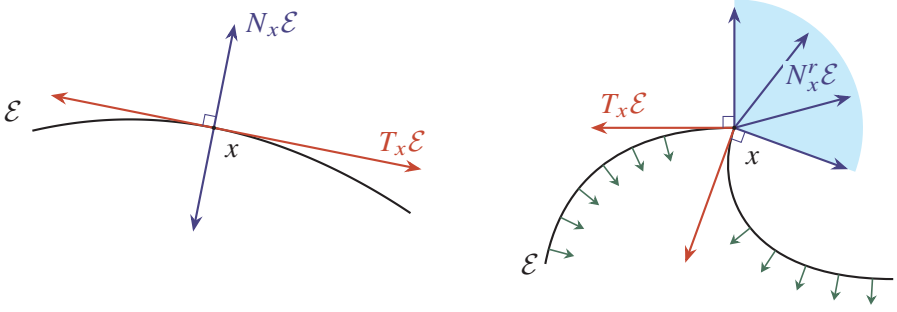
\begin{figure}
    \centering
    \colorlet{tancol}{red!80!black}
    \colorlet{normcol}{blue!75!black}
    \colorlet{conefill}{cyan!20}
    \colorlet{hatchcol}{green!45!black}
     
    \tikzset{
      setbdry/.style = {thick},
      tanvec/.style  = {tancol, thick, -{Stealth[length=2.2mm]}},
      normvec/.style = {normcol, thick, -{Stealth[length=2.2mm]}},
      setside/.style n args = {3}{postaction={decorate}, decoration={markings,
          mark=between positions #2 and #3 step 4.5mm with
            {\draw[hatchcol, semithick, -{Stealth[length=1.3mm, width=1.4mm]}]
               (0,0) -- (0,#1*3.5mm);}}},
    }
     
    \newcommand{\rightangle}[4][normcol]{%
      \draw[#1, thin] ($(#2)+(#3:0.3)$) -- ($(#2)+(#3:0.3)+(#4:0.3)$)
                      -- ($(#2)+(#4:0.3)$);}
     
    \begin{tikzpicture}[scale=0.5, baseline=(x)]
      \coordinate (x) at (0,0);
     
      \draw[setbdry, domain=-4.8:5.2, smooth, samples=60]
        plot (\x, {-0.2*\x - 0.045*\x*\x});
      \node[above left] at (-4.8,-0.08) {$\mathcal{E}$};
     
      \draw[tanvec]  (x) -- (168.69:4.4);
      \draw[tanvec] (x) -- (-11.31:5.6);
      \node[tancol, above] at (-11.31:5.1) {$T_x\mathcal{E}$};
     
      \draw[normcol, thick, {Stealth[length=2.2mm]}-{Stealth[length=2.2mm]}]
        (258.69:2.8) -- (78.69:2.8) node[right] {$N_x\mathcal{E}$};
     
      \rightangle{x}{78.69}{168.69}
     
      \fill (x) circle (1.6pt);
      \node at (0.5,-0.6) {$x$};
    \end{tikzpicture}%
    \hspace{1cm}%
    \begin{tikzpicture}[scale=0.5, baseline=(x)]
      \coordinate (x) at (0,0);
     
      \fill[conefill] (x) -- (90:3.4) arc[start angle=90, end angle=-20, radius=3.4]
                      -- cycle;
     
      \draw[setbdry, setside={-1}{2mm}{-6mm}]
        (-5,-3.6) .. controls (-4.5,-1) and (-2.5,0) .. (x);
      \draw[setbdry, setside={-1}{14mm}{-2mm}]
        (x) .. controls (-0.55,-1.5) and (0.1,-4) .. (4.2,-4);
      \node at (-5.4,-3.6) {$\mathcal{E}$};
     
      \draw[tanvec] (x) -- (180:3.8);
      \draw[tanvec] (x) -- (250:3.8);
      \node[tancol, above] at (180:2.9) {$T_x\mathcal{E}$};
     
      \foreach \a in {90, 52, 15, -20}
        \draw[normvec] (x) -- (\a:3.2);
      \node[normcol, fill=conefill, inner sep=1pt] at (33:2.45) {$N^r_x\mathcal{E}$};
     
      \rightangle{x}{90}{180}
      \rightangle{x}{-20}{-110}
     
      \fill (x) circle (1.6pt);
      \node[below] at (0.5,-0.35) {$x$};
    \end{tikzpicture}
        \caption{Two examples of normal cones of a set $\mathcal{E}$ (drawn in black). Tangent directions are displayed in red, directions in the regular normal cone are displayed in blue. In the picture on the left, $\mathcal{E}$ is locally a one-dimensional smooth manifold around $x$; the tangent space $T_x\mathcal{E}$ and normal space $N_x\mathcal{E}$ are affine spaces. In the picture on the right, $\mathcal{E}$ has a kink in $x$, and the normal directions in $N_x^r\mathcal{E}$ fill the whole cone shaded in light blue. The directions marked with smaller green arrows belong to normal cones to $\mathcal{E}$ in points $z\neq x$; hence their limits when $z\to x$ must belong to the limit normal cone $N_x^\ell\mathcal{E}$. Indeed, $N_x^\ell \mathcal{E}$ is the union of the blue cone $N_x^r\mathcal{E}$ and two half-lines coming from the limits of the green arrows (not drawn). Note that $N_x^\ell\mathcal{E}$ is non-convex in this example.}
    \label{fig:cones}
\end{figure}
Two illustrative examples are drawn in Figure~\ref{fig:cones}.

Normal cones give constraints on the possible candidate values for $\delta = a-x$, as the following example shows.

\begin{example} \label{ex: rank deficient via normal}
    Having fixed $\F \in \{\C,\R\}$, let $\mathcal{E}_0 = \{X \in \mathbb{F}^{m\times n} \colon \rank(X) < \nu\}$ be the set of rank-deficient matrices with coefficients in $\F$. The set $\mathcal{E}_0$ is locally a manifold in the neighborhood of any $X\in\mathcal{E}_0$ with rank exactly $\nu-1$, and 
    \[
    N_X \mathcal{E}_0 = \operatorname{span}\{uv^*\colon u \in \ker X^*, \, v \in \ker X\}.
    \]
    On the other hand, if $X \in \mathcal{E}_0$ has rank strictly smaller than $\nu-1$, then $N^r_X \mathcal{E}_0 = \{0\}$: Indeed, for any pair of vectors $u,v$, the curve $z(t) = X + t uv^* \in \mathcal{E}_0$ converges to $X$, so any matrix $\Delta$ in the normal cone must satisfy $\scalar{\Delta, uv^*} \leq 0$ for every pair of vectors $u,v$, and only $\Delta=0$ satisfies this condition. 
    
    Given $A\in\F^{m\times n}$, for any  minimizer $X$ of
    \begin{equation} \label{nearest_singular_example_cones}
        \min \norm{X-A}_F \subjectto X \in \mathcal{E}_0, 
    \end{equation}
    we must have $\Delta = A-X \in N^r_X \mathcal{E}_0$: This follows by applying Theorem~\ref{thm:reg min in normal} to $F(X) = \frac12\norm{X-A}_F^2$. In particular,
    \begin{itemize}
        \item A matrix $X$ with rank strictly smaller than $\nu-1$ is never a distance minimizer in $\mathcal{E}_0$, with the sole exception of the trivial case when $A=X$;
        \item When $\rank(A)=\nu$, all minimizers are matrices of the form $X = A - \Delta$, $\Delta = uv^*$, i.e., rank-1 modifications of $A$. Moreover, if $X$ is such a minimizer, then the low-rank factors $u,v$ of $\Delta$ satisfy $u^*X = 0$, $Xv=0$.
    \end{itemize}
    It is one of the most celebrated results in the numerical linear algebra literature that the solution to problem~\eqref{nearest_singular_example_cones} can be expressed explicitly via the SVD of the input (Schmidt-Eckart-Young-Mirsky Theorem, see Example~\ref{ex:seym} below).
    Hence, we can verify directly that the claimed properties hold, even in the case of a matrix $A$ where the smallest singular value has multiplicity higher than 1. But for more complicated problems where a closed-form solution is not available, the information from normal cones is useful to restrict the search; we shall see examples in Section~\ref{sec:structured distance from instability}.
\end{example}
\citet{hosseini2019} study systematically the normal cones of the rank varieties $\mathcal{R}_k = \{X \colon \rank(X)=k\}$ and $\mathcal{R}_{\leq k} = \{X \colon \rank(X) \leq k\}$, and obtain more general results than Example~\ref{ex: rank deficient via normal}. 
\begin{theorem}[\citet{hosseini2019}] \label{thm: normals of rank varieties}
Let $X \in \F^{m\times n}$ have rank $k$. Let $U \in \F^{m\times (m-k)}$ and $V \in \F^{n\times (n-k)}$ be basis matrices for its left and right kernel, respectively. Then,
    \begin{enumerate}
        \item  $N_X \mathcal{R}_{\leq k} =N_X \mathcal{R}_k = \{U Y V^* \colon Y \in \F^{(m-k)\times (n-k)}\}$;
        \item For any $r$ with $k < r \leq \nu$, $N^r_X \mathcal{R}_{\leq r}=\{0\}$, and
        \[
        N^\ell_X \mathcal{R}_{\leq r} = \{U Y V^* \colon Y \in \F^{(m-k)\times (n-k)}, \, \rank(Y)\leq \nu-r \}.
        \]
    \end{enumerate}
\end{theorem}

\section{The choice of the distance}\label{sec:uinorms}

\subsection{Frobenius norm or other norms?}

For most of this review, we will measure distances using the norm \eqref{eq:fro}. However, occasionally a mathematician may be interested in a variation on the theme of \eqref{eq:matrixdistfrob}, where the Frobenius norm is replaced by a different matrix norm \cite{Byers1988a, GG11, Higham1988}. The spectral norm $\norm{A}_2 = \sigma_{\max}(A)$ is an especially common alternative; see for example its use in the context of computing the distance to uncontrollability
\cite{gu06} or for the nearest defective matrix problem \cite{malyshev1999}. Moreover, beyond norm-based distances, matrix nearness problems can also be formulated using Bregman divergences \cite{Dhillon2007}, which generalize the squared Frobenius distance. When the chosen norm is not induced by an inner product, this carries the big disadvantage of losing the real Hilbert space structure of $\F^{m \times n}$, and in particular we cannot recycle the orthogonality arguments that are used to obtain formulae for the Frobenius projection. For some problems, this disadvantage is tangible. For example, suppose that we seek a symmetric positive semidefinite
matrix nearest to a real input matrix $A$. In the Frobenius norm the solution can be computed by mapping to $0$ all the negative eigenvalues, if any, of the symmetric part $(A+A^T)/2$; but in the spectral norm, the matrix thus constructed may not be the nearest one to $A$ (see Example~\ref{ex:hh} below for a more detailed discussion), and \citet{Higham1988} develops instead iterative methods based on a characterization due to \citet{halmos-positive}. On the other hand, the Hilbert space structure does not by itself make a nearness problem easy to solve. For example, the weighted Frobenius norm given by $\|X\|^2 = \sum_{i,j} \omega_{ij} |X_{ij}|^2$, where the weights $\omega_{ij}$ are positive, is still induced by an inner product. Nevertheless, \citet{GG11} prove that, for real matrices, computing a nearest matrix of rank $\leq 1$ in the associated distance is generally NP-hard; a somewhat surprising result since this problem is relatively easy for a vast class of norms (see Example~\ref{ex:seym} below). Even in the Frobenius distance itself, some matrix nearness problems can be proved to be NP-hard \cite{ZZ23}.

It is clear that the minimizer, in general, depends on the chosen norm. We illustrate this fact in Example~\ref{ex:1} and Example~\ref{ex:1bis}.
\begin{example}\label{ex:1}
    Let $A=0$ and $\mathcal{E}=\left\{ \begin{bsmallmatrix}
        a & 0\\
        0 & b
    \end{bsmallmatrix}, a,b \in \R, a+2b=3\right\}$. Table~\ref{tab:counter} displays minimizers of $\| X-A \|$ over $X \in \mathcal{E}$ for the Frobenius, spectral, and nuclear norms, which are all different. (Recall that the nuclear norm is the sum of the singular values.)
\begin{table}[h!]
    \centering
    \begin{tabular}{ccccc}
    \toprule
    Norm type & $\min \norm{X} \subjectto X\in \mathcal{E}$
         & $\begin{bmatrix}
        \frac35 & 0\\
        0 & \frac65
    \end{bmatrix}$ & $\begin{bmatrix}
        1 & 0\\
        0 & 1
    \end{bmatrix}$ & $\begin{bmatrix}
        0 & 0\\
        0 & \frac32
    \end{bmatrix}$\\
    \midrule
    Frobenius $\norm{\cdot}_F$ & $\frac{3}{\sqrt{5}} \approx 1.34$ & {\color{red}$\frac{3}{\sqrt{5}} \approx 1.34$ } & $\sqrt{2} \approx 1.41$ & $\frac32 = 1.5$ \\[1ex]
    Spectral $\norm{\cdot}_2$& $1$ & $\frac65=1.2$ & {\color{red}$1$} & $\frac32 = 1.5$ \\[1ex]
    Nuclear $\norm{\cdot}_* $& $\frac32 = 1.5$& $\frac95=1.8$ & $2$ & {\color{red}$\frac32=1.5$}\\
    \bottomrule
    \end{tabular}
    
    \caption{Minimizers for Example~\ref{ex:1} in various unitarily invariant norms. The last three columns give the distance, induced by the norms specified in each row, from $A=0$ to the matrices in $\mathcal{E}$ displayed on top. Red colour highlights the global minima of each distance.}
    \label{tab:counter}
\end{table}    
\end{example}
Our second example highlights a surprising connection with some elementary results in Euclidean geometry.
\begin{example}\label{ex:1bis}
    Let $A \in \C^{3 \times 3}$ be diagonal matrix, such that its diagonal entries  $z_1,z_2,z_3 \in \C$ are the vertices of an acute triangle in the complex plane, i.e., a triangle whose angles are all strictly less than $\pi/2$. Moreover, let $\mathcal{E} := \{\lambda I_3 \colon \lambda\in\mathbb{C}\}$ be the space of multiples of the identity matrix. Then,
    \[
    \min_{X\in\mathcal{E}} \norm{X-A}_2 = \min_{\lambda\in\mathbb{C}} \max(\abs{\lambda-z_1},\abs{\lambda-z_2}, \abs{\lambda-z_3})
    \]
    is attained by $\lambda I_3$ if and only if $\lambda$ is the circumcenter (the center of the circle through the three vertices) of the triangle with vertices $z_1,z_2,z_3$. On the other hand,
    \[
    \min_{X\in\mathcal{E}} \norm{X-A}_F = \min_{\lambda\in\mathbb{C}} (\abs{\lambda-z_1}^2+\abs{\lambda-z_2}^2+\abs{\lambda-z_3}^2)^{1/2}
    \]
    is achieved if and only if $\lambda$ is the centroid (or barycenter, defined as $(z_1+z_2+z_3)/3$) of the same triangle. These two points generally differ.
\end{example}

It is convenient to introduce a notation that will be handy throughout the rest of our review paper. Recalling our notation $\nu=\min\{m,n\}$, for a matrix $X \in \F^{m \times n}$ we denote by $\diag(X) \in \F^{\nu}$ the vector of its diagonal elements; similarly, for a vector $x \in \F^\nu$ and given a pair $m,n$ (often implicit in the context) such that $\nu=\min\{m,n\}$, we denote by $\diagm(x) \in \F^{m \times n}$, or sometimes also $\diagm(x_1,\dots,x_\nu)$, the diagonal matrix whose diagonal entries coincide with the elements of $x$. 

\subsection{Smooth norms}
A matrix norm on $\F^{m\times n}$ is called \emph{smooth} if the function $F(X) = \norm{X}$ is differentiable at every $X\neq 0$. Notable examples of smooth norms are the Schatten $p$-norms with $1<p<\infty$, i.e.,
\[
\norm{X}_{S,p} = \left(\sigma_1(X)^p + \dots + \sigma_{\nu}(X)^p\right)^{1/p},
\]

with gradient
\begin{equation} \label{gradient schatten}
    \nabla F(X) = \frac{1}{\norm{X}_{S,p}^{p-1}} U \diagm(\sigma_1^{p-1},\dots,\sigma_\nu^{p-1})V^*    
\end{equation}
and the elementwise $p$-norms with $1<p<\infty$, i.e.,
\[
\norm{X}_{E,p} = \left(\sum_{i=1}^m\sum_{j=1}^n |X_{ij}|^p \right)^{1/p},
\]
with gradient
\begin{equation} \label{gradient elementwise}
    \left(\nabla F(X)\right)_{ij} = \frac{1}{\norm{X}_{E,p}^{p-1}} X_{ij} \abs{X_{ij}}^{p-2},    
\end{equation}
as long as one follows the convention that $X_{ij} |X_{ij}|^{p-2}=0$ when $1 < p < 2$ and $X_{ij}=0$.
(The function $G(X) = \frac{1}{p}\norm{X}^p$ has a simpler formula for the gradient than $F(X)$, since the denominators in~\eqref{gradient schatten} and~\eqref{gradient elementwise} disappear, but this does not simplify our following arguments involving these norms).

The following easy result will be useful.
\begin{lemma} \label{lem:nonzero gradient}
    Let $\norm{\cdot}$ be a smooth norm, and $X\neq 0$. Then, $\nabla \norm{X} \neq 0$.
\end{lemma}
\begin{proof}
    Norms are convex functions, so $\nabla F(X) = 0$ would imply that $X$ is a global minimizer for $F$, and this is impossible for $X \neq 0$.
\end{proof}
Smooth norms share some of the properties of the Frobenius norm, like for instance the following one.
\begin{theorem} \label{thm:rank exactly n-1}
    Let $\norm{\cdot}$ be a smooth norm, and $A \in \F^{m\times n}$ a matrix with full rank. Then, every nearest rank-deficient matrix to $A$ has always rank exactly $\nu-1$.
\end{theorem}
\begin{proof}
    The proof is essentially the same as the one for the Frobenius norm in Example~\ref{ex: rank deficient via normal}.
    If $X$ is a minimizer for the problem
    \[
    \min \norm{X-A} \subjectto X \in \mathcal{E}_0,
    \]
    then by Theorem~\ref{thm:reg min in normal}
    \begin{equation} \label{step with normal cones}
        -\nabla_X F(X-A) \in N^r_{X} \mathcal{E}_0.    
    \end{equation}
    Whenever $\rank(X) < \nu-1$, the normal cone in~\eqref{step with normal cones} is $\{0\}$, hence we must have $A-X=0$, which is impossible since $A$ has full rank.
\end{proof}
Notably, a norm that is \emph{not} smooth everywhere outside zero is the spectral norm $\norm{X}_2 = \sigma_1(X)$. Its subdifferential for $X \neq 0$ is \cite{Watson1992}
\[
\partial F(X) = \{U_1WV_1^*: W\succeq 0, \traccia(W)=1,\}
\]
where $U_1$ and $V_1$ are bases, appearing in the \emph{same} SVD, for the spaces of left and right singular vectors associated to $\sigma_1$ (which may have multiplicity larger than 1). For this norm, Theorem~\ref{thm:rank exactly n-1} does not hold: for instance, a nearest singular matrix to the identity matrix $A=I$ is $X=0$, with spectral distance 1. Indeed, the argument in the proof breaks down: $0\not \in \partial F(X)$ is still true for every $X\neq 0$, but we cannot apply Theorem~\ref{thm:reg min in normal} since $F(X)$ is nonsmooth, and we cannot replace it with Theorem~\ref{thm:nonreg min in normal} because $N^\ell_X \mathcal{E}_0 \neq \{0\}$.

\subsection{Unitarily-universal nearest matrices}

Interestingly, there are also some matrix nearness problems for which the same matrix is simultaneously a minimizer for the distances induced by many norms. In particular, there are several remarkable results in this sense if we restrict our attention to unitarily invariant norms. The theme of unitary invariance will reoccur later in our article, see Subsection~\ref{sec:uisets}.

Recall that a matrix norm $\| \cdot \|$ is \emph{unitarily invariant} if $\| A \|=\|UAV\|$ for every pair of unitary matrices $U,V$ of suitable size. For example, the Frobenius norm \eqref{eq:fro} and the spectral norm are unitarily invariant. Following~\citet{bhatia_book}, we typically assume that unitarily invariant norms are always \emph{normalized}, i.e., $\norm{e_1 e_1^*} = 1$. A good amount of interest has been given in the literature to matrix nearness problems admitting a ``unitarily universal" minimizer, that is, a matrix that solves the problem \eqref{eq:matrixdistfrob} for \emph{every }unitarily invariant norm. Historically, this line of  research can be traced back to the classic studies on low-rank approximation \cite{Schmidt1907, EY36, Mirsky60}, where an appropriate truncation of the SVD is a universal minimizer for unitarily invariant norms; indeed, statements about properties that hold in every unitarily invariant norm are also known as \emph{Mirsky-type results}. In the context of nearness problems, there is a strong computational motivation. Indeed, sometimes an application dictates that a unitarily invariant norm other than the Frobenius norm (often the spectral norm \cite{dasbora20,gu06,malyshev1999}) is the preferable choice to compute distances. If a unitarily universal minimizer exists and coincides with the unique Frobenius minimizer, the Hilbert space structure associated with Frobenius optimization can be exploited also to compute the different distance of interest.

The following lemma gives a first noteworthy case.
\begin{lemma} \label{lem: 2 norm implies all ui norms}
    Let $A\in\F^{m\times n}$, and $\mathcal{E}\subseteq \F^{m\times n}$ be a closed set. Suppose that $X_\star  \in \arg \min_{X\in\mathcal{E}} \norm{X-A}_2$ is such that $X_\star -A$ has rank $1$. Then, $X_\star  \in \arg \min_{X\in\mathcal{E}} \norm{X-A}$ for every unitarily invariant norm $\norm{\cdot}$.
\end{lemma}
\begin{proof}
    We may suppose without loss of generality (up to rescaling) that $\norm{\cdot}$ is normalized. For any normalized unitarily invariant norm we have $\norm{M}_2 \leq \norm{M}$ for all matrices $M$, with equality if $M$ has rank 1 \cite[(IV.38)]{bhatia_book}. With this result, the proof is simple: For every $X\in\mathcal{E}$,
    \[
    \norm{X_\star -A} = \norm{X_\star -A}_2 \leq \norm{X-A}_2 \leq \norm{X-A}.
    \]
\end{proof}

\begin{example}\label{ex:firstoccurrence}
    Let $A\in\F^{m\times n}$, and assume for simplicity that $m\geq n$ and that $A$ has full column rank $n$. By the Schmidt-Eckart-Young-Mirsky Theorem (Example~\ref{ex:seym}), a nearest rank deficient matrix to $A$ in the spectral norm is $X_\star = A-\Delta_\star$ with $\Delta_\star = u_n \sigma_n v_n^*$, where $u_n, \sigma_n,v_n$ come from a (possibly non-unique) SVD of $A$. Since the optimal perturbation $\Delta_\star$ has rank 1, it is also a unitarily universal minimizer by Lemma~\ref{lem: 2 norm implies all ui norms}.

    Note also that $X_\star$ is not the unique minimizer in the spectral norm: For instance, we can take any matrix of the form $X_\star = U \diagm(\tau_1, \tau_2,\dots,\tau_{n-1},0)V^*$, where $\tau_i \in \F$ with $\abs{\sigma_i-\tau_i} \leq \sigma_n$ for $i=1,2,\dots,n-1$. But not all these matrices are unitarily universal minimizers.
\end{example}
Another pioneering contribution is due to \citet{FH55}, who showed that the (skew-)Hermitian part of a square matrix is the nearest (skew-)Hermitian matrix in every unitarily invariant norm.

\begin{theorem}[\citet{FH55}] \label{thm:FH0}
    Let $A\in \F^{n\times n}$, and $\mathcal{E}_H = \{X \in \F^{n\times n} \colon X = X^*\}$. Then, the Hermitian part $\frac{A+A^*}{2}$ is a minimizer of $d(A, \mathcal{E}_H)$ in every unitarily invariant norm.
\end{theorem}
\begin{proof}
    Let $X_\star = \frac{A+A^*}{2}$, $\Delta_\star = \frac{A-A^*}{2}$ be the (unique) decomposition of $A$ into a Hermitian matrix $X_\star = X_\star^*$ and a skew-Hermitian $\Delta_\star = -\Delta_\star^*$. Let $A = X+\Delta$ be any other decomposition with $X=X^*$. Then,
    \[
    \norm{\Delta_\star} = \norm*{\frac{A-A^*}{2}} = \norm*{\frac{\Delta-\Delta^*}{2}} \leq \frac{\norm{\Delta} + \norm{-\Delta^*}}{2} = \norm{\Delta}.
    \]
\end{proof}
This proof uses only two ingredients: the triangle inequality, and the fact that $\norm{\Delta}=\norm{-\Delta^*}$, which holds for every unitarily invariant norm because $\Delta^*$ can be obtained from $\Delta$ via multiplication by unitary matrices (this fact is easy to prove via the SVD). For this reason, the proof strategy can be extended to a more general setting.

\begin{theorem}
    Let $\mathcal{E} = B + \mathcal{V}$ be a real affine space, where $B\in\F^{m\times n}$ and $\mathcal{V} \subseteq \mathbb{F}^{m\times n}$ is a real vector subspace. Suppose that the projection onto its orthogonal complement $\mathcal{V}^\perp$ can be written as
    \[
    \Pi_{\mathcal{V}^\perp}(X) = \frac{F_1(X) +F_2(X) + \dots + F_k(X)}{k},
    \]
    where the $F_i(X)$ are isometries, i.e., $\R$-linear operators that preserve unitarily invariant norms. Then, the Frobenius norm projection $X_\star = A - \Delta_\star$, where $\Delta_\star = \Pi_{\mathcal{V}^\perp}(A-B)$, is a minimizer of $d(A,\mathcal{E})$ in every unitarily invariant norm.
\end{theorem}
\begin{proof}
    The proof is a direct extension of the one of Theorem~\ref{thm:FH0}, which is the special case $F_1(X)=X$, $F_2(X)=-X^*$, $B=0$. Let $A=X+\Delta$ be another decomposition of $A$ with $X\in\mathcal{E}$. Then,
    \begin{align*}
        \norm{\Delta_\star} &= \norm{\Pi_{\mathcal{V}^\perp}(A-B)} = \norm{\Pi_{\mathcal{V}^\perp}(X-B)+\Pi_{\mathcal{V}^\perp}(\Delta)} = \norm{\Pi_{\mathcal{V}^\perp}(\Delta)}\\
        &\leq \frac{\norm{F_1(\Delta)} + \dots + \norm{F_k(\Delta)}}{k} = \norm{\Delta}.
    \end{align*}
\end{proof}
This averaging strategy, related to \emph{pinching}, is classic in matrix analysis; see for instance~\citet[Section~IV.2]{bhatia_book}, \citet{bhatia_pinching}. Usually the family of maps $\{F_i\}$ has some structure, such as being a matrix group.

More minimizers can be proved to be unitarily universal using this strategy. 
\begin{theorem} \label{thm:pinching applications}
    Let $A\in \F^{n\times n}$. Then, in every unitarily invariant norm:
    \begin{enumerate}
        \item $\frac{A-A^*}{2}$ is a skew-Hermitian matrix closest to $A$ \cite{FH55};
        \item $A-\diagm(A_{11},\dots,A_{nn})$ is a zero-diagonal matrix closest to $A$ (cfr. \citet[Problem~II.5.5]{bhatia_book}, \citet{bhatia_pinching});
        \item $A - \frac{\traccia(A)}{n}I_n$ is a trace-zero matrix closest to $A$ \cite{zietak_zerotrace}. 
    \end{enumerate}
\end{theorem}
\begin{proof}
    Item 1 follows from taking $\Pi_{\mathcal{V}^\perp}(X) = \frac{X+X^*}{2}$. Item 2 follows from taking $\Pi_{\mathcal{V}^\perp}(X) = \frac1n \sum_{k=0}^{n-1} Z^k X (Z^k)^*$, where $Z = \diagm(1,\omega,\dots,\omega^{n-1})$, $\omega$ being a primitive $n$-th root of unity. Item 3 follows from taking \[ \Pi_{\mathcal{V}^\perp}(X) = \frac{1}{n^2} \sum_{k,\ell=0}^{n-1} S^\ell Z^k X (Z^k)^* (S^\ell)^*,\] where $S$ is the cyclic shift matrix.
\end{proof}
Another classic application where this strategy appears is projecting a Hermitian $A$ onto positive semidefinite matrices. In the following, we use the Löwner order on Hermitian matrices: $A \preceq B$ means $B-A$ is positive semidefinite. Note that $0 \preceq A \preceq B$ implies $\norm{A}\leq \norm{B}$ for every unitarily invariant norm; this is a consequence of Ky Fan's Domination Theorem \cite[Theorem 4]{KyFan51} and Weyl's monotonicity theorem~\cite[Corollary~III.2.3]{bhatia_book}.
\begin{theorem}\cite{BK-positive} \label{thm:positivepart}
    Let $\mathcal{E}_+ = \{X\in \F^{n\times n} \colon X \succeq 0\}$. Let $A\in \F^{n\times n}$ be Hermitian, and let $Q\diagm(\lambda_1,\dots,\lambda_n)Q^*$ be an eigenvalue decomposition. Then, the \emph{positive part}
    \[
    X_\star = Q\diagm(\mu_1,\dots,\mu_n)Q^*, \quad \mu_i = \max(\lambda_i,0) \ \forall \ i=1,\dots,n,
    \]
    is a minimizer of $d(A,\mathcal{E}_+)$.
\end{theorem}
\begin{proof}
    Let $\Delta_\star = A - X_\star \preceq 0$, and let $A = X + \Delta$ be another decomposition with $X \succeq 0$. Let $P$ be the orthogonal projector onto the image of $\Delta_\star$, so that $\Delta_\star = PAP$. Then,
    \[
    \norm{\Delta_\star} \leq \norm{\Delta_\star - PXP} = \norm{P(A - X)P} \leq \norm{A-X} = \norm{\Delta},
    \]
    where the first inequality holds because $\Delta_\star - PXP \preceq \Delta_\star \preceq 0$, and the second follows from averaging.
\end{proof}

\begin{example}\label{ex:hh}
    The reader may be tempted to combine Theorems~\ref{thm:FH0} and~\ref{thm:positivepart}, and conclude that the positive part of $\frac12(A+A^*)$ is a nearest positive semidefinite matrix to a non-Hermitian $A$ in every unitarily invariant norm. Remarkably, this is false. For instance, when $A = \begin{bsmallmatrix}
    0 & 4\\0 & 0
\end{bsmallmatrix}$, if we project $A$ first onto Hermitian matrices and then onto positive semidefinite matrices we get $H = \Pi_{\mathcal{E}_H}(A) = \begin{bsmallmatrix}
    0 & 2\\
    2 & 0
\end{bsmallmatrix}$ and then $X = \Pi_{\mathcal{E}_+}\left(H\right) =\begin{bsmallmatrix}
    1 & 1\\
    1 & 1
\end{bsmallmatrix}$. This resulting matrix $X$ is a distance minimizer in the Frobenius norm \cite{Higham2002}, but not in the spectral norm: $\norm{A-2X}_2 < \norm{A-X}_2$. See~\citet{halmos-positive,BK-positive} for more analysis and a proof that this double-projection returns the correct minimizer when $A$ is normal.
\end{example}

\citet{sun} proves that, if the target set $\mathcal{E}$ is described by a certain class of linear matrix equations, then the Frobenius projection is a universal minimizer for all unitarily invariant norms. We record this result as Theorem~\ref{thm:suntheorem}; in its statement and proof, and throughout the paper, $X^\dagger$ denotes the Moore-Penrose pseudoinverse of the matrix $X$ \cite[page 290]{GolubVLoan}, whereas $A \otimes B$ denotes the Kronecker product of $A$ and $B$ \cite[Section 1.3.6]{GolubVLoan}.
\begin{theorem}\cite[Lemma 1.3]{sun}\label{thm:suntheorem}\\
        Let $B \in \C^{p \times m}, C \in \C^{n \times q}, D \in \C^{p \times q}, A \in \C^{m \times n}$. Define
        \begin{equation} \label{sunstructure}
            \mathcal{E} = \left\{ X \in \C^{m \times n} : BXC=D\right\}.
        \end{equation} 
         Then:
        \begin{enumerate}
            \item $\mathcal{E} \neq \varnothing$ if and only if $BB^\dagger D C^\dagger C = D$;
            \item Under the condition of the first item, every $X \in \mathcal{E}$ can be written as 
            \[ X= B^\dagger D C^\dagger + Z - B^\dagger B Z CC^\dagger\]
            for some $Z \in \C^{m \times n}$;
            \item Under the conditions of the first item, $X_\star = B^\dagger D C^\dagger + A -B^\dagger B A CC^\dagger \in \mathcal{E}$ is the unique matrix such that, for all unitarily invariant norms,
            \[ \| A-X_\star\|=\min_{X \in \mathcal{E}} \|A-X \|  = \| B^\dagger(D-BAC) C^\dagger \|. \]
        \end{enumerate}
        \label{thm:sun}
    \end{theorem}
    \begin{proof}
        The first two points follow from the setup in Theorem~\ref{thm:affinedist}, whose proof is postponed to Section \ref{sec:directmethods}, with $\mathbf{M} = C^T \otimes B$ and $b = \mvec(D)$. For the set $\mathcal{E}$ not to be empty, we need to have $b\in \range(\mathbf{M})$, or equivalently $b = \mathbf{M} \mathbf{M^\dagger}b$, and since $\mathbf{M}^\dagger = (C^\dagger)^T \otimes B^\dagger$ this produces the formula in item 1. Item 2 is now simply~\eqref{deltastarlstnrm}.
    
        We prove item 3 by averaging/pinching. The projection $Y = \Pi_{\mathcal{V}^\perp}(X)$ is given by $\mvec(Y) = \mathbf{M}^\dagger \mathbf{M} \mvec(X)$, or using the properties of vectorization
        \[
    Y = B^\dagger B X C C^\dagger = \Pi_{\mathcal{U}_1} X \Pi_{\mathcal{U}_2}.
    \]
    with $\mathcal{U}_1 = \range(B^*)$, $\mathcal{U}_2 = \range(C)$. For $j=1,2$, the projection $\Pi_{\mathcal{U}_j}$ is the average of $I$ and the orthogonal reflection $R_j = 2\Pi_{\mathcal{U}_j}-I$. Hence
    \begin{equation} \label{fouriso}
    \Pi_{\mathcal{V}^\perp}(X) = \Pi_{\mathcal{U}_1} X \Pi_{\mathcal{U}_2} = \frac{I+R_1}{2} X \frac{I+R_2}{2};    
    \end{equation}
    expanding parentheses we get an expression of $\Pi_{\mathcal{V}^\perp}$ as the average of 4 isometries.
    \end{proof}
    The last step in the proof is well visualized if one applies changes of bases on both sides to reduce to the case where $\mathcal{U}_1$ and $\mathcal{U}_2$ take the form $\range(\begin{bsmallmatrix}
        I\\0
    \end{bsmallmatrix})$: In that case, the four isometries in~\eqref{fouriso} are
    $
    X \mapsto \begin{bsmallmatrix}
        I & 0\\ 0 & \pm I
    \end{bsmallmatrix}X \begin{bsmallmatrix}
        I & 0\\ 0 & \pm I
    \end{bsmallmatrix},
    $
    whose average is $\begin{bsmallmatrix}
        X_{11} & 0\\
        0 & 0
    \end{bsmallmatrix}$.
    
Generalizing the affine subspaces in the above examples, we say that a nonempty closed subset $\mathcal{E} \subseteq \C^{m \times n}$ has the \emph{strong UUM property}\footnote{UUM is an abbreviation for unitarily-universal minimizer.} if, for every $A \in \C^{m \times n}$, there exists a unique $X_\star \in \mathcal{E}$ such that $\|X_\star-A\| \leq \|X-A\|$ for every unitarily invariant norm $\| \cdot \|$ and every $X \in \mathcal{E}$. We emphasize that it is not required that for each norm the minimizer is unique, just that there is a unique minimizer which is valid for every norm.

If the uniqueness requirement is removed and only existence is required for the unitarily-universal minimizer $X_\star$, we say that $\mathcal{E}$ has the \emph{weak UUM property}. For example, the set of unitary matrices or the set of matrices of rank at most $r$ have the weak UUM property. Motivated to study which sets have the (weak or strong) UUM property, let us recall in Lemma~\ref{lem:majorizedonsphere} a useful property of weak majorization.
    \begin{lemma}\label{lem:majorizedonsphere}
    Let $x,y \in \R^N$ be two nonnegative vectors, both ordered nonincreasingly, i.e., $i \geq j \Rightarrow x_i \leq x_j$ and similarly for $y$. If $\|x\|_2=\|y\|_2$ and $\sum_{i=1}^k x_i \leq \sum_{i=1}^k y_i$ holds for all $k=1,\dots,N$, then $x=y$.
\end{lemma}
\begin{proof}
Define the nonnegative vector $0 \leq s \in \R^N$ as $s_j := \sum_{i=1}^j (y_i -x_i)$. Moreover, set conventionally $x_{N+1}=y_{N+1}=0$. Then,
\[  0 = \norm{y}_2^2-\norm{x}_2^2 = s_1 (y_1+x_1) + \sum_{i=2}^N (s_i - s_{i-1} )  (y_i+x_i) = s^T v   \]
where
\[  v =   \begin{bmatrix}
    x_1-x_2 + y_1 -y_2\\
    x_2-x_3+y_2-y_3\\
    \vdots\\
    x_{N-1}-x_{N} + y_{N-1}-y_N\\
    x_N + y_N
\end{bmatrix} \geq 0.\]
Hence, $s_j v_j = 0$ for all $j$. Let $j_0$ be the smallest value of $j$ such that $s_j > 0$. Then, by the minimality of $j_0$, the $j$-th components of $x$ and $y$ coincide for $j < j_0$, whereas $y_{j_0} > x_{j_0}$. If $j_0=N$, we conclude $v_N s_N > 0$ which is a contradiction. If instead $j_0 < N$, observe that it must be $v_{j_0}=0$, implying $x_{j_0}=x_{j_0+1}$, $y_{j_0}=y_{j_0+1}$ and in particular $y_{j_0+1} > x_{j_0+1}$. Moreover,
\[ s_{j_0+1} = s_{j_0} + y_{j_0+1}-x_{j_0+1} = 2 s_{j_0} > 0.\]
Propagating inductively the arguments above we conclude that $s_N > 0$ and $y_N>x_N$, yielding the contradiction $v_N s_N > 0$. Hence, $s=0$ and thus $x=y$.
\end{proof}

    Recall that if $\sigma_1(A) \geq \dots \geq \sigma_\nu(A)$ are the singular values of $A$ then the $i$-th Ky Fan norm of $A$ ($1 \leq i \leq \nu$) is defined as $\|A\|_{(i)}=\sum_{j=1}^i \sigma_j(A)$. We argue in Theorem~\ref{thm:dom} that these norms, which are manifestly unitarily invariant, play a central role in characterizing sets that have the UUM properties.
\begin{theorem}\label{thm:dom}
    Let $\mathcal{E} \subseteq \C^{m \times n}$ be closed. Then:
    \begin{enumerate}
        \item $\mathcal{E}$ has the weak UUM property if and only if, for every $A \in \C^{m \times n}$, there exists $X_\star \in \mathcal{E}$ such that $X_\star$ is a projection of $A$ onto $\mathcal{E}$ in the Frobenius distance and $\|X_\star-A\|_{(i)} \leq \|X-A\|_{(i)}$ for every Ky Fan norm and every $X \in \mathcal{E}$;
        \item If $\mathcal{E}$ has the weak UUM property then, for every $A \in \C^{m \times n}$, every Frobenius projection $X$ of $A$ onto $\mathcal{E}$ is also a matrix in $\mathcal{E}$ nearest to $A$ in every distance induced by a unitarily invariant norm. Moreover, the matrices $X-A$ all have the same singular values;
        \item $\mathcal{E}$ has the strong UUM property if and only if it is convex and has the weak UUM property.
    \end{enumerate}
\end{theorem}
\begin{proof}
\begin{enumerate}
    \item This is a corollary of Ky Fan's Domination Theorem \cite[Theorem 4]{KyFan51}: The inequality $\|X\| \leq \|Y\|$ holds for every unitarily invariant norm if and only if it holds for every Ky Fan norm.
    \item Let $X_\star$ be a universal minimizer. If $X_\star$ is the unique Frobenius projection, there is nothing to prove. Otherwise, let $X_1 \neq X_\star$ be another minimizer in the Frobenius norm. Then, 
    $\|X_\star-A\|_F=\|X_1-A\|_F$ and, for every Ky Fan norm $\| \cdot \|_{(i)}$ ($i=1,\dots,\nu$), $\|X_\star-A\|_{(i)} \leq \|X_1-A\|_{(i)}$. Applying Lemma~\ref{lem:majorizedonsphere} to the vectors of the singular values of $X_\star-A$ and $X_1-A$, we deduce that these vectors are equal. As a further consequence, $X_1$ must be a minimizer over $Y \in \mathcal{E}$ of the distance $\|Y-A\|$ induced by every unitarily invariant norm.
    \item Let  first $\mathcal{E}$ have the strong UUM property; trivially, it must then also have the weak UUM property. Suppose that $\mathcal{E}$ is not convex. Then, there exists $A \in \C^{m \times n}$ lacking a unique projection onto $\mathcal{E}$ in the Frobenius distance. Applying item 2 to such $A$, we reach a contradiction.

    Conversely, suppose that $\mathcal{E}$ is convex and has the weak UUM property. Having fixed $A$, let $X_\star$ be its unique projection onto $\mathcal{E}$. Since there are no other minimizers for the Frobenius norm, it is the only candidate to be the universal minimizer for every unitarily invariant norm, and the statement follows.
\end{enumerate}
\end{proof}

\subsection{The UUM property for affine subspaces}

It is interesting to note that Theorem~\ref{thm:sun} does not exhaustively characterize all the affine subspaces that have the strong UUM property. For example, the Hermitian matrices considered in Theorem~\ref{thm:FH0} are a real subspace
but not a complex subspace; this is in contrast with Theorem~\ref{thm:sun}, because every (nonempty) set of the form \eqref{sunstructure} is a complex affine subspace.

Motivated by this observation, let us now focus on sets of the form $\mathcal{E}=\{ X \in \C^{m \times n} : X - B \in \mathcal{V} \}$ where $B \in \C^{m \times n}$ and $\mathcal V$ is a real subspace of $\C^{m \times n}$, that is, affine real subspaces. Obviously, affine subspaces are always convex sets, and hence they either have the strong UUM property or do not even have the weak UUM property. Below in this section, we will provide in Theorem~\ref{thm:projectorsnorm1} and in Theorem~\ref{thm:bigthmminsubspaces} two characterizations of affine real subspaces that satisfy the UUM property.

In Theorem~\ref{thm:projectorsnorm1} we provide an operator-theoretic characterization of the strong UUM property for subspaces.
     \begin{theorem}\label{thm:projectorsnorm1}
         Let $\Pi_\mathcal{U}$ denote the orthogonal (Frobenius) projection onto  the real subspace $\mathcal{U} \subseteq \C^{m \times n}$. Fix a matrix $B \in \C^{m \times n}$ and a proper real subspace $\mathcal{V} \subsetneq \C^{m \times n}$. The affine real subspace  $\mathcal{E}=\{X \in \C^{m \times n} : X - B  \in \mathcal{V}\}$ has the strong UUM property if and only if $\|  \Pi_{\mathcal{V}^\perp} \| = 1$ for every operator norm induced by a Ky Fan norm on $\C^{m \times n}$.
     \end{theorem}

  \begin{proof}
    Fix a Ky Fan norm $\| \cdot  \|_{(i)}$.  For each $A \in \C^{m \times n}$, the point in $\mathcal{E}$ where the minimal Frobenius distance is achieved is $X_\star = B + \Pi_{\mathcal{V}}(A-B)$. By Theorem~\ref{thm:dom},  $\mathcal{E}$ has the strong UUM property if and only if $X_\star$ is the minimizer of the distance induced by every Ky Fan norm. Denoting $C=B-A$, the Ky Fan distance from $A$ to $X_\star$ is
    \[ \| \Pi_{\mathcal{V}}(C)-C  \|_{(i)} = \|\Pi_{\mathcal{V}^\perp}(C)  \|_{(i)} = \| \Pi_{\mathcal{V}^\perp}(B-A+W)  \|_{(i)} \ \forall \ A \in \C^{m\times n}, \  \forall \ W \in \mathcal{V} .\] On the other hand, the Ky Fan distance from $A$ to an arbitrary point $B+W$ in the target set is $\| B+W-A \| _{(i)}$. Since $B+W-A$ can take any value in $\C^{m \times n}$, we deduce that the strong UUM property is equivalent to $\| \Pi_{\mathcal{V}^\perp}\|  \leq 1$ in each operator norm induced by a Ky Fan norm.

    To complete the proof, observe that $\mathcal{V}^\perp \neq 0$ and $\Pi_{\mathcal{V}^\perp}$ is a projection. Therefore, it must be $\| \Pi_{\mathcal{V}^\perp}\|  \geq 1$ in every operator norm, and the statement follows.
  \end{proof}

\emph{Complex} subspaces that are ranges of
projections having unit operator norm induced by a Schatten $p$-norm, for
$1<p<\infty$ and $p\neq2$, are classified by
\citet{ArazyFriedman92} (see also \cite{ArazyFriedman78} for $p=1$ and $p=\infty$). However, Theorem~\ref{thm:projectorsnorm1} analyzes specifically the Frobenius orthogonal projection, and the existence of
a contractive projection alone does not allow us to infer much.

For a complex linear subspace
$\mathcal{U}\subseteq\C^{m\times n}$, consider the condition
\begin{equation}\label{eq:abc}
    A,B,C \in \mathcal{U} \Rightarrow \frac{1}{2} \left(A B^* C + C B^* A \right) \in \mathcal{U}.
\end{equation}
It is a consequence of Theorem~\ref{thm:bigthmminsubspaces},
that we state and prove below, that a complex affine subspace $B+\mathcal{V}$ has the strong UUM property if and only if $\mathcal{V}^{\perp}$ satisfies~\eqref{eq:abc}.
This conclusion does not generally hold, however, for \emph{real} affine subspaces. In particular, we illustrate in Example~\ref{ex:3} that this condition is not necessary.

\begin{example}\label{ex:3}
    In $\C^{2 \times 1}$, every unitarily invariant norm is a multiple of the Frobenius norm, and therefore every nonempty convex subset has the strong UUM property. Consider the subspace
    \[  \mathcal{V}^\perp = \mathrm{span}_\R \left( \begin{bmatrix}
        1\\
        0
    \end{bmatrix}, \begin{bmatrix}
        i\\
        1
    \end{bmatrix}\right).  \]
    Taking $A=C=\begin{bmatrix}
        1\\
        0
    \end{bmatrix}$, $B=\begin{bmatrix}
        i\\
        1
    \end{bmatrix}$ we obtain
$ \displaystyle \frac{1}{2} \left(A B^* C + C B^* A \right) = -\begin{bmatrix}
        i\\
        0
    \end{bmatrix} \notin \mathcal{V}^\perp$, showing that the subspace does not satisfy \eqref{eq:abc}.
 \end{example}
For complex subspaces, but not for real subspaces, \eqref{eq:abc} is equivalent to
\begin{equation}\label{eq:xxstarx}
    X \in \mathcal{U} \Rightarrow XX^*X \in \mathcal{U}.
\end{equation} 
(The nontrivial implication is seen by taking $X=A+i^\alpha B+(-1)^\beta C$, $\alpha \in \{0,1,2,3\}, \beta \in \{0,1\}$, and performing tedious but straightforward manipulations.) It turns out that, not only for complex subspaces but even for real subspaces of $\C^{m \times n}$, \eqref{eq:xxstarx} (for the orthogonal complement) is in fact equivalent to the strong UUM property. We will prove this linear algebraic characterization in Theorem~\ref{thm:bigthmminsubspaces}. While it is relatively easy to prove that the UUM property implies that \eqref{eq:xxstarx} holds for the normal space of the affine subspace, our proof of sufficiency in Theorem~\ref{thm:bigthmminsubspaces} is not very immediate, and thus we find it convenient for the exposition to first establish some auxiliary results.

    \begin{lemma}\label{lem:oddstaysin}
        Let $\mathcal{U}$ satisfy \eqref{eq:xxstarx}. Then, 
        \begin{enumerate}
            \item $Y X^* X + X Y^* X+X X^* Y \in \mathcal{U}$ for all $X,Y \in \mathcal{U}$;
            \item $X \in \mathcal{U} \Rightarrow (XX^*)^kX \in \mathcal{U}$ for all $k \in \mathbb{N}$.
        \end{enumerate}
    \end{lemma}
    \begin{proof}
    \leavevmode
    \begin{enumerate}
        \item For all $X,Y \in \mathcal{U}$ and all $\varepsilon \in \R$, we can apply \eqref{eq:xxstarx} to $Z(\varepsilon):=X+\varepsilon Y$. Expanding,
        \[  f(\varepsilon):=Z(\varepsilon)Z(\varepsilon)^*Z(\varepsilon)= XX^*X + \varepsilon \left( Y X^* X+ X Y^* X + X X^* Y \right) + O(\varepsilon^2).    \]
        Since $\mathcal{U}$ is a finite-dimensional real vector space, it is topologically closed. Moreover, $f(\varepsilon) \in \mathcal{U}$ for all $\varepsilon$, whence we conclude that \[ f'(0) = \lim_{\varepsilon \rightarrow 0} \frac{f(\varepsilon)-f(0)}{\varepsilon}  \in \mathcal{U}.\] This implies the statement.
        \item We give a proof by induction on $k$. The base case $k=0$ is obvious. By applying the property in the previous item to $Y=(X X^*)^k X$, the induction step follows.
    \end{enumerate}
    \end{proof}

Lemma~\ref{lem:subgradient} below enhances the characterizations of subdifferentials given by \citet[Theorem 2]{Watson1992}
and \citet[Corollaries 2.5 and 2.6]{Lewis1995}. In particular, item (i) can be found in these references. We include a self-contained proof, including a novel ``symmetrization trick" that proves also items (ii) and (iii).

    \begin{lemma}\label{lem:subgradient}
        Let $\| \cdot \|$ be a unitarily invariant norm of $\C^{m \times n}$. Suppose $X=U \Sigma V^* \in \C^{m \times n}$ is an SVD, where $U \in \C^{m \times m}$ and $V \in \C^{n \times n}$ are unitary matrices and $\Sigma = \mathrm{diagm}(\sigma_1,\sigma_2,\dots,\sigma_{\nu}) \succeq 0 \in \R^{m \times n}$. Then, the subgradient of $\| \cdot \|$ at $X$ contains a matrix of the form $G=U F V^*$ where $F=\mathrm{diagm}(f) \in \R^{m \times n}$ is a real diagonal matrix, and $f \in \R^{\nu}$ has the properties
        \begin{itemize}
            \item[(i)] $f \geq 0$ and ordered in a non-increasing manner,
            \item[(ii)] $\sigma_i = \sigma_j \Rightarrow f_{i} = f_{j}$, and
            \item[(iii)]$\sigma_i=0 \Rightarrow f_{i}=0$.
        \end{itemize}
    \end{lemma}
\begin{proof}
By \citet[Theorem 3.5.18]{HJtopics} (see also \citet{vonneumann,Watson1992}), $\| X \| = \Psi(\sigma)$ where $\sigma \in \R^\nu$ is the vector of singular values of $X$ and $\Psi : \R^\nu\rightarrow \R$ is a symmetric gauge function. In particular, $\Psi(\sigma)=\Psi(P\sigma)$ for every signed permutation matrix $P$. We claim that there exists a vector $f$ with the stated properties which is a subgradient of $\Psi$ at $\sigma$. Note in particular that from the claim it follows
\[ 0= \Psi(0) \geq \Psi(\sigma) - \sigma^T f, \ 2 \Psi(\sigma)=\Psi(2\sigma) \geq \Psi(\sigma) + \sigma^T f \]
and therefore
\begin{equation}\label{eq:subgrsvequalnorm}
     \sigma^T f = \Psi(\sigma) = \| X \|.
\end{equation}
Letting $F=\diagm(f) \in \R^{m \times n}$, define now $G:=U F V^* \in \C^{m \times n}$ as in the statement. Observe that 
\[  \Re \traccia (G^* X) = \Re \traccia ( F^T \Sigma ) = f^T \sigma = \| X\|.  \]
For every $Y \in \C^{m \times n}$ with vector of singular values $y \in \R^\nu$, \eqref{eq:subgrsvequalnorm} and the fact that $f$ is a subgradient of $\Psi$ at $\sigma$ imply
that 
\begin{equation}\label{eq:subgrsvleqothernorms}
    \| Y \| = \Psi(y) \geq f^T y.
\end{equation} 
Therefore, using \eqref{eq:subgrsvequalnorm} and \eqref{eq:subgrsvleqothernorms}, and also applying Von Neumann's trace inequality \cite[Problem 10, Section 3.3]{HJtopics} by observing that (up to permutation) $f$ and $y$ contain the singular values of $G$ and $Y$,
\[  \langle G, Y - X \rangle =  \Re \traccia(G^*Y) - \Re \traccia (G^* X) \leq f^T y - \| X \| \leq \| Y \| - \| X \|,  \]
proving that $G$ is a subgradient of $\| \cdot \|$ at $X$.

It remains to prove the claim. Let $g \in \R^\nu$ be an arbitrary subgradient of $\Psi$ at $\sigma$ (the existence of $g$ is guaranteed by the convexity of $\Psi$). By definition of subgradient, $\Psi(y) - \Psi(\sigma) \geq (y-\sigma)^T g$ for all $y \in \R^\nu$. Let $\mathbb{S}$ be the set of all signed permutation matrices $P$ such that $P\sigma=\sigma$; obviously $\# \mathbb{S} \geq 1$ since $I \in \mathbb{S}$, and it is easy to verify that $\mathbb{S}$ is a finite group. Fix $P \in \mathbb{S}$ and define $z:=Py$. Then, $\Psi(z) - \Psi(\sigma) \geq (P^T z - \sigma)^T g = (z-\sigma)^T P g$ and therefore $Pg$ is also a subgradient.  Define now
\[  f : = (\# \mathbb{S})^{-1} \sum_{P \in \mathbb{S}} P g .  \]
For all $P \in \mathbb{S}$, the action of $P$ on the group $\mathbb{S}$ is a permutation of the elements of $\mathbb{S}$, and hence $Pf = f$. Now observe that if $\sigma_i=0$ then $ I - 2 e_ie_i^T \in  \mathbb{S}$, implying $f_i=0$. Similarly, if $\sigma_i=\sigma_j$ then $I + e_i e_j^T + e_j e_i^T - e_i e_i^T - e_j e_j^T \in \mathbb{S}$, and thus $f_i=f_j$. On the other hand, for every $y \in \R^\nu$ and using that $P g$ is a subgradient of $\Psi$ at $\sigma$ for all $P \in \mathbb{S}$,
\[  (y-\sigma)^T f = (\# \mathbb{S})^{-1} \sum_{P \in \mathbb{S}}(y-\sigma)^T Pg \leq (\# \mathbb{S})^{-1} \sum_{P \in \mathbb{S}} (\Psi(y)-\Psi(\sigma)) = \Psi(y)-\Psi(\sigma), \]
and hence $f$ is a subgradient of $\Psi$ at $\sigma$. 

Finally, we must show that $f \geq 0$ and non-increasing. If $\sigma_i>0$, let $P_i$ be the signed permutation whose only effect is changing the sign of the $i$-th
coordinate. Then, \(0=\Psi(P_i\sigma)-\Psi(\sigma)
 \geq (P_i\sigma-\sigma)^Tf
 =-2\sigma_i f_i\), and hence $f_i\geq0$. To conclude, suppose $\sigma_i>\sigma_j$, and let $P_{ij}$ be the permutation that swaps the $(i,j)$
coordinates leaving everything else unchanged. By a similar argument,
\( 0=\Psi(P_{ij}\sigma)-\Psi(\sigma)
 \geq (P_{ij}\sigma-\sigma)^Tf
 =(\sigma_i-\sigma_j)(f_j-f_i)\) implying $f_i\geq f_j$.
\end{proof}

To state Proposition~\ref{prop:genodd}, for an odd real polynomial $p(z) = \sum_{j=0}^k p_j z^{2j+1} \in \R[z]$ and a matrix $X \in \C^{m \times n}$, we define $\displaystyle p^\diamond(X)=\sum_{j=0}^{k} p_j (X X^*)^j X$. See \citet{Noferini2017} and the references therein for some properties of such ``generalized matrix functions".

\begin{proposition}\label{prop:genodd}
    Let $\| \cdot \|$ be a unitarily invariant norm, let $X \in \C^{m \times n}$, and let $G \in \C^{m \times n}$ be a subgradient of $\| \cdot \|$ at $X$ having the properties described in Lemma~\ref{lem:subgradient}. Then, there exists an odd real polynomial $p(z)$ such that $G = p^\diamond (X)$.
\end{proposition}

\begin{proof}
    If $X=0$ then $G=0$ and we can just take $p=0$. Otherwise, the argument is similar to \citet[Theorem 3.3]{Noferini2017}: Let $\sigma_{i_j}$, $j=1,\dots,q$, be the \emph{distinct positive} singular values of $X=U \Sigma V^*$ and let $f_{i_j}$ be the corresponding diagonal entries of $F$ in the permuted SVD $G=U F V^*$. It suffices to solve the interpolation problem $p(\sigma_{i_j}) = f_{i_j}$ with the extra constraints that $p(z)$ must be odd and have degree at most $2 q-1$. This is tantamount to solving for $x$ the linear system $x ^T M D = \begin{bmatrix}
        f_{i_1} & \dots & f_{i_q}
    \end{bmatrix}$. Here, the coefficient matrix $M D$ is the product of the Vandermonde matrix $M \in \R^{q \times q}$ at the distinct points $\sigma^2_{i_j}$ and the real diagonal matrix $D = \diagm(\sigma_{i_1},\dots,\sigma_{i_{q}}) \succ 0$. Hence, $MD$ is invertible and thus we can uniquely find the sought polynomial $p(z)$.
\end{proof}

\begin{theorem}\label{thm:bigthmminsubspaces}
    Let $\mathcal{E}= B + \mathcal{V}$ be a real affine subspace. Then, $\mathcal{E}$ has the strong UUM property if and only if $\mathcal{V}^\perp$ satisfies \eqref{eq:xxstarx}.
\end{theorem}
\begin{proof}
Suppose first that $\mathcal{E}$ has the strong UUM property. 
Let $\| \cdot \|_{S,4}$ denote the (unitarily invariant) Schatten $4$-norm,
\[  \| M \|_{S,4}^4 = \sum_{j=1}^{\nu} \sigma_j(M)^4 = \traccia((M^*M)^2).  \]
   Let now $\Delta \in \mathcal{V}^\perp$ and note that every point $X_\star \in \mathcal{E}$ minimizes over $X \in \mathcal{E}$ the function $F(X) = \frac{1}{4}\| A - X \|_{S,4}^4$, for all inputs $A$ of the form $A=X_\star+\Delta$. Indeed, $\Pi_\mathcal{E}(X_\star+\Delta)=X_\star$, and hence $X_\star$ is the minimizer of $\|A-X\|$ in the Frobenius norm and thus, by assumption, also in every other unitarily invariant norm. 
   By Theorem~\ref{thm:reg min in normal}, we must have $\Delta \Delta^* \Delta = -\nabla F(X_\star) \in N_{X_\star} \mathcal{E} = \mathcal{V}^\perp$.

    We now prove the converse implication. Let $A \in \C^{m \times n}$, let $\| \cdot \|$ be a unitarily invariant norm, and suppose that \eqref{eq:xxstarx} holds for $\mathcal{V}^\perp$. Let $X_\star \in \mathcal{E}$ be the unique minimizer of $\|A - X \|_F $ over $X \in \mathcal{E}$. 
    We may assume with no loss of generality that $\Delta:=A-X_\star$ is not zero, otherwise $X_\star = A \in \mathcal{E}$ and there is nothing else to prove. Hence, suppose that $\Delta$ has $q>0$ distinct positive singular values. Then, by Proposition~\ref{prop:genodd} and its proof, it follows that there exists a nonzero real odd polynomial $p(z) = \sum_{j=0}^{q-1} p_j z^{2j+1}$, with $p_0,\dots,p_{q-1} \in \R$ and such that $G:=p^\diamond(\Delta) \in \C^{m \times n}$ is a subgradient of $\| \cdot \|$ at $\Delta$. On the other hand, $\Delta$ is orthogonal to $\mathcal{V}$, and hence $\Delta \in \mathcal{V}^\perp$. It follows by Lemma~\ref{lem:oddstaysin} and by $\R$-linearity that $\displaystyle G = \sum_{j=0}^{q-1} p_j (\Delta \Delta^*)^j \Delta \in \mathcal{V}^\perp$. Then, by definition of subgradient, for every $X \in \mathcal{E}$ it holds
    \[ \|A-X\| - \| A - X_\star\| \geq  \langle G, X_\star-X \rangle= 0 \ \mathrm{because} \ G \in \mathcal{V}^\perp, X_\star-X \in \mathcal{V}. \]
    Therefore, $X_\star=A-\Delta$ minimizes $\| X-A \|$ over all $X \in \mathcal{E}$.
\end{proof}

\begin{remark}
   The proof of Theorem~\ref{thm:bigthmminsubspaces} reveals that, in practice, it suffices to only check two norms: A real affine subspace $B + \mathcal{V}$ has the strong UUM property if and only if, for every input $A$, the Frobenius projection $X_\star$ minimizes the distance from $A$ in the Schatten $4$-norm. 
\end{remark}

\begin{remark}\label{rem:nolabel}
    A natural question is whether a local analogue of Theorem~\ref{thm:bigthmminsubspaces} can be stated when $\mathcal{E}$ is a smooth manifold, by replacing the subspace $\mathcal{V}^\perp$ with the local normal space at the Frobenius minimizer(s) $X_\star$, but when the manifold is not flat the situation is more intricate.

The first part of the proof of Theorem~\ref{thm:bigthmminsubspaces}, showing that \eqref{eq:xxstarx} is necessary, extends locally to manifolds that are not affine subspaces, in the following somewhat restrictive sense. If a Frobenius projection $X_\star$ of $A \in \C^{m \times n}$ onto $\mathcal{E}$ is also a global minimizer for every unitarily invariant distance from $A$, then $\Delta:= A-X_\star$ satisfies
\[ \Delta \Delta^* \Delta \in N_{X_\star} \mathcal{E}.   \]
However, while this equation is satisfied by the matrix $\Delta$, it does not imply that the whole normal space satisfies \eqref{eq:xxstarx}. Moreover, the other direction is also problematic: Even if the whole normal space satisfies \eqref{eq:xxstarx}, the Frobenius projection
does not necessarily minimize every unitarily invariant distance from $A$. Consider for instance
    \[ \mathcal{E} = \left\{ \begin{bmatrix}
        8 z_1 & 3 z_1^2\\
        8 z_2 & 0
    \end{bmatrix} : z_1, z_2 \in \C \right\} \subseteq \C^{2 \times 2}.  \]
    Then, $0 \in \mathcal{E}$ and \[ (T_0 \mathcal{E})^\perp = \mathrm{span}_\C \left\{ \begin{bmatrix}
        0&1\\
        0&0
    \end{bmatrix}, \begin{bmatrix}
        0&0\\
        0&1
    \end{bmatrix} \right\} \]
    satisfies \eqref{eq:xxstarx}, because for every $z_3,z_4 \in \C$ it holds $\displaystyle \begin{bmatrix}
        0 & z_3\\
        0 & z_4
    \end{bmatrix} \left(\begin{bmatrix}
        0 & z_3\\
        0 & z_4
    \end{bmatrix}\right)^* \begin{bmatrix}
        0 & z_3\\
        0 & z_4
    \end{bmatrix} = (|z_3|^2 + |z_4|^2 ) \begin{bmatrix}
        0 & z_3\\
        0 & z_4
    \end{bmatrix}.$ By minimizing a quadratic function one can check that $X_\star=0$ is the unique Frobenius projection onto $\mathcal{E}$ of $A=\begin{bmatrix}
        0&8\\
        0&8
    \end{bmatrix}$. However, $X_1 = \begin{bmatrix}
        8 & 3\\
        0 & 0
    \end{bmatrix} \in \mathcal{E}$ and, in the spectral norm,
    \[ 10.88 \approx  \frac{5 + \sqrt{281}{}}{2} = \| A -X_1 \|_2  < \|A-X_\star\|_2 = \| A \|_2 = 8 \sqrt{2} \approx  11.31.                \]

    In the next section, namely in Theorem~\ref{thm:sunui}, we
will characterize the weak UUM property
for unitarily invariant sets.
\end{remark}

\section{Problems solvable by direct methods}\label{sec:directmethods}

Many matrix nearness problems require iterative optimization methods to be approached numerically, but some subclasses display symmetries that allow us to apply a direct method. By this, we do not necessarily mean that a closed formula exists, but that the problem can be decoupled based on a classic matrix decomposition such as the SVD\footnote{Of course, the SVD is itself typically computed by iterative methods \cite{GolubVLoan, TrefethenBau1997}, but given the general difficulty of matrix nearness problems we take the liberty of viewing the numerical solution of an eigenvalue problem as a ``basic" task. When estimating operation counts, we will assume that the
SVD of an $m \times n$ matrix can be numerically approximated, in a fixed working precision, in
$\mathcal{O}(mn \nu )$ flops; see \citet[Section~8.6]{GolubVLoan} for a more thorough discussion.
}. In this section, we now return to focus on the Frobenius distance and, following
\citet{Higham1989}, we begin our discussion of
methods to solve nearness problems with this numerically friendlier class. We split into four subclasses: problems whose target set is a real affine subspace, those whose target set is completely determined by properties of singular values, those where the row and column spaces of the target matrix are prescribed, and a specific problem where both the input and the target matrix are assumed to be normal.

\subsection{Affine subspaces} \label{subsec:affine}

When the target set is an affine subspace of $\F^{m\times n}$, finding the minimizer reduces to a standard least-squares problem after vectorization. 

\begin{theorem} \label{thm:affinedist}
Consider the affine subspace $\mathcal{E} = B + \mathcal{V}$, where $B\in \F^{m\times n}$ and the vector subspace $\mathcal{V}$ has basis $P_1,\dots, P_\ell$. Then, setting 
\begin{equation} \label{defP}
\mathbf{P} = \begin{bmatrix}
    \mvec(P_1) & \cdots & \mvec(P_\ell)
\end{bmatrix} \in \F^{mn\times \ell},
\end{equation}
the projection $X_\star = \Pi_{\mathcal{E}}(A)$ is obtained as
\begin{equation} \label{xstarlstsq}
\mvec(X_\star)  = \mvec(B) + \mathbf{P}\mathbf{P}^\dagger \mvec(A-B),
\end{equation}
or equivalently $\Delta_\star = A - X_\star = \Pi_{\mathcal{V}^\perp}(A-B)$, i.e.,
\begin{equation} \label{deltastarlstsq}
\mvec(\Delta_\star) = (I - \mathbf{P}\mathbf{P}^\dagger) \mvec(A-B).    
\end{equation}
Alternatively, we can specify $\mathcal{E}$ via linear equations: $\mathcal{E} = \{X \in \F^{m\times n} \colon  \mathbf{M} \mvec(X) = b\}$, for given $\mathbf{M}\in\F^{\ell\times mn}, b\in\range(\mathbf{M})$. Then, $X_\star = A - \Delta_\star$, with
\begin{equation} \label{deltastarlstnrm}
    \mvec(\Delta_\star) = \mathbf{M}^\dagger(\mathbf{M}\mvec(A)-b).    
\end{equation}
\end{theorem}
\begin{proof}
    These formulas follow from classic results on least-squares problems; see e.g. \citet[Section 6.2.3]{GolubVLoan}, \citet[Theorem~1.2.10]{Bjoerck1996}.
\end{proof}

The formulae~\eqref{defP}--\eqref{deltastarlstnrm} can be turned into direct algorithms by means of classic matrix factorizations, such as the SVD or the QR decomposition; see, e.g., \citet[Sections~5 and~6]{GolubVLoan}. However, the cost of doing so might be prohibitive unless $\ell$ is small; for instance, computing the SVD of the full column rank matrix $\mathbf{P}$ with standard dense algorithms costs $\mathcal{O}(mn\ell^2)$ operations. If we assume for concreteness that $\ell\approx mn$, the cost of dense algorithms is in the worst case $\mathcal{O}(m^3n^3)$, or $\mathcal{O}(n^6)$ if $m=n$. It should however be noted that a special structure of the coefficient matrices $\mathbf{P}$ or $\mathbf{M}$, such as the Kronecker product structure in Theorem~\ref{thm:sun}, can often be exploited to reduce these costs. Alternatively, one can use algorithms for large-scale least-squares problems, e.g., Krylov methods~\cite{Saad2003} (for instance LSQR~\cite{PaigeSaunders}), or randomized numerical linear algebra~\cite{HMT2011} (for instance Blendenpik~\cite{blendenpik, avron2026blendenpik-matlab}). These two algorithms can solve both overdetermined linear systems (as we need in~\eqref{xstarlstsq} and \eqref{deltastarlstsq}) and underdetermined ones (as we need in~\eqref{deltastarlstnrm}).

For many common vector spaces, it is easy to produce natural bases, even orthonormal ones, simplifying the computation.
\begin{example}
    Let $\mathcal{V} \subseteq \F^{3\times 3}$ be the vector subspace of upper triangular Toeplitz matrices. The vectorizations of the three matrices
    \[
    M_1 = \begin{bmatrix}
        1 & 0 & 0\\
        0 & 1 & 0\\
        0 & 0 & 1
    \end{bmatrix}, \quad M_2 = \begin{bmatrix}
        0 & 1 & 0\\
        0 & 0 & 1\\
        0 & 0 & 0
    \end{bmatrix}, \quad
    M_3 = \begin{bmatrix}
        0 & 0 & 1\\
        0 & 0 & 0\\
        0 & 0 & 0
    \end{bmatrix}
    \]
    are orthogonal, since their nonzero entries are in different positions. Then, an orthonormal basis for $\mathcal{V}$ is obtained by normalizing them: $P_1 = \frac{1}{\sqrt{3}}M_1, P_2 = \frac{1}{\sqrt{2}}M_2, P_3 = M_3$. Hence $\mathbf{P}^*\mathbf{P}=I_3$ and $\mathbf{P}^\dagger = \mathbf{P}^*$, so~\eqref{xstarlstsq} reduces to
    \[
            X_\star = P_1 \traccia(P_1^*A)  + P_2 \traccia(P_2^*A) + P_3 \traccia(P_3^*A) = \begin{bmatrix}
        \frac{a_{11}+a_{22}+a_{33}}{3} & \frac{a_{12}+a_{23}}{2} & a_{13}\\
        0 & \frac{a_{11}+a_{22}+a_{33}}{3} & \frac{a_{12}+a_{23}}{2}\\
        0 & 0 & \frac{a_{11}+a_{22}+a_{33}}{3}
    \end{bmatrix}.
    \]
\end{example}

\begin{example}
    Let $\mathcal{E} = \mathcal{V}$ be the vector space of sparse matrices with nonzero pattern $\mathcal{N} \subseteq \{1,\dots,m\} \times\{1,\dots,n\}$; i.e., $\mathcal{V} = \{X \in \mathbb{C}^{m\times n} \colon \text{$X_{ij}=0$ whenever $(i,j)\not\in \mathcal{N}$}\}$. Then, an orthonormal basis for $\mathcal{V}$ is the family of matrices with a single $1$ entry $(e_i e_j^T \colon (i,j)\in\mathcal{N})$. The projection $X_\star = \Pi_{\mathcal{E}}(A)$ is obtained by keeping only the entries of $A$ that belong to the pattern $\mathcal{N}$.
\end{example}

In the case where $\mathcal{E}$ is only a \emph{real} subspace of $\C^{m\times n}$ without a complex structure, analogues of~\eqref{defP}--\eqref{deltastarlstnrm} hold after replacing all the symbols $\mvec(\cdot)$ with $\cvec(\cdot)$; in that case the dimensions become $\mathbf{P} \in \R^{2mn \times \ell}$, $\mathbf{M}\in \R^{\ell\times 2mn}$.

Perhaps the most remarkable example of projections onto a real subspace of $\mathbb{C}^{mn}$ is the following, where formulas can be obtained directly.
\begin{example}
    Given $A \in \mathbb{C}^{n\times n}$, let
    \[
    H = \frac{A+A^*}{2}, \quad K = \frac{A-A^*}{2}
    \]
    be its \emph{Hermitian} and \emph{skew-Hermitian part}, respectively. Then, $H$ is the projection of $A$ onto the (real) subspace of Hermitian matrices, and $K$ is its projection onto the (real) subspace of skew-Hermitian matrices. 
    
    One can prove this fact directly by showing that $\scalar{H,K}=0$ and hence $A = H+K$ is a decomposition into two orthogonal components; we also proved it in Theorem~\ref{thm:FH0} and, in more generality, in Theorem~\ref{thm:pinching applications}.
\end{example}

\subsection{Unitarily invariant sets}\label{sec:uisets}

Another especially convenient class of target sets $\mathcal{E}$ consists of those that can be described entirely by the allowed singular values of their elements. Let us start from a formal definition.

 \begin{definition}\label{def:unitarilyinvariantset}
      A closed set $\mathcal{E} \subseteq \C^{m \times n}$ is called \emph{unitarily invariant} if $\mathcal{E} \neq \varnothing$ and $X \in \mathcal{E} \Rightarrow UXV^* \in \mathcal{E}$ for every pair of unitary matrices $U \in \C^{m \times m},V \in \C^{n \times n}$.
    \end{definition}

Definition~\ref{def:unitarilyinvariantset} immediately implies that singular values alone determine membership of a unitarily invariant set. Therefore, denoting by $\sigma(M)$ the vector of (ordered) singular values of the matrix $M$, a unitarily invariant set always takes the form $\mathcal{E} = \{ X \in \C^{m \times n} : \sigma(X) \in \mathcal{E}_\sigma \}$ for some nonempty closed set $\mathcal{E}_\sigma \subseteq \R^{\nu}$ (where we recall that $\nu=\min\{m,n\}$); we henceforth implicitly assume that $\mathcal{E}_\sigma$ is only allowed to contain nonnegative vectors ordered in a non-increasing manner, and (except in formal theorem statements) we explicitly specify only further constraints. We record some notable instances in Example~\ref{ex:456}.
\begin{example}\label{ex:456}
Among matrix nearness problems whose target set is unitarily invariant, we may mention:
\begin{itemize}
    \item \textbf{Low rank approximation.} The set of matrices having rank at most $r$, that is, $\mathcal{E}=\{ X \in \C^{m \times n} : \rank(X) \leq r\}$  is unitarily invariant. It is associated with the singular values set $\mathcal{E}_\sigma=\{ \sigma \in \R^{\nu } :\sigma_i=0 \ \forall \ i > r \} $.
    \item \textbf{Orthogonal Procrustes problem.}  Assuming $m \geq n$, the set of matrices with orthonormal columns $\mathcal{E} = \{  X \in \C^{m \times n} : X^* X = I \}$  is unitarily invariant; finding the nearest matrix in $\mathcal{E}$ is a special case of the so-called \emph{orthogonal Procrustes problem} \cite{Higham1986}. The corresponding singular values set consists of a single element, the vector of all ones: $\mathcal{E}_\sigma=\{ \sigma \in \R^{\nu } :\sigma_i=1 \ \forall \ i\} $.
    \item \textbf{Nearest matrix with a given norm.}   Given a unitarily invariant norm $\| \cdot \|$ and a constant $c \geq 0$, the problem of finding the nearest (in the Frobenius distance) matrix whose such norm is $c$ is equivalent to projecting onto the unitarily invariant set  $\mathcal{E} = \{ X \in \C^{m \times n} : \| X \| = c \}$. This set is associated with the singular value set $\mathcal{E}_\sigma=\{ \sigma \in \R^{\nu } :\Psi(\sigma)=c \}$, where $\Psi$ is the symmetric gauge function corresponding to the norm $\| \cdot \|$.
\end{itemize}
\end{example}

Similarly to affine sets, distance problems from unitarily invariant sets always admit a solution based on the SVD. To state Theorem~\ref{thm:minimizerforunitinv}, we introduce the notion of an orthogonally invariant subset of $\R^{m \times n}$, i.e., $X \in \mathcal{E} \Rightarrow Q X Z^T \in \mathcal{E}$ for every pair of real orthogonal matrices $Q,Z$ of compatible sizes. Moreover, every SVD in the statement and proof of Theorem~\ref{thm:minimizerforunitinv} is taken over $\F$.

    \begin{theorem}\label{thm:minimizerforunitinv}
        Let $\mathcal{E}=\{ X \in \F^{m \times n} : \sigma(X) \in \mathcal{E}_\sigma \}$ be either a unitarily invariant set (for $\F=\C$) or an orthogonally invariant set (for $\F=\R$), where $\mathcal{E}_\sigma \subseteq \R^{\nu}$ is a nonempty closed set only containing nonnegative and non-increasingly ordered vectors. Let $A \in \F^{m \times n}$ and let $\Sigma \in \R^{m \times n}$ be the diagonal matrix whose diagonal elements  are the singular values of $A$, ordered non-increasingly. Then:
        \begin{enumerate}
            \item If $\sigma_\star \in \mathcal{E}_\sigma$ is a minimizer of the vector distance problem
            \begin{equation}\label{eq:vectorproblem}
                \sigma_\star \in \mathrm{argmin}_{x \in \mathcal{E}_\sigma} \|x-\diag(\Sigma)\|_2^2,
            \end{equation}
        and $A=U \Sigma V^*$ is an arbitrary SVD of $A$ over $\F$, then $X_\star = U \diagm(\sigma_\star) V^*$ is a minimizer of the matrix distance problem
\begin{equation}\label{eq:matrixproblem}
    X_\star \in \mathrm{argmin}_{X \in \mathcal{E}} \|X-A\|_F^2.
\end{equation}
        \item Conversely, if $X_\star$ is a minimizer of the matrix distance problem \eqref{eq:matrixproblem}, then its vector of ordered singular values $\sigma_\star$ minimizes the vector distance problem \eqref{eq:vectorproblem} and there exist either unitary (if $\F=\C$) or real orthogonal (if $\F=\R$) matrices $U,V$ such that $X_\star=U\diagm(\sigma_\star)V^*$ and $A=U \Sigma V^*$.
        \end{enumerate}
    \end{theorem}

    \begin{proof}
    \begin{enumerate}
        \item  Suppose first that $\sigma_\star$ minimizes the vector distance problem \eqref{eq:vectorproblem}. Then,
       \[ \min_{X \in \mathcal{E}} \| A - X \|_F^2 = \min_{Y=U^* X V \in \mathcal{E}} \| \Sigma - Y \|_F^2 = \| \Sigma\|_F^2 + \min_{Y \in \mathcal{E}} \left( \| Y \|_F^2  -2 \Re \traccia(\Sigma^T Y)\right) \]
        Denoting by $y \in \mathcal{E}_\sigma$ the (non-increasingly ordered) vector of singular values of $Y$, by Von Neumann's trace inequality \cite[Problem 10, Section 3.3]{HJtopics} we have $\Re \traccia(\Sigma^T Y) \leq \diag(\Sigma)^T y$. Therefore,
       \[\min_{X \in \mathcal{E}} \| A - X \|_F^2 \geq \min_{y \in \mathcal{E}_\sigma}  \| \diag(\Sigma)-y \|_2^2 = \| \diag(\Sigma)-\sigma_\star\|_2^2,   \] 
       and this lower bound for the minimum is attained by taking $X=U \diagm(\sigma_\star) V^*$, thus proving that the latter matrix minimizes \eqref{eq:matrixproblem}. 
       \item Conversely, suppose that the matrix $X_\star$, having ordered singular value vector $\sigma_\star$, minimizes the matrix distance problem \eqref{eq:matrixproblem}. Take an arbitrary SVD $A=U_A \Sigma V_A^*$. Then, the inequality $\Re \traccia(\Sigma^T U_A^* X_\star V_A)  \leq \diag(\Sigma)^T \sigma_\star$ must be attained as an equality; by \citet[Theorem 2.1]{Lewis1995}, this implies the existence of the unitary matrices $U,V$ having the desired property. Suppose now that $\sigma_\star$ does not minimize \eqref{eq:vectorproblem}, and let $\hat{\sigma}$ be an ordered vector such that $\| \hat{\sigma}-\diag(\Sigma)\|_2 < \| \sigma_\star - \diag(\Sigma)\|_2$. Then, by following the same arguments in the proof of item $1$, we reach the contradiction $\| U \diagm(\hat{\sigma})V^* -A \|_F < \| X_\star - A \|_F$.     
    \end{enumerate}
    \end{proof}

    Theorem~\ref{thm:minimizerforunitinv} is a powerful result that reduces a matrix nearness problem to a potentially much lower-dimensional one on vectors. Furthermore, its applicability goes even beyond unitarily invariant sets, as other types of inverse problems can sometimes be reduced to unitarily invariant sets. A striking example is the nearest Kronecker product problem: Given $A\in\mathbb{C}^{mn\times pq}$, find the nearest matrix of the form $B\otimes C$, with $B\in\mathbb{C}^{m\times p}, C\in\mathbb{C}^{n\times q}$. \citet{VLP1993} show that for this problem, even if the target set is not unitarily invariant itself, a clever isometry can be used to reshape the input and target matrices. The ultimate result of this trick is that the problem is shown to be equivalent (exclusively in the Frobenius distance) to the nearest matrix of rank at most $1$, whose target set is unitarily invariant.

    Given its importance, we now give some further remarks on Theorem~\ref{thm:minimizerforunitinv}, of both theoretical and algorithmic nature.

\begin{remark}
     Unlike affine sets, unitarily invariant sets need not be convex, and hence uniqueness is not guaranteed for the matrix minimizer $X_\star$ in Theorem~\ref{thm:minimizerforunitinv}. This does not contradict Theorem~\ref{thm:minimizerforunitinv}, for two reasons. First, $\mathcal{E}_\sigma$ may in turn not be convex and hence the minimizer $\sigma_\star$ of the vector distance problem need not be unique itself. Second, even if $\sigma_\star$ is unique, the non-uniqueness of the SVD can potentially create multiple matrix minimizers $X_\star$ sharing the same singular values. The latter phenomenon can be illustrated by the Procrustes problem when the input matrix $A$ is rank deficient; the minimizer of the vector problem is always unique and equal to the vector of all ones but still there can be many matrix minimizers. For the most extreme example, when $m=n$ and $A=0$, observe that every unitary matrix $U$ is at exactly the same Frobenius distance $\| U-A \|_F=\sqrt{n}$.
\end{remark}
   
\begin{remark}
Theorem~\ref{thm:minimizerforunitinv} provides a very elegant reduction from either $2mn$ ($\F=\C$) or $mn$ ($\F=\R$) to $\nu$ real dimensions. Nonetheless, it is worth emphasizing that solving the reduced vector distance problem is not necessarily easy. Thus, matrix distance problems to unitarily invariant sets are solvable by direct methods \emph{provided} that we can solve the corresponding vector distance problem by direct methods. This can be very far from being true. For example, one can take $\mathcal{E}_\sigma$ to be the intersection of a submodule of $\mathbb{Z}^{\nu}$ with the set of ordered nonnegative real vectors and argue that even computing an approximate, but possibly suboptimal, solution is NP-hard \cite{DKRS}. On the other hand, for many problems of practical interest, the vector projection turns out to be much easier; for example, here is how to solve the three problems listed in Example~\ref{ex:456}:
\begin{enumerate}
    \item To compute the nearest matrix of rank at most $r$ one simply sets to zero the singular values beyond the $r$-th;
    \item In the Procrustes problem, $\mathcal{E}_\sigma$ contains only one element and hence the projection is trivial. Note, however, that a nearest matrix in $\mathcal{E}$ is not necessarily unique, e.g., every unitary matrix is at the same distance from the zero matrix;
    \item The matrix nearest to $A \neq 0$ and of Frobenius norm $c$ is $X_\star=(c / \| A\|_F)  A$ (if $A=0$, the problem is trivial in the sense that every matrix in the feasible set is a minimizer). For a more general unitarily invariant norm one needs to project the singular values of $A$ onto the ``sphere" $\Psi(x)=c$. For example, if $\mathcal{E}$ is given by spectral norm constraint $\| X \|_2=c$, the associated symmetric gauge function is the vector infinity norm $\Psi(x)=\| x \|_\infty$, and hence the projection is $\sigma_1 \mapsto c$ and $\sigma_i \mapsto \min(c,\sigma_i)$ for $i>1$. 
\end{enumerate}
\end{remark}

Another remarkable feature of unitarily invariant sets is that it is relatively simple to characterize when they possess the weak UUM property defined in Section~\ref{sec:uinorms}, thus explaining why certain classic results \cite{Schmidt1907, EY36, FH55, Mirsky60, Higham1986} hold simultaneously for every unitarily invariant norm. We  record this in Theorem~\ref{thm:sunui}; to state it, recall, e.g. from \citet[Chapter~II]{bhatia_book}, that the nonnegative non-increasingly ordered vector $x \in \R^n$ weakly majorizes the nonnegative non-increasingly ordered vector $y \in \R^n$ if $\sum_{i=1}^k x_i \geq \sum_{i=1}^k y_i$ for all $k=1,\dots,n$.

\begin{theorem}\label{thm:sunui}
    Let $\mathcal{E}=\{ X \in \C^{m \times n} : \sigma(X) \in \mathcal{E}_\sigma \}$ be a unitarily invariant set, where $\mathcal{E}_\sigma \subseteq \R^{\nu}$ is a nonempty closed set only containing nonnegative and non-increasingly ordered vectors. Then, $\mathcal{E}$ has the weak UUM property if and only if, for every vector of (non-increasingly ordered) singular values $\sigma \in \R^{\nu}$, there exists $\sigma_\star \in \mathcal{E}_\sigma$ (possibly depending on $\sigma$) such that, for all $y \in \mathcal{E}_\sigma$ and after reordering the two following vectors nonincreasingly, $|\sigma-y|$ weakly majorizes $|\sigma-\sigma_\star|$.
\end{theorem}

\begin{proof}
 Let $A=U\diagm(\sigma)V^*$ be an SVD, and for every $y\in\mathcal{E}_\sigma$ consider the matrix valued function
    \[ F: \mathcal{E}_\sigma \rightarrow \C^{m \times n}, \qquad y \mapsto F(y)= U \diagm(y) V^*.\]
For every $X\in\mathcal{E}$ with singular values equal to the elements of the vector $y$
and for every unitarily invariant norm,
by \citet[Theorem~5]{Mirsky60} we have
\begin{equation}\label{eq:mirsky}
 \|A-F(y)\|
 = \|\diagm(\sigma-y)\|
 \leq \|A-X\|.
\end{equation}
We first prove sufficiency. If the weak majorization assumption holds, then Ky Fan's
domination theorem \cite{KyFan51} (see also the proof of Theorem~\ref{thm:dom}) and \eqref{eq:mirsky} yield $\|A-F({\sigma_\star})\|\leq\|A-F(y)\|$ for every $y\in\mathcal{E}_\sigma$ and every unitarily invariant norm. Hence, $F({\sigma_\star})$ minimizes the distance from $A$ induced by an arbitrary unitarily invariant norm.

Conversely, suppose that $X_\star$ is a common minimizer for every unitarily invariant distance, and let $\sigma_\star$ be the vector whose elements are the singular values of $X_\star$, ordered non-increasingly. Then, for every $y \in \mathcal{E}_\sigma$ and for every unitarily invariant norm
\[ \|\diagm(\sigma-\sigma_\star)\| \leq \|A-X_\star\| \leq \|A-F(y)\| = \| \diagm(\sigma-y)\|.\]  
In particular, we can specialize the inequality to the Ky Fan norms $\|\cdot\|_{(i)}$, for $i=1,\dots,\nu$. This proves necessity.
\end{proof}

We illustrate Theorem~\ref{thm:sunui} with four examples of unitarily invariant sets, of which the first three have the weak UUM property, but the last one does not. The first two examples are once again a reprise of Example~\ref{ex:456}.

\begin{example}\label{ex:seym}
   Let $\mathcal{E}=\{ X \in \C^{m \times n} : \rank(X) \leq r \}$ be the set of matrices of rank at most $r$. Recall from Example~\ref{ex:456} that $\mathcal{E}$ is unitarily invariant, with $\mathcal{E}_\sigma$ being defined by the constraint that the bottom $\nu-r$ elements must be zero. Given an input matrix $A$ with singular values $\sigma_1(A) \geq \sigma_2(A) \geq \dots \geq \sigma_{\nu}(A)$, the vector distance minimizer is obviously $\sigma_\star=\begin{bmatrix}
    \sigma_1(A) & \dots & \sigma_r(A) & 0 & \dots & 0
\end{bmatrix}^T$. Moreover, minimality of $|\sigma(A)-\sigma_\star|$ in the sense of weak majorization is manifest, because this vector is minimal elementwise:
\begin{equation}\label{eq:seym}
     |\sigma(A)-\sigma_\star| = \begin{bmatrix}
    0\\
    \vdots\\
    0\\
    \sigma_{r+1}(A)\\
    \vdots\\
    \sigma_{\nu}(A)
\end{bmatrix} \leq  \begin{bmatrix}
    |y_1-\sigma_1(A)|\\
    \vdots\\
    |y_r-\sigma_r(A)|\\
    \sigma_{r+1}(A)\\
    \vdots\\
    \sigma_{\nu}(A)
\end{bmatrix}=|\sigma(A)-y| \ \forall \ y \in \mathcal{E}_\sigma  .
\end{equation}
Via Theorem~\ref{thm:sunui}, \eqref{eq:seym} yields a proof of the celebrated Schmidt--Eckart--Young--Mirsky Theorem \cite{Schmidt1907, EY36, Mirsky60}.
 \end{example} 

\begin{example}
    In the special case of the orthogonal Procrustes problem discussed in \ref{ex:456}, $\# \mathcal{E}_\sigma =1$, and hence Theorem~\ref{thm:sunui} trivially applies, showing that every Frobenius projection universally minimizes the distance for every unitarily invariant norm.
 \end{example}

\begin{example}
    Another slightly more complicated structure is the following one from~\citet{DPRV}.

    For a fixed $r < n=m$, let $\mathcal{E} = \{Q + UV^* \colon Q^*Q = I_n,\, U,V\in\F^{n\times r} \}$ be the set of unitary-plus-rank-$\leq r$ matrices. \citet[Theorem~3]{DPRV} show that $A\in \mathcal{E}$ if and only if $A$ has at most $r$ singular values strictly larger than $1$ and at most $r$ strictly smaller; hence the associated singular value set is
    $\mathcal{E}_\sigma = \{\sigma \colon \text{$\sigma_{r+1} \leq 1$ and $\sigma_{n-r} \geq 1$} \}$.
    For a given $A = U \diagm(\sigma) V^*$, the vector distance minimizer $\sigma_\star$ has elements
    \[
    (\sigma_\star)_j = \begin{cases}
        \max(\sigma_j, 1) & j \leq \min(r, n-r),\\
        \sigma_j & n-r < j \leq r,\\
        1 & r < j \leq n-r,\\
        \min(\sigma_j, 1) & j > \max(r, n-r).
    \end{cases}
    \]
    (Note that the condition in the second line takes effect only when $n<2r$, and the one in the third only when $n>2r$.) As in Example~\ref{ex:seym}, $\sigma_\star$ is an elementwise optimum: $\abs{(\sigma_\star)_j - \sigma_j} \leq \abs{y_j-\sigma_j}$ for each $j$ and $y \in \mathcal{E}_\sigma$; hence, $X_\star = U \diagm(\sigma_\star)V^*$ is a UUM.

    This result has an analogue for the set $\mathcal{E} = \{H + UV^* \colon H=H^*,\, U,V\in\F^{n\times r}\}$ of Hermitian-plus-rank-$\leq r$ matrices. Let $S_A=S_A^*$ be the matrix such that $iS_A$ is the anti-Hermitian part of $A$, i.e., $S_A = \frac{A-A^*}{2i}$, and take an eigenvalue decomposition $S_A = W\diagm(\lambda_1,\dots,\lambda_n)W^*$, where the $\lambda_j$ are sorted non-increasingly. \citet[Theorem~12]{DPRV} show that $A\in\mathcal{E}$ if and only if $\lambda_{r+1}\leq 0 \leq \lambda_{n-r}$, and by a similar argument (see also Section~\ref{subsec:spectral normal} below) $S_\star = W\diagm(\lambda_\star)W^*$ with
    \[
    (\lambda_\star)_j = \begin{cases}
        \max(\lambda_j,0) & j \leq \min(r,n-r),\\
        \lambda_j & n-r < j \leq r,\\
        0 & r < j \leq n-r,\\
        \min(\lambda_j,0) & j > \max(r,n-r)
    \end{cases}
    \]
    is a nearest matrix to $S_A$ satisfying $\lambda_{r+1}\leq 0 \leq \lambda_{n-r}$. If we set $\Delta_\star = i(S_A-S_\star)$, then $X_\star = A - \Delta_\star \in \mathcal{E}$, since $S_{X_\star} = S_\star$, and moreover for any other $X \in \mathcal{E}$ and every unitarily-invariant norm we have
    \[
    \norm{A-X} \geq \norm{S_{A-X}} = \norm{S_A - S_X} \geq \norm{S_A - S_{X_\star}} = \norm{\Delta_\star}.
    \]
    Hermitian-plus-rank-$r$ and unitary-plus-rank-$r$ matrices appear in the study of short recurrences for Krylov spaces (see~\citet{Liesen2008} and the references therein), and in structured algorithms for eigenvalue problems \cite{eidelman_book,aurentz_book}.
\end{example}

\begin{example}
    Consider the unitarily invariant set \[ \C^{2 \times 2} \supseteq \mathcal{E} := \{ 3 Q : Q \in \C^{2 \times 2}, Q^*Q=I \} \cup \{ 4 uv^* : u,v \in \C^2, \| u \|_2 = \|v\|_2=1  \},\] which is associated with the singular values set $\mathcal{E}_\sigma = \left\{ \begin{bmatrix}
        4\\
        0
    \end{bmatrix}, \begin{bmatrix}
        3\\
        3
    \end{bmatrix}  \right\} \subseteq \R^2$. To see that $\mathcal{E}$ does not have the weak UUM property, take the input $A=0$. Then, for every $Q,u,v$ as in the definition of $\mathcal{E}$, we have in the spectral norm \(  \| 3 Q\|_2=3 < 4 = \| 4 u v^*\|_2   \) but in the Frobenius norm \( \| 4 u v^*\|_{F}= 4 < 3 \sqrt{2} =\|3 Q \|_{F}   \).
 \end{example}

\subsection{Prescribed row and column spaces}
Given two matrices $B\in\mathbb{C}^{m\times p}$, $C\in\mathbb{C}^{q\times n}$, consider the subspace 
\[
\mathcal{E}_{B,C} = \{BYC \colon Y \in \mathbb{C}^{p\times q} \}\subseteq \mathbb{C}^{m\times n}.
\]
Despite the similar-looking definition, this structure is not a special case of~\eqref{sunstructure}; rather, it is the orthogonal complement of a set in the form~\eqref{sunstructure}. One can see that $X \in \mathcal{E}_{B,C}$ if and only if $\range(X)  \subseteq \range(B)$ and $\range(X^*) \subseteq \range(C^*) $: The nontrivial inclusion follows from noting that $BB^\dagger$  and $C^\dagger C$ are the orthogonal projections onto $\range(B) $ and $\range(C^*) $, and if $X$ satisfies the two image inclusions then $BB^\dagger XC^\dagger C = X$.

The results on affine spaces in Section~\ref{subsec:affine} apply to $\mathcal{E}_{B,C}$, but one can obtain also a more compact formula for the nearest matrix $\Pi_{\mathcal{E}_{B,C}}(A)$ that does not require vectorization.
\begin{theorem}[\citet{Penrose56}]
    The (Frobenius-norm) projection of $A\in\mathbb{C}^{m\times n}$ onto $\mathcal{E}_{B,C}$ (i.e., the solution in the least-squares sense of $A=BYC$) is $BY_\star C$ with $Y_\star = B^\dagger AC^\dagger $, where $\cdot^\dagger $ denotes the Moore--Penrose pseudoinverse.
\end{theorem}
\begin{proof}
 This follows from facts that are now standard on vectorization, Kronecker products, and pseudoinverses: $\mvec(\mathcal{E}_{B,C}) = \range(C^\top \otimes B)$, so the solution of $(C^\top \otimes B)\mvec(Y)=\mvec(A)$  is given by
\[
\mvec(Y_\star ) = (C^\top \otimes B)^\dagger  \mvec(A) = ((C^\dagger )^\top \otimes B^\dagger )\mvec(A) = \mvec(B^\dagger AC^\dagger ).
\]
\end{proof}
\citet{friedlandtorokhti07, lim22, lim26} study extensions in which one adds additional constraints on $Y$: Given a nonempty closed set $\mathcal{Y}\subseteq \mathbb{C}^{p\times q}$, define
\begin{equation} \label{limstructure}
    \mathcal{E}_{B,C,\mathcal{Y}} = \{BYC \colon Y \in \mathcal{Y} \}\subseteq \mathbb{C}^{m\times n}.    
\end{equation}
It is interesting to note that $\mathcal{E}_{B,C,\mathcal{Y}}$ is not automatically closed, even if $\mathcal{E}_{B,C}$ and $\mathcal{Y}$ are, as shown in the following example, which is a variant of one in~\cite[p.~1721]{lim22}.

\begin{example}
Let $B=I_2$, $C=\begin{bsmallmatrix}
    0 & 0\\
    0 & 1
\end{bsmallmatrix}$, $\mathcal{Y}$ the closed cone of symmetric positive semidefinite (SPSD) matrices. Then,
\[
BYC = \begin{bmatrix}
    0 & y_{12}\\
    0 & y_{22}
\end{bmatrix}.
\]
For every $\varepsilon>0$ the matrix $A_{\varepsilon} = \begin{bsmallmatrix}
    0 & 1\\
    0 & \varepsilon
\end{bsmallmatrix}$ belongs to $\mathcal{E}_{B,C,\mathcal{Y}}$, since there is an SPSD matrix $Y$ with second column $\begin{bsmallmatrix}
    1\\
    \varepsilon
\end{bsmallmatrix}$, for instance $\begin{bsmallmatrix}
    \varepsilon^{-1} & 1\\
    1 & \varepsilon
\end{bsmallmatrix}$, but $A_{0}\not\in \mathcal{E}_{B,C,\mathcal{Y}}$, since no matrix with second column $\begin{bsmallmatrix}
    1\\0
\end{bsmallmatrix}$ is SPSD. This shows that $\mathcal{E}_{B,C,\mathcal{Y}}$ is not closed. In particular, for instance, in the distance problem $d\left(\begin{bsmallmatrix}
    0 & 1\\
    0 & 0
\end{bsmallmatrix}, \mathcal{E}_{B,C,\mathcal{Y}}\right)$ the minimum is not attained by a point in $\mathcal{E}_{B,C,\mathcal{Y}}$.
\end{example}

The following distance formula simplifies nearness problems on $\mathcal{E}_{B,C,\mathcal{Y}}$.
\begin{lemma}{\cite{lim22}}
    Let
    \[
    B = U_B \begin{bmatrix}
        S_B & 0\\
        0 & 0
    \end{bmatrix}V_B^*, \quad
    C = U_C \begin{bmatrix}
        S_C & 0\\
        0 & 0
    \end{bmatrix}V_C^*
    \]
    be SVDs, where we have chosen the block sizes so that $S_B, S_C$ are square invertible. Partition consistently
    \[
    U_B^*AV_C =
    \begin{bmatrix}
        A_{11} & A_{12}\\
        A_{21} & A_{22}
    \end{bmatrix}, \quad V_B^*YU_C = \begin{bmatrix}
        Y_{11} & Y_{12}\\
        Y_{21} & Y_{22}
    \end{bmatrix}.
    \]
    Then,
    \begin{equation} \label{limblockformula}
        \norm{A-BYC}_F^2 = \norm{A_{11}-S_B Y_{11}S_C}_F^2 + \norm{A_{12}}_F^2 + \norm{A_{21}}_F^2 + \norm{A_{22}}_F^2.    
    \end{equation}
\end{lemma}
\begin{proof}
    Since the Frobenius norm is orthogonally invariant, $\norm{A-BYC}_F^2 = \norm{U_B^*(A-BYC)V_C}_F^2$; now expand and use the block forms.
\end{proof}
In particular, this computation shows that the solution of the distance problem involves a condition on the sole block $Y_{11}$, and this condition depends only on the $A_{11}$ block of the input.

A particular distance problem in this category is studied by~\citet{friedlandtorokhti07}.
\begin{example}[\citet{friedlandtorokhti07,lim22}] \label{ex:friedlandtorokhti}
    Let $A,B,C$ be fixed, and $\mathcal{Y}$ be the space of matrices with rank at most $r$; we look for $d(A, \mathcal{E}_{B,C,\mathcal{Y}})$, i.e., the nearest matrix to $A$ of the form $BYC$, with $Y$ of rank at most $r$.
    
    Note that if $Y \in \mathcal{Y}$, then $S_B Y_{11} S_C$ has rank at most $r$, too. Hence, the term $\norm{A_{11} - S_B Y_{11} S_C}_F^2$ in~\eqref{limblockformula} is at least as large as $\norm{A_{11} - A_{11}^{(r)}}_F^2$, where $A_{11}^{(r)}$ is the best rank-$r$ (or lower) approximation of $A_{11}$ obtained by truncated SVD. The matrix $Y_{11}$ that achieves this minimum can be extended to a suitable $Y$ without increasing its rank if we take the blocks $Y_{12},Y_{21},Y_{22}$ to be zero, i.e.,
    \[
    Y_\star  = V_B \begin{bmatrix}
        S_B^{-1} A_{11}^{(r)} S_C^{-1} & 0\\
        0 & 0
    \end{bmatrix} U_C^*.
    \]
    This minimizer is unique if $\sigma_{r}(A_{11}) > \sigma_{r+1}(A_{11})$ (so that $A_{11}^{(r)}$ is unique) and $B,C^*$ have full column rank (so that the arbitrary blocks $Y_{12},Y_{21},Y_{22}$ are empty).
\end{example}
\citet{lim22,lim26} study the minima for several other sets $\mathcal{Y}$; in some cases the authors obtain closed formulas as in Example~\ref{ex:friedlandtorokhti}, in other cases algorithms based on alternating projections as in Section~\ref{subsec:vonneumann}.

\subsection{Spectral optimization on normal matrices} \label{subsec:spectral normal}

We conclude our overview of directly solvable problems with one that will smoothly lead us to the next subsection. Suppose that we restrict to a normal input $A$, having an eigendecomposition
\[  A=U\diagm(\lambda_1,\ldots,\lambda_n)U^*,
    \qquad U\in U(n),\]
and let $\Omega\subseteq\C$ be a nonempty closed set. Let now $\mathcal{E}=\{X \in \C^{n \times n} : X^*X=XX^*, \Lambda(X) \subseteq \Omega\}$. In other words, we seek the matrix nearest to $A$ (in the Frobenius distance) among the normal matrices whose eigenvalues all lie in $\Omega$. It turns out that this problem is completely equivalent to just projecting every eigenvalue of $A$ onto $\Omega$. Indeed,  by the Hoffman-Wielandt Theorem \cite[Theorem 6.3.5]{Horn2013},  the eigenvalues $\nu_1,\ldots,\nu_n$ of every feasible normal matrix $X$
can be ordered in such a way that
\[  \|A-X\|_F^2
    \geq \sum_{i=1}^n|\lambda_i-\nu_i|^2
    \geq \sum_{i=1}^n d(\lambda_i,\Omega)^2.\]To attain the lower bound, it suffices to choose
\[  \mu_i\in\arg\min_{z\in\Omega}|\lambda_i-z|,
    \qquad \forall \ i=1,\ldots,n \]
which is tantamount to picking as a minimizer the matrix
\[    X_\star=U\diagm(\mu_1,\ldots,\mu_n)U^*,
    \qquad
    \|A-X_\star\|_F^2
    =\sum_{i=1}^n d(\lambda_i,\Omega)^2.\]
Note that this argument only works because both the input and the target set are normal matrices. In contrast, in the next section we will analyze spectral nearness problems without these further constraints, and we will see that the situation becomes way more complicated. See also Section~\ref{sec:nearestnormal} for the problem of finding the normal matrix nearest to a general, i.e., not necessarily normal, input.

\section{Spectral nearness}\label{sec:spectralnearness}

In this section we consider \emph{spectral nearness problems}, in which the target set is defined by conditions on the eigenvalues of a square matrix. For a broader treatment, see also the survey by \citet{lewisoverton_acta} and the book by \citet{GuglielmiLubichBook}. Given a nonempty closed set $\Omega \subseteq \C$, two natural families of target sets arise:

\begin{equation} \label{eq:existsforall}
    \mathcal{E}_{\Omega,\exists} = \{X \in \C^{n \times n} : \Lambda(X) \cap \Omega \neq \varnothing\},    \qquad
    \mathcal{E}_{\Omega,\forall} = \{X \in \C^{n \times n} : \Lambda(X) \subseteq \Omega\}.
\end{equation}
\begin{proposition}\label{prop:bothclosed}
    Let $\Omega \subseteq \C$ be a nonempty closed set. Then, the sets $\mathcal{E}_{\Omega,\exists} \subseteq \C^{n \times n}$ and $
    \mathcal{E}_{\Omega,\forall} \subseteq \C^{n \times n}$, defined as in \eqref{eq:existsforall}, are both closed.
\end{proposition}
\begin{proof}
\begin{enumerate}
    \item Let $(X_k)_k \subset \mathcal{E}_{\Omega,\exists}$ be a converging subsequence such that $X_k \to X$. Then, for all $k$, there exists $\lambda_k \in \Lambda(X_k) \cap \Omega$. If more than one exists for some (or all) values of $k$, pick one arbitrarily and thus construct a sequence $(\lambda_k)_k \subset \Omega$. For sufficiently large $k$, we have $\| X_k - X \|_2 \leq 1$, and thus
    \[ |\lambda_k| \leq \| X_k \|_2 \leq \| X \|_2+1.  \]
    This shows that $(\lambda_k)_k$ is bounded. Therefore, using the closure of $\Omega$, there exists a convergent subsequence $(\lambda_{k_j})_j \to \lambda \in \Omega$. Since eigenvalues are continuous (as a set), we also have that $\lambda \in \Lambda(X)$, yielding the statement.
    \item  Take now $(X_k)_k \subset \mathcal E_{\Omega,\forall}$ with $X_k \to X$, and fix an arbitrary eigenvalue $\lambda\in\Lambda(X)$. Again by continuity of the spectrum as a set, there exist
$\lambda_k\in\Lambda(X_k)$ such that $\lambda_k\to\lambda$.
On the other hand, $\lambda_k\in\Omega$ for every $k$ and $\Omega$ is closed, and therefore 
$\lambda\in\Omega$. By repeating this argument for every eigenvalue of $X$,
we conclude that $\Lambda(X)\subseteq\Omega$.
\end{enumerate}

\end{proof}

The notation in \eqref{eq:existsforall} is suggestive of two contrasting interpretations as an ``$\exists$ problem'' and a ``$\forall$ problem''. In the former case the region $\Omega$ must be reached by at least one eigenvalue of the perturbed matrix, while in the latter case the whole spectrum must be confined within $\Omega$. Up to taking closures, the two classes of sets are related by taking set complements in the sense that we make precise in Proposition~\ref{prop:complement}. 
\begin{proposition}\label{prop:complement}
    Let $\Omega \subseteq \C$ be a nonempty closed set and define $\mathcal{E}_{\Omega,\exists}$ and $\mathcal{E}_{\Omega,\forall}$ as in \eqref{eq:existsforall}. Then,  
\[ \overline{\C^{n \times n} \setminus \mathcal{E}_{\Omega,\forall}}
   = \mathcal{E}_{\overline{\C \setminus \Omega},\exists}.\] 
\end{proposition}
\begin{proof}
    Let first $X \not \in \mathcal{E}_{\overline{\C \setminus \Omega},\exists}$. Then, $\Lambda(X) \subseteq \Omega^\circ$. Take now an arbitrary sequence $(X_k)_k \subset \C^{n \times n}$ converging to $X$.  Since $\Omega^\circ$ is open, for sufficiently large $k$ it must happen $\Lambda(X_k) \subseteq \Omega^\circ$, implying $X_k \not\in (\C^{n \times n} \setminus \mathcal{E}_{\Omega,\forall})$.

    Conversely, let $X \in \mathcal{E}_{\overline{\C \setminus \Omega},\exists}$, implying the existence of $\lambda \in \Lambda(X) \cap \overline{\C \setminus \Omega}$. Let $X=S J S^{-1}$ be a Jordan decomposition. Construct a sequence $(\lambda_k)_k \subset \C \setminus \Omega$ such that $\lambda_k \to \lambda$, and construct $J_k$ from $J$ by replacing $\lambda$ with $\lambda_k$. Then, for all $k$, $S J_k S^{-1} \not\in \mathcal{E}_{\Omega,\forall}$, but $(S J_k S^{-1}) \to X$.
\end{proof}

Below and throughout the paper, we call a square matrix \emph{Hurwitz stable} if all its eigenvalues have negative real part, and \emph{weakly Hurwitz stable} if all its eigenvalues have nonpositive real part (without further constraints on the Jordan blocks of imaginary eigenvalues). Note in particular that the set of weakly Hurwitz stable matrices is the closure of the set of Hurwitz stable matrices. Similarly, we call a matrix \emph{weakly Schur stable} if all its eigenvalues lie in the closed unit disk.

\begin{example}
    Let $\Omega = \{ z \in \C : \Re(z) \leq 0 \}$ be the closed left-half plane, so that $\mathcal{E}_{\Omega,\forall}$ is the set of weakly Hurwitz stable matrices. Then, $\C \setminus \Omega = \{ z \in \C : \Re(z) > 0 \}$ is the open right half plane, and $\mathcal{E}_{\overline{\C \setminus \Omega},\exists}$ is the set of matrices with at least one eigenvalue in the closed right half plane. Proposition~\ref{prop:complement} says that every matrix in the latter set is the limit of a sequence of matrices having at least one eigenvalue in the open right half plane. The subtlety here is that the two closed sets $\mathcal{E}_{\Omega,\forall}$ and
$\mathcal{E}_{\overline{\C\setminus\Omega},\exists}$ have nonempty intersection, e.g., $0$ belongs to both. Note that, for example, $\frac1k I_n$ is a sequence of non-(weakly) Hurwitz stable matrices that converges to $0$. 
\end{example}

   In spite of Proposition~\ref{prop:complement}, from the computational viewpoint the classes associated with the ``$\exists$ problem'' and the ``$\forall$ problem'' behave very differently. Indeed, as we shall see, distances to $\mathcal{E}_{\Omega,\exists}$ can be computed by reducing the problem to an optimization over the boundary $\partial\Omega$, which becomes a univariate problem whenever $\partial\Omega$ admits a smooth parametrization (as it does for the sets $\Omega$ most common in applications, such as half planes or discs). Instead, distances to $\mathcal{E}_{\Omega,\forall}$ do not admit any such reduction. Moreover, the minimizers often exhibit nongeneric features, such as defective eigenvalues $\lambda \in \partial \Omega$ of very high algebraic multiplicity \cite{Noferini2021a, NN25}.

Both sets in \eqref{eq:existsforall} are special cases in the class of sets that are invariant under similarity, i.e., that satisfy $X \in \mathcal{E} \Rightarrow SXS^{-1} \in \mathcal{E}$ for every $S \in GL(n,\C)$; for example, $\mathcal{E}=\{X \in \C^{3 \times 3} : X^2 = 0 \}$ belongs to the latter class but is not of either form in \eqref{eq:existsforall}. An even larger class is that of sets that are invariant by unitary similarity, that is, $X \in \mathcal{E} \Rightarrow UXU^* \in \mathcal{E}$ whenever $U \in \C^{n \times n}$ is unitary; for example, the set of normal matrices is invariant under unitary similarity but not under similarity.

We have seen in Theorem~\ref{thm:minimizerforunitinv} that, for target sets $\mathcal{E}$ that can be entirely characterized by singular values, or in other words that are invariant under unitary equivalence, the associated matrix nearness problem can be reduced to an optimization task on the vector of singular values.  Our readers might already have wondered whether an ``eigenvalue analogue" of Theorem~\ref{thm:minimizerforunitinv} holds when $m=n$ and the nonempty closed set $\mathcal{E}$ belongs to one of the classes mentioned above, and that all imply invariance under unitary similarity. More concretely, the question can be rephrased as follows: We know that for unitarily invariant sets it suffices to take an SVD of the input $A$ and to project singular values appropriately; does it similarly suffice, when the set is invariant under unitary similarity, to take a Schur form of the input $A$ and project eigenvalues appropriately? Unfortunately, eigenvalues are much less tame than singular values (also) in this sense, and as we shall see below the answer is generally negative even for the smaller subclass of sets described entirely by eigenvalues. A generalization of Theorem~\ref{thm:minimizerforunitinv} is only possible if rather restrictive further assumptions are made, say, $\mathcal{E}$ only contains normal matrices and $A$ itself is normal.

\subsection{Nearest $\Omega$-unstable matrix}

A typical example of a spectral nearness problem is the following: Given a closed ``target'' set $\Omega \subseteq \mathbb{C}$, we wish to find the nearest matrix to $A\in\mathbb{C}^{n\times n}$ which has an eigenvalue inside $\Omega$, i.e., solve the matrix nearness problem~\eqref{eq:matrixdistfrob} with
\begin{equation} \label{eq:E with an eigenvalue in Omega}
    \mathcal{E}_{\Omega} = \{X \in \mathbb{C}^{n\times n} : \text{$X$ has an eigenvalue in $\Omega$}\}.
\end{equation}

\begin{example}
    When $\Omega = \{0\}$, we are looking for the \emph{nearest singular matrix}, which is readily found using the Schmidt-Eckart-Young-Mirsky Theorem (Example~\ref{ex:seym}): In every unitarily invariant norm (normalized such that $\| e_1 e_1^T \|=1$),
\begin{equation} \label{eq:distance from singularity}
    \min \{\norm{\Delta} \colon \text{$A-\Delta$ is singular}\} = \sigma_{n}(A),
\end{equation}
the smallest singular value of $A$, and a minimizer $X=A-\Delta$ is given by taking the rank-1 matrix $\Delta_\star = u_{\min} \sigma_{\min} v_{\min}^*$ constructed using its associated singular vectors.

The
Gastinel-Kahan Theorem \cite[pp.~775--776]{Kahan66} provides a related characterization for another family of norms: If $A$ is nonsingular and for every matrix norm induced by a vector norm,
\[    \min\{\|\Delta\| : A-\Delta \text{ is singular}\}
    = \frac{1}{\|A^{-1}\|}.\]
Also in this setting, there always exists a rank-$1$ minimizing perturbation $\Delta_\star$.
 \end{example}
A more interesting problem comes from the theory of dynamical linear systems. We recall that a square matrix $A\in\mathbb{C}^{n\times n}$ is called \emph{Hurwitz stable} if its spectrum lies in the (open) left half-plane, i.e., $\Re \lambda < 0$ for each $\lambda \in \Lambda(A)$. 
\begin{example} \label{ex:distance Hurwitz instability unstructured}
    Given a Hurwitz stable matrix $A$, we look for its \emph{distance to (Hurwitz) instability}, i.e., take $\Omega$ equal to the closed right-half plane $ RHP = \{z\in \mathbb{C} \colon \Re z \geq 0\}$ in~\eqref{eq:E with an eigenvalue in Omega}. This distance is commonly called the \emph{stability radius} of $A$ \cite{Hinrichsen1986, Hinrichsen2011}; it has applicative meaning in systems theory, because it tells how much we can perturb $A$ before losing stability.
 \end{example}

\subsection{Relation to pseudospectra} \label{subsec:pseudospectra}

Some properties of these nearness problems can be better understood in terms of pseudospectra~\cite{TrefethenEmbreeBook}. The (closed) \emph{$\varepsilon$-pseudospectrum} of a square matrix $A\in\mathbb{C}^{n\times n}$ (in a certain matrix norm $\norm{\cdot}$) is the set
\[
\Lambda_\varepsilon(A) = \{\lambda \in \mathbb{C} \colon \text{there exists $\Delta\in\mathbb{C}^{n\times n}$ with $\norm{\Delta} \leq \varepsilon$ such that $\lambda \in \Lambda(A-\Delta)$}\},
\]
i.e., the set of eigenvalues of all possible $\varepsilon$-perturbations of $A$. Pseudospectra are a widely studied topic; we refer the reader to the monograph~\citet{TrefethenEmbreeBook} for more information and many illustrations. 

We recall some of their key properties here. A useful way to rewrite the definition is $\Lambda_{\varepsilon}(A) = \{\lambda\in\mathbb{C} \colon d(A-\lambda I, \mathcal{E}_0) \leq \varepsilon\}$; hence, it follows from~\eqref{eq:distance from singularity} that pseudospectra are the same for every normalized unitarily invariant norm. Following the literature, when we speak of pseudospectra in the following we always assume that we are working with such a norm. 

Pseudospectra can be visualized in $\mathbb{C}\cong \mathbb{R}^2$ as the sub-level sets  of the function $f(z) = \sigma_{\min}(A-zI_n)$. In particular, they are a nested family: $\Lambda_{\varepsilon_1}(A) \subsetneq \Lambda_{\varepsilon_2}(A)$ if $\varepsilon_1 < \varepsilon_2$. Pseudospectra are fairly regular: Since singular values can be defined as zeros of a determinant (a polynomial function in $\Re(z),\Im(z)$), the function $f(z)$ is semi-algebraic. Consequently, the boundary $\partial \Lambda_{\varepsilon}$ lies on a real algebraic curve, and is smooth apart from finitely many points. For a normal $A$, the $\varepsilon$-pseudospectrum is simply the union of balls of radius $\varepsilon$ centered in each eigenvalue; for a non-normal matrix, it can be more complicated.
\begin{figure}
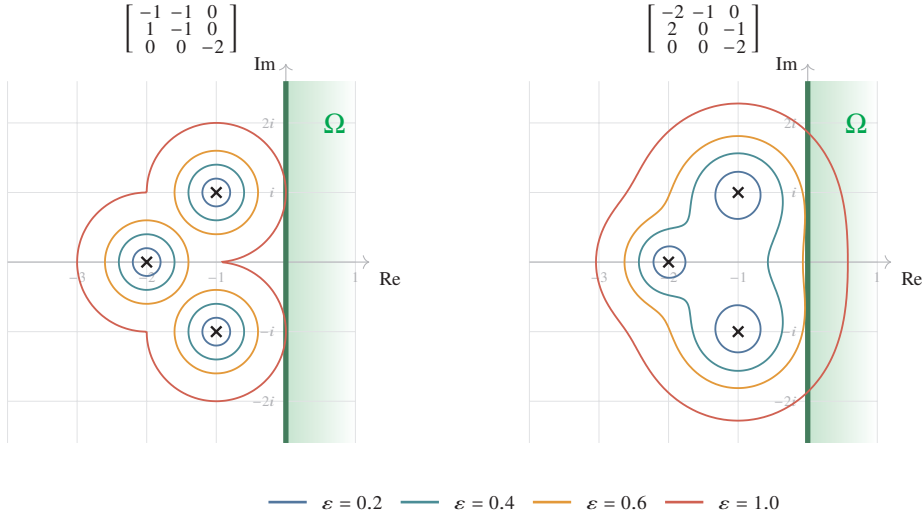

    \centering

    
    \caption{Boundaries of the $\varepsilon$-pseudospectra of two Hurwitz stable matrices, one normal (left) and one non-normal (right). The distance to Hurwitz instability of $A$ ($\Omega = \text{right half-plane}$) is the value of $\varepsilon$ for which $\Lambda_\varepsilon(A)$ first touches the imaginary axis.}
    \label{fig:pseudospectra}
\end{figure}

In a spectral nearness problem of the form~\eqref{eq:E with an eigenvalue in Omega}, the minimum distance is equal to the value of $\varepsilon$ for which the $\varepsilon$-pseudospectrum of $A$ first touches $\Omega$; this connection provides a way to visualize the solutions to spectral nearness problems. An illustrative example is displayed in Figure~\ref{fig:pseudospectra}.

\subsection{Reduction to optimizing singular values}

If there are no additional constraints, spectral nearness problems with a sufficiently regular region $\Omega$ can be solved by reducing them to univariate minimization on $\partial\Omega$.
\begin{theorem}\label{thm:1dmin} Let $\Omega \subseteq \mathbb{C}$ be closed and nonempty, and  $A \not \in \mathcal{E}_{\Omega}$. For every unitarily invariant norm normalized such that $\| e_1 e_1^T \|=1$, we have
    \begin{equation} \label{eq:optimum distance from stability}
        d(A, \mathcal{E}_\Omega) = \min_{\lambda \in \partial \Omega} \sigma_{\min}(A-\lambda I_n).
    \end{equation}
    Moreover, a minimizer can be constructed from an SVD of $A-\lambda I_n$ as $X=A-\Delta$, where $\Delta = u_{\min}\sigma_{\min}v_{\min}^*$ has rank 1.
\end{theorem}
 \begin{proof}
Since eigenvalues are (in the quotient topology, i.e., up to permutations) continuous functions of the matrix entries, the minimizer must be a matrix with an eigenvalue on the boundary $\partial \Omega$. For any unitarily invariant norm, we have
\begin{subequations} \label{eq:distance to instability criteria}
\begin{align}
    d(A, \mathcal{E}_{\Omega}) &=\min \{\norm{\Delta} \colon \text{$A-\Delta$ has an eigenvalue in $\partial \Omega$}\}\\
    &= \min_{\lambda \in \partial \Omega} \min \{\norm{\Delta} \colon \text{$A-\Delta$ has eigenvalue $\lambda$}\} \\
    &= \min_{\lambda \in \partial \Omega} \min \{\norm{\Delta} \colon \text{$A-\Delta-\lambda I_n$ is singular}\} \\
    &= \min_{\lambda \in \partial \Omega} \sigma_{\min}(A-\lambda I_n), \label{eq: minimum singular value char}
\end{align}
\end{subequations}
where we have used~\eqref{eq:distance from singularity}. The expression of the minimizer follows from the Schmidt--Eckart--Young--Mirsky Theorem.
\end{proof}
\begin{remark} \label{rem:double role of u and v}
    Given a Frobenius distance  minimizer $X$ and the corresponding eigenvalue $\lambda \in \partial\Omega$, note the interesting property that $u_{\min}$ and $v_{\min}$ are not only (up to multiples) the low-rank factors of $\Delta=A-X$, but also the left and right eigenvectors of $X$ associated with $\lambda$. Observe also that the present remark can be seen as a shifted version of Example~\ref{ex: rank deficient via normal}.
\end{remark}
An example of this reduction to univariate minimization is the distance to Hurwitz instability, which we see in more detail in the next section.

\subsection{Distance to Hurwitz instability}
In this section, we focus on the problem of computing the distance to Hurwitz instability (Example~\ref{ex:distance Hurwitz instability unstructured}). This problem is natural in systems theory, and has attracted much interest in the literature. 

In this problem, $\partial \Omega$ is the imaginary axis $i\mathbb{R}$; hence, the distance to Hurwitz instability can be computed by minimizing the univariate function 
\begin{equation} \label{eq:function f from distance to instability on the imaginary axis}
 f(\omega) = \sigma_{\min}(A- i\omega I_n), \quad f : \mathbb{R} \to \mathbb{R}.   
\end{equation}
This function is continuous, but may fail to be differentiable in points where $\sigma_{\min}$ is not simple. Luckily, as Lemma~\ref{lem:localminimagood} shows, local minima are always differentiable points. See  also \cite[Theorem 3.8]{DNN24} for an analogous result on manifolds. Note that Lemma~\ref{lem:localminimagood} holds in particular for \emph{global} minima $\omega_\star \in \mathrm{argmin} f(\omega)$.

\begin{lemma}\label{lem:localminimagood}
    Let $A$ be a Hurwitz stable matrix. Then, $f(\omega)$ in~\eqref{eq:function f from distance to instability on the imaginary axis} is differentiable in a neighborhood of any local minimum $\omega_\star$, and $\frac{d}{d\omega}f(\omega_\star) = 0$.
\end{lemma}
\begin{proof}
    By classic results on perturbation of Hermitian operators \cite{Rellich1937}, since $A- i\omega I_n$ is an analytic function of a single real parameter, there are $n$ real analytic functions $\psi_1(\omega), \dots, \psi_n(\omega)$ such that the singular values of $A-i\omega I_n$ are $\abs{\psi_1(\omega)}, \dots, \abs{\psi_n(\omega)}$ (unordered-unsigned analytic SVD; see~\citet{deMoor1989analytic, Bunse91}). Since $A$ is Hurwitz stable, the functions $\psi_k(\omega)$ never vanish; up to changing sign we can assume that they are all positive. Even if there is more than one value $k$ that achieves the minimum $\delta_\star = f(\omega_\star) = \min_k \psi_k(\omega_\star)$, then $\delta_\star$ must be a local minimum of each such $\psi_k(\omega)$; in particular, it holds that $\frac{d}{d\omega}\psi_k(\omega_\star)=0$. Hence $\psi_k(\omega) - \delta_\star = \mathcal{O}((\omega-\omega_\star)^2)$ for each such $\psi_k(\omega)$, and this is enough to conclude that $f(\omega) = \min_k \psi_k(\omega)$ is differentiable in $\omega_\star$.
\end{proof}
Various optimization methods have been suggested to find local minima of the (generally non-convex) $f$. Conveniently, a result by Byers allows one to check if a minimum is global.
\begin{theorem}[\protect{\citet[Theorem~1]{Byers1988a}}] \label{thm: Byers trick}
    The following are equivalent.
    \begin{enumerate}
        \item $\delta \in \mathbb{R}$, $\delta\geq 0$ is greater than or equal to the global minimum $\delta_\star = d(A, \mathcal{E}_{RHP})$ in~\eqref{eq:optimum distance from stability}
        \item There exists $\omega \in \mathbb{R}$ such that $\delta$ is a singular value of $A-i\omega I_n$
        \item There exists $\omega\in \mathbb{R}$ such that $i\omega$ is an eigenvalue of
        \[
        H_\delta = \begin{bmatrix}
            A & -\delta I_n\\
            \delta I_n & -A^*
        \end{bmatrix} \in \mathbb{C}^{2n\times 2n}.
        \]
    \end{enumerate}
\end{theorem}
\begin{proof}
    We begin by showing that 1 is equivalent to 2. If $\delta$ is the singular value of a matrix of the form $A-i\omega I_n$, then $\delta \geq \delta_\star $ in view of~\eqref{eq: minimum singular value char}. The converse follows from noting that the range of the continuous function $f$ in~\eqref{eq:function f from distance to instability on the imaginary axis} is precisely the interval $[\delta_\star,+\infty[$: Indeed, when $\omega \to \pm\infty$, $\sigma_{\min}(A-i\omega I_n)$ tends to $\infty$ because $\sigma_{\min}(A-i\omega I_n) \geq \sigma_{\min}(i\omega I_n) - \sigma_{\max}(A)$ by Weyl's inequalities.
    
    To show that 2 is equivalent to 3, we expand blocks to compute
    \[
    \begin{bmatrix}
        A & -\delta I_n\\
        \delta I_n & -A^*
    \end{bmatrix}\begin{bmatrix}
        v\\u
    \end{bmatrix} = i\omega \begin{bmatrix}
        v\\u
    \end{bmatrix} \iff \begin{cases}
        (A-i\omega I_n)v = \delta u,\\
        (A-i\omega I_n)^*u = \delta v,
    \end{cases} 
    \]
    The two rightmost equalities imply that $\delta$ is a singular value of $A-i\omega I_n$, with left and right singular vectors $u$ and $v$.
\end{proof}
The matrix $H_\delta$ is Hamiltonian, a structure that arises frequently in control theory. The reader may find it surprising that there is a whole interval of values $\delta$ for which $H_\delta$ has a pure imaginary eigenvalue; but these generic imaginary eigenvalues are a consequence of the structural properties of Hamiltonian matrices; see for instance~\citet[Section~3.5]{LancasterRodmanARE}.

Theorem~\ref{thm: Byers trick} forms the basis of a bisection algorithm in~\citet{Byers1988a}: We start from an interval $[a_0,b_0]$ enclosing $\delta_\star$; for instance $[0, \sigma_{\min}(A)]$. At each step, we solve a Hamiltonian eigenvalue problem with $\delta_k = \frac{a+b}{2}$ to determine whether the midpoint of $[a_k,b_k]$ is smaller or larger than $\delta_\star$, and so on.

An evolution of this algorithm is the quadratically convergent algorithm introduced independently by \citet{BoydBalakrishnan} and \citet{BS90} for computing the
$H_\infty$ norm of a transfer matrix. Computing the distance to instability is a special case of this problem. The \emph{criss-cross algorithm}
of \citet{Burke2003} for the pseudospectral abscissa is based on closely related ideas. Below, we describe the algorithm of \citet{BoydBalakrishnan, BS90} for the distance to instability by considering the graph of $f(\omega)$ in~\eqref{eq:function f from distance to instability on the imaginary axis}. Each iteration of the algorithm starts from a candidate minimizer $\delta_k \geq \delta_\star$. The equation $f(\omega)=\delta_k$ has at most $2n$ solutions: Indeed, if $\delta = f(\omega) = \sigma_{\min}(A-i\omega I_n)$, then $i\omega \in \Lambda(H_\delta)$ by Theorem~\ref{thm: Byers trick}. Hence the sub-level set $\{\omega\in\mathbb{R} : f(\omega)\leq\delta_k\}$ is the union of disjoint line segments $[a_1,b_1] \cup [a_2,b_2] \cup \dots \cup [a_\ell,b_\ell]$ (some of them possibly degenerating into points, $a_j=b_j$). 
\begin{figure}
    \centering
    \begin{tikzpicture}
\begin{axis}[
  width=13cm, height=6cm,
  xlabel={$\omega$}, ylabel={$\delta$},
  xmin=-4, xmax=4, ymin=0, ymax=4,
  samples=200, grid=both, axis lines=left,
  legend pos=north east, legend cell align=left,
]
  \addplot[gray!50, dashed, domain=-5:5, forget plot] {sqrt((x-2)^2+1)};
  \addplot[gray!50, dashed, domain=-5:5, forget plot] {sqrt((x+2)^2+4)};

  \addplot[red!60, thin, domain=-4:4] {2};

  \addplot[very thick, blue, domain=-5:-0.375] {sqrt((x+2)^2+4)};
  \addlegendentry{$f(\omega)=\sigma_{\min}(A-i\omega I_n)$}
  \addplot[very thick, blue, domain=-0.375:5, forget plot] {sqrt((x-2)^2+1)};

  \addplot[only marks, mark=*, red, forget plot]   coordinates {(-2,2) (2,1)};
  \addplot[only marks, mark=*, black, forget plot] coordinates {(-0.375,2.5769)};
  \node[anchor=north, red]  at (axis cs:-2,2) {\small $-2$};
  \node[anchor=north, red]  at (axis cs: 2,1) {\small $2$};
\end{axis}
\end{tikzpicture}

    \caption{Plot of $f(\omega)$ in~\eqref{eq:function f from distance to instability on the imaginary axis} for $A=\begin{bsmallmatrix}
        -1.5 & -0.5-2i\\
        -0.5-2i & -1.5
    \end{bsmallmatrix}$. For $\delta_k=2$, $H_{\delta_k}$ has not only the double imaginary eigenvalue $-2i$, but also two simple eigenvalues $\lambda_1 \approx 0.2679i$ and $\lambda_2 \approx 3.7321i$, corresponding to the other intersections of the plot with $y=2$; this allows one to conclude that $2$ is not the global minimum. Starting from $\delta_k=2$, the midpoint iteration algorithm would pick $\omega_{k+1} = \frac{\lambda_1+\lambda_2}{2i}=2$ as the abscissa of the next candidate minimum, and $\delta_{k+1} = f(\omega_{k+1})=1$. In this simple problem, $\delta_{k+1}$ is the exact global minimum $\delta_{\star}$. Indeed, for $\delta_{k+1}=1$, $H_{\delta_{k+1}}$ has no other pure imaginary eigenvalues in addition to the double eigenvalue $2i$, and this confirms that we have reached a global minimum.}
    \label{fig:spectralplot}
\end{figure}
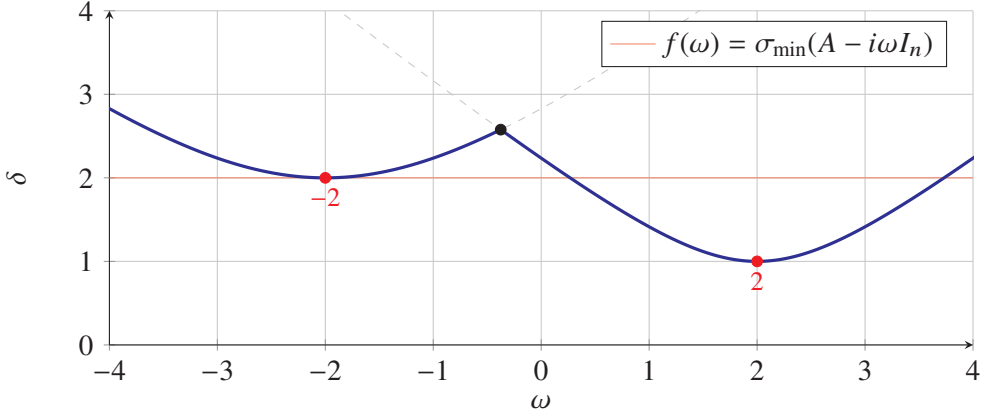
If $\delta$ is close to $\delta_\star= \min f$, then 
geometrical intuition suggests that these segments are few and short, and that their midpoints $\mu_j = \frac{a_j + b_j}{2} \in \mathbb{R}$ are good candidate locations for $\arg\min f$. Thus, we take as the new candidate minimizer for the next iteration of the algorithm
\[
\delta_{k+1} = \min_{j=1,2,\dots,\ell} f(\mu_j) \geq \delta_\star.
\]
The analysis in~\citet{BoydBalakrishnan} shows that in this setup indeed $\delta$ converges quadratically to $\delta_\star$. For $H_\infty$-norm and stability radius computations, \citet{GeninVV98} study a refinement of this approach that, generically, yields a quartic convergence rate at a moderate extra cost.  We see an illustration of the midpoint iteration algorithm in Figure~\ref{fig:spectralplot}. One point in the algorithm needs some care: While the endpoints $a_j,b_j$ are all imaginary parts of pure imaginary eigenvalues of $H_\delta$, in general it is impossible to recover the segments from $\Lambda(H_\delta)$ alone. An example is in Figure~\ref{fig:spectralplot_counterexample}.
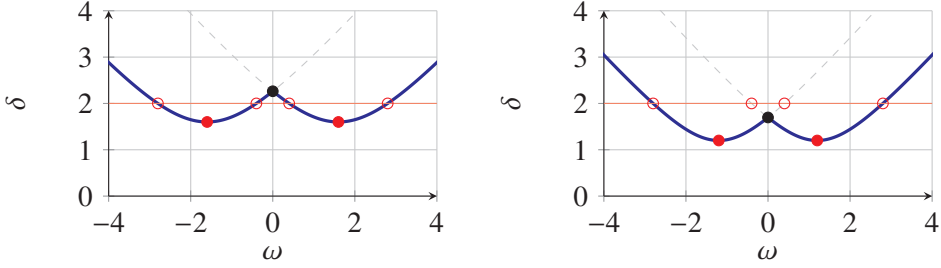
\begin{figure}
    \centering
    \begin{tikzpicture}
    \begin{axis}[
      width=0.47\textwidth, height=0.32\textwidth,
      xlabel={$\omega$}, ylabel={$\delta$},
      xmin=-4, xmax=4, ymin=0, ymax=4,
      samples=400, grid=both, axis lines=left,
    ]
      \addplot[gray!50, dashed, domain=-5:5] {sqrt((x-1.6)^2+2.56)};
      \addplot[gray!50, dashed, domain=-5:5] {sqrt((x+1.6)^2+2.56)};
      \addplot[very thick, blue, domain=-5:5]
        {min(sqrt((x-1.6)^2+2.56), sqrt((x+1.6)^2+2.56))};
      \addplot[red!60, thin, domain=-4:4] {2};
      \addplot[only marks, mark=o, red] coordinates {(-2.8,2) (-0.4,2) (0.4,2) (2.8,2)};
      \addplot[only marks, mark=*, red] coordinates {(-1.6,1.6) (1.6,1.6)};
      \addplot[only marks, mark=*, black] coordinates {(0,2.2627)};
    \end{axis}
    \end{tikzpicture}%
    \hfill
    \begin{tikzpicture}
    \begin{axis}[
      width=0.47\textwidth, height=0.32\textwidth,
      xlabel={$\omega$}, ylabel={$\delta$},
      xmin=-4, xmax=4, ymin=0, ymax=4,
      samples=400, grid=both, axis lines=left,
    ]
      \addplot[gray!50, dashed, domain=-5:5] {sqrt((x-1.2)^2+1.44)};
      \addplot[gray!50, dashed, domain=-5:5] {sqrt((x+1.2)^2+1.44)};
      \addplot[very thick, blue, domain=-5:5]
        {min(sqrt((x-1.2)^2+1.44), sqrt((x+1.2)^2+1.44))};
      \addplot[red!60, thin, domain=-4:4] {2};
      \addplot[only marks, mark=o, red] coordinates {(-2.8,2) (-0.4,2) (0.4,2) (2.8,2)};
      \addplot[only marks, mark=*, red] coordinates {(-1.2,1.2) (1.2,1.2)};
      \addplot[only marks, mark=*, black] coordinates {(0,1.6971)};
    \end{axis}
    \end{tikzpicture}
    \caption{Plots of $f(\omega)=\sigma_{\min}(A-i\omega I_2)$, for
    $A_1=\begin{bsmallmatrix} -1.6 & 1.6\\ -1.6 & -1.6 \end{bsmallmatrix}$ (left) and
    $A_2=\begin{bsmallmatrix} -1.2 & 1.2\\ -1.2 & -1.2 \end{bsmallmatrix}$ (right).
    In both cases, for $\delta=2$ the matrix $H_\delta$ has the same pure imaginary eigenvalues $\pm 0.4i$ and $\pm 2.8i$ (circles), but in one case, the sublevel set $\{\omega \colon f(\omega)\leq 2\}$ is composed of two segments $[-2.8,-0.4]$ and $[0.4,2.8]$; in the other, it is only one segment $[-2.8,2.8]$.
    }
    \label{fig:spectralplot_counterexample}
\end{figure}
One can reconstruct the segments by evaluating $f(\omega)$ in suitable points; alternatively, a method to recover these segments exactly is by considering instead of $H_\delta$ the Hermitian matrix pencil
\begin{equation*}
A_\delta(\omega) = \begin{bmatrix}0&I_n\\-I_n&0\end{bmatrix}(H_\delta-i\omega I_{2n})
=
\begin{bmatrix}
\delta I_n & -A^*-i\omega I_n\\
-A+i\omega I_n & \delta I_n
\end{bmatrix}.
\end{equation*}
Together with the generalized eigenvalues of this pencil, one can also compute the  \emph{sign characteristics} of those eigenvalues that lie on the real line. The sign characteristics give us more detailed information about the crossings; see \citet{MNTX16} for definitions and more details.

Extensions of the criss-cross algorithm to general spectral value sets \cite{karow2003}, together with improvements in efficiency and robustness, are developed by \citet{Benner2019}. A criss-cross algorithm for the real pseudospectral abscissa in the spectral norm is developed by \citet{Lu17}.

A more sophisticated method to compute new candidates $\delta$ and $\omega$ is the \emph{implicit determinant method}~\cite{Freitag11newtoninst}. It is based on solving the bordered linear system
\begin{equation} \label{impdet}
    M(\omega,\delta) \begin{bmatrix}
        x\\s
    \end{bmatrix} = \begin{bmatrix}
        0\\1
    \end{bmatrix}, \quad M(\omega, \delta) = \begin{bmatrix}
        H_\delta - i\omega I & c_1\\
        c_2^* & 0
    \end{bmatrix} \in \mathbb{C}^{(2n+1)\times (2n+1)},
\end{equation}
for two suitable vectors $c_1,c_2\in\mathbb{C}^{2n}$. The following results hold.

\begin{theorem}[\protect{\citet{Freitag11newtoninst}}] \label{thm:implicitdeterminant}
    Assume $M(\omega,\delta)$ is invertible, and choose $c_1,c_2$ so that $c_1 = Jc_2$, $J=\begin{bsmallmatrix}
        0 & I_n\\
        -I_n & 0
    \end{bsmallmatrix}$. Let $s = s(\omega,\delta)$ be the last entry of the solution of~\eqref{impdet}, seen as a function of $\omega$ and $\delta$, and similarly let $x = x(\omega,\delta)\in\mathbb{C}^{2n}$. Then, 
    \begin{enumerate}
        \item $s(\omega,\delta) \in\mathbb{R}$, and
        \begin{equation} \label{impdet_d}
            s(\omega,\delta) = \frac{\det(H_\delta - i\omega I_{2n})}{\det M(\omega,\delta)}.    
        \end{equation}
        \item The partial derivatives $s_{\omega}, s_{\delta}, s_{\omega\omega}, s_{\omega\delta}$ of $s(\omega,\delta)$ satisfy the equations
        \begin{align*}
            M(\omega,\delta)\begin{bmatrix}
                x_\delta\\
                s_\delta
            \end{bmatrix} &= \begin{bmatrix}
                Jx\\0
            \end{bmatrix}, &
            M(\omega,\delta)\begin{bmatrix}
                x_\omega\\
                s_\omega
            \end{bmatrix} &= \begin{bmatrix}
                ix\\0
            \end{bmatrix},\\
            M(\omega,\delta)\begin{bmatrix}
                x_{\omega\omega}\\
                s_{\omega\omega}
            \end{bmatrix} &= \begin{bmatrix}
                2ix_\omega\\0
            \end{bmatrix}, &
            M(\omega,\delta)\begin{bmatrix}
                x_{\omega\delta}\\
                s_{\omega\delta}
            \end{bmatrix} &= \begin{bmatrix}
                Jx_\omega + ix_\delta\\0
            \end{bmatrix}.
        \end{align*}
    \end{enumerate}
\end{theorem}
\begin{proof}
    The fact that $s(\omega,\delta)$ is real (and hence also its derivatives are) follows from multiplying the first row of~\eqref{impdet} by $x^*J$ and noting that $J(H_\delta-i\omega I_{2n})$ is Hermitian. The equations for the partial derivatives follow by differentiating~\eqref{impdet} a suitable number of times. 
\end{proof}
Thanks to~\eqref{impdet_d}, the task of finding pure imaginary eigenvalues of $H_\delta$ can be replaced by finding zeros of $s(\omega, \delta)$.

Assume now that the minimum of $f$ is attained in a point $\omega_\star$ where $\sigma_{\min}(A-i\omega_\star I_n)$ is simple. Then, by computing $f'(\omega)$ using standard formulas for the derivatives of simple singular values, one sees that in a minimum $u_{\min}^*v_{\min} \in \R$, and this in turn is equivalent to the fact that the eigenvalue 
$i \omega_\star$
of $H_{\delta_\star}$ is defective, since its left and right eigenvectors are respectively $\begin{bsmallmatrix}
    u_{\min}^* & -v_{\min}^*
\end{bsmallmatrix}$ and $\begin{bsmallmatrix}
    v_{\min}\\
    u_{\min}
\end{bsmallmatrix}$. The fact that this eigenvalue is not simple is not surprising, since in view of the previous algorithm we know that in the limit $\delta\to\delta_\star$ at least one segment $[a_j,b_j]$ degenerates to a point $\omega_\star = \lim_{\delta\to\delta_\star} a_j = \lim_{\delta\to\delta_\star} b_j$.

Hence, to find the minimum $\delta_\star$ and its associated $\omega_\star$, we solve for $(\omega,\delta)\in\mathbb{R}^2$ the real bivariate system of equations $s(\omega,\delta) = s_{\omega}(\omega,\delta) = 0$ using the Newton method in $\mathbb{R}^2$; the associated Jacobian is
\begin{equation} \label{impdet_jac}
    \begin{bmatrix}
    s_\omega & s_\delta\\
    s_{\omega\omega} & s_{\omega\delta}
    \end{bmatrix}.    
\end{equation}
\citet{Freitag11newtoninst} argue that generically $H_{\delta_\star}$ has a Jordan block in $i\omega_\star$ of size $2$; in turn this implies that $M(\omega_\star,\delta_\star)$ is nonsingular for almost all choices of $c$, and that the Jacobian~\ref{impdet_jac} is nonsingular in $(\delta_\star,\omega_\star)$, so the Newton method is quadratically convergent. 

Note that all the derivatives in~\eqref{impdet_jac} are computed by solving linear systems with the same matrix $M(\omega,\delta)$; so the method requires only one matrix factorization, and can be used effectively also in cases where $A$ is large and sparse. We refer the reader to~\citet{Freitag11newtoninst} for more details on this method.

All the results in this subsection have simple analogues for \emph{Schur (or discrete-time) stability}, i.e., when the right half-plane is replaced by the outside of the unit disc; for instance, an analogue of Theorem~\ref{thm: Byers trick} for that case is in the same paper~\citet{Byers1988a}. Algorithms for computing the pseudospectral radius in this  context are developed by \citet{Mengi2005}. See also~\citet{Kostic2015} for an extension to more general domains.

Further related work includes an algorithm for the distance to instability designed primarily for sparse matrices \cite{He1999}, iterative methods for the pseudospectral abscissa and radius, developed by \citet{Guglielmi2011}, and
subspace acceleration for the pseudospectral abscissa and the stability radius, studied by \citet{Kressner2014}.

\subsection{Nearness problems solvable via optimization on $U(n)$}\label{sec:Un}

In this subsection, we analyze an approach to matrix nearness problems based on the Schur canonical form and variable projection. Let $\mathcal{V}_T \subseteq \C^{n \times n}$ denote the subspace of upper triangular matrices, and suppose that the target set $\mathcal{E} \subseteq \C^{n \times n}$ is invariant by unitary similarity, i.e., $X \in \mathcal{E} \Rightarrow Q X Q^* \in \mathcal{E}$ for all $Q \in U(n)$. We choose a closed subset $\mathcal{E}_T \subseteq \mathcal{E}\cap\mathcal{V}_T$ such that
\begin{equation}\label{eq:siminvset}
    \mathcal{E}
    = \{QTQ^* : Q\in U(n),\ T\in\mathcal{E}_T\}.
\end{equation}
An obvious choice for $\mathcal{E}_T$ in \eqref{eq:siminvset} always exists by the Schur decomposition, namely, we can take $\mathcal{E}_T=\mathcal{E}\cap\mathcal{V}_T$. For instance, if $\mathcal{E}$ is the set of matrices whose eigenvalues lie in $\Omega\subseteq \mathbb{C}$, then $\mathcal{E} \cap \mathcal{V}_T$ is the set of triangular matrices whose diagonal elements are in $\Omega$, whereas the strictly upper triangular part is unconstrained; if $\mathcal{E}$ is the set of normal matrices then $\mathcal{E} \cap \mathcal{V}_T$ reduces to the set of diagonal matrices, and so on. However, Example~\ref{ex:smaller} shows that letting $\mathcal{E}_T$ be a strictly smaller generating set than $\mathcal{E} \cap \mathcal{V}_T$ may suffice, and indeed even be more convenient.  
\begin{example}\label{ex:smaller}
    Taking $\mathcal{E}_T=\left\{\begin{bsmallmatrix}
    1&0\\
    0&0
\end{bsmallmatrix}\right\}$ generates the set $\mathcal{E}$ of all orthogonal projections onto a one-dimensional subspace of $\C^2$, but $\displaystyle \mathcal{E}_T \not\ni \begin{bsmallmatrix}
    0&0\\
    0&1
\end{bsmallmatrix} \in \mathcal{E} \cap \mathcal{V}_T$. The same example also illustrates that a less obvious choice than just the intersection can be advantageous from the viewpoint of optimization; indeed, while $\mathcal{E}_T=\left\{\begin{bsmallmatrix}
    1&0\\
    0&0
\end{bsmallmatrix}\right\}$ is a convex set, $\mathcal{E} \cap \mathcal{V}_T = \left\{\begin{bsmallmatrix}
    1&0\\
    0&0
\end{bsmallmatrix},\begin{bsmallmatrix}
    0&0\\
    0&1
\end{bsmallmatrix}\right\}$ is not.
\end{example}

Before going ahead, let us quickly verify that \eqref{eq:siminvset} is coherent with our treatment in generating a closed  target set $\mathcal{E}$.

\begin{proposition}
    Let $\mathcal{E}_T \subseteq \mathcal{V}_T$ be a closed subset of upper triangular matrices. Then, $\mathcal{E}$ as in \eqref{eq:siminvset} is closed.
\end{proposition}
\begin{proof}
    Let $(X_k)_k \subset \mathcal{E}$ converge to $X \in \C^{n \times n}$ and choose Schur decompositions $X_k=Q_kT_kQ_k^*$ with
$Q_k\in U(n)$ and $T_k\in\mathcal E_T$,
as allowed by \eqref{eq:siminvset}. Because $U(n)$ is compact, there is a convergent subsequence $Q_{k_j} \rightarrow Q \in U(n)$, and then $T_{k_j} \to Q^* X Q \in \mathcal{E}_T$ because $\mathcal{E}_T$ is closed. This proves $X \in \mathcal{E}$.
\end{proof}

Let us now discuss why \eqref{eq:siminvset} can be exploited to design numerical methods for nearness problems whose target set is invariant by unitarily similarity. Every $X \in \mathcal{E}$ has a Schur decomposition $X=QTQ^*$ with $Q \in U(n)$ and $T \in \mathcal{E}_T$, and conversely every pair $(Q,T) \in U(n) \times \mathcal{E}_T$ generates a matrix $X=QTQ^* \in \mathcal{E}$. Hence,

\begin{equation}\label{eq:schurtrick}
   d(A,\mathcal{E})^2 = \min_{Q \in U(n)} \min_{T \in \mathcal{E}_T} \| Q^* A Q - T \|_F^2.  
\end{equation}

In fact, \eqref{eq:schurtrick} remains valid even after replacing the Frobenius norm with any unitarily invariant norm. What makes \eqref{eq:schurtrick} particularly attractive is that, for many problems of practical interest, the inner minimization often reduces to elementary projections of elements and blocks, leading to an exact solution. This is an instance of a \emph{variable projection method} \cite{Golub1973,O’Leary2013}, and it allows us to express the optimal distance as a function of the unitary matrix $f(Q) = \min_{T\in\mathcal{E}_T} \norm{Q^*AQ-T}_F^2$. It follows that the original matrix nearness problem is equivalent to
\begin{equation} \label{optimization on unitaries}
   d(A,\mathcal{E})^2 = \min_{Q \in U(n)} f(Q). 
\end{equation}
Moreover, since unitary matrices are a well studied real Riemannian manifold, the tools of modern Riemannian optimization \cite{Absil2008, Boumal2014b, Boumal2023} are available to tackle this new problem~\eqref{optimization on unitaries} numerically; often, this leads to very efficient algorithms. 

Before this specialization, let us give in Theorem~\ref{thm:equivalentgeneral} a general framework for this class of problems. This result has previously appeared for specific instances of problems in this class, such as \cite[Theorem 3.2]{Noferini2021a} (whose statement for the case of local minima is incorrect as given; item 2 in Theorem~\ref{thm:equivalentgeneral} provides a correction for the case where $\Omega$ is convex) or \cite[Theorem 2.2]{NZ26} for normal matrices.

\begin{theorem}\label{thm:equivalentgeneral}
    Let $\mathcal{V}_T \subseteq \C^{n \times n}$ be the subspace of upper triangular matrices, and let $\mathcal{E}_T \subseteq \mathcal{V}_T$ be a nonempty closed convex subset of $\mathcal{V}_T$. Let $\mathcal{E}$ be defined as in
\eqref{eq:siminvset}. For every $Q\in U(n)$, define
\begin{equation*}
B(Q)=Q^*A Q, \quad T(Q)=\Pi_{\mathcal{E}_T}(B(Q)), \quad X(Q)=QT(Q)Q^*, \quad f(Q)=\|B(Q)-T(Q)\|_F^2,
\end{equation*}
where $\Pi_{\mathcal{E}_T}$ denotes the orthogonal (in the Frobenius inner product) projection onto $\mathcal{E}_T$. Then:
\begin{enumerate}
    \item $Q_0$ is a global minimizer of $f$ if and only if $X(Q_0)$ is a
global minimizer of $\|A-X\|_F^2$ over $X\in\mathcal{E}$. Moreover, $f(Q_0)=\|A-X(Q_0)\|_F^2$;
\item Given $X_0\in\mathcal{E}$, define
\[
 \mathcal{R}_{\mathcal{E}_T}(X_0)
 =\{Q\in U(n):Q^*X_0Q\in\mathcal{E}_T\}.
\]
Then $X_0$ is a local minimizer of $\|A-X\|_F^2$ over $\mathcal{E}$ if and
only if, for every $Q\in\mathcal{R}_{\mathcal{E}_T}(X_0)$, one has
$X(Q)=X_0$ and $Q$ is a local minimizer of $f$.
\end{enumerate}
\end{theorem}

\begin{proof}
    \begin{enumerate}
        \item Using the unitary invariance of the Frobenius norm and that (by convexity) $T(Q)=\Pi_{\mathcal{E}_T}(B(Q))$ minimizes the distance from $B(Q)$ over  $\mathcal{E}_T$, we have that for every $Q \in U(n)$ and every $T \in \mathcal{E}_T$
\begin{equation}\label{eq:fQleq}
    f(Q) = \min_{Y \in \mathcal{E}_T} \|B(Q)-Y\|_F^2 \leq \|B(Q)-T\|_F^2  = \|A-Q T Q^*\|_F^2,
\end{equation}
with equality if and only if $T=T(Q)$. Taking minima yields the sought equivalences.
        \item Suppose first that $X_0 \in \mathcal{E}$ is a local minimizer of the squared distance from $A$, define $T_0=Q^*X_0Q$, and fix $Q \in \mathcal{R}_{\mathcal{E}_T}(X_0)$. If $T_0 \neq T(Q)$, then $f(Q) = \|B(Q)-T(Q)\|_F^2 < \| B(Q) - T_0\|_F^2 = \| A - X_0\|_F^2$, because the projection of $B(Q)$ onto $\mathcal{E}_T$ is unique by convexity. Now construct the sequence $(T_k)_{k \geq 1}$ defined as
        \[  T_k = \frac{k-1}{k} T_0 + \frac1k T(Q), \]
        and note that indeed $T_k \in \mathcal{E}_T$ as it is a convex combination of $T_0 \in \mathcal{E}_T$ (because $Q \in \mathcal{R}_{\mathcal{E}_T}(X_0)$) and $T(Q) \in \mathcal{E}_T$ (because it is the projection of $B(Q)$ onto $\mathcal{E}_T$). Therefore, if we define $X_k=Q T_k Q^*$, we have that $(X_k)_{k \geq 1} \in \mathcal{E}$ and $X_k \to X_0$. Moreover, by the triangle inequality, for all $k$
        \[ \| B(Q)-T_k\| \leq \frac{k-1}{k} \|B(Q)-T_0\|_F +\frac1k \|B(Q)-T(Q)\|_F < \|B(Q)-T_0\|_F  \]
        and therefore
        \[  \|A - X_k\|_F < \|A - X_0\|_F.    \]
This is a contradiction, and therefore we conclude $T_0=T(Q)$ implying in turn $X(Q)=X_0$. Take now a sequence $Q_k \to Q$. The function $\Pi_{\mathcal{E}_T}$ is continuous, because it is a projection onto a nonempty closed convex set. Therefore, $X(Q_k) \to X_0$ and
\[  f(Q_k) =  \| A - X(Q_k) \|_F^2 \geq \|A-X_0\|_F^2  = f(Q).  \]

Now suppose that $X_0 \in \mathcal{E}$ is not a local minimizer over $\mathcal{E}$ of the squared distance from $A$. Then, we can pick a sequence $X_k \to X_0$ satisfying $X_k \in \mathcal{E}$ and $\|A-X_k\|^2_F < \|A-X_0\|^2_F$ for all $k$. Take Schur decompositions $X_k = Q_k T_k Q_k^*$ such that $T_k \in \mathcal{E}_T$ (the existence of such a Schur decomposition is guaranteed by \eqref{eq:siminvset}). By compactness, we can extract a convergent subsequence $(Q_{k_j})_j$, with $Q_{k_j} \to Q_0$.
    Then, $T_{k_j} \to Q_0^* X_0 Q_0 \in \mathcal{E}_T$ implying in particular $Q_0 \in \mathcal{R}_{\mathcal{E}_T}(X_0)$. If $X(Q_0) \neq X_0$, there is nothing else to prove. Otherwise, by \eqref{eq:fQleq},
    \[ f(Q_k) \leq \|A - X_k \|_F^2 < \|A-X_0\|_F^2  = f(Q_0),         \]
    and therefore $Q_0$ is not a local minimizer for $f$.
    \end{enumerate}
\end{proof}

Example~\ref{ex:mustbeforall} below illustrates that, in item 2 of Theorem~\ref{thm:equivalentgeneral}, it is crucial for the ``if and only if" that \emph{every} $Q \in \mathcal{R}_{\mathcal{E}_T}(X_0)$ is a local minimizer.

\begin{example}\label{ex:mustbeforall}
    Let $\displaystyle A=\begin{bmatrix}
        1&0\\
        0&-1
    \end{bmatrix}, \mathcal{E}_T=\left\{ T = \begin{bmatrix}
        0&0\\
        0&x
    \end{bmatrix} \ \textrm{for} \ \textrm{some} \ x \geq 0
    \right\}.$ Then, $\mathcal{E} \subsetneq \C^{2 \times 2}$ is the set of Hermitian positive semidefinite matrices of rank at most $1$. Observe that, calling $B=B(Q)=Q^*AQ$,
    \[  \left\| B - \begin{bmatrix}
        0&0\\
        0&x
    \end{bmatrix} \right\|_F^2 = \|A\|_F^2 -  (B_{22})^2  + (B_{22}-x)^2 = 2 - 2x B_{22} + x^2  \]
    which (for a given $Q$) is minimized at $x(Q)=\max\{q^*Aq,0\}$ where $q=Qe_2$. Moreover, $f(Q)=2-x(Q)^2$. At $Q=I_2$, we have $q^*Aq=-1$ and hence (by continuity) $x(Q)=0$ and $f(Q)=2$ in a neighborhood of the identity, which is thus a local minimizer of $f$. However, the corresponding $X_0=0$ is not a local minimizer of the distance, as can be seen by the sequence $(X_k)_{k \geq 1}$,
\[ X_k = \begin{bmatrix}
    \frac1k&0\\
    0&0
\end{bmatrix} = \begin{bmatrix}
    0&1\\
    1&0
\end{bmatrix} \begin{bmatrix}
    0&0\\
    0&\frac1k
\end{bmatrix} \begin{bmatrix}
    0&1\\
    1&0
\end{bmatrix} \to 0=X_0\]
but
\[  \|A-X_k\|_F^2 = \frac{2k^2-2k+1}{k^2} < 2 = \|A-X_0\|_F^2.   \]
\end{example}

To perform Riemannian optimization, it is important to calculate the Riemannian gradient of $f(Q)$, i.e., the projection of $\nabla f(Q)$ onto the tangent space $T_Q U(n)$. Let us now extend the approach of \citet[Lemma 2.6]{NZ26}, who study objective functions defined
through a linear orthogonal projection, to the squared distance from an arbitrary nonempty closed convex set; see also \citet[Section~6.2]{Noferini2021a}
for similar gradient calculations in the context of the nearest $\Omega$-stable matrix problem. First, we can compute the Euclidean gradient of $f$ by some tedious, but straightforward, calculations: $f(Q)$ is the composition of the quadratic map $Q \mapsto Q^*AQ$ and the squared-distance map $g(B) = d(B, \mathcal{E}_T)^2$, which has gradient $\nabla g(B) = 2(B-T)$, with $T=\Pi_{\mathcal{E}_T}(B)$, by Theorem~\ref{thm:ae}. 
Abbreviating $B=B(Q)$ and $R=R(Q)=B(Q)-T(Q)$, we obtain for $Q \in U(n)$,
\[  \nabla f(Q)= 2 Q(B^*R+B R^*). \]
The tangent space of $U(n)$ at $Q$ is $T_Q(U(n)) = \{QK : K^*=-K\}$, and the orthogonal projection onto it is thus $\Pi_{T_Q(U(n))}(Z)=Q \mathrm{skew}(Q^*Z)$, where $\mathrm{skew}(Y)=(Y-Y^*)/2$. Bringing all together, we obtain Proposition~\ref{prop:RGfQ}.
\begin{proposition}\label{prop:RGfQ}
    Under the assumptions of Theorem~\ref{thm:equivalentgeneral}, the 
    Riemannian gradient on $U(n)$ of $f(Q)$ is
    \[  \mathrm{grad} f(Q) = 2 Q \mathrm{skew} (B^*R-RB^*) = 2Q \mathrm{skew}(T^* R - R T^*). \]
\end{proposition}
\begin{proof}
    It only remains to prove the last equality, which follows by noting that $RR^*-R^*R$ is Hermitian.
\end{proof}

We now move from the general overview of methods based on optimization over unitary matrices to a closer look at some important examples for which this approach has been especially successful.
\subsubsection{Nearest normal matrix}\label{sec:nearestnormal}

If $\mathcal{E} = \{ N \in \C^{n \times n} : NN^*=N^*N \}$ is the set of normal matrices, then we can take  $\mathcal{E}_T$ to be the subspace of diagonal matrices, because an upper  triangular matrix is normal if and only if it is diagonal. Hence, in this case, the objective function $f(Q)$ to be minimized over the manifold of unitary matrices is precisely the squared Frobenius norm of the off-diagonal part of $B(Q)=Q^*AQ$:
\[ f_{normal}(Q) = \sum_{i \neq j} |(Q^*AQ)_{ij}|^2.  \]
Equivalently, finding a nearest normal matrix is tantamount to finding a unitary matrix $Q$ that maximizes the Frobenius norm of the diagonal of $Q^*AQ$. Note that, since $\mathcal{E}_T$ is a linear subspace, $f_{normal}$ is smooth. Formulae for its gradient and Hessian can be found in \citet[Section 2]{NZ26}; they can also be derived by specialization from the formulae in \citet{Bierly} or (for the gradient) in Proposition~\ref{prop:RGfQ}.

The observation that finding a nearest normal matrix is tantamount to minimizing $f_{normal}(Q)$ (or, equivalently, maximizing $\norm{\diag(Q^*AQ)}_F$) dates back to~\citet{gabriel}, who also compute the gradient of this function. Algorithmically, this approach was used to formulate the classic Jacobi-type algorithms reviewed by
\citet{Higham1989}; see in particular \citet{ruhe87}. More recently, the same objective function has been exploited in the context of more modern trust-regions algorithms \cite{Bierly,NZ26}. A reader generally interested in the Riemannian trust-region framework might consult \citet{Absil2007}. For a thorough study of the theoretical properties of the objective function $f_{normal}$, see \citet[Section 3]{NZ26}; here, we limit ourselves to observing some symmetries of $f_{normal}$. If $D$ is diagonal and unitary,
then
$ f_{normal}(QD)=f_{normal}(Q)$. Thus, the optimization can also be carried out on the flag manifold
$U(n)/U(1)^n$, by identifying unitary matrices
whose corresponding columns differ only by phases \cite{NZ26}. The gradient and Hessians on the flag manifold are computed in \citet[Section 3]{NZ26}. Similarly, $f_{normal}(Q)=f_{normal}(QP)$ for every permutation matrix $P$, because permuting the columns of $Q$ cannot change the value of the objective function.

If the input $A$ is real, there does not necessarily exist a real matrix among the normal matrices nearest to $A$ (the question was posed by \citet{Higham1989} and answered by \citet[Theorem 3.4]{NZ26b}, negatively for $n\geq 3$ and positively for $n \leq 2$). If one wishes to compute a nearest real normal matrix, a real variant of the algorithm can be devised exploiting the fact that a real normal matrix
can be orthogonally block diagonalized, with scalar blocks and
$2\times2$ blocks of the form
$\begin{bsmallmatrix}a&b\\-b&a\end{bsmallmatrix}$.
These blocks must be taken into account in a real variant of the Riemannian algorithms discussed above \cite{Bierly,NZ26}.

\subsubsection{Nearest circulant matrix}

A simpler variant of the nearest normal matrix problem arises when the unitary matrix $Q \in U(n)$ is fixed. This is the case, for example, if we let $\mathcal{E}$ be the set of circulant matrices \cite{Davis1979}. If we denote by $\mathcal{E}_D$  the set of diagonal matrices, then no optimization on $U(n)$ is needed at all, because every circulant matrix shares the same set of eigenvectors. Namely, let
\[  F =  \frac{1}{\sqrt{n}}\begin{bmatrix}
\omega^{n-1} & \omega^{2n-2} & \dots & \omega^{(n-1)^2} & 1\\
\vdots  & \vdots & \ddots & \vdots  & \vdots \\
\omega^2 & \omega^4 & \dots & \omega^{2n-2} & 1\\
    \omega & \omega^2 & \dots & \omega^{n-1} & 1\\
    1 & 1 & \dots & 1 & 1\\
\end{bmatrix} \in U(n)      \]
be the Fourier matrix, where $\omega$ is an arbitrary, but fixed, primitive $n$-th root of unity. A matrix $X$ is circulant if and only if  $X=FDF^*$ for some diagonal $D\in\mathcal{E}_D$ \cite{Davis1979}. Hence, invoking once again the unitary invariance of the Frobenius norm and having chosen $\mathcal{E}$ and $\mathcal{E}_D$ as explained above,
\[ \min_{X\in \mathcal{E}} \|A-X\|_F^2
    = \min_{D\in \mathcal{E}_D} \|F^*AF-D\|_F^2
    = \sum_{i\neq j}|(F^*AF)_{ij}|^2.\]
Hence, we conclude that the circulant matrix nearest to $A$ must be
\[ B=F \mathcal{D}(F^* A F)F^*,    \]
where the linear operator $\mathcal{D}:\C^{n \times n} \to \mathcal{E}_D$ maps a matrix $X$ to a diagonal matrix whose diagonal elements coincide with those of $X$.

\subsubsection{Nearest $\Omega$-stable matrix}

Let $\Omega\subseteq\C$ be nonempty and closed, and consider the target set
$\mathcal{E}_{\Omega,\forall}$ in \eqref{eq:existsforall}, which is closed by Proposition~\ref{prop:bothclosed} and obviously nonempty because, if $z \in \Omega$, then $z I \in \mathcal{E}_{\Omega,\forall}$. The problem of finding a nearest matrix in $\mathcal{E}_{\Omega, \forall}$ is classic, especially when $\Omega$ is the closed left-half plane. Relevant references include \citet{Higham1989, Gillis2017, Guglielmi2017a, Guglielmi2018, Orbandexivry2013}.

Following the treatment
of \citet{Noferini2021a}, in the context of \eqref{eq:schurtrick} the set up above corresponds to
\begin{equation}\label{eq:ETOmega}
     \mathcal{E}_{T}
 =\{T\in\C^{n\times n}:T\text{ is upper triangular},\ T_{ii}\in\Omega\}.
\end{equation}
\begin{lemma}
    If $\Omega$ is nonempty (resp. closed, convex) then $\mathcal{E}_{T}$ as defined in \eqref{eq:ETOmega} is.
\end{lemma}
\begin{proof}
    The same example above proves nonemptiness because $zI$ is upper triangular. To prove closedness we may follow the same proof of Proposition~\ref{prop:bothclosed}, taking into account that taking limits preserves the property of being upper triangular. Finally, if $T_1,T_2 \in \mathcal{E}_T$ then for every $0 \leq t \leq 1$ the convex combination $tT_1+(1-t)T_2$ also does, because upper triangular matrices are a subspace and, by the convexity of $\Omega$, the diagonal elements must remain in $\Omega$.
\end{proof}

For every $z\in\C$, let its projection $p_{\Omega}(z)$ be any point of $\Omega$ nearest to $z$ (chosen in an arbitrary, but fixed, way if there is more than one)\footnote{We use $p_\Omega$ rather than our usual notation $\Pi_{\Omega}$ for a Euclidean projection, to highlight visually the fact that $p_\Omega(z)$ is a scalar.}.
Writing $B=B(Q)=Q^*AQ$, the inner problem in \eqref{eq:schurtrick} reduces to appropriate elementwise projections, and one minimizer is given by
\[
 \mathcal{T}_{\Omega}(B)_{ij}
 =\begin{cases}
 B_{ij},&i<j,\\
 p_{\Omega}(B_{ii}),&i=j,\\
 0,&i>j.
 \end{cases}
\]
Hence, $ d(A,\mathcal{E}_{\Omega,\forall})^2
 =\min_{Q\in U(n)} f_{\Omega}(Q),$ having defined the objective function
\[
 f_{\Omega}(Q)
 =\sum_{i>j}|B_{ij}|^2+\sum_{i=1}^n |B_{ii}-p_\Omega(B_{ii})|^2.
\]
In practical applications, the two most relevant cases correspond to when $\Omega$ is either the closed left-half plane (weak Hurwitz stability) or the closed unit disk (weak Schur stability); note that both these sets are convex and therefore Theorem~\ref{thm:equivalentgeneral} applies. In particular, a weakly (Hurwitz or Schur) stable
matrix $X_0$ is a local minimizer if and only if each of its Schur
decompositions $X_0=QTQ^*$ satisfies
$T=\mathcal{T}_{\Omega}(Q^*AQ)$ and $Q$ is a local minimizer of
$f_{\Omega}$. Note that, instead, the equivalence of \emph{global} minimizers in Theorem~\ref{thm:equivalentgeneral} does not require the convexity of $\Omega$. 

The projection formulae also take an especially simple form for Hurwitz and Schur stability. In the former case,
\[ \Omega = \{ z \in \C : \Re(z) \leq 0 \}, \quad p_\Omega(z) = \min\{0,\Re(z)\} + i \Im(z),  \]
whereas in the latter case
\[ \Omega = \{ z \in \C : |z| \leq 1 \}, \quad p_\Omega(z) = \frac{z}{\max\{1,|z|\}}.    \]
Up to translating or rotating the input $z$, it is clear that minor modifications of the same formulae yield the projections onto an arbitrary closed disk or half-plane.

When $\Omega$ is convex, $f_\Omega$ is continuously differentiable over $U(n)$, with gradient given by Proposition~\ref{prop:RGfQ} after substituting $R=B-\mathcal{T}_{\Omega}(B)$.
When $\Omega$ is nonconvex, the same gradient formula holds whenever every
$B_{ii}$ has a unique orthogonal projection onto $\Omega$, but not everywhere. 

Other parametrizations have been used to develop optimization
methods for the nearest Schur stable matrix
\cite{Gillis2019} and for spectral regions defined by
intersections of conic sectors, vertical strips, and disks
\cite{Choudhary}. However, an attractive feature of the Riemannian optimization approach is that the constraint is imposed on the diagonal of
the computed triangular factor. When the matrix in $\mathcal{E}_{\Omega,\forall}$ has repeated or defective eigenvalues of the
approximant, the Riemannian
algorithm returns its Schur factors directly, completely avoiding any tricky computation of Jordan chains. At least in the case of Hurwitz stability, Example~\ref{ex:jordanblocks} below (which is a reprisal and generalization of \citet[Example 5.6]{Noferini2021a}) shows that it happens with positive probability that the nearest weakly stable matrix has nontrivial Jordan blocks in its Jordan canonical form; indeed, there is ample numerical evidence \cite{Noferini2021a} that this phenomenon happens with very high probability at least for the case of nearest weakly Hurwitz or weakly Schur stable matrices.

\begin{example}\label{ex:jordanblocks}
    As shown by \citet[Example~5.6]{Noferini2021a}, the unique real weakly Hurwitz stable matrix nearest to
\[
A_0=\begin{bmatrix}1&2\\1&1\end{bmatrix}
\]
in the Frobenius distance is
\[
X_\star=\begin{bmatrix}0&2\\0&0\end{bmatrix},
\qquad \|A_0-X_\star\|_F=\sqrt{3}.
\]
By \citet[Proposition 2.2 and Proposition~4.3]{NZ26b},  $X_\star$ is also a (possibly complex) nearest weakly Hurwitz matrix, and it is also the unique complex minimizer.
It is clear that the Jordan form of $X_\star$ is $X_\star/2$, i.e., a nontrivial Jordan block.

It is already shown in \citet{Noferini2021a} that there exists an open neighbourhood of real matrices centered around $A_0$ such that the Frobenius projection remains defective. We show here that the same conclusion holds for a complex neighbourhood of $A_0$. Take a sequence $(A_k)_k \subset \C^{2 \times 2}$ with $A_k \to A_0$, and construct a second sequence $(X_k)_k \subset \mathcal{E}$, where $\mathcal{E}$ is the set of weakly stable matrices, such that $X_k$ is nearest to $A_k$ for all $k$. By the proof of Theorem~\ref{thm:ae}, we have $X_k \to X_\star$. Observing that $\ker X_\star = \spann(\{e_1\})$, we can take Schur decompositions $X_k=Q_kT_kQ_k^*$ with $Q_k\to I_2$. Thus, $B_k:=Q_k^*A_kQ_k\to A_0$. Hence, for large enough $k$, $\Re \diag(B_k) > 0$. Minimizing $\norm{B_k-T_k}_F^2$ over weakly Hurwitz stable triangular matrices we therefore see that $T_k$ must have the form
\[ T_k = \begin{bmatrix}
    i\Im (B_k)_{11} & (B_k)_{12}\\
    0 & i\Im (B_k)_{22}
\end{bmatrix}  \Rightarrow B_k-T_k = \begin{bmatrix}
    \Re (B_k)_{11} & 0\\
    (B_k)_{21} & \Re (B_k)_{22}
\end{bmatrix}. \]
Moreover, applying Theorem~\ref{thm:equivalentgeneral} and Proposition~\ref{prop:RGfQ}, and using that $B_k^* B_k-B_k B_k^*$ is Hermitian, it must be that $T_k^*(B_k-T_k)-(B_k-T_k) T_k^*$ is real and diagonal for all $k$. Equating to $0$ the $(2,1)$ element of this matrix we get
\[
\overline{(B_k)_{12}}\,
\Re\left((B_k)_{11}-(B_k)_{22}\right)
+i(B_k)_{21}\,
\Im\left((B_k)_{11}-(B_k)_{22}\right)=0.
\]
Multiplying by $(B_k)_{12}$ and taking imaginary parts, this yields
\[
\Im\left((B_k)_{11}-(B_k)_{22}\right)
\cdot \Re\left((B_k)_{12}(B_k)_{21}\right)=0.
\]
For sufficiently large $k$, the second factor of the latter equation is close to $2$, and hence the first factor must be $0$, implying that $T_k$ has equal diagonal entries. Moreover, $(T_k)_{12} \to 2$, showing that $X_k$ is defective.
\end{example}

A related phenomenon occurs in spectral abscissa minimization
over affine matrix families: \citet{burke2001optimizing}
construct examples whose optimal Jordan structure consists
of a single block and persists under small perturbations
of the family. The presence of large Jordan blocks in the Jordan canonical form of the minimizers quite possibly explains why the Riemannian approach has been so successful on this problem: Unlike the algorithm in \citet{Noferini2021a}, the competing algorithms face the issue of having to compute Jordan chains, or more generally numerical difficulties related to defective eigenvalues; in some cases, the method actively avoids that situation.

It is worth noting that the problem of  the nearest Hurwitz stable matrix is motivated by its applications to differential-algebraic equations and dynamical systems \cite{KunkelMerhmann}. For this reason, Hurwitz stability is usually defined by imposing that every eigenvalue have \emph{strictly negative} real part; a weaker notion allows  for semisimple pure imaginary eigenvalues. However, this observation cannot be used to dismiss the distances and nearest matrices found by the algorithm in \cite{Noferini2021a}, because the corresponding two sets are not closed and their closure is precisely the set $\mathcal{E}_{\Omega,\forall}$ discussed above (with $\Omega$ being the closed left-half plane). Therefore, even assuming that the nearest weakly Hurwitz stable matrix $B$ does not belong to the smaller sets (for example because it has a huge Jordan block with pure imaginary eigenvalues), the distance remains the same, as it is clear by taking the sequence of strictly stable matrices $(B_k)_{k \geq 1}$ with $B_k = B - \frac1k I \to B$.

Similarly to the nearest normal matrix problem, it was conjectured in \citet{Noferini2021a} that a real matrix $A$ may always have a real weakly Hurwitz stable matrix among the (complex) weakly Hurwitz stable matrices nearest to $A$. Some time after the publication of the paper, the algorithm of \citet{Noferini2021a} found some potential counterexamples; a formal disproof of the conjecture, together with a proof that the conjectured behaviour happens with probability $1$ when $n \leq 2$ and with a probability strictly between $0$ and $1$ when $n \ge 3$, can be found in \citet{NZ26b} (limited to the case of Hurwitz stability). Moreover, again for every fixed size $n\geq3$, the ratio between the real and complex distances is unbounded, and  this is  true even when the Frobenius norm is replaced by any other unitarily invariant norm \citep[Theorem~4.1]{NZ26b}.

For a general $\Omega$, if one is interested in a nearest \emph{real} matrix in $\mathcal{E}_{\Omega,\forall}$, in \citet{Noferini2021a} a real variant of the main algorithm is described, based on optimization over orthogonal matrices and generalized real Schur forms, allowing $2\times2$ diagonal blocks even when all eigenvalues are real. This approach requires solving exactly the inner problem of projecting real $2\times2$ matrices onto the
corresponding stable set. We refer the reader to \citet{Noferini2021a} for a full solution of this problem when $\Omega$ is either the closed left half-plane or the closed unit disk; the former case is also treated by \citet{Kuo2020}.

\subsubsection{Exact spectral assignment}
Suppose now that the complete spectrum is prescribed, including algebraic
multiplicities. Given a list $L$, possibly including repetitions, of the eigenvalues $\lambda_1,\ldots,\lambda_n\in\C$, define
\[
 \mathcal{E}_{L}
 =\left\{X\in\C^{n\times n}:
 \det(zI_n-X)=\prod_{i=1}^n(z-\lambda_i)\right\}.
\]
The set is closed, because the definition is equivalent to fixing the coefficients of the
characteristic polynomial of $X$. Every matrix in $\mathcal{E}_L$ has a Schur decomposition
with the eigenvalues in the specified order. We can therefore use
\[
 \mathcal{E}_{T}
 =\{T:T\text{ is upper triangular},\ T_{ii}=\lambda_i\},
\]
which is an affine subspace. For $B=Q^*AQ$, the inner minimization in \eqref{eq:schurtrick} is solved exactly by keeping
$B_{ij}$ for $i<j$, replacing $B_{ii}$ by $\lambda_i$, and setting the strictly
lower triangular entries to zero. Thus, the problem reduces to minimizing the objective function
\[f_L(Q) = \sum_{i>j}|B_{ij}|^2+
       \sum_{i=1}^n|B_{ii}-\lambda_i|^2, \quad B = Q^*AQ. \]
The function $f_L$ is smooth, and Theorem~\ref{thm:equivalentgeneral}
applies. In particular, repeated prescribed eigenvalues cause no additional difficulty, and all the possible corresponding Jordan structures are allowed.

Alternatively, one may choose $\mathcal{E}_T$ by allowing every possible reordering of the eigenvalues. This gives the alternative objective function
\[
 \sum_{i>j}|B_{ij}|^2+
 \min_{\pi\in S_n}\sum_{i=1}^n|B_{ii}-\lambda_{\pi(i)}|^2,
\]
where $S_n$ is the permutation group. The inner minimization is
a linear assignment problem with costs $|B_{ii}-\lambda_j|^2$. The global
minimum remains the same, but the objective function may become nonsmooth at the points
where the solution to the optimal assignment problem changes.

Spectral assignment has a long history within inverse eigenvalue problems;
see \citet{chu98,chugolub}. In particular, \citet{chu90} developed projected
gradient methods for general matrix nearness problems with spectral constraints, and
\citet{chu91} treated the problem of computing the nearest real normal matrices with prescribed spectrum. Going back to the complex case, the normality constraint changes the triangular set; in particular, $\mathcal{E}_T$ reduces to a set with a single element, namely, \[ \mathcal{E}_T = \left\{\diagm(\lambda_1,\dots,\lambda_n) \right\}.\] The objective function becomes instead
\[ 
 f(Q) = 
 \left\|Q^*AQ-\diagm(\lambda_1,\dots,\lambda_n)\right\|_F^2.
\]

\subsubsection{The normal Procrustes problem}\label{sec:nProcr}

Given a pair of matrices $A,B\in\mathbb{C}^{n\times m}$, the normal Procrustes problem, analyzed by \citet{Bierly}, is a variant of the more commonly studied orthogonal Procrustes problem where the request that the sought matrix is unitary (or real orthogonal) is replaced by a normality constraint:
\[
  \inf_{N\in\mathbb{C}^{n\times n},\,NN^*=N^*N}
      \|NA-B\|_F^2.
\]
Note that, when $A=I_n$, this problem reduces to the nearest normal matrix problem discussed earlier.
Since $N$ is normal, we can once again write $N=QDQ^*$ with $Q$ unitary and $D$ diagonal, leading once again to an optimization problem over $U(n)$. If $N$ is required to be not only normal, but also real, once again one can use orthogonal matrices and real normal blocks to design a real variant of the approach.

\subsubsection{Other examples}

Several further examples can be obtained by changing only $\mathcal{E}_T$. We limit ourselves to one example: It is clear that one can combine the normal and stable problems by imposing that $\mathcal{E}_T$ contains diagonal matrices whose diagonal elements are restricted to lie in a certain set $\Omega$. The problem studied in \cite{chu91} overlaps with this hybrid problem when $ \#  \Omega = 1$. Some other choices recover
problems for which direct methods are already available. For example,
$\Omega=\R$ gives the problem of finding the Hermitian matrix nearest to $A$, which can be simply computed as $(A+A^*)/2$.

\subsection{A bi-objective optimization viewpoint} \label{subsec:biobjective}

We describe now a method to convert a matrix nearness problem into optimization of a matrix function on a ball. This method will be instrumental for the gradient system approach in Section~\ref{subsec:gradientsystem_unstructured}.

Let $F: \mathbb{F}^{m\times n} \to \mathbb{R}$ be a function; for now we assume only that it is continuous. Given $c < F(A)$, we wish to solve the matrix distance problem in the Frobenius norm for the sublevel set (assumed nonempty) $\mathcal{E} = F^{-1}((-\infty,c])$, i.e., find
\begin{equation} \label{dist_sublevel}
    d(A, \mathcal{E}) = \min_X \norm{X-A}_F \subjectto X \in\mathbb{F}^{m\times n},\, F(X) \leq c.    
\end{equation}
We call this distance $\psi(c)$, to highlight the dependence on the function value $c$. Together with $\psi$, we consider also the function obtained by swapping the roles of $F$ and $\norm{X-A}_F$, i.e.,
\begin{equation} \label{gsphi}
    \varphi(\varepsilon) :=\min_X F(X) \subjectto X \in\mathbb{F}^{m\times n},\,\norm{X-A}_F \leq \varepsilon
\end{equation}
These maps are two different ways to study the same problem via the trade-off between the function value $c$ and the distance $\varepsilon$. We see immediately an example to make this setup more concrete.
\begin{example}[distance from instability and pseudospectral abscissa] \label{ex:spectral abscissa}
    Consider again the problem of distance to Hurwitz instability. For a matrix $A\in\mathbb{C}^{n\times n}$, define the \emph{spectral abscissa}
    \[
        \alpha(A) = \max \Re(\lambda) \subjectto \lambda \in \Lambda(A),
    \]
    and set $F(X) = -\alpha(X)$. A matrix is Hurwitz stable if and only if $\alpha(A) < 0$; hence, computing the distance to Hurwitz instability reduces to computing $\psi(c)$ in~\eqref{dist_sublevel} with $c=0$. The function $\varphi$ in this case satisfies
    \begin{align*}
        -\varphi(\varepsilon) &= \max \alpha(X) \subjectto X \in\mathbb{C}^{n\times n},\,\norm{X-A}_F \leq \varepsilon\\
        &= \max \Re(\lambda) \subjectto \lambda \in \Lambda_\varepsilon(A).
    \end{align*}
    The function $-\varphi(\varepsilon)$ for this problem is known as \emph{pseudospectral abscissa}, and usually denoted by $\alpha_\varepsilon(A)$.

For Schur stability, instead, one can take $F(X) = -\rho(X)$ and $c=-1$. More generally, one can consider the problem $F(X)=\min f(\lambda) \subjectto \lambda\in\Lambda(X)$ for an arbitrary scalar function $f:\mathbb{C} \to \mathbb{R}$~\cite{GLS23}.
\end{example}

The following result describes the connection between $\psi$ and $\varphi$.
\begin{theorem} \label{thm:psiphi}
    Suppose $F$ is continuous, and has no local minima in the interior of the ball $B(A,r)=\{X \colon \norm{X-A}_F\leq r\}$. Then,
    \begin{enumerate}
        \item in~\eqref{dist_sublevel} and~\eqref{gsphi}, the minima are reached only in points $X$ for which equality holds in the constraint;
        \item $\psi$ and $\varphi$ are continuous;
        \item $\varphi$ is strictly decreasing from $F(A) = \varphi(0)$ to $c_0:=\varphi(r)$, and $\psi$ is strictly decreasing from $r = \psi(c_0)$ to $0 = \psi(F(A))$, and the two functions are the inverse of each other on these intervals;
    \end{enumerate}
\end{theorem}
\begin{proof}
    We start from proving the properties of $\varphi$. Let $\varepsilon\in (0,r]$. If $X$ minimizes $F$ inside $B(A,\varepsilon)$ and it is not on the boundary, then $X$ must be a local minimizer of $F$; this contradicts the assumptions.

    Let now $\varepsilon_1,\varepsilon_2$ be given with $0 \leq \varepsilon_1<\varepsilon_2 \leq r$. Let $X$ be a minimizer of $F$ inside $B(A,\varepsilon_1)$. If $\varphi(\varepsilon_2) \geq \varphi(\varepsilon_1)$, then $X$ must minimize $F$ also inside the larger ball $B(A,\varepsilon_2)$, and this is impossible because it would be a non-boundary minimizer. Hence $\varphi$ is strictly decreasing.

    The function $F$ is uniformly continuous on the compact set $B(A,r)$, by the Heine--Cantor theorem: For any given $\sigma>0$ there is $\delta$ such that $\norm{X-X'}_F\leq \delta$ implies $\abs{F(X)-F(X')}\leq\sigma$. Let $X = A+\varepsilon E$, with $\norm{E}_F=1$, be an arbitrary point on the sphere $\partial B(A,\varepsilon)$. Then, for the corresponding point $X' = A+(\varepsilon+\delta)E$ on the sphere $\partial B(A,\varepsilon+\delta)$ it holds that $\abs{F(X)-F(X')} \leq \sigma$; in particular, the minima of $F$ on these two spheres cannot differ by more than $\sigma$, i.e., $\abs{\varphi(\varepsilon+\delta) - \varphi(\varepsilon)} \leq \sigma$. This proves continuity of $\varphi$.
    
    Since $\varphi$ is continuous and strictly decreasing, it has a continuous strictly decreasing inverse; we show that this inverse is $\psi$. Let $\varepsilon,c$ be given so that $\varphi(\varepsilon)=c$. Then, since the minimum is attained (only) on the boundary, it holds that:
    \begin{itemize}
        \item For all $X$ with $\norm{X-A}_F < \varepsilon$, $F(X)>c$;
        \item For all $X$ with $\norm{X-A}_F = \varepsilon$, $F(X)\geq c$, and there is a point where equality is attained.
    \end{itemize}
    These two properties together imply that the minimum in~\eqref{dist_sublevel} is $\varepsilon$, hence $\psi(c)=\varepsilon$; and moreover the equality $F(X)=c$ holds for all minimizers of~\eqref{dist_sublevel}.
\end{proof}
Under certain regularity conditions, we can easily obtain the derivative $\varphi'$ once we have computed a minimizer.
\begin{theorem}[cfr. \protect{\citet[Theorems IV.1.2, VI.1.2, VII.2.4]{GuglielmiLubichBook}}] \label{thm:gspsiderivative}
    Under the assumptions of Theorem~\ref{thm:psiphi}, assume $\varepsilon \in (0,r)$. Assume further that the minimum in~\eqref{gsphi} is achieved by a unique $X_\star$ on the sphere $\norm{X-A}_F=\varepsilon$, and that $F$ is continuously differentiable in a neighborhood of $X_\star $. Then, $\varphi'(\varepsilon) = -\norm{\nabla F(X_\star )}_F$.
\end{theorem}
\begin{proof}
    By Theorem~\ref{thm:reg min in normal}, $-\nabla F(X_\star)$ must belong to the normal cone of the ball $\norm{X-A}_F \leq \varepsilon$, i.e.,    
    \begin{equation} \label{phi first order optimality}
        -\nabla F(X_\star ) = \mu (X_\star -A), \quad \mu \geq 0
    \end{equation}
    We shall show that $\mu = -\frac{\varphi'(\varepsilon)}{\varepsilon}$; and then the result follows by taking norms.
    We reparametrize the problem: Write $X = A + \varepsilon E$, so
    \begin{equation} \label{minF_reparametrized}
    \varphi(\varepsilon) = \min_E F(A+\varepsilon E) \subjectto E \in \mathbb{F}^{m\times n}, \, \norm{E}_F=1.    
    \end{equation}
    The assumptions of Danskin's theorem~\cite[Proposition~B.25]{Bertsekas1999}  hold for~\eqref{minF_reparametrized}, so we can compute $\varphi'(\varepsilon)$ ignoring the contribution $\frac{d}{d\varepsilon} E_\star$, i.e.,
    \[
    \varphi'(\varepsilon) = \frac{d}{d\varepsilon} F(A+\varepsilon E_\star ) = \scalar{\nabla F(X_\star ), E_\star } = \scalar{-\mu(X_\star -A),E_\star } = \scalar{-\mu\varepsilon E_\star ,E_\star  } = -\mu\varepsilon.
    \]
\end{proof}
These results suggest that (when the no-local-minima condition holds) we can compute $\psi(c)$ (and hence solve the distance problem~\eqref{dist_sublevel}) via $\varphi(\varepsilon)$, using a form of univariate rootfinding to solve $\varphi(\varepsilon) - c = 0$: for instance, bisection or Newton's method, given sufficient regularity. Newton's method is particularly convenient since Theorem~\ref{thm:gspsiderivative} shows that $\varphi'$ is available explicitly once we have computed a minimizer $X_\star $.

Using $\varphi$ as a means to solve nearness problems is a popular approach in the literature. An informal motivation to explain why $\varphi$ is generally easier to compute than $\psi$ is that it is easier to handle complexity in the function to optimize (via derivatives) than complexity in the optimization domain (via projections or reparametrizations).

\subsection{The gradient system approach} \label{subsec:gradientsystem_unstructured}
In this section, we describe a general method to compute $\varphi(\varepsilon)$ for a large class of problems; the approach is described in more detail in~\citet{GuglielmiLubichBook}.

Let $\mathcal{M}$ be a Riemannian manifold immersed in $\mathbb{R}^N$ (with the induced metric). A \emph{gradient system} of ordinary differential equations on $\mathcal{M}$ is the system of ODEs in the function $X(t)$
\begin{equation} \label{gradientsystem}
    \dot{X} = -\Pi_{T_X(\mathcal{M})}(\nabla V(X)),
\end{equation}
where $V \in \mathcal{C}^1(\mathcal{M})$ is a real-valued function.

The projection of $\nabla V(X)$ onto the tangent space in $X$ is known as the \emph{Riemannian gradient} of $V$, and denoted by $\grad(V)$. Since the right-hand side lies in the tangent space of $\mathcal{M}$, one can show that $X(t)\in\mathcal{M}$ for all $t$ such that $X(t)$ is defined. Moreover, $\frac{d}{dt} V(X(t)) = -\norm{\grad(V)}_2^2 \leq 0$, hence the function $V$ decreases monotonically along the solutions. If the solutions converge to a point $X_\star\in \mathcal{M}$, then $X_\star$ is a critical point of $V$ on $\mathcal{M}$.

From an optimizer's point of view, this gradient system is a continuous-time version of a gradient descent method to compute a (local) minimum of $V(X)$ disguised as a differential equation. Indeed, to solve a gradient system in the simplest possible way, the forward Euler method, we need to:
\begin{enumerate}
    \item compute the right-hand side $F_k$ of~\eqref{gradientsystem} for the current iterate $X_k$;
    \item compute $X_k + h F_k$, with a suitable step-size $h$; note that this sum generally is not an element of the manifold $\mathcal{M}$;
    \item project $X_k + h F_k$ back to a point on the manifold, obtaining $X_{k+1} \in \mathcal{M}$.
\end{enumerate}
This is essentially gradient descent on a Riemannian manifold (see, e.g., \citet[Section~4.3]{Boumal2023}); note in particular Step 3, the \emph{retraction}, which is a characteristic feature of optimization on manifolds, not needed in $\mathbb{R}^N$. One must take care to include a retraction when solving~\eqref{gradientsystem} numerically: While the exact solution is guaranteed to be in $\mathcal{M}$, numerically one must enforce this invariant. In practice, the distance between the ODE and optimization worlds is small, since also in the ODE one typically uses optimization ideas such as the Armijo step-size; see, e.g., \citet[Algorithm~1]{GuglielmiLubichBook}.

To compute~\eqref{gsphi} for given $\varepsilon$, we reuse the parametrization~\eqref{minF_reparametrized}, and integrate numerically the gradient system on the sphere $\mathcal{S} = \{E \colon \norm{E}_F=1\}$ given by
\begin{equation} \label{gradient system on the sphere}
    \dot{E} = -\Pi_{T_E\mathcal{S}}(G) = -G + \scalar{G,E}E, \quad G = \nabla_X F(A+\varepsilon E),
\end{equation}
following a single trajectory from an arbitrary starting point, until we reach a stationary point: In practice this stationary point is often a local minimum, especially with a well-chosen starting point.

Caution is needed because in many applications $F$ is not differentiable everywhere, in particular when it depends on the eigenvalues or singular values of $X$: For instance, in Example~\ref{ex:spectral abscissa}, $F(X) = -\alpha(X)$ is differentiable only when $X$ has a unique, simple rightmost eigenvalue $\lambda$, and in this case $\nabla F(X) = -\frac{1}{x^*y}xy^*$, where $x$ and $y$ are its left and right associated eigenvectors, normalized with $\norm{y}_2=\norm{x}_2=1$ and $y^*x>0$. 

In many practical cases one still gets a usable result out of these methods even by ignoring differentiability issues~\cite[Remark II.2.8]{GuglielmiLubichBook}: Computing an eigenvalue decomposition of $X$ and sorting the eigenvalues automatically picks a subgradient in cases where there is an exact tie, which is in any case unlikely (non-generic). Moreover, the method ensures monotonic descent of $F$, if implemented properly, for instance with Armijo step-size. However, the usual concerns of non-smooth optimization apply, especially if the minimizer is a non-differentiable point~\cite{asl20}.

For more robust methods, there is a long line of research on BFGS-based methods to optimize functions based on eigenvalues, which can be non-smooth and non-convex: See for instance~\citet{burke2005_robust,Lewis2009,Curtis2017}; in particular, the GRANSO method for constrained optimization. It is a BFGS-based method for constrained optimization, with carefully chosen tuning parameters and termination criteria, designed especially for functions that are differentiable almost everywhere, but may have non-smooth points. This is exactly the setup we have when studying functions based on eigenvalues, like the spectral abscissa.

\citet{grundel2014} studied the normal cones of weakly Hurwitz stable matrices and the subgradients of the spectral abscissa, focusing on a special case for which there is a defective and derogatory eigenvalue on the imaginary axis, or in other words, an eigenvalue associated with both more than one Jordan block and at least one Jordan block of size greater than one; see also \citet{burke2001_variational} for results on the more general case of a function $F(X)$ depending only on the spectrum of $X$.

\subsection{Rank-constrained gradient systems}

A significant practical improvement in this setup comes from exploiting low-rank structure in $E_\star$. As we have seen in Theorem~\ref{thm:1dmin}, in distance to $\Omega$-instability problems $\Delta_\star $ (and hence also $E_\star  = -\Delta_\star  / \norm{\Delta_\star }_F$) has rank $1$ for complex perturbations. More generally, if $\nabla F(X_\star)$ exists and $\varphi'(\varepsilon_\star )\neq 0$, we can see that $E_\star $ has the same rank as $\nabla F(X_\star)$  also from~\eqref{phi first order optimality}, or from imposing stationarity in~\eqref{gradient system on the sphere}. In a problem where the optimal $E_\star $ is guaranteed to have rank $1$, we can restrict the minimization to
\begin{align}
    \varphi(\varepsilon) &= \min_{E} F(A+\varepsilon E) \subjectto E \in \mathcal{S} \cap \mathcal{R}_1 , \label{varphi_rango} \\
    \mathcal{S} \cap \mathcal{R}_1 &:= \{E\in \mathbb{F}^{m\times n} \colon \norm{E}_F=1, \,\rank(E)=1\}. \nonumber
\end{align}
(We denote with $\mathcal{R}_r$ the manifold of matrices with rank exactly $r$.)

The following lemma simplifies the projection on the tangent space of this manifold.
\begin{lemma} \label{lem:sphere and cone}
    Let $\mathcal{M}$ be a manifold that is closed by rescaling, i.e., $E\in\mathcal{M}$ implies $tE \in \mathcal{M}$ for each $t>0$. Then, $\mathcal{S} \cap \mathcal{M} $ is a clean intersection, and for each $E\in\mathcal{S} \cap \mathcal{M}$ we have
    \[
    \Pi_{T_E(\mathcal{S}\cap \mathcal{M})} (G) = \Pi_{T_E\mathcal{S}} \Pi_{T_E\mathcal{M}} G = 
    P - \scalar{P,E}E, \quad P = \Pi_{T_E\mathcal{M}}(G).
    \]
\end{lemma}
\begin{proof}
    Note that $E \in T_E\mathcal{M}$, since it is the tangent vector of the rescaling curve $f(t) = tE$. Since $T_E\mathcal{S} = (\operatorname{span}(E))^\perp$, the intersection is transversal, and hence $T_E(\mathcal{S} \cap \mathcal{M}) = T_E\mathcal{S} \cap T_E\mathcal{M}$.
    
    Note also that $P - \scalar{P,E}E \in T_E\mathcal{M}$, since it is the difference of two elements of $T_E\mathcal{M}$. Given two vector subspaces $\mathcal{W},\mathcal{V}$, there are various elementary ways to show that $\Pi_{\mathcal{W}}\Pi_{\mathcal{V}} x \in \mathcal{V}$ for all $x$ implies $\Pi_{\mathcal{W}}\Pi_{\mathcal{V}} = \Pi_{\mathcal{W}\cap\mathcal{V}}$; 
    or we can also swat a fly with a cannon and invoke Von Neumann's alternating projection theorem (Theorem~\ref{th:vonneumann} below). In this case, the algorithm reaches convergence after the first step.
\end{proof}
We can thus compute the minimum~\eqref{varphi_rango} by integrating the gradient system restricted to $\mathcal{S} \cap \mathcal{R}_1$
\begin{gather}
    \dot{E} = -\Pi_{T_E(\mathcal{S} \cap \mathcal{R}_1)}(G) = -P + \scalar{P,E}E, \label{gradientsys_R1}
    \\ G = \nabla_X F(A+\varepsilon E), \quad P = \Pi_{T_E \mathcal{R}_1}(G) \nonumber
\end{gather}
We introduce the parametrization
\[
\mathcal{S} \cap \mathcal{R}_1 = \{uv^* \colon u\in\mathbb{F}^m, v\in\mathbb{F}^n, \norm{u}_2=\norm{v}_2=1\},
\]
which is unique only up to a phase factor in $u$ and $v$. After some work to compute the tangent space of this manifold and the corresponding projection \cite[Lemma~II.3.4]{GuglielmiLubichBook}, \eqref{gradientsys_R1}~reduces to the two coupled differential equations 
\begin{equation} \label{parametrized tangent system}
\begin{aligned}
    \dot{u} &= -(I_m-uu^*)Gv - \tfrac{i}{2}\Im (u^*Gv)u,\\
    \dot{v} &= -(I_n-vv^*)G^*u - \tfrac{i}{2}\Im (v^*G^*u)v.
\end{aligned}    
\end{equation}
Since~\eqref{parametrized tangent system} is still a gradient system, $F(X(t)) = F(A+\varepsilon u(t)v(t)^*)$ is monotonic along its trajectories, and if these trajectories converge they reach a stationary point of $F$ on $\mathcal{S} \cap \mathcal{R}_1$. \citet[Theorem~3.6 and Remark~3.7]{GuglielmiLubichBook} show that this is also a stationary point on $\mathcal{S}$, under the assumptions that $\rank(G)=1$ and $P\neq0$. In practice, it is observed numerically that this is a local minimizer of~\eqref{gsphi}, and of course the hope is that it is the global minimizer.

For other target functions $F$, the rank of the optimal perturbation $E_\star$ can be larger than $1$, and hence one needs to consider generalizations of the manifold $\mathcal{S} \cap \mathcal{R}_1$ to higher rank (with more complicated projections). Several individual problems are studied in~\citet{GuglielmiLubichBook}; in general, the heuristic (under various regularity assumptions) is that the rank of $E_\star $ is the rank of $\nabla F(X_\star)$, which in turn depends on how many eigenvalues, eigenvectors or singular values of $X_\star $ enter the expression of the function $F$ (or of the target set $\mathcal{E}$) as ``active constraints''. For instance, we shall see below in Section~\ref{sec:nearestdef} that, for the distance to defectivity, a rank-$2$ gradient flow is relevant. This is still compatible with the possible existence of a rank-$1$ optimal perturbation (see Theorem~\ref{thm:defectiveuin}). For the distance to stability, an arbitrary number $r$ of eigenvalues of $X$ may lie on the boundary $\partial\Omega$ \cite[Section~VI.1]{GuglielmiLubichBook}.

This gradient system approach to compute the distance of $A$ to $\mathcal{E} = F^{-1}((-\infty, c])$ can be summarized in an algorithm as follows. Preliminarily, identify a suitable submanifold $\mathcal{M} \subseteq \mathcal{S}$ of the Frobenius-norm unit sphere, which contains $E_\star$. Start from an estimate $\varepsilon_0$ of the distance, and then for $k=0,1,2,\dots$:
\begin{enumerate}
    \item Starting from an arbitrary initial point $E_{k,0}$, integrate the gradient system $\dot{E} = -\Pi_{T_E\mathcal{M}}(\nabla F(A+\varepsilon_k E))$ using an ODE solver (the ``inner iteration''), until $\dot{E}=0$ is reached; this gives a stationary point $X=A+\varepsilon_k E_{k,\star}$ of $F$ on $\mathcal{M}$.
    \item If $F(X)$ is sufficiently close to $c$, terminate. Otherwise, compute a new guess $\varepsilon_{k+1}$, for instance with the Newton method $\varepsilon_{k+1} = \varepsilon_k - \varphi'(\varepsilon_k)^{-1} (\varphi(\varepsilon_k)-c) = \varepsilon_k + \norm{\nabla_{X}(F)}_F^{-1}(F(X)-c)$ (where we use Theorem~\ref{thm:gspsiderivative} and assume that the computed $X$ is a global minimizer), and continue.
\end{enumerate}
A good initial value $E_{k,0}$ is beneficial not only to reduce the number of inner iterations in the ODE solver, but also to drive the method toward the global minimum. At each $k>0$, we can use as $E_{k+1,0}$ the stationary point $E_{k,\star}$ obtained at the previous step (\emph{warm-starting}). At the first step $k=0$, we resort to problem-specific heuristics: For instance, for the distance to $\Omega$-instability, we can take as $u,v$ the left and right eigenvectors of $A$ associated with its eigenvalue closest to $\Omega$, normalized so that $\norm{u}_2=\norm{v}_2=1$, mimicking the property of the solution $X$ given in Remark~\ref{rem:double role of u and v}; or we can choose a target eigenvalue $\lambda\in\Omega$ close to $\Lambda(A)$ and compute $u_{\min},v_{\min}$ from the SVD of $A-\lambda I_n$.

An alternative to Newton's method and bisection is the \emph{hybrid expansion-contraction} algorithm \cite{mitchell23}. Start from a sufficiently large $\varepsilon_0>0$ and a matrix $E_0$ with $\norm{E_0}_F=1$ such that $F(A+\varepsilon_0 E_0)< c$. Then, at each step $k=0,1,2,\dots$, alternate between the following two steps:
\begin{description}
    \item[Contraction phase] Keep $E_k$ fixed, and compute $\varepsilon_{k+1} \in (0,\varepsilon_k)$ such that $F(A+\varepsilon_{k+1}E_k) = c$ (the ``contraction phase''). This can be done cheaply via a bisection-type algorithm, since $F(A + 0 \cdot E_k) > c$, and the algorithm ensures at each step the invariant $F(A+\varepsilon_k E_k) < c$.
    \item[Expansion phase] Keep $\varepsilon_{k+1}$ fixed, and use the gradient system described above (or any other eigenvalue optimization method) to compute $\varphi(\varepsilon_{k+1})$, as well as the minimizer matrix $E_{k+1,\star}$ (with $\norm{E_{k+1}}_F=1$) such that 
    \[
    \varphi(\varepsilon_{k+1}) = F(A + \varepsilon_{k+1} E_{k+1,\star}) = \min_{\norm{E}_F=1} F(A+\varepsilon_{k+1} E).
    \]
    Clearly, $\varphi(\varepsilon_{k+1}) \leq c$, since $E_k$ is a possible value for $E$. If equality holds, we terminate; otherwise, we continue onto the next step setting $E_{k+1} = E_{k+1,\star}$.
\end{description}
\citet{mitchell23} describe this approach for more general problems, and prove its quadratic convergence.

\section{Distance from an intersection} \label{sec:projections}

In several contexts, especially involving structured matrices, one wishes to solve a (Frobenius-norm) distance problem with an intersection $\mathcal{E} = \mathcal{E}_1 \cap \mathcal{E}_2$, where the two individual projections $\Pi_{\mathcal{E}_1}$ and $\Pi_{\mathcal{E}_2}$ are simple to compute.

A famous example of such a problem is the nearest correlation matrix \cite{Higham2002}: A \emph{correlation matrix}, in statistics and economics, is a real positive semidefinite matrix with diagonal entries equal to 1, i.e., an element of $\mathcal{E}_1 \cap \mathcal{E}_2$, where $\mathcal{E}_1,\mathcal{E}_2 \subseteq \mathbb{R}^{n\times n}$ are
\begin{equation} \label{nearestcorr}
\begin{aligned}
\mathcal{E}_1 &= \mathcal{E}_{\succeq 0} = \{X \colon \text{$X$ is positive semidefinite}\},\\
\mathcal{E}_2 &= \{X \colon \diag(X) = \begin{bmatrix}
    1,1,\dots,1
\end{bmatrix}\}.
\end{aligned}
\end{equation}
Estimation methods designed to reconstruct a correlation matrix from noisy data often return a matrix $A$ which is not in $\mathcal{E}_1\cap\mathcal{E}_2$, so one wishes to project it to the nearest matrix that satisfies the two constraints. The norm of interest often is a weighted Frobenius norm, to reflect the fact that different components of the data may have different statistical errors.

A classic technique for computing the (Frobenius) distance from an intersection is that of \emph{alternating projections} (see, e.g., the book \cite{raydan_book}). We now present two famous alternating projection algorithms in the general vector case $a,x\in\mathbb{R}^N$, using lowercase letters for the variables, even though in our setting we are typically interested in the case where $\mathbb{R}^N \cong \mathbb{F}^{m\times n}$ is a matrix space.

\subsection{Von Neumann's alternating projections} \label{subsec:vonneumann}
The earliest alternating projection algorithm is due to~\citet{vonneumann_alternating}. We present it in the case when $\mathcal{E}_1 =\mathcal{V}_1,\mathcal{E}_2 = \mathcal{V}_2$ are vector subspaces, for simplicity, but with a translation the setup transfers to affine subspaces. In the algorithm, we start from $x_0=a$, and apply repeatedly the two projection operators: For $k=0,1,2,\dots$
\begin{equation} \label{voonneumannproj}
y_{k+1} = \Pi_{\mathcal{V}_1}(x_k), \quad x_{k+1} = \Pi_{\mathcal{V}_2}(y_{k+1}).    
\end{equation}
The iterates $y_k$ and $x_k$ zig-zag between the two subspaces, as shown in Figure~\ref{fig:vonneumann}. 
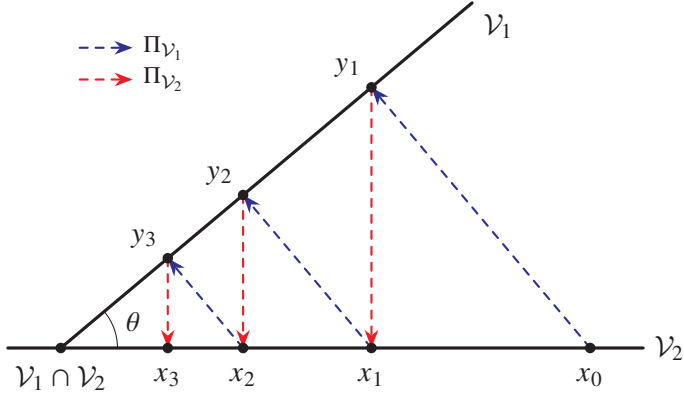
\begin{figure}
    \centering
    \begin{tikzpicture}[
        >={Stealth[length=2.4mm]},
        line cap=round,
        dot/.style={circle,fill,inner sep=1.4pt},
        projV1/.style={->,thick,dashed,blue},
        projV2/.style={->,thick,dashed,red},
      ]
     
      \def\ang{40}      
     
      \draw[very thick] (180:0.7) -- (0:7.7)            node[right]      {$\mathcal{V}_2$};
      \draw[very thick] (0,0) -- (\ang:7.1)  node[below right]{$\mathcal{V}_1$};
     
      \coordinate (O) at (0,0);
     
      \draw (0:0.75) arc (0:\ang:0.75);
      \node at ({\ang/2}:1.02) {$\theta$};
     
      \coordinate (x0) at (0:7.0000);
      \coordinate (y1) at (\ang:5.3623);
      \coordinate (x1) at (0:4.1078);
      \coordinate (y2) at (\ang:3.1468);
      \coordinate (x2) at (0:2.4107);
      \coordinate (y3) at (\ang:1.8468);
      \coordinate (x3) at (0:1.4148);
     
      \draw[projV1] (x0) -- (y1);
      \draw[projV2] (y1) -- (x1);
      \draw[projV1] (x1) -- (y2);
      \draw[projV2] (y2) -- (x2);
      \draw[projV1] (x2) -- (y3);
      \draw[projV2] (y3) -- (x3);
     
      \node[dot] at (x0){}; \node[below=4pt]        at (x0){$x_0$};
      \node[dot] at (y1){}; \node[above left]       at (y1){$y_1$};
      \node[dot] at (x1){}; \node[below=4pt]        at (x1){$x_1$};
      \node[dot] at (y2){}; \node[above left]       at (y2){$y_2$};
      \node[dot] at (x2){}; \node[below=4pt]        at (x2){$x_2$};
      \node[dot] at (y3){}; \node[above left]       at (y3){$y_3$};
      \node[dot] at (x3){}; \node[below=4pt]        at (x3){$x_3$};
      \node[dot] at (O){};  \node[below=3pt]        at (O){$\mathcal{V}_1\cap\mathcal{V}_2$};
     
      \begin{scope}[shift={(0.25,3.55)}]
        \draw[projV1] (0,0.42) -- (0.7,0.42)
            node[right,black]{\footnotesize $\Pi_{\mathcal V_1}$};
        \draw[projV2] (0,0) -- (0.7,0)
            node[right,black]{\footnotesize $\Pi_{\mathcal V_2}$};
      \end{scope} 
    \end{tikzpicture}
    
    \caption{Von Neumann's alternating projection algorithm}
    \label{fig:vonneumann}
\end{figure}

The analysis of this algorithm is not complicated with modern linear algebra tools, since the projections $\Pi_{\mathcal{V}_i}$ are linear operators.
\begin{theorem}[\citet{vonneumann_alternating}] \label{th:vonneumann}
    In the iteration~\eqref{voonneumannproj}, $\lim_{k\to\infty} x_k = \lim_{k\to\infty} y_k = \Pi_{\mathcal{V}_1\cap\mathcal{V}_2}(a)$.
\end{theorem}
\begin{proof}
 Let $Q_1 \in \mathbb{R}^{N\times r_1}, Q_2 \in \mathbb{R}^{N\times r_2}$ be rectangular matrices whose columns are orthonormal bases for $\mathcal{V}_1,\mathcal{V}_2$. Then, $\Pi_{\mathcal{V}_i} = Q_i Q_i^*$ for $i=1,2$. The eigenvalues of $\Pi_{\mathcal{V}_2}\Pi_{\mathcal{V}_1}$ coincide (up to additional zeros) with those of $Q_2^* Q_1 Q_1^*Q_2$, which are the squares of the singular values of $Q_1^*Q_2$. These singular values are the cosines of the so-called \emph{principal angles} between the subspaces $\mathcal{V}_1,\mathcal{V}_2$. In particular, calling $W$ the eigenvector matrix, we have
\[
\lim_{k\to\infty} (\Pi_{\mathcal{V}_2}\Pi_{\mathcal{V}_1})^k = \lim_{k\to\infty} W \begin{bmatrix}
    \cos(\theta_1)^{2k} \\ & \ddots  \\ && \cos(\theta_{\min(r_1,r_2)})^{2k} \\ & && 0 \\ && & & \ddots \\ &&&&&0
\end{bmatrix} W^{-1}
\]
The $j$th eigenvalue converges to $0$ if $\cos(\theta_j)<1$, and is constantly equal to $1$ if $\cos(\theta_j)=1$, i.e., $\theta_j=0$. Hence, some bookkeeping (which we omit) shows that the limit displayed above is precisely the orthogonal projection onto $\mathcal{V}_1\cap\mathcal{V}_2$.
\end{proof}
This proof also reveals that the convergence speed is geometric with rate $\cos(\theta_F)^{2}$, where $\theta_F$ is the smallest non-zero principal angle (\emph{Friedrichs angle}). In particular, the method may be very slow when $\theta$ is small, as one can infer also from Figure~\ref{fig:vonneumann}. More details on the subspace case are in~\citet{botspit}. For two smooth manifolds that intersect transversally, local linear convergence results for the alternating projections method are proved in \cite[Theorem 4.3]{LM08}.

\subsection{Dykstra's alternating projection algorithm}
For closed convex sets $\mathcal{E}_1,\mathcal{E}_2$ with nonempty intersection, Von Neumann's alternating projection algorithm converges to a point in the intersection \cite{BB96}. However, when the sets are not affine subspaces, this point is usually not the projection of $a$: See, e.g., \citet[Example~5.1]{raydan_book}. A modification introduced by \citet{dykstra1983algorithm}
restores convergence to the nearest point. The classic formulation uses two additional variables $\delta_1,\delta_2$, one associated to each set, called \emph{Dykstra corrections}: These are updated at each step and added to $x_k,y_k$ before every projection. Starting from $x_0=a, \delta_{1,0} = \delta_{2,0} = 0$, we compute for $k=0,1,\dots$
\begin{equation} \label{dykstra}
\begin{aligned}
    y_{k+1} &= \Pi_{\mathcal{E}_1}(x_k + \delta_{1,k}),\\
    \delta_{1,k+1} &=\delta_{1,k} + (x_k - y_{k+1}),\\
    x_{k+1} &= \Pi_{\mathcal{E}_2}(y_{k+1} + \delta_{2,k}),\\
    \delta_{2,k+1} &=\delta_{2,k} + (y_{k+1} -x_{k+1}).
\end{aligned}
\end{equation}
The two sequences $x_k,y_k$ converge to $\Pi_{\mathcal{E}_1\cap \mathcal{E}_2}(a)$ whenever the target set $\mathcal{E}_1\cap \mathcal{E}_2$ is nonempty~\cite{boyle1986}. The meaning of the two corrections is quite obscure at first, so we give some more intuition behind them.

Recall from Section~\ref{subsec:normal cones} that, for a convex set $\mathcal{E} = \mathcal{E}_1 \cap \mathcal{E}_2$, $x\in\mathcal{E}$ is the projection of $a$ if and only if $\delta = a-x$ belongs to the normal cone $N_x \mathcal{E}$~\eqref{normal cone for convex E}; and that if the constraint qualification $\operatorname{relint}(\mathcal{E}_1) \cap \operatorname{relint}(\mathcal{E}_2) \neq \emptyset$ holds, then $N_x\mathcal{E} = N_x\mathcal{E}_1 + N_x\mathcal{E}_2$, i.e., there are $\delta_1 \in N_x\mathcal{E}_1$ and $\delta_2\in N_x\mathcal{E}_2$ such that $\delta = \delta_1 + \delta_2$. The two sequences of Dykstra corrections approximate these two summands. Indeed, the way the corrections are computed in the algorithm ensures that at each step $k$ we have
\begin{gather*}
    a = x_k + \delta_{1,k} + \delta_{2,k} = y_{k+1} + \delta_{1,k+1} + \delta_{2,k},\\
    y_{k} \in \mathcal{E}_1, \quad \delta_{1,k} \in N_{y_{k}}\mathcal{E}_1,\\
    x_k \in \mathcal{E}_2, \quad \delta_{2,k} \in N_{x_k}\mathcal{E}_2.
\end{gather*}
Putting $x_k$ and $y_k$ in the sideline, we can interpret Dykstra's algorithm as ``improving'' twice in each step the current iterate $a - \delta_{1,k} - \delta_{2,k}$: first by replacing $\delta_{1,k}$ with $\delta_{1,k+1}$ (obtained by projecting $a-\delta_{2,k}$ on $\mathcal{E}_1$), and then by replacing $\delta_{2,k}$ with $\delta_{2,k+1}$, (obtained by projecting $a-\delta_{1,k+1}$ on $\mathcal{E}_2$). If the correction sequences converge, their limits satisfy $a = x_\star + \delta_{1,\star} + \delta_{2,\star}$, with $x_\star  \in \mathcal{E}_1\cap\mathcal{E}_2$, $\delta_{1,\star} \in N_{x_\star }\mathcal{E}_1$ and $\delta_{2,\star} \in N_{x_\star }\mathcal{E}_2$; this relation proves that $x_\star $ is the sought projection.

To make this description more formal, we need to remove the quotes around ``improving''. A way to do this is through Lagrangian duality. Formulate the projection problem as
\begin{equation} \label{dykstraprimal}
    \tfrac12 d(a, \mathcal{E}_1\cap \mathcal{E}_2)^2 = \min \tfrac12\norm{x-a}^2 \subjectto x=z_1,\, x=z_2,\, z_1 \in \mathcal{E}_1,\, z_2\in\mathcal{E}_2,
\end{equation}
with two additional copies of $x$ and equality constraints between them. Write the Lagrangian dual of~\eqref{dykstraprimal}
\begin{equation} \label{dykstradual}
g(\delta_1,\delta_2) = \inf_{x,z_1,z_2} \tfrac12\norm{x-a}^2 + \scalar{\delta_1, x-z_1} + \scalar{\delta_2, x-z_2} \subjectto z_1\in \mathcal{E}_1, z_2\in\mathcal{E}_2;
\end{equation}
then $\tfrac12 d(a, \mathcal{E}_1\cap \mathcal{E}_2)^2 = \sup_{\delta_1,\delta_2} g(\delta_1,\delta_2)$, by duality. The following result holds.
\begin{proposition}[see e.g.~\citet{tibshirani2017dykstra}] \label{prop:dykstraascent}
    Dykstra's projection algorithm~\eqref{dykstra} is coordinate ascent on $g$: 
    \begin{subequations}
        \begin{align}
        \delta_{1,k+1} &\in \arg\max_{\delta} g(\delta,\delta_{2,k}), \label{dykstraascent1}\\       
        \delta_{2,k+1} &\in \arg\max_{\delta} g(\delta_{1,k+1},\delta). \label{dykstraascent2}
        \end{align}
    \end{subequations}
\end{proposition}
\begin{proof}
    It is enough to show~\eqref{dykstraascent1}; then~\eqref{dykstraascent2} is analogous by symmetry in the two variables. Recall that $y_{k+1} = \Pi_{\mathcal{E}_1}(a-\delta_{2,k})$, and $y_{k+1} = a - \delta_{1,k+1} - \delta_{2,k}$.

    We get an upper bound that does not depend on $\delta_1$ if we plug in $x=z_1=y_{k+1}$ instead of minimizing in~\eqref{dykstradual}: For each $\delta_1$,
    \begin{equation} \label{dykstraupper}
        g(\delta_{1},\delta_{2,k}) \leq \min_{z_2\in\mathcal{E}_2} \tfrac12 \norm{y_{k+1}-a}^2 + \scalar{\delta_{2,k},y_{k+1}-z_2}.
    \end{equation}
    We shall show that equality in~\eqref{dykstraupper} holds when $\delta_1 = \delta_{1,k+1}$. Using gradients, we see that the unconstrained minimum of~\eqref{dykstradual} in $x$ is achieved when $x-a +\delta_{1,k+1}+\delta_{2,k} = 0$, i.e., $x=y_{k+1}$; hence
    \[
    g(\delta_{1,k+1}, \delta_{2,k}) = \min_{z_1\in\mathcal{E}_1,z_2\in\mathcal{E}_2} \tfrac12\norm{y_{k+1}-a}^2 + \scalar{\delta_{1,k+1},y_{k+1}-z_1} + \scalar{\delta_{2,k},y_{k+1}-z_2}.
    \]
    The only summand which depends on $z_1$ is the middle one; it achieves its minimum $0$ when $z_1=y_{k+1}$, since $\delta_{1,k+1} \in N_{y_{k+1}}(\mathcal{E}_1)$.
    Hence
    \[
    g(\delta_{1,k+1}, \delta_{2,k}) = \min_{z_2\in\mathcal{E}_2} \tfrac12\norm{y_{k+1}-a}^2 +  \scalar{\delta_{2,k},y_{k+1}-z_2},
    \]
    which coincides with the upper bound~\eqref{dykstraupper}.
\end{proof}
Proposition~\ref{prop:dykstraascent} is useful as interpretation because it gives a practical description of Dykstra's algorithm as a dual ascent method, and gives an interpretation of the corrections $\delta_1,\delta_2$ as dual variables. It does not prove immediately convergence of Dykstra's algorithm~\eqref{dykstra} by itself, since it does not exclude that the dual coordinate ascent could stall at a smaller value than the primal optimum; also note that convergence in Dykstra's algorithm holds even without constraint qualifications~\cite{boyle1986}, but a full proof is more involved.

When one of the two sets (say, $\mathcal{E}_2$) is an affine subspace $B+\mathcal{V}$, then $N_{x}\mathcal{E}_2 = \mathcal{V}^\perp$ independently of $x\in\mathcal{E}_2$, and hence thanks to orthogonality we can omit the correction $\delta_{2,k}\in N_{x_k}\mathcal{E}_2$ when projecting:
\[
x_{k+1} = \Pi_{\mathcal{E}_2}(y_{k+1} + \delta_{2,k}) = \Pi_{\mathcal{E}_2}(y_{k+1})
\]
When both $\mathcal{E}_1,\mathcal{E}_2$ are affine, we can omit both corrections, and we recover Von Neumann's algorithm.
\begin{example}[Nearest correlation matrix, \citet{Higham2002}]
    We solve the problem~\eqref{nearestcorr} with Dykstra's algorithm, omitting the correction term corresponding to the affine subspace $\mathcal{E}_2$. Starting from $X_0=A, \Delta_0=0$, we compute for $k=0,1,\dots$
    \begin{align*}
        Y_{k+1} &= \Pi_{\mathcal{E}_1}(X_k + \Delta_k),\\
        \Delta_{k+1} &= \Delta_k + (X_k - Y_{k+1}),\\
        X_{k+1} &= \Pi_{\mathcal{E}_2}(Y_{k+1}).
    \end{align*}
    \citet{Higham2002} actually solves this problem in the more general case of the norm $\norm{A}_W = \norm{W^{1/2}AW^{1/2}}_F$, which is natural in the applications. The analysis of Dykstra's algorithm is unchanged, since~\eqref{dykstra} is unaffected by a linear change of variables, but computing the two projections $\Pi_{\mathcal{E}_1}, \Pi_{\mathcal{E}_2}$ in this norm requires a little more work.
\end{example}

We have seen that the alternating projection algorithm can converge at a linear rate arbitrarily close to 1, already in its version for subspaces. A quadratically convergent algorithm to compute the nearest correlation matrix is proposed by~\citet{Qi2006}. We give an account of the main ideas, in the case of the Frobenius norm. 

We have seen that the projection $X$ is the unique matrix in $\mathcal{E}_1\cap \mathcal{E}_2$ that satisfies the normal cone condition $X = A - \Delta_1 - \Delta_2$. Since $N_X\mathcal{E}_2$ is the space of diagonal matrices, we can write $\Delta_2 = \diagm(\mu)$ for some $\mu\in\mathbb{R}^n$; then, $X = \Pi_{\mathcal{E}_1}(A - \diagm(\mu))$, by the normal cone condition on $\mathcal{E}_1$. In particular, $\mu$ is the unique solution of $\varphi(\mu)=0$, with
\begin{equation} \label{ncmqi}
    \varphi(\mu) = \diag(\Pi_{\mathcal{E}_1}(A - \diagm(\mu))) - e,
\end{equation}
and $e$ is the vector of all ones. We would like to solve this nonlinear system of $n$ equations in $n$ unknowns with Newton's method; however, applying this method is not obvious, since the projection onto positive semidefinite matrices $\Pi_{\mathcal{E}_1}$ is not differentiable everywhere. \citet{Qi2006} show that $\Pi_{\mathcal{E}_1}$ (and hence also $\varphi$ in~\eqref{ncmqi}) is \emph{strongly semi-smooth}: A function $\varphi: \mathbb{R}^N \to \mathbb{R}^N$ is strongly semi-smooth if it is  directionally differentiable, and for each $V$ in the Clarke subdifferential $\partial \varphi(x+h)$ it holds that
\begin{equation} \label{semismooth}
    \varphi(x+h) = \varphi(x) + Vh + \mathcal{O}(\norm{h}^2).
\end{equation}
Note that~\eqref{semismooth} is different from a standard Taylor expansion, since the subdifferential $V$ is computed in $x+h$, not in $x$. Thanks to semi-smoothness (and assuming that every $V\in \partial\varphi(\mu_\star )$ is invertible), the Newton method
\[
\mu_{k+1} = \mu_k - V_k^{-1}\varphi(\mu_k), \quad V_k \in \partial \varphi(\mu_k)
\]
is locally quadratically convergent~\cite[Theorem~2.1 and Corollary~3.7]{Qi2006}. \citet{Borsdorf2010} give practical details to implement this method as an efficient $\mathcal{O}(n^3)$ algorithm, using Krylov methods and preconditioning.

An algorithm which is formally similar to Dykstra's projection algorithm is the \emph{alternating-direction method of multipliers} (see, e.g., \citet{admm}). Given two convex functions $f,g: \mathbb{R}^n \to\mathbb{R}\cup \{+\infty\}$, the method computes $\min f(x)+g(x)$ via the iteration (started from any $x_0, \delta_0$)
\begin{subequations} \label{admm}
    \begin{align}
        y_{k+1} &= \arg\min_y f(y) + \frac12 \norm{y - (x_k + \delta_k)}^2, \label{admmprox1} \\
        x_{k+1} &= \arg\min_x g(x) + \frac12 \norm{x - (y_{k+1}-\delta_k)}^2, \label{admmprox2} \\
        \delta_{k+1} &= \delta_k - y_{k+1} + x_{k+1}. \label{admmdual}
    \end{align}
\end{subequations}
Under very mild assumptions, the two sequences $x_k,y_k$ converge to $x_\star  \in \arg\min f(x)+g(x)$, and $\delta_k$ converges to a dual optimal solution of the Lagrangian formulation $f(y)+g(x) + \scalar{\delta, x-y}$. One can specialize ADMM to the case in which $f$ (or $g$, or both) is the \emph{convex indicator function} of a convex set $\mathcal {E}$, i.e.,
\[
\iota_{\mathcal{E}}(x) = \begin{cases}
    0 & x\in \mathcal{E},\\
    +\infty & x\not\in\mathcal{E}.
\end{cases}
\]
With this function, $\arg\min \iota_{\mathcal{E}}(x) + \frac12 \norm{x-a}^2 = \Pi_{\mathcal{E}}(a)$, so~\eqref{admmprox1} and~\eqref{admmprox2} (which are, in general, so-called \emph{proximal problems}) become projections, $y_{k+1} = \Pi_{\mathcal{E}_1}(x_k+\delta_k)$ and $x_{k+1} = \Pi_{\mathcal{E}_2}(y_{k+1}-\delta_k)$, resembling closely Dykstra's algorithm~\eqref{dykstra}. However, the similarities stop here: ADMM uses only one correction variable $\delta$, with alternating signs, rather than two; and even in the version with $f=\iota_{\mathcal{E}_1}$ and $g=\iota_{\mathcal{E}_2}$ it converges to some point inside $\mathcal{E}_1\cap\mathcal{E}_2$, not necessarily the projection $\Pi_{\mathcal{E}_1\cap\mathcal{E}_2}(a)$. See~\citet{tibshirani2017dykstra} for a more detailed comparison of ADMM and Dykstra's algorithm.

\citet{lim22} use this version of ADMM in the matrix distance problem $d(A, \mathcal{E}_{B,C,\mathcal{Y}})$, with target set as in~\eqref{limstructure}. Notice that this target set is not an intersection, since $A$ and the elements of $\mathcal{Y}$ have different sizes. However, we can solve the problem with the ADMM algorithm~\eqref{admm}, setting $f(Y) = \iota_{\mathcal{Y}}(Y)$ and $g(Y) = \norm{A-BYC}_F^2$:
    \begin{align}
        Y_{k+1} &= \Pi_{\mathcal{Y}}(X_k + \Delta_k) \nonumber\\
        X_{k+1} &= \arg\min_X \norm{A-BXC}_F^2 + \frac12 \norm{X - (Y_{k+1}-\Delta_k)}_F^2, \label{admmspecialprox2}  \\
        \Delta_{k+1} &= \Delta_k + X_{k+1} - Y_{k+1}. \nonumber 
    \end{align}
This is essentially~\citet[Equations (4.2)--(4.4)]{lim22}, with the minor difference that the proximal problem~\eqref{admmspecialprox2} is replaced by a constrained problem; but Section~2.3 of the same article explains how the two are equivalent. \citet[Algorithm~2]{lim26} expand the technique to deal with more general distance functions.

\section{Nearest structured singular matrix}\label{sec:structuredsingular}

\subsection{General properties of structured nearness problems}
A challenging class of problems comes from combining a singularity or spectral constraint with the request that $X = A-\Delta$ lies in an affine subspace $B+\mathcal{V}$, where $B\in\F^{m\times n}$, and $\mathcal{V} \subseteq \F^{m\times n}$ is a linear subspace. (Or, equivalently, we can request that $\Delta$ lies in the affine subspace $A-B+\mathcal{V}$.) For instance, looking for the nearest sparse weakly Hurwitz stable matrix with a prescribed sparsity pattern, or the nearest singular Toeplitz matrix. 

Structured nearness problems may have additional pathologies. For instance, the feasible region may be empty (e.g., finding the nearest singular matrix to any given $A$ in the affine space $B+\mathcal{V} = \{\begin{bsmallmatrix}a & 1\\1 & 0\end{bsmallmatrix},\, a\in\F\}$), or the sought eigenvalues may be multiple (e.g., finding the nearest weakly Hurwitz stable matrix in $B+\mathcal{V} = \{\begin{bsmallmatrix}a & 0\\0 & a\end{bsmallmatrix},\, a\in\F\}$), and therefore formulae for the derivatives that assume simple eigenvalues are not directly applicable. Moreover, \citet[Theorem~2]{ZZ23} show that determining whether the feasible set is nonempty is NP-hard even for a fairly simple case: The intersection of an affine subspace $B+\mathcal{V} \subseteq \R^{n\times n}$ with the set $\mathcal{E}_0$ of singular matrices.

We open this section by giving some general properties that are true for general regions $\mathcal{E}$, and then we study in particular the case of finding the nearest structured singular matrix. This problem is both useful in itself and because it is a building block for more general problems of the form~\eqref{eq:structured nearness problem}. Throughout this section, distances are measured
in the Frobenius norm.

First of all, one can easily reduce to the case where we have a purely linear condition $\hat{\Delta}\in \mathcal{V}$:
\begin{lemma}\label{lem:A=B}
    Let $A,B\in\F^{m\times n}$; let $\mathcal{V} \subseteq \F^{m\times n}$ be a vector subspace, and $\mathcal{E} \subseteq \F^{m\times n}$ be a nonempty closed subset. Let $\hat{A} = \Pi_{B+\mathcal{V}}(A)$ be the orthogonal projection of $A$ on the affine subspace $B+\mathcal{V}$. 
    
    The matrix $\Delta_0$ is a minimizer of
    \begin{equation} \label{eq:affine minimum Delta}
        \min \norm{\Delta}_F^2 \subjectto A-\Delta \in B+\mathcal{V}, \,A-\Delta \in \mathcal{E}
    \end{equation}
    if and only if $\hat{\Delta}_0 = \Delta_0 + \hat{A}-A $ is a minimizer of
    \begin{equation} \label{eq:linear minimum Deltahat}
        \min \norm{\hat{\Delta}}_F^2 \subjectto \hat{\Delta} \in \mathcal{V}, \hat{A}-\hat{\Delta} \in \mathcal{E}.
    \end{equation}
\end{lemma}
\begin{proof}
    For a given $\Delta$, we define $\hat{\Delta} = \Delta + \hat{A} - A$, and make a change of variable to express the minimum problem~\eqref{eq:affine minimum Delta} in terms of $\hat{\Delta}$. Since $\hat{A} \in B+\mathcal{V}$, we have that $A-\Delta = \hat{A} - \hat{\Delta} \in B+\mathcal{V}$ if and only if $\hat{\Delta} \in \mathcal{V}$. Moreover, $\hat{A}-A\in \mathcal{V}^\perp$ by definition of orthogonal projection, so the decomposition $\Delta = \hat{\Delta} + (A-\hat{A})$ is orthogonal and
    \[
    \norm{\Delta}_F^2 = \norm{\hat{\Delta}}_F^2 + \norm{\hat{A}-A}_F^2.
    \]
    Hence the minimum problem~\eqref{eq:affine minimum Delta} becomes
    \[
        \min \norm{\hat{\Delta}}_F^2 + \norm{\hat{A}-A}_F^2 \subjectto \hat{\Delta} \in \mathcal{V}, \, \hat{A} - \hat{\Delta}\in\mathcal{E},
    \]
    which is equivalent to~\eqref{eq:linear minimum Deltahat} since the two objective functions differ only by the additive constant $\norm{\hat{A}-A}_F^2$.
\end{proof}
Therefore, in the following we shall restrict to the problem
\begin{equation} \label{eq:structured nearness problem}
    d^2(A, \mathcal{E}\cap (A+\mathcal{V})) = \min \norm{\Delta}_F^2 \subjectto \Delta \in \mathcal{V}, \, A-\Delta \in \mathcal{E}.
\end{equation}
The properties of tangent cones allow us to express the optimal $\Delta$ as a projection.

\begin{theorem}\label{thm:V82}
   Let $\mathcal{E} \subseteq \F^{m \times n}$ be a nonempty closed set, and let $\mathcal{V} \subseteq \F^{m \times n}$ be a real vector subspace. Let $X_\star = A- \Delta_\star$ be a Frobenius-norm projection of $A \in \F^{m \times n}$ (i.e., a minimizer of $\norm{A-X}_F^2$) onto $\mathcal{E}_S = \mathcal{E} \cap (A+\mathcal{V})$. If the sum rule for normal cones~\eqref{sum rule with inclusion} holds, i.e., $N_{X_\star}^\ell(\mathcal{E}_S) \subseteq N_{X_\star}^\ell(\mathcal{E}) + \mathcal{V}^\perp$, then there exists $\Delta_{1} \in N^\ell_{X_\star} \mathcal{E}$ such that $\Delta_\star = \Pi_{\mathcal{V}}(\Delta_{1}).$
\end{theorem}
\begin{proof}
    Theorem~\ref{thm:reg min in normal} ensures that $\Delta_\star = A-X_\star \in N_{X_\star}^\ell(\mathcal{E}_S)$. By the sum rule, we can write
    \[
    \Delta_\star = \Delta_1 + V_\perp, \quad \Delta_{1} \in N^\ell_{X_\star} \mathcal{E}, \, V_\perp \in \mathcal{V}^\perp,
    \]
    and since $\Delta_\star \in \mathcal{V}$ we get the thesis by applying the linear operator $\Pi_{\mathcal{V}}$ to both sides.
\end{proof}
In particular, applying Theorem~\ref{thm:V82} to the structured distance from singularity problem (and recalling our notation $\nu=\min\{m,n\}$),
\begin{equation} \label{structured distance from singularity}
    d^2(A,\mathcal{E}_0 \cap (A+\mathcal{V})) = \min \norm{\Delta}_F^2 \subjectto \Delta \in\mathcal{V},\, \rank(A-\Delta) < \nu,
\end{equation}
we obtain a direct generalization of the rank-1 structure that follows from the Schmidt-Eckart-Young-Mirsky Theorem.
\begin{corollary} \label{cor:V82 for singular}
    Let $X_\star$ be a minimizer of~\eqref{structured distance from singularity}, and $\Delta_\star = A-X_\star$. If $\Pi_{\mathcal{V}}(uv^*) \neq 0$ for every pair of nonzero vectors $u\in \ker(X_\star^*)$, $v\in\ker(X_\star)$, then there exist vectors $u_\star \in \ker(X_\star^*), \, v_\star \in \ker(X_\star)$ such that $\Delta_\star = \Pi_{\mathcal{V}}(u_\star v_\star^*)$.
\end{corollary}
\begin{proof}
    Note that $\mathcal{E}_0 = \mathcal{R}_{\leq \nu-1}$, in the language of Theorem~\ref{thm: normals of rank varieties}. Hence, by the same theorem,
    \[
    N^\ell_{X_\star} \mathcal{E}_0 = \{uv^* \colon u \in \ker(X_\star^*), \, v \in \ker(X_\star)\}.
    \]
    The assumption $\Pi_{\mathcal{V}}(uv^*) \neq 0$ ensures that the constraint qualification condition in Theorem~\ref{thm:sum rule} holds, and we conclude by applying Theorem~\ref{thm:V82}.
\end{proof}
\citet[Theorem~V.8.3]{GuglielmiLubichBook} use gradients to prove the same result in the special case where $m=n$ and $\rank(X_\star)=n-1$, hence the vectors in $\ker(X_\star^*)$ and $\ker(X_\star)$ are unique up to multiples. The power of our approach with normal cones is that it requires no assumptions on the rank of $X_\star$ nor the singular values of $A$.

\subsection{Structured distance from singularity} \label{subsec:structuredsing}

We now focus specifically on the set of rank-deficient matrices $\mathcal{E} = \mathcal{E}_0$ that we have encountered in the previous example. We assume (up to transposition) that $m\geq n$, so that $X=A-\Delta$ is singular if and only if there is $v\neq 0$ such that $Xv=0$; and we can take $v$ to be normalized with $\norm{v}_2=1$. Moreover, let $P_1,P_2,\dots,P_\ell$ be a basis for $\mathcal{V}$; we shall assume that this basis is orthonormal, i.e., $\mathbf{P}$ as in~\eqref{defP} has orthonormal columns.

Now the main idea is reparametrizing the problem, by using as optimization variables $v$ and the vector of basis coefficients $p \in \F^\ell$ such that
\begin{equation*} 
    \Delta = P_1 p_1 + P_2 p_2 + \dots + P_\ell p_\ell.
\end{equation*}
Then, we have
\begin{subequations} \label{eq:Lagrangian formulation}
\begin{align}
    d^2(A, \mathcal{E}_0 \cap (A+\mathcal{V})) &= \min_{\Delta, v} \norm{\Delta}_F^2 \subjectto (A-\Delta)v = 0, \, \norm{v}_2^2 = 1, \, \Delta \in \mathcal{V} \label{eq:Lagrangian formulation with A+Delta}\\
    &= \min_{p,v} \norm{p}_2^2 \subjectto M_v p = r_v, \, \norm{v}_2^2 = 1, \label{eq:Lagrangian formulation with Mp}
\end{align}    
\end{subequations}
where we used the fact that $\mathbf{P}$ has orthonormal columns to rewrite the objective function as
\[
    \norm{\Delta}_F^2 = \norm{\mvec(\Delta)}_2^2 = \norm{\mathbf{P}p}_2^2 = \norm{p}_2^2,
\]
and introduced
\begin{align} 
    M_v &= (v^T \otimes I_m) \mathbf{P} = \begin{bmatrix}
        P_1 v & P_2 v & \dots P_\ell v
    \end{bmatrix} \in\F^{m\times \ell},  \label{defM} \\ 
    r_v &= Av = (v^T \otimes I_m)\mvec(A) \nonumber
\end{align}
to rewrite the first constraint as
\[
    0 = (A-\Delta)v = (v^T \otimes I_m)\mvec(A-\Delta) = r_v - M_v p.
\]
We can now rewrite the matrix distance problem~\eqref{eq:Lagrangian formulation with Mp} as a nested minimization problem on the sphere $\mathcal{S} = \{v \in \mathbb{F}^n \colon \norm{v}_2 = 1\}$
\begin{gather}
    d^2(A, \mathcal{E}_0 \cap (A+\mathcal{V})) = \min_{v \in \mathcal{S}} f(v),  \nonumber \\ 
    \begin{aligned}
    f(v) &= \min_{\Delta} \norm{\Delta}_F^2 \subjectto \Delta \in \mathcal{V}, \, (A-\Delta)v = 0\\
    &= \min_p \norm{p}_2^2 \subjectto M_v p = r_v.\label{fv}
    \end{aligned}
\end{gather}
The inner function $f$ is the optimal value of a linearly-constrained least-squares problem, and we can express it in closed form.
\begin{lemma} \label{lem:weakrk1}
    Let $A \in \mathbb{F}^{m\times n}$, with $m\geq n$, and a linear subspace $\mathcal{V}\subseteq\mathbb{F}^{m\times n}$ over $\mathbb{F}$ be given. For each given $v\in\mathbb{F}^n, v\neq 0$, such that $r_v \in \range(M_v)$, 
    the problem~\eqref{fv} has the unique solution $\Delta(v) = \Pi_{\mathcal{V}}(uv^*)$, where $u = (M_v M_v^*)^\dagger  r_v$.
    (Here $X^\dagger $ denotes the Moore--Penrose pseudoinverse of the matrix $X$.)

    In particular, let $X_\star$, $\Delta_\star = A - X_\star$ be a solution of the structured distance from singularity problem~\eqref{structured distance from singularity} and $v_\star \in \ker X_\star$, $v_\star\neq0$. Then, $\Delta_\star = \Delta(v_\star) = \Pi_{\mathcal{V}}(u_\star v_\star^*)$, where $u_\star = (M_{v_\star} M_{v_\star}^*)^\dagger  r_{v_\star} \in\F^m$.
\end{lemma}
\begin{proof}
    For a fixed $v$, the minimization problem~\eqref{fv} is a linearly-constrained minimum-norm problem, and its unique solution can be computed using the pseudoinverse:
    \[
    p = M_v^\dagger r_v = M_v^* (M_v M_v^*)^\dagger r_v = M_v^*u.
    \]
    (See, e.g., \citet[Theorem~1.2.10]{Bjoerck1996}.) To conclude the proof, we show that for every vector $u$ the relation $p = M_v^*u$ is equivalent to $\Delta(v) = \Pi_{\mathcal{V}}(uv^*)$, by plugging in the formulas that define $M_v$ and $p$:
    \begin{equation} \label{projzv}
    \mvec(\Delta(v)) = \mathbf{P}p = \mathbf{P} ((v^T\otimes I_m)\mathbf{P})^* u = \mathbf{P}\mathbf{P}^* (\overline{v}\otimes I_m)u = \mathbf{P}\mathbf{P}^* \mvec(uv^*).
    \end{equation}
    Since we took $\mathbf{P}$ with orthonormal columns, $\mathbf{P}\mathbf{P}^*$ is the (vectorized) orthogonal projection on $\mathcal{V}$.

    The second part of the theorem follows by using this result with $v=v_\star$; since there is $\Delta_\star$ such that $(A-\Delta_\star)v_\star = 0$, the condition $r_{v_\star}\in \range(M_{v_\star})$ is satisfied.
\end{proof}
It is interesting to compare Lemma~\ref{lem:weakrk1} with the results obtained in Theorem~\ref{thm:V82} and Corollary~\ref{cor:V82 for singular}. To apply Theorem~\ref{thm:V82}, we need a constraint qualification condition to ensure that the sum rule for normal cones holds. On the other hand, Lemma~\ref{lem:weakrk1} holds without additional technical assumptions, but does not guarantee that $u_\star \in \ker(X_\star^*)$. Example~\ref{ex:glarewrong} in the next section shows that there are in fact cases where one holds but not the other.

Using the machinery introduced in this section, we can also give an alternative direct proof of Corollary~\ref{cor:V82 for singular} that does not rely on tangent cones. This second proof is instructive, because it also reveals that $u_\star$ can be interpreted as a Lagrange multiplier. We present it only for the case $\F=\R$, to avoid some technical complications.
\begin{proof}(\emph{Alternative proof of Corollary~\ref{cor:V82 for singular}}.)
    Note that in~\eqref{eq:Lagrangian formulation} we can replace the constraint $\norm{v}_2^2=1$ with $v\neq 0$, since the objective function is homogeneous in $v$. Minimizers then come in rays: if $(p_\star, v_\star)$ is a minimizer, then so is $(p_\star, \alpha v_\star)$ for each $\alpha \neq 0$.

    Introduce a Lagrange multiplier $y\in\R^m$ to write the associated Lagrangian function
    \begin{equation} \label{lagr}
        \begin{aligned}
            \mathcal{L}(p,v;y) &= \frac{1}{2}\norm{p}_2^2 + \left\langle y, r_v - M_v p \right\rangle\\
            &= \frac{1}{2}\norm{\Delta}_F^2 + \scalar{y, (A-\Delta)v}.
        \end{aligned}    
    \end{equation}
    (Here and in the following, consider $\Delta$ not as a separate variable, but merely as a shorthand for $\sum_{j=1}^\ell P_j p_j$.)
    
    The Jacobian of the constraints is
    \begin{equation} \label{Jacobian of constraints}
        J_{p,v} \begin{bmatrix}
        (A-\Delta)v
    \end{bmatrix} = \begin{bmatrix}
        -M_v & A-\Delta
    \end{bmatrix}.    
    \end{equation}
    Let now $X_\star,\Delta_\star$ define a minimizer of~\eqref{structured distance from singularity}. Let $v_\star \in \ker X_\star$, $v_\star\neq 0$, and $p_\star$ be the coordinates of $\Delta_\star$.
    So, $(p_\star, v_\star)$ is a minimizer of~\eqref{eq:Lagrangian formulation with Mp}.
    
    We would like to prove that this Jacobian has full row rank in $(p_\star,v_\star)$, so that LICQ (linear independence constraint qualification, \cite[Definition~12.4]{Nocedal2006}) holds. Let us take $u\in \F^m$ such that
    \[
    u^*
    \begin{bmatrix}
        -M_{v_\star} & A-\Delta_\star
    \end{bmatrix} = 0,
    \]
    i.e., $u \in \ker(X_\star^*)$ and $u^* M_{v_\star} = 0$. By the definition of $M_{v_\star}$, this is equivalent to $0 = u^* P_j v_\star = \scalar{P_j, uv_\star^*}$ for each $j=1,\dots,\ell$. Since the $P_j$ are a basis of $\mathcal{V}$, this property is equivalent to $uv_\star^* \in \mathcal{V}^\perp$, or equivalently to $\Pi_{\mathcal{V}}(uv_\star^*)=0$. Hence we must have $u=0$ not to contradict the assumptions of Corollary~\ref{cor:V82 for singular}. This completes the proof that the Jacobian of the constraints has full row rank in the minimizer $(p_\star,v_\star)$, and hence LICQ holds.
    
    Since constraint qualifications hold, in an optimum $(p_\star,v_\star)$ of~\eqref{eq:Lagrangian formulation with Mp} there must be $y_\star\in\R^m$ which satisfy the KKT conditions
    \begin{subequations}
    \begin{align}
        0 &= \nabla_p \mathcal{L} = p_\star - M_{v_\star}^*{y_\star}, \label{firstlag}\\
        0 &= \nabla_v \mathcal{L} = (A-\Delta_\star)^*{y_\star}. \label{secondlag}
    \end{align}
    \end{subequations}
    The second equation~\eqref{secondlag} tells us that $y_\star \in \ker X_\star^*$, and the first~\eqref{firstlag} that $p_\star = M_{v_\star}^* y_\star$, which is equivalent to $\Delta_\star = \Pi_{\mathcal{V}} (y_\star v_\star^*)$, by following the argument in~\eqref{projzv}. Hence $\Delta_\star = \Pi_{\mathcal{V}}(u_\star v_\star^*)$ holds with $u_\star$ equal to the Lagrange multiplier $y_\star$.
\end{proof}

\begin{remark}
    Let $m=n$ and consider the case $\rank(X_\star)=n-1$, in which the vectors $u\in \ker X_\star^*$ and $v\in \ker X_\star$ are unique up to multiples. The three conditions 
    \begin{enumerate}
        \item $\Pi_{\mathcal{V}}(uv^*) \neq 0$ \cite[Theorem V.8.3]{GuglielmiLubichBook}
        \item $\rank\begin{bmatrix}
            M_v & A-\Delta_\star
        \end{bmatrix} = n$ (LICQ for~\eqref{Jacobian of constraints})
        \item $\mathcal{V} + (T_{X_\star} \mathcal{E}_0) = \F^{m \times n}$ (transversal intersection of tangent spaces)
    \end{enumerate}
    are then all equivalent to the constraint qualification in Theorem~\ref{thm:sum rule}.
\end{remark}

\subsection{Some examples}
\label{subsec:struct sing excamples}
In this section, we give a few examples of interesting cases involving the structured distance to singularity, and in particular the various conditions that let us prove the property that $\Delta$ is the projection on $\mathcal{V}$ of a rank-1 matrix.

Our first example is a case where the sum rule for normal cones does not hold, and hence Lemma~\ref{lem:weakrk1} applies but not Corollary~\ref{cor:V82 for singular}, and in particular $u_\star\not\in \ker(X_\star^*)$.
\begin{example}\label{ex:glarewrong}
    Let $A=\begin{bsmallmatrix}
        1 & 1\\
        0 & 1
    \end{bsmallmatrix} \in \mathcal{V}=\left\{ \begin{bsmallmatrix}
        p_1 & p_2\\
        0 & p_1
    \end{bsmallmatrix} \colon \ p_1,p_2 \in \F \right\} \subseteq \F^{2 \times 2}$. In the notation of Theorem~\ref{thm:V82}, we have $\displaystyle \mathcal{E}_S = \left\{\begin{bsmallmatrix}
        0 & p_2\\
        0 & 0
    \end{bsmallmatrix} \colon \ p_2 \in \F \right\}$; in particular, $\mathcal{E}_S$ is a vector subspace of $\mathcal{V}$. Evidently, the (unique) matrix in $\mathcal{E}_S$ nearest to $A$ is $X_\star=\begin{bsmallmatrix}
        0 & 1\\
        0 &0
    \end{bsmallmatrix}$, with $\ker(X_\star^*)= \spann(\begin{bsmallmatrix}
        0\\
        1
    \end{bsmallmatrix})$ and $\ker (X_\star) = \spann(\begin{bsmallmatrix}
        1\\0
    \end{bsmallmatrix})$. So we have
    \[
    N_{X_\star} \mathcal{E}_0 = \spann\left(\begin{bmatrix}
        0 & 0\\
        1 & 0
    \end{bmatrix}\right) \subseteq N_{X_\star}(A+\mathcal{V}) = \mathcal{V}^\perp = \left\{
    \begin{bmatrix}
        p_1 & 0\\
        p_2 & -p_1
    \end{bmatrix}\colon p_1,p_2\in\F\right\};
    \]
    in particular, the constraint qualification in Theorem~\ref{thm:sum rule} $N_{X_\star}^\ell \mathcal{E}_0 \cap (-\mathcal{V}^\perp) = \{0\}$ does not hold, and neither does the sum rule itself, since
    \[
    N_{X_\star}\mathcal{E}_S = \mathcal{E}_S^\perp = \left\{
    \begin{bmatrix}
        p_1 & 0\\
        p_2 & p_3
    \end{bmatrix}\colon p_1,p_2,p_3\in\F\right\} \not \subseteq N_{X_\star}(\mathcal{E}_0) + \mathcal{V}^\perp = \mathcal{V}^\perp.
    \]
    So this example does not satisfy the assumptions of Theorem~\ref{thm:V82}; and neither does it satisfy its thesis, since for every $\Delta_1 \in N_{X_\star}\mathcal{E}_0$ we have $\Pi_{\mathcal{V}}(\Delta_1) = 0 \neq \Delta_\star$.
    
    On the other hand, we can apply Lemma~\ref{lem:weakrk1}, and indeed taking $v_\star = \begin{bsmallmatrix}
            1\\0
        \end{bsmallmatrix}$ we have
    $u_\star= \begin{bsmallmatrix}
        2\\0
    \end{bsmallmatrix} \not \in \ker(X_\star^*)$, and 
    \[
    \Pi_\mathcal{V}(u_\star v_\star^*)=\Pi_\mathcal{V} \left( \begin{bmatrix}
            2 & 0 \\
            0 & 0
        \end{bmatrix} \right)  = \begin{bmatrix}
            1 & 0\\
            0 & 1
        \end{bmatrix} = \Delta_\star,
    \]
    i.e., $\Delta_\star$ is the projection onto $\mathcal{V}$ of a rank-1 matrix, but not one built from vectors in the left and right kernel of $X_\star$.
 \end{example}
 \begin{remark}
    The fact that $A$ and $X_\star$ have a defective eigenvalue in Example~\ref{ex:glarewrong} is a red herring, in the sense that  
    the failure of Theorem~\ref{thm:V82} is not due to the spectrum of $X_\star$. For instance, one could swap the rows in all the matrices appearing in our example, obtaining an analogous example with simple eigenvalues (which is essentially \citet[Example 8.5]{GuglielmiLubichBook}). More generally, the structured distance to singularity problem~\eqref{structured distance from singularity} is invariant under left and right multiplication (of both $A$ and $\mathcal{V}$) by two unitary matrices, so eigenvalues are not a meaningful concept here.
 \end{remark}

Note that the conclusion of Corollary~\ref{cor:V82 for singular} is verified whenever the sum rule $N_{X_\star}^{\ell}\mathcal{E}_S \subseteq N_{X_\star}^{\ell}\mathcal{E}_0 + N_{X_\star}^{\ell}(A+\mathcal{V})$ holds; however, for simplicity we formulated the corollary using as hypothesis the constraint qualification condition $\Pi_{\mathcal{V}}(uv^*)\neq 0$, which is a sufficient (but not necessary!) condition for the sum rule to hold, in view of Theorem~\ref{thm:sum rule}. The following example displays a case where the constraint qualification does not hold, but the sum rule does.
\begin{example}\label{ex:gl2}
    Let $A=\begin{bsmallmatrix}
        1 & 0\\
        0 & 0
    \end{bsmallmatrix}$ and $\mathcal{V}=\spann \{ A \}$. Since $A$ is already singular, it is clear that the optimal perturbation is $\Delta_\star=0$. Every (either left or right) singular vector of $X_\star=A$ associated with the smallest singular value $0$ must be a nonzero multiple of $\begin{bsmallmatrix}
        0\\
        1
    \end{bsmallmatrix}$, and hence 
    \[ N^\ell_{X_\star} \mathcal{E}_0 = \spann \left\{ \begin{bmatrix}
    0&0\\
    0& 1 \end{bmatrix}\right\} \subsetneq \mathcal{V}^\perp = \spann\left\{\begin{bmatrix}
        0 & 1\\
        0 & 0
    \end{bmatrix},
    \begin{bmatrix}
        0 & 0\\
        1 & 0
    \end{bmatrix},
    \begin{bmatrix}
        0 & 0\\
        0 & 1
    \end{bmatrix}\right\}.\]
    The constraint qualification does not hold, since $\begin{bsmallmatrix}
        0 & 0\\
        0 & 1
    \end{bsmallmatrix}\in N^\ell_{X_\star}\mathcal{E}_0 \cap (-\mathcal{V}^\perp)$; however, $\mathcal{E}_S = \mathcal{E}_0 \cap (A+\mathcal{V}) = \mathcal{V}$, and hence the sum rule holds:
    \[
    \mathcal{V}^\perp = N_{X_\star}^\ell \mathcal{E}_S = N^\ell_{X_\star}\mathcal{E}_0 + \mathcal{V}^\perp.
    \]
    We can also verify directly that the thesis of Theorem~\ref{thm:V82} (and Corollary~\ref{cor:V82 for singular}) holds in this case, since $0 = \Delta = \Pi_{\mathcal{V}} \left(\begin{bsmallmatrix}
    0\\
    z
\end{bsmallmatrix} \begin{bsmallmatrix}
    0 & 1
\end{bsmallmatrix}\right)$ for each $z\in\F$.
 \end{example}

The following example displays a case where $X_\star$ has a rank drop of more than 1.

\begin{example}\label{ex:gl3}
    Let
    \[
    A=\begin{bmatrix}
        1 & 0 & 0\\
        0 & 1 & 0\\
        0 & 0 & 1
    \end{bmatrix}, \quad \mathcal{V} = \left\{
    \begin{bmatrix}
        0 & 0 & 0\\
        0 & x_1 & x_2\\
        0 & -x_2 & x_1
    \end{bmatrix} \colon x_1, x_2 \in\F
    \right\}.
    \]
    The determinant of the general matrix in $A+\mathcal{V}$ is $(1+x_1)^2 + x_2^2$. So, when $\F=\R$, the only singular matrix in $A+\mathcal{V}$ is $X_\star = \diagm(1,0,0)$. Note that $\rank(X_\star)<n-1$, and hence $\mathcal{E}_0$ is not a manifold in a neighborhood of $X_\star$. The normal cones are
    \[
    N^\ell_{X_\star} \mathcal{E}_S = N_{X_\star} \{X_\star\} = \mathbb{R}^{3\times 3}, \quad N^\ell_{X_\star} \mathcal{E}_0 = \{ 0 \oplus st^T \colon s,t \in \R^2\},
    \]
    and
    \[
     N^\ell_{X_\star}(A+\mathcal{V}) = \mathcal{V}^\perp = \left\{\begin{bmatrix}
        x_1 & x_2 & x_3\\
        x_4 & x_5 & x_6\\
        x_7 & x_6 & -x_5
    \end{bmatrix} \colon x_1,\dots, x_7\in \R\right\}.
    \]
    One can verify directly Theorem~\ref{thm:sum rule} in this example, because both the constraint qualification condition $N^\ell_{X_\star} \mathcal{E}_0 \cap (-\mathcal{V}^\perp) = \{ 0 \}$ and the sum rule $N^\ell_{X_\star} \mathcal{E}_S = N^\ell_{X_\star} \mathcal{E}_0 + \mathcal{V}^\perp$ hold. If we pick, for instance, $v_\star = \begin{bsmallmatrix}
        0\\1\\0
    \end{bsmallmatrix} \in \ker(X_\star)$, $u_\star = \begin{bsmallmatrix}
        0\\2\\0
    \end{bsmallmatrix}\in \ker(X_{\star}^*)$, then we have 
    \[ \Delta_\star = \begin{bmatrix}
        0&0&0\\
        0&1&0\\
        0&0&1
    \end{bmatrix}  = \Pi_{\mathcal{V}}\left( \begin{bmatrix}
        0 & 0 & 0\\
        0 & 2 & 0\\
        0 & 0 & 0
    \end{bmatrix} \right)  =\Pi_{\mathcal{V}}\left(u_\star v_\star^* \right).    \]
    So the conclusion of Corollary~\ref{cor:V82 for singular} holds.
 
    \citet[Theorem V.8.3]{GuglielmiLubichBook} is simply inapplicable to this case, since one of its assumptions is $\rank(X_\star)=n-1$. 
 \end{example}

\subsection{Structured low-rank approximation}

The reformulation~\eqref{fv} can be turned into a practical algorithm, by eliminating the variable $p$: Indeed, we have seen that once $v$ is fixed, \eqref{fv} can be readily solved with a pseudoinverse. We present the resulting algorithm, starting from the simpler case when $M_v$ has full row rank.

\begin{theorem} \label{thm:nearest structured singular, regular case}
    If $M_v$ has full row rank, then 
    \begin{enumerate}
        \item $f(v) = \norm{M_v^\dagger r_v}^2 = r_v^* (M_v M_v^*)^{-1} r_v$;
        \item for any fixed $v$, the unique minimizer $\Delta(v)$ of the problem in~\eqref{fv} is given by $\mvec(\Delta(v)) = \mathbf{P}M_v^\dagger r_v = \mathbf{P} M_v^*u$, with $u=(M_v M_v^*)^{-1} r_v$;
        \item the (Euclidean) gradient of $f$ is $\nabla_v f = 2(A-\Delta(v))^*u$.
    \end{enumerate}
    
\end{theorem}
\begin{proof}
    For a fixed $v$, \eqref{fv} is an equality-constrained linear least-squares problem. The problem must be feasible because of the assumption that $M_v$ has full row rank; hence (see, e.g.,~\citet[Theorem~1.2.10]{Bjoerck1996}) its unique minimizer is $p_v = M_v^\dagger r_v = M_v^*(M_v M_v^*)^{-1} r_v$, and the minimum is $\norm{p_v}_2^2 = p_v^*p_v = r_v^* (M_v M_v^*)^{-1} r_v$. Standard differentiation techniques give the gradient; see, e.g., \citet[Proposition~2]{Usevich2014} or \citet[Theorem~3.1]{Oracle}.
\end{proof}
\begin{remark}
    Under suitable regularity conditions, the formula for the gradient can also be obtained using the so-called \emph{envelope theorem}~\cite[Theorem~3.4.1]{fiacco_book}, which states that $\nabla_v \frac{1}{2}f(v) = \nabla_v \mathcal{L}(p, v; \lambda)$
\end{remark}

Hence we obtain the optimal $d^2(A,\mathcal{E})$ (or, at least, a local minimizer) and the associated minimizers $v_\star, p_\star, \Delta_\star, X_\star$ by solving numerically the minimum problem
\[
\min_{v\in\mathcal{S}} r_v^*(M_v M_v^*)^{-1}r_v.
\]
This method, known as \emph{variable projection}, is studied extensively by Markovsky and Usevich; see~\citet{Markovsky2013, Markovsky2014,Usevich2014a,Usevich2014}. The authors obtain expressions for the approximate Hessian of $f(v)$, and use first- and second-order optimization techniques on the sphere to compute the minimum. Notable precursors of this technique are~\citet{Golub1973}, where the idea of removing quadratic variables from a nonlinear optimization problem appears, and~\citet{Rosen1996}, where the same structured matrix parametrization is used in a similar context for a total least-squares problem. See also \citet{Markovsky2008} for a broader review of structured low-rank approximation and its applications.

The attentive reader probably already noticed the non-trivial assumption in Theorem~\ref{thm:nearest structured singular, regular case} that $M_v$ must have full row rank. This condition must hold at least in  the optimum $v_\star$, for the method to be usable, and is a stronger assumption than our earlier constraint qualifications. The following two examples show that we cannot assume it blindly, or handwave it away as ``true generically''.
\begin{example} \label{ex:infeasibility}
    This example shows that the minimum problem~\eqref{fv} may have an empty feasible region. Consider the structured distance from singularity problem~\eqref{structured distance from singularity} with
    \[
    A = \begin{bmatrix}
        0 & 4\\
        1 & 0
    \end{bmatrix}, \quad \mathcal{V} = \left\{\begin{bmatrix}
        p_1 & 0\\
        0 & p_1
    \end{bmatrix} : p_1 \in \F\right\}.
    \]
    Then, with $P_1 = \frac{1}{\sqrt{2}}I_2$, 
    \[
    M_v = \frac{1}{\sqrt{2}}\begin{bmatrix}
        v_1\\
        v_2
    \end{bmatrix}, \quad r_v = \begin{bmatrix}
        4v_2\\v_1
    \end{bmatrix}.
    \]
    The matrix $M_v M_v^*$ is always singular, so we cannot rely on the formula $M_v^\dagger  = M_v^*(M_v M_v^*)^{-1}$, and the constraint $M_v p = r_v$ in~\eqref{eq:Lagrangian formulation with Mp} is infeasible for almost all $v$. If we agree that $\min \varnothing = +\infty$, we have
    \[
    f(v) = \begin{cases}
        8 & v_1 = \pm 2v_2,\\
        +\infty & \text{otherwise}.
    \end{cases}
    \]
\end{example}
Hence, we might have to restrict $v$ to a suitable submanifold of $\F^n$ to ensure that $f$ takes finite values. We note that a case where the feasibility of~\eqref{fv} is always guaranteed is when $A \in \mathcal{V}$: In that case $\Delta = A$ is always a feasible solution irrespective of $v$.

The following example shows another fundamental issue.
\begin{example}[\protect{\citet[Example~2.7]{Oracle}}] \label{ex:oracle discontinuous f}
    Consider the structured distance from singularity problem~\eqref{structured distance from singularity} with
    \[
    A = \begin{bmatrix}
        1 & 1\\
        0 & 2
    \end{bmatrix}, \quad \mathcal{V} = \left\{\begin{bmatrix}
        p_1 & 0\\
        0 & p_2
    \end{bmatrix} : \begin{bmatrix}
        p_1\\
        p_2
    \end{bmatrix} \in \F^2\right\}.
    \]
    Then,
    \[
    M_v = \begin{bmatrix}
        v_1 & 0\\
        0 & v_2
    \end{bmatrix}, \quad r_v = \begin{bmatrix}
        v_1 + v_2\\
        2v_2
    \end{bmatrix}.
    \]
    The matrix $M_v$ becomes singular if $v_1=0$ or $v_2=0$; for $v\in \mathcal{S}$ we have
    \[
    f(v) = \begin{cases}
        4 + \abs*{1 + \frac{v_2}{v_1}}^2 & v_1,v_2\neq 0,\\
        +\infty & v_1=0,\\
        1 & v_2=0.\\
    \end{cases}
    \]
    We plot this function over the real unit circle in Figure~\ref{fig:oracle discontinous f}.
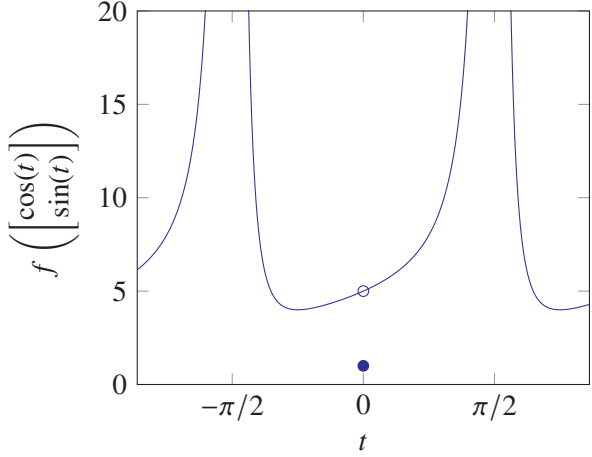
\begin{figure}
    \centering
\begin{tikzpicture}
\begin{axis}[ymin=0, ymax=20, width=0.6\textwidth,
xtick={-pi,-pi/2,0,pi/2,pi},
xticklabels={$-\pi$, $-\pi/2$, $0$, $\pi/2$,$\pi$}, enlarge x limits=false, xlabel=$t$, 
ylabel=$f\left(\begin{bmatrix}\cos(t)\\\sin(t)\end{bmatrix}\right)$
]
\addplot[color=blue, domain=-2.708:2.708, samples=400]{min(50, (1*cos(deg(x))+1*sin(deg(x)))^2/(cos(deg(x))^2) + 2^2)};
\addplot[only marks, mark=*, color=blue]
table{
x  y
0 1
};
\addplot[only marks, mark=o, color=blue]
table{
x  y
0 5
};
\end{axis}
\end{tikzpicture}
\caption{A plot of the function $f(v)$ in Example~\ref{ex:oracle discontinuous f}, for $\F=\R$, over the circle $\norm{v}=1$. Note the discontinuity in $t=0$.} \label{fig:oracle discontinous f}
\end{figure}
The global minimum $1$ is achieved in a discontinuity point, where the matrix $M_v$ becomes singular. Gradient-based optimization on the function $f$ cannot locate this discontinuous minimum, but only the local minimum $f\left(\frac{1}{\sqrt{2}}\begin{bsmallmatrix}
    1\\-1
\end{bsmallmatrix} \right)=4$.
\end{example}
We can get some insight on the behavior of the problem~\eqref{fv} for singular $M_v$ in terms of an SVD $M_v = U_v \Sigma_v V_v^*$. Set $q = V_v^*p$ and $w_v = U_v^* r_v$; then the constraint $M_v p = r_v$ becomes $\Sigma_v q = w_v$. This diagonal system of equations is solvable if and only if $(w_v)_i=0$ whenever $(\sigma_v)_i = 0$ (or $i>\dim\mathcal{V}$), and in this case the minimum-norm solution is
\[
    f(v) = \norm{q}^2 = \sum_{i : (\sigma_v)_i \neq 0} \abs*{\frac{(w_v)_i}{(\sigma_v)_i}}^2.
\]
In particular, consider a family of vectors $v(t)$ for which we have (for a certain index~$j$)
\[
(w_{v(0)})_j = (\sigma_{v(0)})_j=0, \quad \lim_{t\to 0} \frac{(w_{v(t)})_j}{(\sigma_{v(t)})_j} = \alpha \neq 0
\]
Then,
\[
f(v(0)) < \lim_{t\to 0} f(v(t)),
\]
because the right-hand side contains a term $\abs{\alpha}^2$ that the left-hand side lacks. This argument shows that $f(v)$ is, in general, only lower semicontinuous in rank-drop points.

We would like to have an algorithm that can solve also problems such as those in Examples~\ref{ex:infeasibility} and~\ref{ex:oracle discontinuous f}. To this purpose, we use a classic strategy from optimization, the \emph{penalty method}, following~\citet{Oracle}. Define the following relaxed version of~\eqref{eq:Lagrangian formulation with Mp}
\begin{equation} \label{penaltyform}
    \min_{p \in \F^\ell,v\in \mathcal{S}} \norm{p}_2^2 + \frac{1}{\varepsilon} G(p,v), \quad G(p,v) := \frac{\norm{(A-\Delta)v}_2^2}{\norm{v}_2^2} = \frac{\norm{r_v - M_v p}_2^2}{\norm{v}_2^2}.
\end{equation}
The penalty method consists in choosing a sequence $\varepsilon_1 \geq \varepsilon_2 \geq \varepsilon_3 \geq \dots$ converging to zero, and for each of these values $\varepsilon_k$ computing a (possibly non-unique) global minimizer of~\eqref{penaltyform}, which we call $(p_k, v_k)$. Theorem~\ref{thm:fepsvconvergence} shows that the distance from $(p_k,v_k)$ to the set of global minimizers of \eqref{eq:Lagrangian formulation with Mp} tends to zero. If there is a unique minimizer, this implies in particular that $p_k$ and $\Delta_k$ converge to it.
\begin{theorem} \label{thm:fepsvconvergence}
Assume the feasible region of the structured distance to singularity problem $(A+\mathcal{V}) \cap \mathcal{E}_0$ is nonempty.
Then, the iterates $(p_k,v_k)$ of the penalty method described above satisfy the following properties.
\begin{enumerate}
    \item $\norm{p_k}_2 \leq d(A, (A+\mathcal{V}) \cap \mathcal{E}_0)$;
    \item every limit point of $(p_k,v_k)$ is an optimum of~\eqref{eq:Lagrangian formulation with Mp};
    \item $\lim_{k\to\infty} d((p_k,v_k), S) = 0$, where $S$ is the set of minimizers of~\eqref{eq:Lagrangian formulation with Mp}. In particular, if $A$ does not belong to the medial axis of $(A+\mathcal{V}) \cap \mathcal{E}_0$, then $p_k$ converges to the unique minimizer $p$.
\end{enumerate}
\end{theorem}
\begin{proof}
    This is a standard result in the theory of penalty methods, but we give a brief proof anyway. 
    \begin{enumerate}
    \item Since the feasible region is nonempty, \eqref{eq:Lagrangian formulation with Mp}~has an optimum, i.e., there exist $p_\star$ of minimum norm and $v_\star\in\mathcal{S}$ such that $r_{v_\star}-M_{v_\star} p_\star = 0$. By minimality, 
    \begin{equation} \label{eq:penalty method convergence}
        \norm{p_k}_2^2 + \frac{1}{\varepsilon_k}G(p_k,v_k) \leq \norm{p_\star}_2^2 + \frac{1}{\varepsilon_k}G(p_\star,v_\star) = \norm{p_\star}_2^2.
    \end{equation}
     In particular, this implies that $\norm{p_k}_2^2 \leq \norm{p_\star}_2^2 = d^2(A, \mathcal{E}_0 \cap (A+\mathcal{V}))$.
     \item Let $(p,v)$ be a limit point of the sequence $(p_k, v_k)$; at least one such limit point must exist because we have shown in the previous point that $(p_k,v_k)$ are uniformly bounded. It follows from~\eqref{eq:penalty method convergence} that $\frac{1}{\varepsilon_{k}}G(p_k, v_k)$ is bounded, hence we must have $G(p,v) = 0$ in the limit. Moreover, in the limit $\norm{p}_2^2 \leq \norm{p_\star}_2^2$, hence $\Delta$ must be an optimum.
    \item Assume that the thesis is false, i.e., there are a subsequence $(p_{k_j},v_{k_j})$ and a radius $\delta$ such that $d((p_{k_j},v_{k_j}), S) \geq \delta$ for each $j$. By compactness, $(p_{k_j},v_{k_j})$ has a limit point, which must be outside $S$; this is a contradiction.
     \end{enumerate}
\end{proof}
To solve~\eqref{penaltyform}, we once again eliminate the variable $p$, defining
\begin{subequations} \label{fepsv}
\begin{align}
    f_\varepsilon(v) &=  \min_{\Delta} \norm{\Delta}_F^2 + \frac{1}{\varepsilon} \frac{\norm{(A-\Delta)v}_2^2}{\norm{v}_2^2} \subjectto \Delta \in \mathcal{V} \label{fepsv_withA}\\
    &=\min_{p} \norm{p}_2^2 + \frac{1}{\varepsilon} \frac{\norm{r_v - M_v p}_2^2}{\norm{v}_2^2}. \label{fepsv_withM}
\end{align}
\end{subequations} 
An analogue of Theorem~\ref{thm:nearest structured singular, regular case} holds.

\begin{theorem} \label{thm:nearest structured singular, singular case}
    For each $\varepsilon>0$, we have
    \begin{enumerate}
        \item $f_\varepsilon(v) = r_v^{*} (M_v M_v^* + \varepsilon \norm{v}_2^2 I_m)^{-1} r_v$;
        \item for any fixed $v$, the unique minimizer $\Delta_\varepsilon(v)$ of the problem in the right-hand side of~\eqref{fepsv_withA} is given by $\mvec(\Delta_\varepsilon(v)) = \mathbf{P}M_v^* u = \mvec(\Pi_{\mathcal{V}}(uv^*))$, with $u=\left(M_v M_v^* + \varepsilon \norm{v}_2^2 I_m\right)^{-1} r_v$;
        \item the (Euclidean) gradient of $f_\varepsilon$ is $\nabla_v f_\varepsilon = 2\left(A-\Delta_{\varepsilon}(v) - \varepsilon uv^*\right)^*u$.
    \end{enumerate}
\end{theorem}
\begin{proof}
    The proof is similar to that of Theorem~\ref{thm:nearest structured singular, regular case}; see~\citet[Theorems~2.9 and~3.1]{Oracle}, where a similar function is defined without the factor $\norm{v}_2^2$ in the denominator.
\end{proof}

Note that our definition of $f_\varepsilon$ differs slightly from the version in~\citet{Oracle}, which is missing the factor $\norm{v}_2^2$ in the denominator. While the two definitions coincide when $v\in\mathcal{S}$, the function $f_\varepsilon(v)$ in our~\eqref{fepsv} is homogeneous in $v$; hence $\nabla_v f_{\varepsilon}$ is always orthogonal to $v$, and we can expect $\nabla_v f_{\varepsilon} = 0$ in a minimum. In contrast, \citet{Oracle} must note the difference between the Euclidean gradient and the Riemannian gradient, and their presentation requires more machinery from optimization on manifolds.

This regularization procedure gives a practical numerical algorithm to solve~\eqref{structured distance from singularity} even in the case where $M_v$ is not guaranteed to have full rank: starting from arbitrary $\varepsilon_0$ and $v_0$, for each $k=1,2\dots,$
\begin{description}
    \item[Outer iteration] Fix a new value of $\varepsilon_{k}$ so that the sequence $(\varepsilon_k)_k$ tends to $0$: for instance $\varepsilon_{k} = 0.5\, \varepsilon_{k-1}$.
    \item[Inner iteration] For a fixed $\varepsilon_k$, use an iterative method, with $v_{k-1}$ as initial value, to solve numerically the optimization problem $v_{k} \in \arg\min_{v\in\mathcal{S}} f_{\varepsilon_k}(v)$, taking advantage of the explicit formulas for $f_\varepsilon$ and its gradient in Theorem~\ref{thm:nearest structured singular, singular case}.
\end{description}
\citet{Oracle} describe this nested iteration in detail. The convergence result that we gave in Theorem~\ref{thm:fepsvconvergence} assumes that each inner problem is solved by returning a global minimizer. Under this assumption, for each $k$ we have
\[
0 = \nabla_{v} f_{\varepsilon_k}(v_k) = 2\left(A-\Delta_{\varepsilon_k}(v_k) - \varepsilon_k u_kv_k^*\right)^*u_k,
\]
where $u_k$ is the vector $u$ defined in point 2 of Theorem~\ref{thm:nearest structured singular, singular case}. If the sequence $(u_k)_k$ converges, passing to the limit along a subsequence $(u_{k_j})_j$ such that the corresponding $v_{k_j}$ converges (the existence of this subsequence is guaranteed because the unit sphere is compact), we re-obtain the projection-of-rank-1 properties $\Delta_\star = \Pi_{\mathcal{V}}(u_\star v_\star^*)$ and $0 = (A-\Delta_\star)^*u_\star$ of Corollary~\ref{cor:V82 for singular}. Note also that, in practice, the inner iteration may return a local, but not global, minimizer. 

\subsection{Variants and alternatives to SLRA}

Eliminating $p$ with the formulas in Theorems~\ref{thm:nearest structured singular, regular case} and~\ref{thm:nearest structured singular, singular case} is not the only way to solve the optimization problems~\eqref{eq:Lagrangian formulation} and~\eqref{penaltyform}. \citet{Oracle} suggest another classical idea from optimization: combining~\eqref{lagr} and~\eqref{penaltyform} into the \emph{augmented Lagrangian}
    \begin{align*}
        \mathcal{L}_\varepsilon(p,v;y) &= \norm{p}_2^2 + \scalar{2y, r_v - M_v p} + \frac{1}{\varepsilon}\norm{r_v - M_v p}_2^2\\
        &= \norm{\Delta}_F^2 + \scalar{2y, (A-\Delta)v} + \frac{1}{\varepsilon}\norm{(A-\Delta)v}_2^2.
    \end{align*}
(We have imposed $\norm{v}_2=1$ to simplify the subsequent formula, and multiplied by 2 to avoid some denominators.) Therefore,
\[\mathcal{L}_\varepsilon(p,v;y)   + \varepsilon\norm{y}_2^2  = \norm{p}_2^2   + \frac{1}{\varepsilon} \norm{(A-\Delta)v+\varepsilon y}_2^2.\]
Minimizing over $p$ having fixed $v$ and $y$, we then get
\begin{equation}\label{auglag_explicit}
    \min_p \mathcal{L}_\varepsilon(p,v;y)
    =
    \hat{r}_v^*
    (M_vM_v^*+\varepsilon I_m)^{-1}
    \hat{r}_v
    -\varepsilon\norm{y}_2^2,
    \qquad
    \hat{r}_v=Av+\varepsilon y.
\end{equation}

Hence, for a fixed $y_k$, we can eliminate $p$ and minimize $\mathcal{L}_\varepsilon(p,v;y_k)$ at the same cost as minimizing $f_\varepsilon(v)$ in~\eqref{fepsv}. The \emph{augmented Lagrangian method} is once again a nested iteration: in every inner step we compute $v_{k+1}$ by minimizing $\mathcal{L}_{\varepsilon_k}(p,v;y_k)$ with an iterative algorithm applied to~\eqref{auglag_explicit}, but at each outer step, in addition to choosing a new $\varepsilon_k$, we update our estimate $y_{k}$ of the Lagrange multiplier with a gradient ascent step. The main appeal of this method is that convergence is driven mainly by the convergence of $y_k$ to the optimal dual variable $y_\star$, and typically one terminates the iteration before $\varepsilon_k$ gets too close to zero, hence avoiding ill-conditioning in $(M_v M_v^*+\varepsilon I_m)^{-1}$. For more details on this application of the method, see~\citet{Oracle}, and for a general treatment of augmented Lagrangian methods, see \citet{Birgin2014}.

\citet{nyman2026_singular_vector_spaces} suggests replacing the inner iteration with alternating minimization to compute the minimum in~\eqref{penaltyform}, since both half-steps of the alternating minimization can be computed in closed form:
\begin{align*}
    p_{j+1} &= \arg\min_p \norm{p}_2^2 + \frac{\norm{r_{v_j} - M_{v_j}p}_2^2}{\varepsilon\norm{v_j}_2^2} = M_{v_j}^*\left( M_{v_j}M_{v_j}^* +\varepsilon\norm{v_j}_2^2 I_m\right)^{-1}r_{v_j},\\
    v_{j+1} &= \arg\min_v  \norm{\Delta_{j+1}}_F^2 + \frac{\norm{(A-\Delta_{j+1})v}_2^2}{\varepsilon\norm{v}_2^2} = v_{\min}(A-\Delta_{j+1}),
\end{align*}
where $v_{\min}(M)$ denotes the (possibly non-unique) right singular vector associated to the smallest singular value of a matrix $M$. The resulting method is particularly efficient for larger-dimensional problems.

An alternative approach suggested by~\citet{gnazzo26_nonlinear} avoids the nested outer-inner iteration entirely: we find the two vectors $u_\star,v_\star$ whose existence is predicted by Corollary~\ref{cor:V82 for singular} (when its assumptions hold!) by solving the nonlinear system of $m+n$ equations in $m+n$ unknowns
\begin{equation} \label{uv_nonlinear}
    \begin{cases}
    (A - \Pi_{\mathcal{V}}(uv^*))^*u = 0;\\
    (A - \Pi_{\mathcal{V}}(uv^*))v = 0.
\end{cases}    
\end{equation}
One needs to impose a suitable normalization condition to eliminate the spurious degree of freedom $uv^* = (u\beta)(v/\overline{\beta})^*$, $\beta\neq0$, but after doing that the nonlinear system~\eqref{uv_nonlinear} can be solved using Newton's method. \citet{gnazzo26_nonlinear} show that the resulting algorithm is very competitive in terms of speed for the case of a sparsity structure.

\subsection{The structured singular value}

In this section, we describe a variant of the structured distance to singularity problem that has received attention in the area of systems and control. Subdivide the dimension $n$ into $k+\ell$ positive integers $n = r_1 + r_2 + \dots + r_k + s_1 + s_2 + \dots + s_\ell$, and define the structure
\begin{equation*}
    \mathcal{V} = \left\{\delta_1 I_{r_1} \oplus \dots 
    \oplus \delta_k I_{r_k} \oplus \Delta_1 \oplus\dots \oplus \Delta_\ell \colon \delta_1,\dots,\delta_k\in\mathbb{C},\, \Delta_j \in \mathbb{C}^{s_j \times s_j}\right\} \subseteq \mathbb{C}^{n\times n},
\end{equation*}
i.e., the space of block diagonal matrices $\Delta$ that have $k$ blocks (of prescribed dimensions) that are multiples of the identity, and $\ell$ blocks (of prescribed dimensions) that are arbitrary matrices. For a given matrix $M\in\mathbb{C}^{n\times n}$, we define the \emph{structured singular value} $\mu_{\mathcal{V}}(M)$ \cite{doyle82} as the inverse of the solution of the minimum problem
\begin{equation} \label{structured singular value definition}
\mu_{\mathcal{V}}(M)^{-1} = \min\, \norm{\Delta}_2 \subjectto \Delta \in \mathcal{V}, \, \text{$I_n-M\Delta $ is singular}.    
\end{equation}
This minimum problem, with its peculiar structure $\mathcal{V}$, arises in the study of robust stability of systems subject to noise or perturbations. If $M$ is invertible, we can rewrite~\eqref{structured singular value definition} as structured distance from singularity:
\begin{align*}
    \mu_{\mathcal{V}}(M)^{-1} &= \min\, \norm{\Delta}_2 \subjectto \Delta \in \mathcal{V}, \, \text{$M^{-1} -\Delta$ is singular} \\
    &= d(M^{-1}, \mathcal{E}_0 \cap (M^{-1} + \mathcal{V})).
\end{align*}
The norm is the spectral norm, so the computational techniques that we have seen in this section cannot be applied. Notably, computing $\mu_{\mathcal{V}}$ is NP-hard \cite{Braatz94nphard,toker1998}: Quadratic programming can be reduced polynomially to computing the structured singular value of a certain $M$. The matrix $M$ appearing in this reduction is singular, so this argument does not immediately imply that structured distance problems in the spectral norm are NP-hard.

Since the exact computation of $\mu_{\mathcal{V}}$ is prohibitive, research in this area has focused on computable bounds. We give a few basic ones here.
\begin{lemma}[\protect{\citet{doyle82}}] \label{lem: easy bound for the structured singular value}
    \[
    \rho(M) \leq \mu_{\mathcal{V}}(M) \leq \sigma_{1}(M).
    \]
\end{lemma}
\begin{proof}
    Let us consider the space of multiples of the identity matrix, 
    $\mathcal{V}_1 = \{\delta I_n \colon \delta\in\mathbb{C}\}$.
    Since $\mathcal{V}_1 \subseteq \mathcal{V}$, we have $\mu_{\mathcal{V}_1}(M) \leq \mu_{\mathcal{V}}(M)$, because the minimum in~\eqref{structured singular value definition} is taken on a smaller set. We can compute directly the structured singular value for this simple structure:
    \begin{align*}
        \mu_{\mathcal{V}_1}(M)^{-1} &= \min \, \norm{\delta I_n}_2 \subjectto \text{$I_n - \delta M$ is singular}    \\
        & = \min \, \abs{\delta} \subjectto \rho(\delta M) = 1\\
        &= \rho(M)^{-1}.
    \end{align*}
    For the upper bound, we consider similarly $\mathcal{V}_2 = \mathbb{C}^{n\times n}$: We have $\mu_{\mathcal{V}}(M) \leq \mu_{\mathcal{V}_2}(M)$ because $\mathcal{V} \subseteq \mathcal{V}_2$, and
    \begin{align*}
        \mu_{\mathcal{V}_2}(M)^{-1} &= \min \norm{\Delta}_2 \subjectto \Delta\in\mathbb{C}^{n\times n}, \, \text{$I_n-M\Delta$ is singular} = \sigma_1(M)^{-1}.
    \end{align*}
    Indeed, if $u_1,v_1$ are the singular vectors associated with $\sigma_1(M)$, then with $\Delta_\star  = v_1 \sigma_1(M)^{-1} u_1^*$ we have $M\Delta_\star  u_1 = u_1$, and for any $\Delta$ with smaller norm the matrix $I_n-M\Delta$ is invertible, since $\norm{M\Delta}_2 \leq \norm{M}_2\norm{\Delta}_2 < 1$.
\end{proof}
For a tighter upper bound, define the space
\[
\mathcal{W} = \left\{\Delta_1 \oplus \dots \oplus \Delta_{k} \oplus \delta_1 I_{s_1} \oplus\dots\oplus\delta_\ell I_{s_\ell} \colon \Delta_i \in \mathbb{C}^{r_i \times r_i}, \, \delta_1,\dots,\delta_\ell\in\mathbb{C}, \right\} \subseteq \mathbb{C}^{n\times n},
\]
i.e., the space of matrices whose block structure is ``reversed'' with respect to that of $\mathcal{V}$: full blocks in place of multiples of the identity, and vice versa. One can see that $\mathcal{W}$ is the space of matrices that commute with $\mathcal{V}$ (centralizer).
\begin{theorem}[\protect{\citet{doyle82}}] It holds that
    \begin{equation} \label{bound for the structured singular value}
        \mu_{\mathcal{V}}(M) \leq \inf\, \norm{DMD^{-1}}_2 \subjectto D \in \mathcal{W},\, \text{$D$ positive definite}.    
    \end{equation}
\end{theorem}
\begin{proof}
    It is sufficient to prove that $\mu_{\mathcal{V}}(M) \leq \norm{DMD^{-1}}_2 = \sigma_1(DMD^{-1})$ for any feasible $D$. Note that $I_n-M\Delta$ is singular if and only if
    \[
    D(I-M\Delta)D^{-1} = I_n - DM\Delta D^{-1} = I_n - DMD^{-1}\Delta
    \]
    is singular; hence $\mu_\mathcal{V}(M) = \mu_{\mathcal{V}}(DMD^{-1}) \leq \sigma_1(DMD^{-1})$, where the last inequality follows from Lemma~\ref{lem: easy bound for the structured singular value}
\end{proof}
The infimum in~\eqref{bound for the structured singular value} can be computed using linear matrix inequality techniques; and the computable bound in this theorem is tight in many practical cases. We refer the reader to~\citet{PackardDoyle} for more details, including algorithms for lower bounds to $\mu_{\mathcal{V}}$, which we have not described here.

\subsection{Several generalizations and applications}

This section focused on $\F$-linear subspaces of $\F^{m \times n}$; however, the theory and algorithms are almost identical if $A\in\C^{m\times n}$ and $\mathcal{V}\subseteq\C^{m\times n}$ is only a \emph{real} subspace; for instance, the space of Hermitian matrices.

The complex constraint $(A-\Delta)v = 0, \Delta\in\mathcal{V}$ becomes the real linear structured constraint $M_v p = r_v$, with
\begin{equation} \label{cveccase}
M_v = \begin{bmatrix}
    \cvec(P_1v) & \cvec(P_2v) & \cdots & \cvec(P_\ell v)
\end{bmatrix}\in\mathbb{R}^{2m\times \ell}, \quad r_v = \cvec(Av) \in \mathbb{R}^{2m},    
\end{equation}
and $\cvec$ as defined in~\eqref{eq:cvec}. Then, the formulas in Theorem~\ref{thm:nearest structured singular, regular case} and~\ref{thm:nearest structured singular, singular case} still hold, up to minor modifications related to interpreting $\C^m \cong \R^{2m}$

The theory behind SLRA extends to the case where one minimizes an arbitrarily weighted 2-norm of $p$:
\[
\min \norm{Wp}^2 \subjectto \mvec(\Delta) = \mathbf{P}p, \rank(A-\Delta)<n,
\]
where $W\in\mathbb{C}^{\ell \times \ell}$ is invertible. This is, in fact, the problem that is treated in~\citet{Usevich2014}. Equivalently, one can set $\hat{p}=Wp$, $\hat{\mathbf{P}}=\mathbf{P}W^{-1}$. Since $\hat{\mathbf{P}}$ no longer has orthonormal columns, the equivalence between~\eqref{eq:Lagrangian formulation with A+Delta} and~\eqref{eq:Lagrangian formulation with Mp} is lost when one replaces $\mathbf{P},p$ with $\hat{\mathbf{P}},\hat{p}$; but one can proceed from~\eqref{eq:Lagrangian formulation with Mp} and formulate the algorithms in the same way. The noteworthy change in the theory is that $\hat{\mathbf{P}}\hat{\mathbf{P}}^*$ is no longer an orthogonal projection, and hence the projection-of-rank-1 property~\eqref{lem:weakrk1} no longer holds.

Another extension of this framework is imposing a rank drop (nullity) of at least $d$. I.e., for a fixed $d>1$, we compute
\[
d(A, \mathcal{R}_{\leq n-d} \cap (A+\mathcal{V}))^2 = \min \norm{\Delta}_F^2 \subjectto \Delta\in\mathcal{V},\, \rank(A-\Delta) \leq n-d,
\]
This problem can be addressed similarly: $A-\Delta$ has nullity at least $d$ if and only if there is $V\in\mathbb{C}^{n\times d}$ with orthonormal columns such that $(A-\Delta)V = 0$; hence we write
\begin{gather*}
    f(V) = \min \norm{\Delta}_F^2 \subjectto \Delta\in\mathcal{V}, \, (A-\Delta)V=0,\\
    d(A, \mathcal{R}_{\leq n-d} \cap (A+\mathcal{V}))^2 = \min_{V\in\mathcal{E}} f(V),
\end{gather*}
where the minimization is over the Stiefel manifold $\mathcal{E} = \mathbf{St}_d(\mathbb{C}^n)$, i.e., the manifold of all $V\in\mathbb{C}^{n\times d}$ with orthonormal columns (or, equivalently for this problem, the Grassmann manifold $\mathbf{Gr}_d(\mathbb{C}^n)$ of all $d$-dimensional subspaces $\range(V) \subseteq \mathbb{C}^n$). The inner minimization problem is once again a least-norm problem on $p$ with linear constraints, while the outer problem can be solved with the help of a package for optimization on manifolds such as Manopt \cite{Bergmann2022, Boumal2014b}. This generalization is the reason why the acronym SLRA includes the initials of ``low-rank'': The algorithm can optimize over fixed-rank constraints on $A-\Delta$. This variant appears in~\citet{Usevich2014}, in a variant where a specific chart of the Grassmann manifold (i.e., the location of an invertible $d\times d$ submatrix of $V$) has to be fixed in advance, and in~\citet{Usevich2014a} in a version in which the algorithm can switch between different charts at need. The regularized version is briefly sketched in~\citet[Section~9]{Oracle}. We also briefly mention Cadzow's algorithm \cite{Cadzow88}, a heuristic approach that alternates projections onto the matrix structure and the low-rank constraint.

\section{Structured spectral nearness}\label{sec:otherstruct}
\subsection{Structured distance from instability}
\label{sec:structured distance from instability}

Another interesting problem comes from combining a spectral constraint with a linear structure: Namely, given a closed region $\Omega \subseteq \mathbb{C}$, we define
\[
\mathcal{E}_{\Omega} = \{X \colon \text{$X$ has an eigenvalue in $\Omega$}\},
\]
and we look for
\[
d^2(A, \mathcal{E}_{\Omega} \cap (A+\mathcal{V})) = \min \norm{\Delta}_F^2 \subjectto \Delta \in \mathcal{V}, \text{$A-\Delta$ has an eigenvalue in $\Omega$}.
\]
As in the unstructured case, we can reduce the problem to the one of distance from singularity:
\begin{gather*}
d^2(A, \mathcal{E}_{\Omega} \cap (A+\mathcal{V})) = \min_{\lambda \in \Omega} f(\lambda),\\
f(\lambda) = \min \norm{\Delta}_F^2 \subjectto \text{$(A-\lambda I_n) - \Delta$ is singular,}\, \Delta\in\mathcal{V}.    
\end{gather*}
Combining this formulation with the results in the previous section, it follows immediately that the minimizer $\Delta$ is the projection on $\mathcal{V}$ of a rank-1 matrix. This time, however, the inner function $f(\lambda)$ that we have to minimize over $\Omega$ (or, better, over its boundary $\partial \Omega$) is more complicated than $\sigma_{\min}(A-\lambda I_n)$. Hence, the idea of combining the methods in the previous section with an outer minimization step on $\lambda$ is less appealing: We would have to compute $f(\lambda)$ in many different points $\lambda$, for a high computational cost.

Instead, we can nest the optimization variables differently:
    \begin{gather}
        d^2(A, \mathcal{E}_{\Omega} \cap (A+\mathcal{V})) = \min_{v \in \mathcal{S}} f_{\Omega}(v), \nonumber\\
        f_{\Omega}(v) = \min_{\lambda \in \Omega} f(v,\lambda) \nonumber\\
        f(v, \lambda) = \min \norm{\Delta}_F^2 \subjectto \Delta \in \mathcal{V}, \, (A-\lambda I_n-\Delta)v = 0. \label{3problems:inner}
    \end{gather}    
Now Theorem~\ref{thm:nearest structured singular, regular case} gives immediately the closed-form solution to~\eqref{3problems:inner}:
\begin{equation} \label{fvlambdaexpr}
    f(v,\lambda) = r_{v,\lambda}^* (M_v M_v^*)^{-1} r_{v,\lambda}, \quad r_{v,\lambda} = (A-\lambda I_n)v,
\end{equation}
with $M_v$ again as in~\eqref{defM}. (We assume for now that $\rank(M_v) = n$: The singular case is discussed below.)

In this form, also the intermediate minimization step on $\lambda$ has a closed form solution.
\begin{theorem} \label{thm:optimal lambda}
    Suppose $\rank(M_v) = n$. Then,
    \begin{equation} \label{fOmegaexpr}
        f_{\Omega}(v) = f\bigl(v, p_{\Omega}(\lambda_0)\bigr), \quad \lambda_0 = \frac{v^*(M_v M_v^*)^{-1} Av}{v^* (M_v M_v^*)^{-1}v}.        
    \end{equation}
    Here, $p_{\Omega}(\lambda_0)$ denotes the complex number closest to $\lambda_0$ inside $\Omega$.
\end{theorem}
\begin{proof}
    The proof (following \citet[Theorem~8.1]{Oracle}) follows from noting that $f(v,\lambda)$ depends on $\lambda$ only through $r_{v,\lambda}$, and it is a simple quadratic function: Indeed, defining the three scalar constants
    \[
        a = v^* (M_v M_v^*)^{-1} v > 0, \quad b = v^* (M_v M_v^*)^{-1}Av,\quad c = v^*A^* (M_v M_v^*)^{-1} Av,
    \]
    we have
    \begin{align*}
        f(v,\lambda) &= v^*(A-\lambda I_n)^*(M_v M_v^*)^{-1}(A-\lambda I_n)v \\
        &= \overline{\lambda}a\lambda - \overline{\lambda}b - \overline{b}\lambda + c\\
        &= a\abs*{\lambda-\frac{b}{a}}^2 + \frac{ac-\abs{b}^2}{a}.
    \end{align*}
    Hence, for a fixed $v$, the value of $f(v,\lambda)$ is an increasing function of the distance between $\lambda$ and $\lambda_0 = \frac{b}{a}$.
\end{proof}
This reduction suggests a simpler algorithm: We can compute numerically with an optimization algorithm $\min_{v\in \mathcal{S}} f_{\Omega}(v)$ using the explicit formula given by~\eqref{fOmegaexpr} and~\eqref{fvlambdaexpr}. No nested minimization is required.

If the projection $p_{\Omega}(\lambda_0)$ is unique (which is always true when $\Omega$ is convex), then
 $\nabla_v f_{\Omega}(v) = \nabla_v f(v,\lambda)$ by Danskin's theorem~\cite[Proposition~B.25]{Bertsekas1999}, i.e., $\lambda$ can be considered as a constant when computing gradients, without the need to differentiate the expression for it. In particular, we obtain an expression for the gradient $\nabla f_\Omega(v)$  by replacing $A$ with $A-\lambda I_n$ in Theorem~\ref{thm:nearest structured singular, regular case}. This simplification is particularly convenient: To implement first-order algorithms, one needs only an expression for the projection $p_{\Omega}$, and not its more complicated derivative.

In the case when $\rank(M_v) < n$, an analogue of Theorem~\ref{thm:optimal lambda} holds for the penalty function~\eqref{fepsv}: It is sufficient to replace $M_v M_v^*$ in~\eqref{fOmegaexpr} with $M_v M_v^* + \varepsilon \norm{v}_2^2 I_n$. Formulas for the gradient are obtained by replacing $A$ with $A-\lambda I$ in Theorem~\ref{thm:nearest structured singular, singular case}. A noteworthy complication appears in the case in which $\mathcal{V}$ is a real subspace instead of a complex one: It is no longer true that $r_{v,\lambda}$ depends linearly on $\lambda$, since one has to split real and imaginary parts in~\eqref{cveccase}. So Theorem~\ref{thm:optimal lambda} no longer applies, and to obtain the optimal $\lambda_\star (v)$ we have to minimize a quadratic function of $[\Re(\lambda),\Im(\lambda)]$. This minimization is not problematic since it is a $2\times 2$ problem, but it is useful to note that in this real-constrained case $\lambda_\star $ is no longer the projection $p_{\Omega}(\lambda_0)$.

Related nearness problems arise in the context of port-Hamiltonian systems. For example, for the system of dissipative differential equations $\dot{x}=(J-R)Qx$, where $J^*=-J$, $R=R^*\succeq0$, and $Q=Q^*\succ0$, \citet{MMS} study the distance to loss of asymptotic stability under structured perturbations. See also \cite{benner2023computation} for further computational development in this direction. A similar question concerns the distance to loss of passivity for systems with inputs and outputs; \citet{MehrmannV20} study this distance, and construct port-Hamiltonian realizations
that maximize it for a given passive transfer function.

\subsection{Structured distances with the ODE approach} \label{subsec:gradientsystem_structured}
The gradient system approach of Section~\ref{subsec:gradientsystem_unstructured} is flexible enough to be adapted to problems with a structure. Let $F(X) : \mathbb{F}^{m\times n} \to \mathbb{R}$ be a function defined on the whole matrix space, for instance the spectral abscissa or radius, and $X \in A+\mathcal{V}$ be an additional structural constraint. Following the setup of Section~\ref{subsec:biobjective}, we replace the structured distance problem
\begin{equation} \label{psistruct}
    \psi(c) = \min \norm{X-A}_F \subjectto F(X) \leq c, \,  X \in A+\mathcal{V}    
\end{equation}
with
\begin{subequations}  \label{phistruct_tutte}
\begin{align}
    \varphi(\varepsilon) &= \min F(X) \subjectto \norm{X-A}_F = \varepsilon,\, X \in A+ \mathcal{V} \label{phistructX}   \\
    &= \min F(A+\varepsilon E) \subjectto \norm{E}_F=1, \, E\in\mathcal{V}. \label{phistruct}
\end{align}
\end{subequations}
An analogue of Theorem~\ref{thm:psiphi} holds (with the same proof) under the assumption that $F$ has no local minima in $\{\norm{X-A}_F\leq r,\, X\in A+\mathcal{V}\}$, so we can recover $\psi$ from $\varphi$. A noteworthy benefit of this replacement is that the feasible region $\mathcal{S}\cap \mathcal{V}$ of~\eqref{phistruct} is always a clean intersection by Lemma~\ref{lem:sphere and cone}, unlike the one $\mathcal{E}\cap (A+\mathcal{V})$ of~\eqref{psistruct}.

To compute $\varphi(\varepsilon)$, we integrate the  system
\begin{equation} \label{gradientsystem_struct}
    \dot{E} = -\Pi_{T_E(\mathcal{S} \cap \mathcal{V})}(G) = -P + \scalar{P,E}E, \quad G = \nabla_X F(A+\varepsilon E), \quad P = \Pi_{\mathcal{V}}(G)
\end{equation}
following a trajectory until we reach a stationary point, hopefully a global minimum. The structure of this algorithm is the same as the one in Section~\ref{subsec:gradientsystem_unstructured}, and we only need to replace the gradient $G$ with $\Pi_{\mathcal{V}}(G)$ in~\eqref{gradient system on the sphere}. 

We have seen in Sections~\ref{subsec:structuredsing} and~\ref{sec:structured distance from instability} that the optimal $\Delta$ (and hence $E$) in that problem has a projection-of-rank-1 property; we generalize this result to a general $F$, assuming regularity.
\begin{lemma} \label{lem:proj rank structure in ODE}
    Assume the minimum in~\eqref{phistructX} is achieved in a regular point $X_\star$ of $F$, which is not a stationary point of $F|_{A+\mathcal{V}}$ (i.e., $\Pi_{\mathcal V}(\nabla F(X_\star)) \neq 0$). Let $r =\rank(\nabla F(X_\star))$. Then, $E_\star$ is the projection on $\mathcal{V}$ of a rank-$r$ matrix.
\end{lemma}
\begin{proof}
The stationarity conditions of~\eqref{phistruct} read
\begin{equation} \label{statstruct}
\nabla F(X_\star) = \mu E_\star + V_\perp \quad \text{for some $\mu\in\mathbb{R}$ and $V_\perp \in \mathcal{V}^\perp$},
\end{equation}
i.e., the gradient is in the normal space of the intersection of the sphere $\{\norm{E}_F=1\}$ with $\mathcal{V}$. Taking projections in~\eqref{statstruct} gives $\Pi_{\mathcal{V}} (\nabla F(X_\star)) = \mu E_\star$. It must be the case that $\mu \neq 0$, otherwise $\Pi_{\mathcal{V}} (\nabla F(X_\star))=0$, contradicting the assumptions. Dividing by $\mu$ we get the thesis.
\end{proof}
We restrict in the following to functions with gradient of rank $r=1$ for simplicity, following~\cite{GLS23}; this covers many useful cases such as the (structured) distance to instability, or to singularity. Taking advantage of this rank structure in computations is more complicated than in the unstructured case in Section~\ref{subsec:gradientsystem_unstructured}. Thanks to Lemma~\ref{lem:proj rank structure in ODE}, $E_\star$ lies in $\mathcal{S} \cap \mathcal{M}$, where 
\[
\mathcal{M} = \{E \colon E = \Pi_{\mathcal{V}}(Y) \text{ for a rank-$1$ matrix $Y$}\} \subseteq \mathcal{V}.
\]
Note that $\mathcal{S} \cap \mathcal{M}$ is \emph{not} the intersection of $\mathcal{V}$ with $\mathcal{S} \cap \mathcal{R}_1$:  The projection $\Pi_{\mathcal{V}}$ preserves neither the rank nor the norm. More importantly, integrating the gradient system $\dot{E} = -\Pi_{T_E(\mathcal{S} \cap \mathcal{M})}(\nabla_X F(A+\varepsilon E))$ produces a stationary point of $F$ in $\mathcal{S} \cap \mathcal{M}$, but this is not necessarily also a stationary point of $F$ in the larger manifold $\mathcal{S} \cap \mathcal{V}$. Even when $\mathcal{S} \cap \mathcal{M}$ is a submanifold (which does not necessarily happen), the gradient system on this submanifold will typically not produce the global minimizer $E_\star$ that we are looking for. 

Instead, \citet{GLS23} consider the ODE over the rank-$1$ manifold $\mathcal{R}_1 = \{Y \in \mathbb{F}^{m\times n}\colon \rank(Y)=1\}$
\begin{gather}
    \dot{Y} = - P + \scalar{P, E} Y, \label{gug_structured_ode} \\ G = \nabla_X F(A+\varepsilon E), \quad P = \Pi_{T_Y \mathcal{R}_1}G, \quad E = \Pi_{\mathcal{V}}(Y). \nonumber
\end{gather}
If we choose the initial value $Y(0)$ in such a way that $\|\Pi_{\mathcal{V}}(Y(0))\|_F=1$, the ODE~\eqref{gug_structured_ode} mimics the structure of~\eqref{gradient system on the sphere}, but it is not a gradient system: Indeed, \citet{GLS23} note that $F(A+\varepsilon E)$ is often, but not always, decreasing on the trajectories. Despite that, this ODE shares several properties with the gradient systems that we have seen earlier:

\begin{itemize}
    \item The right-hand side of~\eqref{gug_structured_ode} is in $T_Y(\mathcal{R}_1)$: Indeed, $P \in T_Y \mathcal{R}_1$ by definition, and $Y \in T_Y \mathcal{R}_1$ because it is the tangent at $t=1$ of the radial curve $f(t) = tY \in \mathcal{R}_1$. Hence the trajectories $Y(t)$ never leave $\mathcal{R}_1$. 
    \item Projecting both sides of~\eqref{gug_structured_ode} onto $\mathcal{V}$ and noting that $\scalar{P, E} = \scalar{P, \Pi_{\mathcal{V}}(Y)} = \scalar{\Pi_{\mathcal{V}}(P), Y}$ (since $\Pi_{\mathcal{V}}$ is self-adjoint), we get
    \[
    \dot{E} = -\Pi_{\mathcal{V}}(P) + \scalar{\Pi_\mathcal{V}(P), E}E = -\Pi_{T_E\mathcal{S}} \Pi_{\mathcal{V}}(P).
    \] from which one sees that $\norm{E}_F^2$ is constant on the trajectories. We conclude the unit sphere is invariant, and our choice of initial
value ensures that $\|E(t)\|_F=1$ along the trajectories.
\end{itemize}
In particular, these two properties imply that $E = \Pi_{\mathcal{V}}(Y) \in \mathcal{S} \cap \mathcal{M}$. Moreover, one can show under mild assumptions that all stationary points for~\eqref{gug_structured_ode} are stationary also for~\eqref{gradientsystem_struct} \cite[Theorem~2.8]{GuglielmiLubichBook}. When the objective function only depends on one eigenvalue of the perturbed matrix,  \citet[Lemma~4.5]{GLS23} show that the difference between the projected and original gradients depends quadratically on the distance to a stationary point with a simple target eigenvalue. Under the assumptions of \citet[Theorem~4.4]{GLS23}, this implies that any trajectory that starts close enough to a local minimum converges exponentially towards it. Using the parametrization $Y = \rho uv^*$, with $\norm{u}=1,\norm{v}=1,\rho>0$, one can reduce~\eqref{gug_structured_ode} to the coupled vector ODEs
    \begin{equation} \label{structured tangent system}
\begin{aligned}
    \dot{u} &= -\tfrac{1}{\rho}(I_m-uu^*)Gv - \tfrac{i}{2\rho}\Im (u^*Gv)u,\\
    \dot{v} &= -\tfrac{1}{\rho}(I_n-vv^*)G^*u - \tfrac{i}{2\rho}\Im (v^*G^*u)v.
\end{aligned}
\end{equation}
with $G = \nabla_X F(A+\varepsilon E)$, $ E = \Pi_{\mathcal{V}}(\rho uv^*)$.
There should be also a third equation for $\dot{\rho}$ in the system, but we can omit it, since we know that $\norm{E}_F=1$ on the trajectories and hence we can recover $\rho$ from the relation $\frac{1}{\rho} = \norm{\Pi_{\mathcal{V}}(uv^*)}_F$. Note that one must assume that $\Pi_{\mathcal{V}}(uv^*)\neq 0$ along the trajectory.

The structure of the algorithm hence is the same one seen in Section~\ref{subsec:gradientsystem_unstructured}: For $k=0,1,2,\dots$
\begin{enumerate}
    \item Starting from suitable $u_{k,0}, v_{k,0},$ integrate~\eqref{structured tangent system} (for $\varepsilon=\varepsilon_k$) up to a stationary point, which is hopefully the global minimum of~\eqref{phistruct_tutte}.
    \item If $\varphi(\varepsilon)$ is sufficiently close to $c$, terminate; otherwise, construct $\varepsilon_{k+1}$ using a suitable strategy (e.g., Newton, bisection, \dots) and continue.
\end{enumerate}
For more details, we refer the reader to~\citet[Section~V.2]{GuglielmiLubichBook}. For certain structures, for instance the real structure~\cite{GugL13-low-rank-real-pseudospectra} or the Hamiltonian structure~\cite{GugKL-Hamiltonian-nearness}, $\Pi_{\mathcal{V}}(uv^*)$ has low rank itself, so an alternative approach is generalizing~\eqref{parametrized tangent system} to integration on a rank-$r$ manifold. Algorithms to estimate the real stability radius are developed by \citet{Guglielmi2015}.

\section{Other nearness problems}\label{sec:othernearness}

\subsection{Distance to uncontrollability}

A classic nearness problem, arising from control theory, concerns pairs of matrices  rather than a single one.
\begin{definition}\label{def:controllable}
    Let $A \in \C^{n \times n}, B \in \C^{n \times m}$. The pair $(A,B)$ is  said to be \emph{controllable} if
    \[ \rank \left( \begin{bmatrix}
        A - \lambda I & B
    \end{bmatrix} \right) = n \ \forall \ \lambda \in \C. \]
\end{definition}
Since it is clear that the rank condition in Definition~\ref{def:controllable} is always attained for $\lambda \not \in \Lambda(A)$, it suffices to check it for at most $n$ values of $\lambda$, corresponding to the eigenvalues of $A$. Uncontrollability means that some left eigenvector of $A$ is orthogonal to the range of $B$. Thus, controllability is a generic property, and robustness of control theory algorithms motivates the computation of the distance of a given matrix pair $(A,B)$ to the nearest \emph{uncontrollable} pair, that is, the distance from the nearest matrix pair that does not satisfy Definition~\ref{def:controllable}; see, for example, \citet{Paige81,wicks91, gu00, gu06, mitchell21} and, for an overview, \cite[Chapter 6]{Datta2004}. We can reduce to a classic matrix nearness problem by identifying the pair $(A,B)$ with the matrix $\displaystyle \begin{bmatrix}
    A & B
\end{bmatrix} \in \C^{n \times (m+n)}$. If we then define
\begin{equation}\label{eq:eunc}
   \mathcal{E}_{\mathrm{unc}} = \left\{ \begin{bmatrix} X & Y \end{bmatrix} \in \C^{n \times (n+m)} : \exists \ \lambda \in \C \mid   \rank \begin{bmatrix} X - \lambda I & Y \end{bmatrix} < n \right\}, 
\end{equation}
this reduces to a matrix nearness problem in the sense of \eqref{eq:matrixdistfrob}. Together with Proposition~\ref{prop:lauriwaswrong}, Proposition~\ref{prop:euncclosed} below guarantees the existence of a nearest uncontrollable pair for every input $(A,B)$. 
\begin{proposition}\label{prop:euncclosed}
    Let $\mathcal{E}_{\mathrm{unc}}$ be the set defined as in \eqref{eq:eunc}. Then, $\mathcal{E}_{\mathrm{unc}}$ is closed.
\end{proposition}
\begin{proof}
    Consider a sequence $(M_k)_k \subset \mathcal{E}_{\mathrm{unc}}$ with $M_k = \begin{bmatrix}
        X_k & Y_k
    \end{bmatrix}$, $X_k \in \C^{n \times n}$, $Y_k \in \C^{n \times m}$, and suppose $M_k \rightarrow M = \begin{bmatrix}
        X & Y
    \end{bmatrix}$ so that $X_k \rightarrow X, \ Y_k \rightarrow Y$. Then, for all $k$, there exists $\lambda_k \in \C$ such that $\begin{bmatrix}
        X_k -\lambda_k I & Y_k
    \end{bmatrix}$ does not have full row rank. Noting that $\lambda_k$ must be an eigenvalue of $X_k$, following the same argument as in the proof of Proposition~\ref{prop:bothclosed} we can construct a convergent subsequence $(\lambda_{k_j})_j \rightarrow \lambda \in \C$. By the lower semicontinuity of the rank, it then follows that
    \[ \rank \left( \begin{bmatrix}
        X-\lambda I & Y
    \end{bmatrix}\right) \leq \liminf_{j \rightarrow \infty} \rank\left(\begin{bmatrix}
        X_{k_j} - \lambda_{k_j} I & Y_{k_j}
    \end{bmatrix}\right)  < n \Rightarrow M \in \mathcal{E}_{\mathrm{unc}}.    \]
\end{proof}

The same reasoning that led to \eqref{eq:distance to instability criteria} yields a two-level formulation also for the distance to uncontrollability problem, with the inner problem solved by the Schmidt-Eckart-Young-Mirsky Theorem (Example~\ref{ex:seym}). To state Theorem~\ref{thm:eising}, which was first proved by \citet{Eising} let us begin by denoting the distance to uncontrollability of the pair $(A,B)$ as
\begin{equation}\label{def:tau}
\tau(A,B) = \min_{\scriptsize \begin{bmatrix}A+\Delta A & B + \Delta B \end{bmatrix} \in \mathcal{E}_{\mathrm{unc}}} \left\| \begin{bmatrix}
    \Delta A & \Delta B 
\end{bmatrix}\right\| .
\end{equation}

\begin{theorem} \label{thm:eising}

    Let $A \in \C^{n\times n}$ and $B \in \C^{n \times m}$, and specialize \eqref{def:tau} to an arbitrary unitarily invariant norm $\| \cdot \|$ such that $\|e_1 e_1^T \|=1$. Then, denoting by $\sigma_k(M)$ the $k$-th singular value of the matrix $M$, it holds

    \begin{equation} \label{eq:eising}
        \tau(A,B) = \min_{\lambda \in \C} \sigma_{n}\left( \begin{bmatrix} A - \lambda I & B \end{bmatrix}\right).
    \end{equation}

    Moreover, if $\lambda_0$ attains the minimum in \eqref{eq:eising} and $u$ (resp. $v$) is a left (resp. right) singular vector of $\begin{bmatrix} A - \lambda_0 I & B \end{bmatrix}$ associated with the singular value $\sigma_n$, then the rank-one perturbation $\begin{bmatrix} \Delta_A & \Delta_B\end{bmatrix} = -\sigma_n u v^*$, partitioned conformably, is a global minimizer.
\end{theorem}

\begin{proof}

    Arguing as in the proof of Theorem~\ref{thm:1dmin},
    \begin{align*}
        \tau(A,B) &= \min_{\lambda \in \C}\, \min \left\{ \left\|{\begin{bmatrix} \Delta_A & \Delta_B \end{bmatrix}}\right\| : \rank\left(\begin{bmatrix} A - \lambda I + \Delta_A & B + \Delta_B \end{bmatrix}\right) < n \right\}\\
       &= \min_{\lambda \in \C} \sigma_{n}\left( \begin{bmatrix} A - \lambda I & B \end{bmatrix}\right),
    \end{align*}
    where the inner minimum and the form of the minimizer directly follow from Example~\ref{ex:seym} applied to the matrix $\begin{bmatrix} A - \lambda I & B\end{bmatrix}$.

\end{proof}

Equation \eqref{eq:eising} is only superficially similar to \eqref{eq:optimum distance from stability}. However, an important difference is that the minimization is now over the whole complex plane, as opposed to the boundary $\partial \Omega$. This reveals a higher complexity. Indeed, at least for a sufficiently regular target set $\Omega$, one can hope to parametrize its boundary by one real variable; on the other hand, the eigenvalue at which controllability is lost can lie anywhere in $\C$, and therefore $\sigma_n(\lambda)$ in \eqref{eq:eising}
must be minimized as a function of the two real variables $\Re(\lambda)$ and $\Im(\lambda)$. Observe also that, by Weyl's inequalities, for all $|\lambda| > \norm{A}_2 + \norm{B}_2$ it holds
\[
\sigma_n(\lambda)^2 = \lambda_{\min}\left( (A-\lambda I_n)(A - \lambda I_n)^* + BB^* \right)
\geq \sigma_n(A-\lambda I_n)^2 \geq (\abs{\lambda} - \norm{A}_2)^2 > \norm{B}_2^2.
\]
Since the perturbation $\displaystyle \begin{bmatrix}
    \Delta_A & \Delta_B
\end{bmatrix} = \begin{bmatrix}
    0 & -B
\end{bmatrix}$ shows that $\tau(A,B)=\min_{\lambda \in \C} \sigma_n(\lambda)
\leq \norm{B}_2$ (recalling that we assume a normalization of the unitarily invariant norm via $\|e_1 e_1^T\|=1$, by Theorem~\ref{thm:eising}, there is no loss of generality in using the spectral norm), we conclude that the minimum in \eqref{eq:eising} is attained in the compact set
$|\lambda| \leq \norm{A}_2 + \norm{B}_2$. However, the real function $\sigma_n(\Re \lambda, \Im \lambda)$ is not generally convex nor smooth, and may possess many local minima; moreover, no analogue of the Hamiltonian eigenvalue test of Theorem~\ref{thm: Byers trick} is known that certifies global optimality via a single structured eigenvalue computation.

Local optimization methods for \eqref{eq:eising} were analyzed in \citet{wicks91}. \citet{gu00} proposed a test that, given $\delta > 0$, verifies at least one of the two inequalities $\tau(A,B) \leq \delta$ and $\tau(A,B) > \delta/2$; Gu's test is based on a very large structured eigenvalue problem. When enhanced with a bisection scheme, this yields an estimate of $\tau(A,B)$ within a factor of $2$ in polynomial time; a trisection variant that computes $\tau(A,B)$ to arbitrary accuracy was proposed in \citet{BLO04}. The dominant cost of the verification step, a dense eigenvalue problem of dimension $O(n^2)$, was later reduced by \citet{gu06} using shift-and-invert iterations, bringing the complexity down to $O(n^4)$ on average and $O(n^5)$ in the worst case. More recently, \citet{mitchell21} proposed interpolation-based \emph{globality certificates}, which, assuming sufficiently accurate interpolation, certify a posteriori that a candidate local minimizer of \eqref{eq:eising} is global, and which perform well in practice also on larger problems.
Earlier algorithms for computing Kreiss constants, based on large eigenvalue problems, were developed in \citet{mitchell20}.

Several variations on this theme have been studied. If $A$ and $B$ are real and the perturbations are constrained to be real, one obtains the \emph{real} distance to uncontrollability, which can be strictly larger than \eqref{def:tau} and is related to real structured singular values; see \citet{Qiu95}. The discussion above transfers verbatim to the distance to unobservability, by simply taking transpose conjugates of all the involved $n \times (m+n)$ matrices.

\subsection{Nearest defective matrix} \label{sec:nearestdef}
Another classic in numerical linear algebra is \emph{Wilkinson's problem}, which consists in finding the nearest defective matrix, i.e., the nearest matrix that has non-semisimple eigenvalues. Note that the set of defective matrices is not closed: For instance, $\begin{bsmallmatrix}
    a & \varepsilon\\
    0 & a
\end{bsmallmatrix}$ is defective for $\varepsilon \neq 0$, but for $\varepsilon=0$ it has merely a double eigenvalue. The closure of the set $\mathcal{D}$ of defective matrices is precisely
\[
\overline{\mathcal{D}} = \{X \colon \text{$X$ has repeated (i.e., non-distinct) eigenvalues}\};
\]
To ensure that the distance problem $d(A,\mathcal{D})$ has a minimum rather than an infimum, we would need to replace it with $d(A,\overline{\mathcal{D}})$, as discussed in Section~\ref{subsec:distsets}. However, we shall see in the following that, for every $A\in \F^{n\times n}$ with $n$ simple distinct eigenvalues, there are minimizers in $\mathcal{D}$ for every unitarily invariant norm, and in particular for the Frobenius norm all such minimizers are in $\mathcal{D}.$ Recalling the notation $\nu=\min\{m,n\}$,we can argue that this situation mimics what happens for the (unstructured) distance to singularity: For any full-rank $A\in\F^{m\times n}$, its nearest rank-deficient matrices in the Frobenius norm all have rank $\nu-1$, but this is not generally true for every other unitarily invariant norm. Indeed, for example, in the spectral norm $0$ is a nearest singular matrix to the identity.

The distance to the nearest defective matrix is related to the conditioning of the eigenvalue and eigenvector problems, since $\mathcal{D}$ is precisely the set for which they are ill-posed. This motivated the initial interest on this problem by Wilkinson and Demmel~\cite{wilkinson1972,Wilkinson1984,Demmel87}. For instance, defining the condition number of the eigenvalue $\lambda\in\Lambda(A)$ as $\kappa(A,\lambda) = \frac{\norm{y}_2\norm{x}_2}{\abs{y^*x}}$, one has (for the distance in $\norm{\cdot}_2$) the inequalities
\[
 \min_{i\neq j}\frac{\abs{\lambda_i-\lambda_j}}{n(\kappa(A,\lambda_i)+\kappa(A,\lambda_j))} \leq d(A, \mathcal{D}) \leq \min_{\lambda\in\Lambda(A)}\frac{\norm{A}}{(\kappa(A,\lambda)^2-1)^{1/2}}
\]
\cite{demmel1990,wilkinson1972}.

Lemma~\ref{lem:orthoeig} gives a useful equivalent condition to check when $X\in\overline{\mathcal{D}}$ that is used by several algorithms. The result is folklore; for a proof, see for instance~\cite{Kalinina2023}.
\begin{lemma} \label{lem:orthoeig}
    An eigenvalue $\lambda$ of a matrix $X\in\C^{n\times n}$ is non-simple (multiple and/or defective) if and only if there exist a left eigenvector $u$ and a right eigenvector $v$ which are orthogonal, i.e., $v^*u=0$.
\end{lemma}

\citet{malyshev1999} gives an algebraic characterization for the solution of Wilkinson's problem. 
\begin{theorem}[\citet{malyshev1999}]\label{thm:malyshev}
    Let $A\in\mathbb{C}^{n\times n}$. Define 
    \[
    \mathcal{E}_{0,0} = \{X\in\C^{n\times n} \colon \text{$X$ has a repeated eigenvalue $0$}\}
    \]
    (so that every matrix in $\mathcal{E}_{0,0}$ has either a multiple semisimple eigenvalue in $0$, or at least one nilpotent Jordan block of size at least $2$).
    
    Then, in the distance induced by the spectral norm $\norm{\cdot}_2$,
    \begin{equation}
        d(A, \mathcal{E}_{0,0}) = \max_{\gamma\geq 0} \sigma_{2n-1}(M(A, \gamma)), \quad \text{where }
    M(A, \gamma) = \begin{bmatrix}
            A  & \gamma I_n\\
            0 & A
        \end{bmatrix}.    
    \label{malyshev00}
    \end{equation}
    In particular, this implies that
    \begin{equation} \label{malyshevformula}
    d(A, \overline{\mathcal{D}}) = \min_{\lambda\in\mathbb{C}} d(A-\lambda I_n, \mathcal{E}_{0,0}) = \min_{\lambda\in\mathbb{C}} \max_{\gamma\geq 0} \sigma_{2n-1}(M(A-\lambda I_n, \gamma)).
    \end{equation}
\end{theorem}
\begin{proof} \emph{(Sketch)}
Proving the inequality $d(A, \mathcal{E}_{0,0}) \geq \max_{\gamma\geq 0} \sigma_{2n-1}(M(A, \gamma))$ in Theorem~\ref{thm:malyshev} is easy: If $A-\Delta$ has a double eigenvalue $0$ or a nilpotent Jordan block in $0$, then for any given $\gamma$ we can use the eigenvectors or the Jordan chain to construct two linearly independent vectors inside $\ker M(A-\Delta, \gamma)$. Hence $M(A-\Delta, \gamma)$ has rank at most $2n-2$, and
\begin{align*}
\norm{\Delta}_2 &= \norm{M(A-\Delta,\gamma) - M(A,\gamma)}_2 \geq d(M(A,\gamma), \mathcal{R}_{\leq 2n-2}) = \sigma_{2n-1}(M(A,\gamma)).    
\end{align*}
The converse inequality is significantly more involved, and a constructive proof requires explicitly finding a matrix $\Delta$ satisfying both $A-\Delta \in \mathcal{E}_{0,0}$ and $\norm{\Delta}_2 = \max_{\gamma\geq 0} \sigma_{2n-1}(M(A,\gamma))$. The construction in~\citet{malyshev1999} splits into several subcases, distinguishing whether $\gamma=0$ or not in the maximum, and whether $\sigma_{2n-1}(M(A,\gamma)) = \sigma_{2n}(M(A,\gamma))$. We have not been able to find a less technical argument, and so we do not include a proof here.
\end{proof}
We may replace the domain $\gamma \geq 0$ with $\gamma\in\mathbb{R}$ with no loss of generality in the statement of Theorem~\ref{thm:malyshev}, because replacing $\gamma$ with $-\gamma$ does not alter the singular values.

Note the similarity of~\eqref{malyshev00} with formulas such as~\eqref{eq:optimum distance from stability} for the distance to instability. Using a language due to~\citet{alam2023}, we may generalize both. Let $C\in\mathbb{C}^{k\times k}$ be given, with $k\leq n$, and define
\[
\mathcal{E}_{\sim C} = \left\{X\in\mathbb{C}^{n\times n} \colon X = V \begin{bmatrix}
    C & *\\
    0 & *
\end{bmatrix}V^{-1} \text{ for some nonsingular $V\in\mathbb{C}^{n\times n}$} \right\},
\]
where $*$ denotes a block whose exact nature is not relevant. In other words, $\mathcal{E}_{\sim C}$ is the set of matrices similar to a block triangular matrix having $C$ in its upper left corner. Every $X\in \mathcal{E}_{\sim C}$ has the same eigenvalues as $C$, counted with multiplicities, plus others. One has~\cite{alam2023}
\begin{equation*} 
    d(A, \mathcal{E}_{\sim C}) \geq \sigma_{nk-k+1}(I_k\otimes A - C^\top \otimes I_n).
\end{equation*}
so maximizing over one or more parameters $\gamma_i$ appearing in $C$ one can obtain bounds for $d(A, \mathcal{E}_{\sim C})$. Unfortunately, in general this bound is only an inequality: A certain rank assumption on the singular vectors must hold in the general case to construct a matrix at that exact distance. Malyshev's result is recovered with $C = \begin{bsmallmatrix}
    0 & 0\\
    \gamma & 0
\end{bsmallmatrix}$; \citet{gracia2005,lippert2005_fixing_two_eigenvalues} studied in more detail the case of a matrix with two distinct eigenvalues $C=\begin{bsmallmatrix}
    \lambda_1 & 0\\
    \gamma & \lambda_2
\end{bsmallmatrix}$, and \citet{Mengi2011} the case of a multiplicity higher than 2.

A more appealing geometric characterization of the solution of Wilkinson's problem based on pseudospectra is due to~\citet{Lippert1999} (in the regular case) and~\citet{Alam2005} (in the general case). To describe it, we start from some observations on the pseudospectra of a matrix $A\in\mathbb{C}^{n\times n}$ with distinct eigenvalues. First-order eigenvalue perturbation theory tells us that for small $\varepsilon>0$ the pseudospectrum $\Lambda_{\varepsilon}(A)$ is composed of $n$ disjoint connected components $\mathcal{R}_{1,\varepsilon},\dots, \mathcal{R}_{n,\varepsilon}$ each resembling (to first order) a small disk of radius $\kappa(A,\lambda_j)\varepsilon$ centered in an eigenvalue $\lambda_j$ of $A$. No additional connected components can appear as we increase $\varepsilon$: Otherwise, $\norm{(A-zI_n)^{-1}}_2 = \max_{\norm{x}_2=\norm{y}_2=1} \abs{x^*(A-zI_n)^{-1}y}$ would have a local maximum outside $\Lambda(A)$, and this would violate the maximum modulus principle, since it is a maximum of absolute values of holomorphic functions in $z$. So, as we increase $\varepsilon$, these $n$ components become larger, and eventually they touch and merge.  We call the smallest $\varepsilon$ for which $\Lambda_\varepsilon(A)$ has fewer than $n$ connected components the \emph{first coalescence level}, and denote it by $\varepsilon_\star$. A point where the closures of two distinct connected components of the open pseudospectrum $\{z:\sigma_{\min}(A-zI_n)<\varepsilon_\star\}$ meet is called a \emph{first coalescence point} $z_\star$. \citet{Alam2005,Alam2011} prove that $\varepsilon_\star = d(A, \mathcal{D})$, in each unitarily invariant norm (recall that the pseudospectra coincide for all such norms, up to normalization by a multiplicative constant). The key result to prove this equality is the following.
\begin{theorem}[\protect{\citet[Theorem~5]{Alam2011}}] \label{thm:Alamortho}
    Let $A\in\mathbb{C}^{n\times n}$ have distinct eigenvalues. If $z_\star$ is a first coalescence point of the pseudospectrum $\Lambda_{\varepsilon_\star}(A)$, then there exist a pair of singular vectors $u_\star,v_\star\in\mathbb{C}^{n}$ associated to $\varepsilon_\star = \sigma_{n}(A-z_\star I_n)$ such that $v_\star^*u_\star = 0$. 
\end{theorem}
\begin{proof}\emph{(Sketch.)}
    The proof is easier when the smallest singular value $\sigma_{n}(A-z_\star I_n)$ is simple. In this case, let $u,v$ be the associated singular vectors (unique up to a phase). By the Schmidt-Eckart-Young-Mirsky Theorem (Example~\ref{ex:seym}), $A-u\varepsilon_\star  v^*$ is the nearest matrix to $A$ that has an eigenvalue $z_\star$; hence $z_\star$ is on the border of the $\varepsilon_\star $-pseudospectrum. Note that $z_\star = x_\star + i y_\star$ must be a saddle point of the bivariate function $f(x,y) = \sigma_{\min}(A-(x+iy)I_n)$: Indeed, it is a ``mountain pass'', since it is the lowest possible maximum point in any path that connects the two eigenvalues $\lambda_i,\lambda_j \in \Lambda(A)$ whose associated connected pseudospectral components $\mathcal{R}_{i,\varepsilon}$ and $\mathcal{R}_{j,\varepsilon}$ coalesce. First-order perturbation theory for simple singular values can be used to prove that the two partial derivatives are $\frac{\partial f}{\partial x}(x_\star,y_\star) = -\Re(v^*u)$ and $\frac{\partial f}{\partial y}(x_\star,y_\star) = -\Im(v^*u)$. Identifying $\C$ with $\R^2$, we can rewrite these two equalities as $\nabla f(z) = -v^*u$. Since in a (regular) saddle point the partial derivatives vanish, $v^*u=0$.

    The case in which $\sigma_{n-1}(A-z_\star I_n) = \sigma_{n}(A- z_\star I_n)$ is more involved, since $f$ may be non-differentiable. We give here only a sketch of the proof in~\cite{Alam2011}. If there exists a sequence $z_k\to z_\star$ with $f(z_k)<\varepsilon_\star$ and $\nabla f(z_k)\to0$, we can select a subsequence such that the corresponding unit singular vectors $u_k,v_k$ converge to $u_\star,v_\star$. Since $\nabla f(z_k)=-v_k^*u_k$, these vectors satisfy $v_\star^*u_\star=0$. If instead no such sequence exists, let $\mathcal{R}_{i,\varepsilon_\star}$ and $\mathcal{R}_{j,\varepsilon_\star}$ denote two connected
components of the open set $\{z:f(z)<\varepsilon_\star\}$ whose closures meet at $z_\star$. The authors show that these components contain open disks whose closures are tangent at $z_\star$. For each of these two connected components, they consider a sequence of points $z_1,z_2, \dots$ in the interior of  $\mathcal{R}_{i,\varepsilon_\star}$ (resp. $\mathcal{R}_{j,\varepsilon_\star}$), and the associated (simple) singular vectors $u_k,v_k$ of $\sigma_n(A-z_k I_n)$. They consider the limit (which exists up to subsequences) $\lim_{k\to\infty} \nabla f(z_k) = -\lim v_k^*u_k$, which is orthogonal to the boundary $\partial \mathcal{R}_{i,\varepsilon_\star}$ (resp. $\partial \mathcal{R}_{j,\varepsilon_\star}$) in $z_\star$. Since the closures of the two disks meet tangentially, these two normals point in opposite directions of $\mathbb{C} \cong \mathbb{R}^2$; hence there is $\mu \in (0,1)$ such that  $\mu \,v^{(i)*}u^{(i)} + (1-\mu)\, v^{(j)*}u^{(j)}=0$. Using this relation, the authors show that for some $c,s$ with $\abs{c}^2+\abs{s}^2=1$ a combination $u_\star = cu^{(i)} + su^{(j)}, v_\star = cv^{(i)} + sv^{(j)}$ satisfies the required property $v_\star^*u_\star=0$.
\end{proof}
The reader might be tempted to see the case of a repeated singular value as an uninteresting non-generic complication, but we note that it actually arises in cases of practical relevance for applications. For example, it always happens when $A$ is normal; see Example~\ref{ex:nearestdefective_diagonal} below. But first, we conclude the geometric characterization of $d(A,\mathcal{D})$.
\begin{theorem}\label{thm:defectiveuin}
    Let $A$ have distinct eigenvalues. For any unitarily invariant norm, the first coalescence level $\varepsilon_\star$ of the pseudospectrum equals the distance $d(A,\mathcal{D})$. There is a minimizer $X$ that belongs to $\mathcal{D}$ (hence it has a defective eigenvalue, rather than merely a multiple semisimple eigenvalue), and $\Delta = A-X$ has rank 1.
\end{theorem}
\begin{proof}
All matrices of the form $A-\Delta$ with $\norm{\Delta} =\varepsilon < \varepsilon_\star$ have distinct eigenvalues: Indeed, by a continuity argument resembling the proofs of Gershgorin's second theorem, $A-\Delta$ has one eigenvalue in each of the regions $\mathcal{R}_{1,\varepsilon},\dots, \mathcal{R}_{n,\varepsilon}$. Hence, $d(A, \mathcal{D}) \geq \varepsilon_\star$. 

By Theorem~\ref{thm:Alamortho}, one can find $u_\star,v_\star,z_\star$ such that $v_\star^*u_\star=0$ and $X_\star = A - u_\star \varepsilon_\star v_\star^*$ has an eigenvalue $z_\star$ with associated left and right eigenvector $u_\star,v_\star$. By Lemma~\ref{lem:orthoeig}, $z_\star$ is non-simple. 

It remains to prove that $z_\star$ is a defective eigenvalue rather than a multiple semisimple one. Since $z_\star$ is on the boundary of a pseudospectrum, it is not an eigenvalue of $A$; hence, $\rank(A-z_\star I_n) = n$. Thus, its rank-1 modification $X_\star-z_\star I_n$ cannot have rank less than $n-1$; i.e., the eigenvalue $z_\star$ has geometric multiplicity exactly $1$.
\end{proof}
Theorem~\ref{thm:defectiveuin} shows that there always exists a defective minimizer. Actually, in any distance induced by a smooth norm, \emph{all} minimizers are defective; we can prove it using the following lemma.
\begin{lemma}
    Let $X\in\C^{n\times n}$, be such that $X \in \overline{\mathcal{D}} \setminus \mathcal{D}$. Then, $N^r_X \overline{\mathcal{D}} = \{0\}$.
\end{lemma}
\begin{proof}
    Let us take an arbitrary element of the normal cone $Z \in N_X^r\overline{\mathcal{D}}$. In the following, we shall show that $Z$ has to be orthogonal to many tangent directions inside $\overline{\mathcal{D}}$, and this forces $Z=0$.

    Since $X\not\in \mathcal{D}$, it is diagonalizable, and we may write $X = VDV^{-1}$, where $D=\diagm(\lambda_1,\lambda_1,\lambda_3,\lambda_4,\dots,\lambda_n)$, with the two leading eigenvalues $\lambda_1=\lambda_2$ equal (and possibly other ones as well). We shall exhibit several matrices $T_{jk}$ such that $D+zT_{jk}\in\overline{\mathcal{D}}$ for all $z\in\C$. Namely:
    \begin{itemize}
        \item The matrix $T_{11}=e_1e_1^T + e_2 e_2^T$, for which $D+zT_{11}$ has a double eigenvalue $\lambda_1+z$.;
        \item The matrix $T_{22} = (e_1-e_2)(e_1+e_2)^T$, for which $D+zT_{22}$ has a leading $2\times 2$ defective matrix;
        \item The matrices $T_{jk} = e_j e_k^T$, for every pair $(j,k)$ with $j,k\in\{1,2,\dots, n\}$ and $(j,k)\neq (1,1),(2,2)$, for which $D+zT_{jk}$ is upper triangular, lower triangular, or diagonal with the entries $\lambda_1=\lambda_2$ unchanged.
    \end{itemize}
    In particular, for each $j,k$, the curve $\gamma_{jk}(t) = V(D \pm tT_{jk})V^{-1}$ is tangent to $X$ along the direction $ \pm VT_{jk}V^{-1}$; hence every element $Z \in N_X^r\overline{\mathcal{D}}$ must satisfy $\scalar{\pm VT_{jk}V^{-1},Z}\leq 0$, and therefore $\scalar{VT_{jk}V^{-1},Z} = 0$. A similar argument on the curve $\gamma_{jk}(t) = V(D \pm itT_{jk})V^{-1}$ (where we have multiplied $t$ by the imaginary unit $i$) shows that $\scalar{iVT_{jk}V^{-1},Z} = 0$ as well.

    We have shown that $Z$ is orthogonal to $2n^2$ matrices that span the whole $\C^{n\times n}$ over the reals; hence we conclude that $Z=0$.
\end{proof}

\begin{corollary} \label{cor:allinD}
    Let $A\in\C^{n\times n}$ have distinct eigenvalues. Then, in every smooth norm (including in particular in the Frobenius norm), any nearest matrix to $A$ in $\overline{\mathcal{D}}$ is defective.
\end{corollary}
\begin{proof}
    The proof is essentially the same as the one of Theorem~\ref{thm:rank exactly n-1}. Let $X \in \overline{\mathcal{D}} \setminus \mathcal{D}$ be a nearest matrix. By Theorem~\ref{thm:reg min in normal}, $F(X) = \norm{X}$ must satisfy
    \[
    -\nabla_X F(X-A) \in N^r_X \overline{\mathcal{D}} = \{0\},
    \]
    and this implies $X-A=0$ by Lemma~\ref{lem:nonzero gradient}. Hence, there is no $A$ such that $X$ is a minimizer of $d(A,\mathcal{D})$, apart from the trivial case $A=X$.
\end{proof}
The smoothness of the norm is once again crucial: indeed, Example~\ref{ex:nearestdefective_diagonal} shows that this result is false for the spectral norm.

\begin{example} \label{ex:nearestdefective_diagonal}
    Let $A = \diagm(15,9,5,3)$. Recall that the $\varepsilon$-pseudospectrum of a normal matrix is the union of the balls of radius $\varepsilon$ centered in each eigenvalue \cite{TrefethenEmbreeBook}. The first coalescence happens for $\varepsilon_\star=1$, when the balls centered in $5$ and $3$ meet (tangentially) in $z_\star=4$. The matrix $A-z_\star I_4 =\diagm(11,5,1,-1)$ has singular value decomposition
    \[
    A-z_\star I_4 = \underbrace{I_4}_{U} \underbrace{\diagm(11,5,1,1)}_{S} \underbrace{\begin{bmatrix}
        e_1 & e_2 & e_3 & -e_4
    \end{bmatrix}^*}_{V^*}.
    \]
    As a pair of singular vectors associated with $\sigma_3=\sigma_4=1$, one can take not only $u_3=e_3,v_3=e_3$ (with $v_3^*u_3=1$) and $u_4=e_4,v_4=-e_4$ (with $v_4^*u_4=-1$), but also a generic combination $u_\star = cu_3 +su_4$, $v_\star = cv_3 +sv_4$, with $\abs{c}^2+\abs{s}^2=1$: Indeed, $(A-z_\star I_4) v_\star = 1 \cdot u_\star$ and $(A-z_\star I_4)^*u_\star = 1 \cdot v_\star$. Theorem~\ref{thm:Alamortho} predicts the existence of a pair of vectors in this form with $v_\star^* u_\star=0$, and indeed we find one with $c=s=\frac{1}{\sqrt{2}}$. With this choice,
    \[
    X = A - u_\star v_\star^* = 15 \oplus 9 \oplus \frac12\begin{bmatrix}
       9 & 1\\
        -1 & 7
    \end{bmatrix},
    \]
    which has a Jordan block of size 2 with eigenvalue $z_\star = 4$.

Applying Theorem~\ref{thm:defectiveuin}, this example shows that the distance from $A$ to the set $\overline{\mathcal{D}}$ of matrices with multiple eigenvalues is $1$ for every distance induced by a unitarily invariant norm. In particular this holds in the spectral norm, and therefore $Y=\diagm(15,9,4,4) \in \overline{\mathcal{D}}\setminus \mathcal{D}$ is a nearest matrix in the spectral distance. Hence, Corollary~\ref{cor:allinD} does not hold for the spectral norm.
\end{example}

We now turn to describing algorithms for Wilkinson's problem. Before doing so, we note that the problem is non-convex, and that none of the iterative methods discussed can, in general, guarantee convergence to a global minimizer. Direct computation of the extrema of~\eqref{malyshevformula} is not appealing, because of the additional nested minimization when compared to~\eqref{eq:optimum distance from stability}. An older algorithm \cite[Section~3.2]{Lippert1999} is based on obtaining a new guess $z_{k+1}$ by averaging the two closest eigenvalues of $A-u_{\min} \sigma_{\min} v_{\min}^*$, where $u_{\min},\sigma_{\min},v_{\min}$ come from an SVD of $A-z_k I_n$. But it turns out that many of the methods that we have seen in the previous sections can be adapted effectively to Wilkinson's problem.

\citet{Akinola2014} use a variant of the implicit determinant method (Theorem~\ref{thm:implicitdeterminant}) to eliminate the variables $u_\star,v_\star$ and compute directly $x_\star, y_\star, \varepsilon_\star$. We describe their approach: Any pair of singular vectors $u,v$ of $A-(x+iy)I_n$, together with its singular value $\varepsilon$, satisfies $K(x,y,\varepsilon)\begin{bsmallmatrix}
    u\\v
\end{bsmallmatrix}=0$, where
\[
K(x,y,\varepsilon) = 
\begin{bmatrix}
    -\varepsilon I_n & A - (x+iy) I_n\\
    (A-(x+iy)I_n)^* & -\varepsilon I_n
\end{bmatrix}.
\]
Assume that the singular value $\varepsilon$ is positive and simple; then, $K(x,y,\varepsilon)$ has rank $2n-1$, and there exists $c\in\mathbb{C}^{2n}$ such that the bordered $(2n+1)\times (2n+1)$ linear system
\[
M(x,y,\varepsilon)
\begin{bmatrix}
    u\\v\\s
\end{bmatrix} = \begin{bmatrix}
    0_n\\0_n\\1
\end{bmatrix}, \quad M(x,y,\varepsilon) = \begin{bmatrix}
    K(x,y,\varepsilon) & c\\
    c^* & 0
\end{bmatrix}
\]
has an invertible $M(x,y,\varepsilon)$. Then, the last entry of its unique solution is
\[
s(x,y,\varepsilon) = \frac{\det K(x,y,\varepsilon)}{\det M(x,y,\varepsilon)}.
\]
Algebraic manipulations show that $s_x = 2\Re(v^*u)$, $s_y= 2\Im(v^*u)$, where $s_x,s_y$ are partial derivatives of $s(x,y,\varepsilon)$. Hence, by Theorem~\ref{thm:Alamortho}, the first coalescence point is a solution of the nonlinear system
\begin{equation} \label{Akisystem}
    s(x,y,\varepsilon) = s_x(x,y,\varepsilon) = s_y(x,y,\varepsilon) = 0,    
\end{equation}
in the 3 real unknowns $x,y,\varepsilon$. A result analogous to Theorem~\ref{thm:implicitdeterminant} shows that the functions $s, s_x, s_y$, as well as all the derivatives required to run Newton's method, can be computed by solving linear systems with the matrix $M(x,y,\varepsilon)$. The system~\eqref{Akisystem} has other solutions besides $x_\star, y_\star, \varepsilon_\star$, hence a good starting point for Newton's method is necessary.

The implicit determinant method is fast, as it requires only on solving linear systems with $M$, but it is problematic to apply in the case where the critical singular value $\varepsilon_\star$ is repeated, since in that case $M(x_\star,y_\star,\varepsilon_\star)$ is singular in the limit. \citet[Section~4]{Alam2011} suggest a similar looking method, which relies on computing singular values rather than solving linear systems, and can be used also in the case where $\sigma_n(A-z_\star I_n)$ is not simple. Motivated by their proof of Theorem~\ref{thm:Alamortho}, they consider another nonlinear system of 3 equations in 3 real variables, $k(x,y,\mu)=0$, with
\[k(x,y,\mu) = \begin{bmatrix}
    \mu \left(\sigma_{n-1}(A-(x+iy)I_n)- \sigma_{n}(A-(x+iy)I_n)\right)\\
    \mu \nabla \sigma_{n-1}(A-(x+iy)I_n) + (1-\mu) \nabla \sigma_{n}(A-(x+iy)I_n)
\end{bmatrix}.\]
When $\mu=0$, this system reduces to $0 = \nabla \sigma_{n}(A-(x+iy)I_n) = -v_n^*u_n$, which is the condition mandated by Theorem~\ref{thm:Alamortho} when $\sigma_n$ is simple. However, with this formulation the numerical method can also be used when the limit point $z_\star = x_\star + y_\star$ has a repeated singular value $\sigma_{n-1}(A-z_\star I_n) = \sigma_{n}(A-z_\star I_n)$: In this case, the second block equation tells us that we can obtain $0$ as a real linear combination of $v_{n-1}^*u_{n-1}$ and $v_n^*u_n$, and (following again the proof of Theorem~\ref{thm:Alamortho}) this implies that $v^*u = 0$ for a suitable pair of singular vectors $cv_{n-1}+sv_{n}$, $cu_{n-1}+su_n$.

\citet[Section~VI.3]{GuglielmiLubichBook} and \citet{Butta2015} suggest an algorithm based on the function
\begin{equation} \label{defectivity_functional}
    F(X) = \min_{\substack{Xv=\lambda v,\, u^*X=\lambda u^*,\\v^*u>0,\, \norm{u}_2=\norm{v}_2=1}} v^*u,
\end{equation}
where the minimum is taken over a pair $u,v$ of left and right eigenvectors associated to the same eigenvalue $\lambda$, normalized so that $\norm{u}_2=\norm{v}_2=1$ and $v^*u\geq 0$. In view of Lemma~\ref{lem:orthoeig}, the set of zeros of $F(X)$ is precisely $\overline{\mathcal{D}}$. Following the approach in Section~\ref{subsec:gradientsystem_unstructured}, the problem is reduced to minimizing $F(X)$ over the sphere $\{X \colon \norm{X-A}_F=\varepsilon\}$. If the minimum appearing in~\eqref{defectivity_functional} is achieved in a unique, simple eigenvalue $\lambda$ (which is not guaranteed even for $\varepsilon < \varepsilon_\star$, since there could be a tie in the value of $v^*u$ between two simple eigenvalues $\lambda_i$ and $\lambda_j$), then $F(X)$ is differentiable, with $\nabla F(X) = (u^*v) (uu^*M^* + M^* vv^*)$, where $M = (X-\lambda I_n)^\#$ is the group inverse of $X-\lambda I_n$ \cite[Lemma~VI.3.1]{GuglielmiLubichBook}. The book authors show that $\nabla F(X)$ is nonzero whenever $x\neq y$ (so once again normal matrices are problematic), and in this case $\nabla F(X)$ has rank 2; so one can set up a rank-2 generalization of the gradient flow equations~\eqref{parametrized tangent system}.

Another algorithm due to~\citet{NNP26} looks directly for a pair $u,v$ satisfying the conditions of Lemma~\ref{lem:orthoeig}, borrowing ideas from Sections~\ref{sec:Un} and~\ref{subsec:structuredsing}. For any two orthonormal vectors $u,v$, we can compute directly  the function $f(u,v)$ defined as the minimum of $\norm{\Delta}_F^2$ subject to the condition that $u$ (resp. $v$) is a left (resp. right) eigenvector of $A-\Delta$ associated to the same eigenvalue $z\in\Lambda(A-\Delta)$. Indeed, complete $[u,v]$ to an orthonormal basis $Q$; then, we can conveniently apply a simultaneous similarity by setting $B = Q^*AQ$ and $R= Q^*\Delta Q$. Denote now with $\mathcal{E}_T$ the set of matrices $X$ such that $e_1$ is a left eigenvector of $X$ and $e_2$ is a right eigenvector of $X$, both associated with the same eigenvalue $z \in \Lambda(X)$. Let 
\[
f(u,v) = \min_{R \ s.t. \ B-R \in \mathcal{E}_T}\norm{R}_F^2,
\]
and observe that
\[
\mathcal{E}_T  = \{ X\in \C^{n\times n} \colon X_{11}=X_{22}, \, X_{12}=0, X_{3:n,2}=0, \,X_{1,3:n}=0,
\}
\]

where $X_{3:n,2}$ denotes the vector composed of the elements of $X$ in rows 3 to $n$ and column $2$, and analogously for $X_{1,3:n}$. In this way, $f(u,v)$ is the squared distance of $B$ from the linear subspace $\mathcal{E}_T$, and it can be computed directly in closed form. Note that this distance can be expressed in terms of $B_{11} = u^*Au, B_{22} = v^*Av$, $B_{12} = u^*Av$, $\norm{B_{3:n,2}} = \norm{(I_n-\Pi_{\range([u,v])})Av}$, $\norm{B_{1,3:n}} = \norm{(I_n-\Pi_{\range([u,v])})A^*u}$, and hence depends on $u,v$ only and not on the (non-unique) choice of $Q$. To compute
\[
d(A,\mathcal{D})^2 = \min_{[u,v]\in \mathbf{St}_2(\C)} f(u,v),
\]
we can use optimization on manifolds. Both this method based on optimization on manifolds and the gradient system method described earlier can be constrained to a linear structure on $\Delta$ \cite{Butta2015,NNP26}.

\subsection{Stratifications}
A generalization of various different problems that we have studied is \emph{stratifications}, a concept originating in differential topology (see, e.g., \citet{whitney1965}) that has been adopted in numerical linear algebra to study perturbation theory \citep{EEKpart1,EEKpart2}; see also \citet{klimenko2014} for a nice introduction to this topic. A \emph{stratification} of a matrix space $\mathcal{M}$ is a decomposition $\mathcal{M} = \bigcup_{\alpha \in I} \mathcal{S}_\alpha$ into disjoint smooth manifolds $\mathcal{S}_\alpha$ (\emph{strata}), usually associated to a canonical form, where one can put a partial order on the manifolds: certain (lower-dimensional, less generic) manifolds lie inside the topological closure of (higher-dimensional, more generic) others: $\mathcal{S}_\alpha \subseteq \overline{\mathcal{S}_\beta}$.
The simplest example is the \emph{rank stratification}
\[
\F^{m\times n} = \{0\} \cup \mathcal{R}_1 \cup \mathcal{R}_2 \cup \dots \cup \mathcal{R}_\nu.
\]
As usual in this paper, $\mathcal{R}_k$ is the set of matrices of rank exactly $k$, and $\nu = \min(m,n)$. The space $\F^{m\times n}$ is the disjoint union of these rank strata, which in this case form not merely a partial order, but a total one: $\mathcal{R}_k \subsetneq \overline{\mathcal{R}_\ell}$ for each $k<\ell$. We can draw the corresponding \emph{stratification graph}
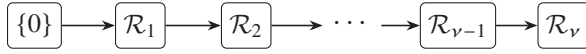
\begin{center}
    \begin{tikzpicture}[
      x=1.4cm, y=1.1cm,
      orb/.style={draw, rounded corners=2pt, inner sep=3pt,
                  minimum height=6mm, font=\small},
      cov/.style={-{Stealth[length=5pt]}, semithick}
    ]
     
    \node[orb] (r0) at (0,0) {$\{0\}$};
    \node[orb] (r1) at (1,0) {$\mathcal{R}_1$};
    \node[orb] (r2) at (2,0) {$\mathcal{R}_2$};
    \node      (dd) at (3,0) {$\cdots$};
    \node[orb] (rm) at (4,0) {$\mathcal{R}_{\nu-1}$};
    \node[orb] (rn) at (5,0) {$\mathcal{R}_{\nu}$};
     
    \draw[cov] (r0) -- (r1);
    \draw[cov] (r1) -- (r2);
    \draw[cov] (r2) -- (dd);
    \draw[cov] (dd) -- (rm);
    \draw[cov] (rm) -- (rn);
     
    \end{tikzpicture}
\end{center}
with arrows that represent the hierarchy from the least generic to the most generic stratum. We can also see these strata as the equivalence classes (\emph{orbits}) under transformations of the form $A \mapsto PAQ^{-1}$, with $(P,Q)\in GL_m(\F)\times GL_n(\F)$: Each class $\mathcal{R}_k$ corresponds to the canonical form $E_k = \begin{bsmallmatrix}
    I_k & 0\\
    0 & 0
\end{bsmallmatrix}$. 

From the point of view of numerical linear algebra, it is interesting to compute the distances of a matrix $A$ to each stratum of a stratification. In the case of the rank stratification, these distances are given by the Schmidt-Eckart-Young-Mirsky Theorem. For example, in the spectral norm and for $0 \leq k < \nu$, $d(A, \mathcal{R}_k) = \sigma_{k+1}(A)$. These distances provide insight on whether a matrix is a perturbation of a less generic one: For instance, if  $\sigma_3(A)$ is large and $\sigma_4(A)$ is small, then we infer that $A$ might have been produced by perturbing a matrix originally in $\mathcal{R}_3$.

The rank stratification serves as the perfect, simple introductory example, because the stratification graph is a line and all distances can be computed exactly; but 
more interesting examples are stratifications related to Jordan structures of matrices, or the Kronecker structure of pencils; and then the distance problems are more challenging. A classical case study is the Jordan stratification of the set of nilpotent matrices $\mathcal{N} = \{A \in \F^{n\times n} \colon A^n = 0\}$. Each stratum is a possible orbit under transformations of the form $A \mapsto PAP^{-1}$, $P\in GL_n(\F)$; i.e., the set of matrices that have the same Jordan canonical form with all eigenvalues equal to zero. For instance, for $n=6$ one of the strata is formed by all matrices with JCF $B = J_4(0) \oplus J_2(0) \in \F^{6\times 6}$. We denote each stratum with the list of its Jordan block sizes, ordered decreasingly (the \emph{Segre characteristic}): For instance, $B$ corresponds to the list $(4,2)$, and the zero matrix in $\F^{6\times 6}$ to $(1,1,1,1,1,1)=(1^6)$. The stratum with list $s = (s_1,s_2,\dots ,s_k)$ is in the closure of the one with list $t = (t_1,t_2,\dots,t_k)$ (where the lists are padded with zeros, if needed, so that they have the same length) if and only if $t$ majorizes $s$, i.e., $\sum_{i=1}^j t_i \geq \sum_{i=1}^j s_i$ for each $j = 1,\dots, k$. For $n\geq 6$, the resulting stratification graph is not a total order; for instance, for $n=6$ we have
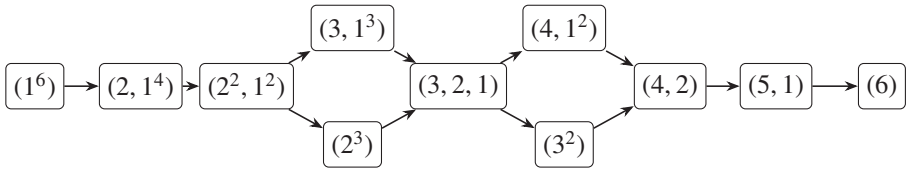
\begin{center}
    \begin{tikzpicture}[
      x=1.4cm, y=1.1cm,
      orb/.style={draw, rounded corners=2pt, inner sep=3pt,
                  minimum height=6mm, font=\small},
      cov/.style={-{Stealth[length=5pt]}, semithick}
    ]
     
    \node[orb] (a)  at (0, 0)    {$(1^6)$};
    \node[orb] (b)  at (1, 0)    {$(2,1^4)$};
    \node[orb] (c)  at (2, 0)    {$(2^2,1^2)$};
    \node[orb] (d1) at (3, 0.7)  {$(3,1^3)$};
    \node[orb] (d2) at (3,-0.7)  {$(2^3)$};
    \node[orb] (e)  at (4, 0)    {$(3,2,1)$};
    \node[orb] (f1) at (5, 0.7)  {$(4,1^2)$};
    \node[orb] (f2) at (5,-0.7)  {$(3^2)$};
    \node[orb] (g)  at (6, 0)    {$(4,2)$};
    \node[orb] (h)  at (7, 0)    {$(5,1)$};
    \node[orb] (i)  at (8, 0)    {$(6)$};
     
    \draw[cov] (a)  -- (b);
    \draw[cov] (b)  -- (c);
    \draw[cov] (c)  -- (d1);
    \draw[cov] (c)  -- (d2);
    \draw[cov] (d1) -- (e);
    \draw[cov] (d2) -- (e);
    \draw[cov] (e)  -- (f1);
    \draw[cov] (e)  -- (f2);
    \draw[cov] (f1) -- (g);
    \draw[cov] (f2) -- (g);
    \draw[cov] (g)  -- (h);
    \draw[cov] (h)  -- (i);
     
    \end{tikzpicture}
\end{center}
The Jordan forms listed in the graph form a stratification of $\mathcal{N}$. 

Given a matrix $A\in\F^{n\times n}$ (not necessarily nilpotent), it is a nontrivial problem to compute its distances to each of these strata. \citet{EJK} suggest two algorithms which deliver respectively a lower and an upper bound. We describe them, starting from the lower bound.
\begin{theorem}[\citet{EEKpart1, EJK}] \label{thm:strata lower bound}
    Let $\mathcal{S} = \{PJP^{-1} \colon P \in GL_n(\F)\}$ be the similarity orbit of a matrix $J\in\F^{n\times n}$, and assume that the dimension over $\F$ of $\mathcal{S}$ as a manifold is $k$. Then, defining $\mathbf{T}_A := I_n \otimes A - A^T \otimes I_n$, we have
    \[
    d(A,\mathcal{S}) \geq \frac{1}{2}\sigma_{k+1}(\mathbf{T}_A)
    \]
    in the spectral norm, and
    \[
    d(A,\mathcal{S}) \geq \frac{1}{\sqrt{2n}} \left(\sum_{j>k} \sigma_j(\mathbf{T}_A)^2 \right)^{1/2}
    \]
    in the Frobenius norm.
\end{theorem}
\begin{proof}
Note first that indeed each stratum $\mathcal{S} = \{PJP^{-1} \colon P \in GL_n(\F)\}$ is a manifold; in particular, its tangent space at $X\in\mathcal{S}$ is $T_X \mathcal{S}=\{XZ-ZX \colon Z \in\F^{n\times n}\}$; this is a standard result in matrix Lie groups. In particular, for every $X\in\mathcal{S}$ we have $k = \dim T_X\mathcal{S} = \rank(\mathbf{T}_X)$, where the last equality follows by vectorizing the map $Z \mapsto XZ-ZX$. Write $A = X + \Delta$, with $X\in\mathcal{S}$. In both norms,
\[
\norm{\mathbf{T}_\Delta} = \norm{\mathbf{T}_{A-X}} =  \norm{\mathbf{T}_A - \mathbf{T}_X} \geq d(\mathbf{T}_A, \mathcal{R}_k),
\]
and the latter distance is given by the Schmidt-Eckart-Young-Mirsky Theorem. We still need to relate the norm of $\mathbf{T}_\Delta$ to that of $\Delta$: For the spectral norm, we do it using the triangle inequality
\[
\norm{\mathbf{T}_\Delta}_2 \leq \norm{I_n\otimes \Delta}_2 + \norm{\Delta^T \otimes I_n}_2 \leq 2\norm{\Delta}_2.
\]
For the Frobenius norm, the same argument would give a spurious additional $\sqrt{2}$ factor, so we argue differently:
\begin{align*}
\norm{\mathbf{T}_\Delta}_F^2 &= \scalar{\mathbf{T}_\Delta,\mathbf{T}_\Delta} = \scalar{I_n\otimes \Delta, I_n\otimes \Delta} -2\scalar{I_n\otimes \Delta,\Delta^T \otimes I_n} + \scalar{\Delta^T \otimes I_n, \Delta^T \otimes I_n}\\
&= \norm{I_n\otimes \Delta}_F^2 -2\Re\traccia(\Delta^*\otimes \Delta^T) + \norm{\Delta^T \otimes I_n}_F^2\\
&= n\norm{\Delta}_F^2 - 2\abs{\traccia(\Delta)}^2 + n\norm{\Delta}_F^2 \leq 2n\norm{\Delta}_F^2.
\end{align*}
\end{proof}
Note that the only information used on $\mathcal{S}$ is its dimension $k$, so Theorem~\ref{thm:strata lower bound} returns the same lower bound for all strata with the same dimension; for instance, $(4,1^2)$ and $(3^2)$.

For the upper bound, we start by describing an algorithm to compute the Jordan block sizes using only unitary transformations. Starting from a nilpotent matrix $A$, compute an orthonormal basis for its kernel $Q_1 \in \F^{n\times w_1}$, and complete it to a unitary matrix $Q = \begin{bsmallmatrix}
    Q_1 & *
\end{bsmallmatrix} \in \F^{n\times n}$. Then, $A$ is unitarily similar to
\[
Q^*AQ = \begin{bmatrix}
    0 & *\\
    0 & A_{2}
\end{bmatrix}.
\]
We then repeat the transformation iteratively on $A_2 \in \F^{(n-w_1)\times (n-w_1)}$; the final result is a matrix unitarily similar to $A$ which is in \emph{staircase form}:
\begin{equation} \label{staircase_form}
    \begin{bmatrix}
        0_{w_1\times w_1} & S_{12} & \dots & S_{1\ell}\\
        0 & 0_{w_2 \times w_2} & \ddots & \vdots\\
        \vdots & \ddots & \ddots & S_{\ell-1,\ell}\\
        0 & \dots & 0 & 0_{w_\ell \times w_\ell}
    \end{bmatrix}.    
\end{equation}
The construction also ensures the blocks $S_{i,i+1}$ in the first superdiagonal have full column rank. The list of obtained kernel dimensions $(w_1,\dots,w_\ell)$ (the \emph{Weyr characteristic}) then determines uniquely the sizes of the Jordan blocks of $A$; indeed, $w_j-w_{j+1}$ is the number of Jordan blocks of size $j$ (having defined $w_{\ell+1}:=0$).

Changing certain steps, the construction in this algorithm can be used to determine, given a starting matrix $A\in\F^{n\times n}$, a nearby matrix $X$ in the closure of the stratum $\mathcal{S}$ associated to a prescribed Weyr characteristic $(w_1,\dots,w_\ell)$. This time, the kernel dimensions $w_1,\dots,w_\ell$ are enforced, rather than computed from $A$. In the first step, given $A$ and a dimension $w_1$, we construct a change of basis $Q$ such that in the block matrix
\[
Q^*AQ = \begin{bmatrix}
    \Delta_{11} & *\\
    \Delta_{21} & A_2
\end{bmatrix}
\]
the norm of the first block column $\begin{bsmallmatrix}
    \Delta_{11}\\
    \Delta_{21}
\end{bsmallmatrix} \in \F^{n\times w_1}$ is as small as possible: We achieve this by taking the first $w_1$ columns of $Q$ to be the last $w_1$ right singular vectors of $A$. Then, we replace $\Delta_{11},\Delta_{21}$ with zeros, obtaining
\[
Q^*X_1Q = \begin{bmatrix}
    0 & *\\
    0 & A_2
\end{bmatrix},
\]
and we continue in this way on $A_2$: We obtain a new matrix $X_2$ with a change of basis built using the last $w_2$ right singular vectors of $A_2$, and so on. The algorithm produces at the end a matrix $X_\ell$ in the staircase form~\eqref{staircase_form} with block sizes $(w_1,\dots,w_\ell)$. This matrix $X_\ell$ is only guaranteed to be in $\overline{\mathcal{S}}$, not necessarily in the stratum $\mathcal{S}$, because the blocks in the first superdiagonal $S_{i,i+1}$ may not have full column rank, and hence $X_\ell$ could have a ``less generic'' Jordan structure. Moreover, these perturbations are chosen greedily: In the first step of the algorithm, we change basis so that the norm of $\begin{bsmallmatrix}
    \Delta_{11}\\
    \Delta_{21}
\end{bsmallmatrix}$ is as small as possible, without looking forward and considering the effect of this change of basis on the next steps. Hence, $\norm{A-X_\ell}$ is only an upper bound for $d(A, \overline{\mathcal{S}}) = d(A,\mathcal{S})$.

See \citet[Table~2]{EJK} for a practical example that shows that the lower and upper bounds that we have described can be quite close.

This \emph{staircase algorithm} dates back to~\citet{kubla}; see also~\citet{ruhe1970,vando1979} for other early works. It can be generalized to many settings, the most famous version being the GUPTRI algorithm~\cite{guptri1, guptri2}, which determines the singular Kronecker structure of a matrix pencil via unitary transformations from the left and the right.

When the eigenvalues are allowed to change as well, one can build stratifications by considering matrices with eigenvalues with higher multiplicities as less generic strata; see, e.g., \citet{klimenko2014}. However,  computing distances becomes more complicated, as the reader can infer from our earlier presentation of the distance to defectivity, which is only a subproblem. 

Stratifications have been studied also for various other spaces and canonical forms: for instance, matrix congruence \cite{straticong}, structured pencils \cite{DK_stratiskew}, matrix polynomials \cite{DJKV_stratipoly}, and
linear dynamical systems \cite{DJK_systems}.

\subsection{Miscellanea}
The line of research of matrix nearness problems is vast, and includes some topics that do not fit well in the categories that we have treated in the previous sections. We conclude our discussion of matrix problems (before having a look at what happens beyond matrices) with a few examples of nearness problems having a different character.

\subsubsection{Reversible Markov chains.}
Given a transition matrix $P$ and a prescribed positive probability vector
$\pi$, one can seek a nearby stochastic matrix satisfying various constraints:
\[
  \min_{X\in\mathbb{R}^{n\times n}}\|X-P\|_F^2
  \quad\text{s.t.}\quad
  X\mathbf{1}=\mathbf{1},\quad X\geq0,\quad
  D_\pi X=X^T D_\pi,
  \qquad D_\pi=\diagm(\pi).
\]
Here the inequality is elementwise, and reducible matrices are allowed. \citet{NW15} reformulate this task as a convex optimization problem.  \citet{durastante2025riemannian} develop a Riemannian approach using
the Fisher metric, while \citet{cipolla2026nearest} study a quadratic
programming formulation with sparsity constraints.
A related problem aims at a reduction in Kemeny's constant, which
measures an average hitting time, while simultaneously minimizing the size of the perturbation.
\citet{durastante2025kemeny} study this double-objective problem.

\subsubsection{Stabilization with nonnegative entries}
A natural structured variant of the nearest weakly Schur stable matrix
problem is
\[
  \min\{\|X-A\|:X\geq0,\ \rho(X)\leq1\},
\]
where $A\geq0$ and $\rho$ denotes the spectral radius. 
Although the feasible set is nonconvex, \citet{Nesterov2020} obtain global
solution methods for the max norm (provably in polynomial time), and the induced matrix
$1$- and $\infty$-norms. Their approach is based on properties of the Perron eigenvalue.

\subsubsection{Matrix compressions}
An application to atomic physics motivates the nearness problem
\[
  \min_{Y\in\mathbb{R}^{n\times p},\,Y^TY=I_p}
      \|Y^TNY-T\|_F^2,
\]
where $N \geq 0$ is a given symmetric positive definite matrix and $T$ is diagonal.
\citet{Borsdorf2014} construct an optimal solution and study related
problems using Riemannian
geometry and an augmented Lagrangian method.

\subsubsection{Parameter optimization}
Some instances of nearness problems can involve a family of matrices depending on just a few parameters, say, $A(\theta)$, where the optimization only involves the vector $\theta \in \R^m$. For example, \citet{Mengi2018} develop various
optimization methods to maximize the distance from instability for matrices depending on a few parameters; these algorithms come with convergence guarantees.

\section{Beyond matrices}\label{sec:beyond}

A common generalization of \eqref{eq:matrixdistfrob} consists of replacing $\F^{m \times n}$ by the vector space of polynomial matrices of degree at most $k$, and taking $\mathcal{E}$ to be a closed subset of this vector space. (Arguably, the case of highest interest in practice is that of \emph{pencils}, where $k=1$.) These are polynomials with matrix coefficients,
\[ P(z) = A_0 + A_1 z + \dots + A_k z^k, \quad A_i \in \F^{m \times n}\  \forall \ 0 \leq i \leq k.  \]
If we consider the norm
\begin{equation}\label{eq:normpoly}
    \| P(z) \|_F := \sqrt{\sum_{i=0}^k\| A_i \|_F^2    },
\end{equation}
it is clear once again that Theorem~\ref{thm:ae} and Corollary~\ref{cor:erdos} apply because this metric vector space is isomorphic to $\F^{(k+1)mn}$ equipped with the Frobenius distance. Hence, with the distance induced by \eqref{eq:normpoly}, it also makes sense to pose the nearness problem \eqref{eq:matrixdistfrob} for matrix pencils, or generally polynomial matrices, thus going beyond single matrices.

Other objects on which nearness problems have appeared include state-space realizations of transfer functions or tensors. It is undeniable that these topics are close neighbors of matrix nearness problems but, to avoid an incontrollable growth of this survey's length, we will not discuss them at the same level of detail as we did for matrices. Instead, in this section we only provide a brief, and necessarily partial, overview of selected results in these adjacent areas.

\subsection{Distance from singularity}

Let us start with the case of polynomial matrices $P(z)=\sum_{i=0}^k A_i z^i$. Recall \cite{Noferini2024} that a square polynomial matrix $P(z)$ is called singular if $\det P(z)$ is the zero polynomial, and regular otherwise.

Possibly, the simplest possible extension to polynomials of the class of spectrally constrained nearness problems is to require that $P(z)$ becomes singular at a prescribed point $\lambda \in \C$; for a regular polynomial, this is tantamount to asking that $\lambda$ is a finite eigenvalue (whereas, if the input is singular, then the distance is $0$ and the problem is trivial). This problem admits a closed-form solution, that we describe in Proposition~\ref{prop:tisseur}, whose proof follows a topos in polynomial eigenvalue perturbation theory; see, e.g., \citet[Theorem 5]{Tisseur00} or \citet[Theorem 1]{NNT15}. For a connection with pseudospectral theory, see \citet{TH01}.
\begin{proposition}\label{prop:tisseur}
    Let $P(z)=\sum_{i=0}^k A_i z^i$ with $A_i \in \C^{n \times n}$ for $i=0,\dots,k$. The squared distance of $P(z)$ from the set of polynomial matrices that are singular at $\lambda$, measured in the distance induced by the Frobenius norm \eqref{eq:normpoly}, is given by
    \[  \| \Delta(z)\|_F^2 =   \frac{\sigma_{\min}^2(P(\lambda))}{\eta(\lambda)},  \]
    where $\sigma_{\min}(X)$ denotes the smallest singular value of the matrix $X$ and $\eta(\lambda):=\sum_{j=0}^k |\lambda|^{2j}$.
\end{proposition}
\begin{proof}
    Let $\Delta(z) = \sum_{i=0}^k E_i z^i$ be such that $P(\lambda)+\Delta(\lambda)$ is singular. Then, using the Schmidt-Eckart-Young-Mirsky Theorem (Example~\ref{ex:seym}) and the Cauchy-Schwarz inequality,
    \[ \sigma_{\min}(P(\lambda)) \leq \| \Delta(\lambda) \|_F \leq \sqrt{\eta(\lambda)} \cdot  \| \Delta(z)\|_F.   \]
    Hence, there cannot be a $\Delta(z)$ having smaller norm. It remains to find a polynomial having precisely that norm; to this goal, let $u,v$ be unit left and right singular vectors of $P(\lambda)$ associated with $\sigma:=\sigma_{\min}(P(\lambda))$, and take \[E_i : = -\frac{\sigma (\overline{\lambda})^i}{\eta(\lambda)} u v^*, \quad i=0,\dots,k.\]
\end{proof}
In contrast with the simple Proposition~\ref{prop:tisseur}, the structured variant of this problem, for example when a linear constraint is imposed on the coefficients of $\Delta(z)$, already grows in complexity \cite{TH01}.

The prescription of one eigenvalue can be made more demanding by moving in at least two different directions. Fixing not just one but several eigenvalues, \citet{KMNT14} study the distance in the spectral norm from a given pencil to pencils with (partially) prescribed spectrum,  by only perturbing the trailing matrix coefficient. Under suitable assumptions, they characterize the distance by optimizing singular values; the analogous problem for higher-degree polynomial matrices is studied by \citet{KM15}. Or one can prescribe an elementary divisor and not just an eigenvalue. \citet{dasbora20} study the distance to matrix polynomials having an elementary divisor $(z-\lambda)^j$ with $j\geq r$, for prescribed $\lambda \in\C$ and $2\leq r\leq nk$. For regular polynomials, this requires a Jordan chain of length at least $r$ at $\lambda$. They characterize the distance through structured low-rank approximation of a block Toeplitz matrix built from derivatives of the polynomial at $\lambda$, obtaining both optimization algorithms and exact bounds.

A stronger requirement is to make the polynomial matrix singular, so that its determinant vanishes at all points. If the maximal degree of a perturbed polynomial matrix $P(z)-\Delta(z)$ is fixed, then the subset of singular polynomial matrices is closed and therefore a least norm perturbation $\Delta(z)$ that makes $P(z)-\Delta(z)$ singular exists. For pencils, the problem of finding the distance to singularity is studied by \citet{Byers1998}, who give several characterizations and bounds. The ODE-based method of \citet{glm17} is applicable both to finding the nearest singular pencil and to finding the nearest pencil having a constant vector in its kernel (right minimal index zero). The two problems are not the same, because for a pencil $A_0 + A_1 z$ to be singular the existence of a shared vector in the null spaces of $A_0$ and $A_1$ is a sufficient, but not necessary condition. This is illustrated by the example
\[ \ker \left( \begin{bmatrix}
    1 & -z & 0\\
    0 & 0 & -z\\
    0&0&1
\end{bmatrix} \right) = \spann \left( \begin{bmatrix}
    z\\
    1\\
    0
\end{bmatrix} \right).   \]
Generally, the least possible degrees in a polynomial basis for the null space of a polynomial matrix are called \emph{minimal indices}; see \citet{Jr.1975, Noferini2024} and the references therein. A method for the nearest singular pencil based on Riemannian optimization over unitary matrices is proposed by \citet{DNN24}. Every pair of matrices can be brought to triangular form by unitary equivalence, that is, $T_0 = Q A_0 Z, T_1 = Q A_1 Z$ for upper triangular $T_0,T_1$ and unitary $Q,Z \in U(n)$. On the other hand, an upper triangular pencil is singular if and only if at least one of its diagonal elements is zero. Hence, a pencil version of \eqref{eq:schurtrick} can be written, leading to the optimization over $(Q,Z) \in U(n) \times U(n)$ of the objective function
\[f(Q,Z)=\sum_{i>j}\left(|(QA_0Z)_{ij}|^2+| (QA_1Z)_{ij}|^2\right)
 +\min_{1\leq k\leq n}
 \left(| (QA_0Z)_{kk}|^2+|(QA_1Z)_{kk}|^2\right).\]
 In \cite{DNN24}, various algorithms to minimize $f(Q,Z)$ are proposed, including smooth alternatives and variants involving
prescribed minimal indices.

For higher-degree polynomials, a useful connection with structured
low-rank approximation comes from polynomial null vectors. If $v(\lambda)=\sum_{j=0}^{\ell}\lambda^j v_j$, then
\[P(\lambda)v(\lambda)\equiv0
 \Leftrightarrow\quad
 \sum_{\substack{i+j=r\\0\leq i\leq k\\0\leq j\leq\ell}}
 A_i v_j=0,
 \quad r=0,\ldots,k+\ell. \]
These equations form a block Toeplitz linear system in
the coefficients $\{v_j\}_{j=0}^{\ell}$.
Both \citet{Das2023} and \citet{Oracle} start from these block Toeplitz convolutions to characterize
distances and derive
optimization formulations and bounds; the resulting algorithms are however different, with \citet{Das2023} exploiting Euclidean optimization and ad hoc formulae and \citet{Oracle} adopting Riemannian optimization and Tikhonov regularization. An alternative approach is to enforce singularity at sufficiently many evaluation
points. Since $\deg\det P\leq nk$, a degree-$k$ $n \times n$ polynomial $P$ is singular if and
only if $P(\lambda_i)$ is singular at $nk+1$ distinct evaluation points $\{\lambda_i\}_{i=0}^{nk}$. 
This approach is developed by \citet{GG25}, who design an algorithm based on sampling the smallest singular value at sufficiently many evaluation points and driving these values to zero through a structured gradient flow. Adding extra structure can sometimes make these problems more tractable. One example arises when studying dissipative Hamiltonian pencils, for which \citet{mehl2021distance} obtain explicit characterizations of the structured distance to singularity. \citet{dmytryshyn2026distance} consider instead the distance to the nearest skew-symmetric polynomial matrix of prescribed normal rank, and tackle it numerically by exploiting a rank-revealing factorization.

\subsection{Approximate polynomial GCD}

A notable application of the structured distance to singularity problem~\eqref{structured distance from singularity} is the approximate polynomial GCD. The problem is as follows: Given a pair of homogeneous bivariate scalar polynomials $\hat{p}(x,z),\hat{q}(x,z)$ (of degrees $d_p,d_q$), find the nearest pair  of polynomials $p(x,z),q(x,z)$ that have a greatest common divisor (GCD) of degree at least $d$. The distance between pairs of polynomials here is measured using the distance induced by the 2-norm of their coefficients 
\[
\norm{(p,q)}^2 = \norm{p}^2 + \norm{q}^2 = \sum_{j=0}^{d_p} \abs{p_j}^2 + \sum_{j=0}^{d_q} \abs{q_j}^2
\]

It can be shown that two homogeneous polynomials $p(x,z),q(x,z)$ (of degrees $d_p,d_q$) have a homogeneous GCD of degree at least $d$ if and only if the \emph{Sylvester matrix}
\begin{equation} \label{sylvesters}
S(p,q) = \left[\begin{array}{cccc|ccccc}
    p_0 & & & & q_0\\
    p_1 & p_0 & & &\vdots & \ddots\\
    \vdots & p_1 & \ddots & & q_{d_q} & \ddots & \ddots \\
    p_{d_p} & \vdots & \ddots & p_0 & & \ddots & \ddots & \ddots\\
    & p_{d_p} & \ddots & p_1& & & \ddots & \ddots& q_0 \\
    & & \ddots & \vdots & & & & \ddots & \vdots\\
    & & & p_{d_p} & & & & & q_{d_q}
\end{array}\right]    
\end{equation}
(where the two blocks have $d_q-d+1$ and $d_p-d+1$ columns, respectively) has a right kernel.  Note that $\norm{S(p,q)}_F^2 = (d_q-d+1)\norm{p}^2 + (d_p-d+1)\norm{q}^2$, hence one needs to rescale suitably the polynomials in $S(p,q)$ so that the target norms of the two problems agree; but once we do that, the approximate GCD problem can be converted into a structured distance from singularity problem: 
\begin{multline} \label{gcd reformulation}
\min \norm{(p-\hat{p}, q-\hat{q})}^2 \subjectto \text{$p,q$ have a GCD of degree at least $d$} \\ = d\left(S\left(\frac{1}{\sqrt{d_q-d+1}} \hat{p}, \frac{1}{\sqrt{d_p-d+1}} \hat{q}\right),\, \mathcal{E}_0 \cap \mathcal{V}\right)^2,    
\end{multline}
where $\mathcal{V}$ is the space of matrices $S(\cdot, \cdot)$ with the structure~\eqref{sylvesters}. 

\begin{remark}
    One can dehomogenize the involved polynomials, recasting the problem as one about univariate polynomials; indeed, this is the most common formulation in the literature. However, in this setting the ``if and only if" characterization discussed above no longer holds, due to the possibility of a shared root at infinity. Consider for example the dehomogenized pair $\hat{p}=2+x$ and $\hat{q}=2-x$. Then the input  Sylvester matrix is $\begin{bsmallmatrix}
        2 & 2\\
        1 & -1
    \end{bsmallmatrix}$, and its (unique) nearest singular Sylvester matrix in the  Frobenius norm is $\begin{bsmallmatrix}
        2&2\\
        0&0
    \end{bsmallmatrix}$. The latter matrix corresponds to $p=q=2$, whose GCD has degree $0$ and not $1$. The issue here is that the homogenized analogues of the output polynomials are $p(x,z)=q(x,z)=2z$, that share the projective root $(x,z)=(1,0)$, and this root blows up to infinity in the affine line.
\end{remark}

The formulation based on the Sylvester matrix plays also a key role in the structured total least norm algorithms of \citet{Kaltofen2007}. Specialized algorithms from polynomial algebra exist for this problem, so it is a good test bench to assess the effectiveness of the algorithms that we have discussed in Section~\ref{sec:structuredsingular}; see also \citet{Nagasaka2021} for a comparison of algorithms for approximate polynomial GCDs. \citet{Usevich2017} and~\citet{Oracle} test SLRA methods on~\eqref{gcd reformulation}, and conclude that unfortunately these methods may have accuracy issues when $d^2(A, \mathcal{E}_0 \cap (A+ \mathcal{V}))$ becomes the order of machine precision or smaller; possibly this is due to the conditioning of the problem reformulation with the Sylvester matrix~\eqref{sylvesters}. Two more effective methods for this problem that do not use~\eqref{sylvesters} appear also in~\citet{Usevich2017} with the names $VP_h$ and  $VP_g$: They still rely on variable projection ideas, but they are no longer based on solving a matrix distance problem.  \citet{fazzi2019} use a variant of the methods that we described in Section~\ref{subsec:gradientsystem_structured}, with target function $F(A) = \sigma_{\min}(A)$. The authors compare to SLRA, and the experiments display the fact that SLRA without the regularization technique in~\eqref{fepsv} may miss the global minimum and converge to a worse local one.

\subsection{Stability, hyperbolicity, and more}

The eigenvalues of a regular matrix pencil $A_0+z A_1$ are usually defined on the
Riemann sphere; an infinite eigenvalue corresponds to 
a zero eigenvalue of the reversed pencil $A_1+z A_0$. If we fix a nonempty closed subset
$\Omega$ of the Riemann sphere, the set of the regular pencils with all eigenvalues in
$\Omega$ is not closed, but one can still compute the distance to this set as well as a nearest pencil in the closure. A Riemannian optimization approach, based on optimization over $U(n) \times U(n)$ is developed by \citet{NN25}. This method is again based on a variant of \eqref{eq:schurtrick}, with the inner minimization solved based on projective projections of pairs of complex numbers. Related nearness problems for stable descriptor systems are
studied by \citet{Gillis2018} for Hurwitz stability and by
\citet{Gillis2020} for Schur stability, the latter with further rank constraints.

Other distance problems arise in the context of Hermitian polynomial matrices.
For example, a quadratic polynomial matrix $P(z)=A_0 + A_1 z + A_2 z^2$ \cite{Tisseur2001} with $A_2 \succ 0$ is called hyperbolic if $x^* P(z) x$ has two distinct real roots for every nonzero vector $x$. Hyperbolicity is equivalent to the inequality
\[      (x^* A_1 x)^2 - 4 (x^* A_0x)(x^* A_2 x) > 0  \quad \forall \ x \neq 0.\]
Analogously,  a Hermitian quadratic polynomial matrix with $A_2 \succ 0$ is called elliptic if  $(x^* A_1 x)^2-4(x^* A_0x)(x^* A_2x)<0$ for every nonzero vector $x$. \citet{HTV02} study 
the distance from a hyperbolic quadratic polynomial to one outside the class, together
with related problems for definite Hermitian pairs and elliptic quadratic
polynomials. \citet{HachezV03} prove more related results, both on the distance to non-hyperbolic and non-elliptic Hermitian polynomials and also on the two reverse nearness problems (distance to hyperbolic polynomials and distance to elliptic polynomials).

State-space realizations provide a further example of a nearness problem. The transfer function of a continuous time realization $(A,B,C,D)$ is defined as the rational matrix
\[
 G(s)=C(sI-A)^{-1}B+D.
\]
In applications, a property called bounded-realness requires that $G$ is analytic on the closed right half-plane and satisfies $G(i\omega)^*G(i\omega)\preceq I$ for every 
$\omega \in \R$. \citet{CGS25} study the distance to bounded-real systems via port-Hamiltonian representations and numerical
optimization.

\subsection{Matrix-valued functions and tensors}

To extend nearness problems beyond polynomial matrices, we can replace the monomials $\{z^j\}_{j=0}^k$ with a finite set $\{f_j(z)\}_{j=0}^k$ of linearly independent entire functions. Thus, a polynomial matrix $P(z)$ is replaced by an analytic matrix-valued function $F(z)=\sum_{j=0}^k A_j f_j(z)$, and perturbations take the form $\Delta(z)=\sum_{j=0}^k E_j f_j(z)$. The norm \eqref{eq:normpoly} extends trivially to this setup. One example of an application comes from time-delay equations, where $F(z)$ takes the form $zI -A_0 - e^{-\tau z} A_1$. \citet{gnazzo25}
develop an approximation method, based on analytic approximation, to compute the distance to singular
matrix-valued functions ($\det F(z) \equiv 0$), including the case where
the coefficient matrices of  the perturbation $\Delta(z)$ must satisfy further linear constraints.  A different variation on this theme is studied by \citet{KKM14}, who look for a constant matrix $E$ of least spectral norm such that $F(z)+E$ has prescribed eigenvalues with prescribed algebraic multiplicities.

Finally, the boldest numerical linear algebraist might dare make the leap from matrices to tensors. Tensor nearness problems provide an extension in which even existence of a best
low-rank approximation can fail. The CP rank of a tensor is the least
number of elementary tensor products needed to express it as a sum.
For order three, the set of tensors of CP rank at most two need not be
closed \citep{SL08}. For instance, let $e_1,e_2$ be the
vectors in the canonical basis of $\R^2$ and write $v^{\otimes 3}=v\otimes v \otimes v$. Then, consider the sequence $(W_k)_{k \geq 1}$ such that
\[
 W_k=
 k\left((e_1+\tfrac1k e_2)^{\otimes3}-e_1^{\otimes3}\right)
 \longrightarrow
 W=e_2\otimes e_1\otimes e_1
  +e_1\otimes e_2\otimes e_1
  +e_1\otimes e_1\otimes e_2.
\]
For all $k \ge 1$, $W_k$ has CP rank at most $2$, but its limit $W$ has CP rank $3$.
Therefore, the distance from $W$ to the set of tensors of CP rank at most $2$ is zero (and there is no nearest such tensor). This behavior is of course very different from that of matrices. To further complicate matters, different notions of tensor rank exist, leading to different results about the existence, or lack thereof, of low-rank tensor approximations. For example, the multilinear rank of a tensor of order $3$ is the triple of matrix ranks obtained by flattening the tensor in three different ways (isolating one index at a time against the other two). In high contrast with the CP rank, a nearest (in the Euclidean distance) tensor always exists among those whose multilinear rank is elementwise bounded above by a triple $(r_1,r_2,r_3)$. For further related notions, we refer to \citet{KB09} for a
survey of tensor decompositions and algorithms.

\section{Conclusions}\label{sec:conc}

Matrix nearness problems combine a simple formulation with a rich variety of theoretical and computational challenges. In this review, we have emphasized a few aspects in particular, such as the geometry of the feasible set, the choice of the distance, and the effect of imposing additional structure on admissible perturbations. We have also examined the interplay between the properties of these problems and the features of the algorithms used to solve them.

Among the (in our opinion) most useful ideas that we hope to have illustrated are matrix decompositions as a source of useful parametrizations, variable projection as a way to reduce dimensionality, and the use of geometric tools such as normal spaces and orthogonal projections. We have also pointed out several subtleties, particularly concerning existence, regularity, and the distinction between local and global optimality. Along the way, we have also filled a few gaps in the literature (see page~\pageref{intro:novelties}), and we are especially proud of this outcome.

Even before starting this review, we knew that the literature on matrix nearness problems was vast, but only while exploring did we realize how many interesting papers we had not encountered before. Unlike mathematics itself, mathematical papers must unfortunately be finite, and therefore we have had to be selective. We apologize to colleagues whose relevant contributions to nearness problems we have overlooked or have been unable to include.

We hope that the mathematical and algorithmic connections presented here will provide a useful guide for future research on matrix nearness problems. We had a lot of fun exploring the literature and writing this article, and we hope that our readers will share some of that enjoyment.

\section*{AI Statement}

The writing, ideas, and (where they differ from the literature) mathematical statements and proofs in this survey paper are human. We acknowledge use of ChatGPT 6.0 and Claude Opus 4.8--5.5 for literature review, generating the TikZ code of figures, and auditing the proofs and the exposition. AI found some minor mistakes, that were then fixed by the human authors.

\section*{Acknowledgements}

 We are grateful to several colleagues for reading a preliminary version of this manuscript and sharing insightful comments, including Nicolas Gillis, Miryam Gnazzo, Nicola Guglielmi, Christian Lubich, Emre Mengi, Lauri Nyman, Michael Overton, Stefano Sicilia, Nick Trefethen, and Paul Van Dooren. We also thank Alice Cortinovis for useful discussions on the literature on randomized algorithms for least-squares problems.

\clearpage 
\addcontentsline{toc}{section}{References}
\bibliography{biblio}
\label{lastpage}

\end{document}